\documentclass{amsart}
\usepackage[utf8]{inputenc}
\usepackage[T1]{fontenc}
\usepackage{amsmath,amssymb}
\usepackage{graphicx}
\usepackage{amsthm, enumerate}
\theoremstyle{remark}
\usepackage[colorlinks=true]{hyperref}
\usepackage{latexsym}
\usepackage{xypic}
\usepackage{pxfonts,kpfonts}
\usepackage{mathrsfs}
\usepackage{fullpage}
\usepackage{wasysym}
\usepackage{enumitem}
\usepackage[affil-it]{authblk}
\usepackage{bbold}
\usepackage[all]{xy}
\usepackage{amssymb}
\usepackage{microtype}
\usepackage[colorinlistoftodos]{todonotes}

\usepackage{color}

\theoremstyle{definition}
\newtheorem{defi}{Definition}[section]
\newtheorem{const}[defi]{Construction}
\newtheorem{expl}[defi]{Example}

\newtheorem{rmk}[defi]{Remark}

\theoremstyle{plain}

\newtheorem*{main theorem}{Main Results}
\newtheorem{pro}[defi]{Proposition}
\newtheorem{thm}[defi]{Theorem}

\newtheorem{cor}[defi]{Corollary}

\newtheorem{lmm}[defi]{Lemma}

\theoremstyle{remark}

\newcommand{\NN}{{\mathbb N}}

\newcommand{\Operad}{{\mathrm{Operad}}}
\newcommand{\Bimod}{{\mathrm{Bimod}}}
\newcommand{\Ibimod}{{\mathrm{Ibimod}}}

\newcommand{\TT}{T}

\newcommand{\calB}{{\mathcal B}}

\usepackage{titlesec}

\titleformat{\section}
{\Large\bfseries\center}
{\thesection}{1em}{}

\titleformat{\subsection}
{\large\bfseries}
{\thesubsection}{1em}{}

\titleformat{\subsubsection}
{\large\bfseries}
{\thesubsubsection}{1em}{}

\usepackage{titletoc}

\titlecontents{section}
[2.8em] % ie, 1.5em (chapter) + 2.3em
{}
{\contentslabel{2.3em}}
{\hspace*{-2.3em}}
{\titlerule*[1pc]{.}\contentspage}

\titlecontents{subsection}
[5.8em] % ie, 1.5em (chapter) + 2.3em
{\contentsmargin{1.5em}}
{\contentslabel{2.3em}}
{\hspace*{-1.3em}}
{\titlerule*[1pc]{.}\contentspage}

\begin{document}

\title{Projective and Reedy model category structures for (infinitesimal) bimodules over an operad}
%\author{J.Ducoulombier, B.Fresse and V.Turchin}

\author{Julien Ducoulombier}
\address{Max Planck Institute for Mathematics  \\
Vivatsgasse 7 \\
53111 Bonn, Germany}
\email{julien@mpim-bonn.mpg.de}

\author{Benoit Fresse}
\address{Univ. Lille, CNRS, UMR 8524\\
Laboratoire Painlev\' e\\
Cit\' e Scientifique - B\^ atiment M2\\
F-59000 Lille, France}
\email{Benoit.Fresse@math.univ-lille1.fr}

\author{Victor Turchin}
\address{Department of Mathematics\\
Kansas State University\\
138 Cardwell Hall\\
Manhattan, KS 66506, USA}
\email{turchin@ksu.edu}

\date{}

\thanks{J.D. was partially supported by the NCCR SwissMAP funded by the Swiss National Science Foundation and the ERC starting grant GRAPHCPX and the Max Planck Institute for Mathematics in Bonn.  B.F. acknowledges support from the Labex CEMPI (ANR-11-LABX-0007-01) and from the FNS-ANR project OCHoTop (ANR-18CE93-0002-01). V.T. has benefited from a visiting position of the Labex CEMPI (ANR-11-LABX-0007-01) at the Universit\'e de Lille and a visiting position at the Max Planck Institute for Mathematics in Bonn for the achievement of this work. He  was also partially supported by the Simons Foundation  grant, award ID: 519474. J.D. and V.T. acknowledge the University of Lille for hospitality.}

%\sloppy

%\begin{document}

\maketitle %\vspace{-10pt}

\begin{center} \begin{Large} Julien Ducoulombier \hspace{20pt} Benoit Fresse \hspace{20pt} Victor Turchin \end{Large}\end{center} %\vspace{-10pt}

\begin{abstract}
We construct and study projective and Reedy model category structures for bimodules and infinitesimal bimodules over  topological operads. Both model structures produce the same homotopy categories. For the model categories in question, we build explicit cofibrant and fibrant replacements. We show that these categories are right proper and  under some conditions left proper. We also study the extension/restriction adjunctions. %We give  a characterization of Reedy cofibrations and we check that both model structures produce the same homotopy categories.% In the bimodules and infinitesimal bimodules cases, the Reedy and projective homotopy categories are the same.
\end{abstract}

\tableofcontents

\section*{Introduction}

In this paper, we set up a homotopy theory for the categories of bimodules and of infinitesimal bimodules over topological operads. More precisely, we  study two model structures, the \textit{projective} model 
structure and the Reedy model structure, which we define for both bimodule categories. The Reedy and the projective model structures have the same class of weak equivalences and, therefore, produce isomorphic homotopy categories. Both model categories find important applications in the manifold functor calculus, 
%structure and the \textit{Reedy} model structure, which we associate to such categories. Both model structures have the same class of weak equivalences
%and, therefore, produce isomorphic homotopy categories. Both of them find important applications in the manifold functor calculus, 
 specifically in the problems of delooping the functor
calculus towers~\cite{DT,Duc3,DTW}. It is well known that the arity zero elements  essentially complicate the homotopy theory
of such objects. However, in practical examples the arity zero component of the studied objects is often reduced to a point.                                                                   Such objects are called {\it reduced}. Motivated by the homotopy theory of the little 2-discs operad, the second author developed the Reedy model structure for reduced
operads~\cite{Fre1,Fre2}. %The purpose of this work is to
  We adapt this theory to the setting of bimodules and of infinitesimal bimodules.
One of the advantages of the Reedy model structure in comparison to the projective one is that the cofibrant resolutions are smaller as they do not take into account
the arity zero component. This makes the constructions of delooping in~\cite{DT,Duc3,DTW} simpler. By contrast,  while all the objects are fibrant for the projective model categories, there is no obvious Reedy fibrant coresolution. Consequently, both model structures have their advantages and it can be convenient to be able to switch from one structure to another.

%Reedy model structures enjoy better homotopy behavior. For example in~\cite{DT}, it is crucial that the  model
%category of bimodules is relative left proper. We establish this result in this paper.  the Reedy case in this paper.
%(In the case of the projective model category, we only prove the left properness for bimodules over operads with a void component in arity zero.)

The starting idea of the Reedy model structure for reduced operads is to encode the operadic composition operations
with the unique point in arity zero in an extension of the diagram structure
which underlies our objects.
In the usual category of symmetric operads, the diagram structure of the
 objects
is governed by the category $\Sigma = \coprod_n\Sigma_n$
which is defined by taking the disjoint union of the symmetric groups $\Sigma_n$.
In what follows, we use the expression `$\Sigma$-sequence'
for the objects of the category of diagrams
over $\Sigma$. We use the notation $\Sigma Seq$ for this category of diagrams and the notation $\Sigma\Operad$ for the category of symmetric operads.
To formalize the construction of the Reedy model structure, we consider the category $\Lambda$,
which has the finite sets $[n] = \{1,\dots,n\}$
as objects and all injective maps of finite sets $u: \{1,\dots,m\}\hookrightarrow\{1,\dots,n\}$
as morphisms.
We use the expression `$\Lambda$-sequence' for the objects of the category of contravariant diagrams
over $\Lambda$,
and we use the notation $\Lambda Seq$ for the category of $\Lambda$-sequences.
The composition operations with the arity zero term $P(0) = *$ in a reduced operad $P$
are equivalent to restriction operators $u^*: P(n)\rightarrow P(m)$,
which can be associated to the injective maps of finite sets $u: \{1,\dots,m\}\hookrightarrow\{1,\dots,n\}$
and hence to the morphisms in the category $\Lambda$.
This observation implies that the category of reduced operads is identified with a category $\Lambda_{\ast}\Operad$,
whose objects are operads shaped on this category of finite sets and injections $\Lambda$
instead of the category of permutations $\Sigma$.

The category of $\Lambda$-sequences inherits a Reedy model structure, in which the fibrations are defined by using a natural
notion of matching object.
The Reedy model structure of reduced operads is precisely defined by transferring this Reedy model structure
on the category of $\Lambda$-sequences $\Lambda Seq$
to our category of operads $\Lambda_{\ast}\Operad$,
while the projective model structure of symmetric operads is defined by transferring the projective model structure
on the category of $\Sigma$-sequences
$\Sigma Seq$ to $\Sigma\Operad$.

Throughout this paper, we work in the category of topological spaces, and we therefore deal with operads in topological spaces. In that context, the projective model category of symmetric operads is known to be left proper relative to $\Sigma$-cofibrant operads (i.e. operads that are cofibrant as $\Sigma$-sequences) and right proper \cite{HRY} making the homotopy colimits and limits easier to identify in this category.
Furthermore, all operads are fibrant in the projective model category of symmetric operads
in topological spaces.
In the Reedy model category of reduced operads, the objects are not necessarily fibrant. We use a notion of matching object
to define the class of fibrations (and unfortunately, we have no explicit definition of a fibrant coresolution functor at the time),
but the class of cofibrations is larger.
In fact,  a morphism of reduced operads is a cofibration with respect to the Reedy model structure
if and only if this morphism defines a projective cofibration of operads
after forgetting the arity zero components~\cite[Theorem~8.4.12]{Fre2} (thus if and only if this morphism defines a cofibration in the projective model category of operads
with a void component in arity zero).
This result implies that the Reedy model category of reduced operads is also left proper relative to $\Sigma$-cofibrant reduced operads.
Together with Willwacher, the second and third authors~\cite{FTW} showed that, for any reduced operads $P$ and $Q$,
there is a weak equivalence of derived mapping spaces
$$
\Sigma\Operad^{h}(P\,;\,Q)\simeq\Lambda_{\ast}\Operad^{h}(P\,;\,Q).
$$
Hence we can use the Reedy model category to compute mapping spaces in the usual category of topological operads.

The main purpose of this work is to extend the results of these operadic homotopy theories to the setting
of bimodules and of infinitesimal bimodules. First of all, we define our counterparts of the projective and Reedy model structures for bimodules and infinitesimal bimodules in topological spaces. We also address the definition of these model structures for truncated bimodules.

For the definition of the projective model structure, we work out difficulties that occur in the context of topological spaces, notably regarding the application of the small object argument (see the discussion of~\cite{Hov2}). This question is independent from other works on the projective model categories of modules over operads carried out in the litterature. In fact, the projective model structure of bimodules was defined in~\cite{Re1} for bimodules in simplicial sets and for bimodules in a category of simplicial bimodules over a ring. The paper~\cite{Ha} gives the definition of an analogous model structure for left modules over non-symmetric operads and for left modules over symmetric operads when every $\Sigma$-sequence is projectively cofibrant in the base category (for instance, when the base category is a category of chain complexes over a characteristic zero field). The book~\cite{Fre3}, by the second author, provides a general study of the homotopy theory of modules and bimodules over operads, but deals with semi-model structures (with a restriction of the application of the axioms to maps with a cofibrant source) to get results that are valid in any base monoidal model category. In the paper, we prove that, when we work in the category of topological spaces, we have a full validity of the definition of the projective model category of bimodules over a pair of operads, and this result holds without any assumption on our operads. We get the same result for the definition of the projective model category of infinitesimal bimodules.

For the definition of the Reedy model categories of bimodules and of infinitesimal bimodules, we rely on a preliminary definition of a fibrant coresolution functor and we apply a transfer argument, using an adjunction between (infinitesimal) bimodules and 
$\Lambda$-sequences. We just need a mild assumption on our operads to ensure the validity of the definition of the Reedy model structures (technically, we just need to consider well-pointed operads,  in which the inclusion of the operadic unit in arity one defines a cofibration of spaces).

%
%The main purpose of this work is to extend the results of these operadic homotopy theories to the setting of bimodules and of infinitesimal bimodules. We define an analogue of the projective and Reedy model structures in both bimodule categories. 
%In general, 

We use the notation $\Sigma \Bimod_{P\,;\,Q}$ for the category of bimodules associated to a pair of operads $(P,Q)$, while we adopt the notation $\Sigma \Ibimod_{O}$ for the category of infinitesimal bimodules over an operad $O$. To distinguish the Reedy model structure from the projective model structure, we adopt the convention to keep these notations $\Sigma \Bimod_{P\,;\,Q}$ and $\Sigma \Ibimod_{O}$ when we equip these bimodule categories with the projective model structure, and we pass to the notations $\Lambda \Bimod_{P\,;\,Q}$ and $\Lambda \Ibimod_{O}$ when we consider the Reedy model structure.
We prove that our model categories have the following features.
% We establish the following results.
%\vspace{5pt}

%\begin{main theorem}
%Under some  assumptions on the operads:
\begin{itemize}[leftmargin=10pt, itemsep=5pt]
\item[$\blacktriangleright$] Sections \ref{Z2} and \ref{ZZ2}: All the objects in $\Sigma \Bimod_{P\,;\,Q}$ and $\Sigma \Ibimod_{O}$ are fibrant. Furthermore, we give explicit fibrant coresolutions in the Reedy model categories $\Lambda \Bimod_{P\,;\,Q}$ and $\Lambda \Ibimod_{O}$. %\vspace{5pt}
\item[$\blacktriangleright$] Sections \ref{C0}, \ref{CH0}, \ref{CP0} and \ref{Fin2}: The categories $\Sigma \Bimod_{P\,;\,Q}$,
$\Lambda\Bimod_{P\,;\,Q}$,  $\Sigma \Ibimod_{O}$ and $\Lambda\Ibimod_{O}$ are right proper.
Moreover, if $P$ is either projectively or Reedy cofibrant, and $Q$, $O$ are componentwise cofibrant,  then the categories 
%Moreover, the categories
   $\Sigma \Bimod_{P\,;\,Q}$ and
$\Lambda\Bimod_{P\,;\,Q}$ are left proper relative to componentwise 
cofibrant objects while $\Sigma \Ibimod_{O}$ and $\Lambda\Ibimod_{O}$ are left proper.% \vspace{5pt}
\item[$\blacktriangleright$] Sections \ref{Fin1} and \ref{Fin2}: Let $Q_{>0}$ and $O_{>0}$ be the sub-operads obtained from $Q$ and $O$, respectively, by removing the arity zero components. Any map in $\Lambda \Bimod_{P\,;\,Q}$ and $\Lambda \Ibimod_{O}$ is a cofibration if and only if the corresponding map in $\Sigma \Bimod_{P\,;\,Q_{>0}}$ and $\Sigma \Ibimod_{O_{>0}}$, respectively, is a cofibration.% \vspace{5pt}
\item[$\blacktriangleright$] Sections \ref{sss:adj_bim} and \ref{E7_2}: If $M$ and $N$ are $(P\text{-}Q)$-bimodules, while $M'$ and $N'$ are $O$ infinitesimal bimodules, then one has weak equivalences between derived mapping spaces
$$
\Sigma \Bimod_{P\,;\,Q}^{h}(M\,;\,N)\simeq \Lambda \Bimod_{P\,;\,Q}^{h}(M\,;\,N) \hspace{15pt}\text{and} \hspace{15pt} \Sigma \Ibimod_{O}^{h}(M'\,;\,N')\simeq \Lambda \Ibimod_{O}^{h}(M'\,;\,N').
$$

\item[$\blacktriangleright$] Sections \ref{E8}, \ref{CH0}, \ref{EE8} and \ref{Fin2}: Let $\phi_{1}:P\rightarrow P'$, $\phi_{2}:Q\rightarrow Q'$ and $\phi:O\rightarrow O'$ be weak equivalences between $\Sigma$-cofibrant operads $P$, $P'$ and
componentwise cofibrant  operads $Q$, $Q'$, $O$, $O'$, then the extension and restriction functors form Quillen equivalences
$$
\begin{array}{cc}%\vspace{6pt}
\phi_{!}:\Sigma \Bimod_{P\,;\,Q}\rightleftarrows \Sigma \Bimod_{P'\,;\,Q'}:\phi^{*},\hspace{5pt} & \phi_{!}:\Sigma \Ibimod_{O}\rightleftarrows \Sigma \Ibimod_{O'}:\phi^{*}, \\
\phi_{!}:\Lambda \Bimod_{P\,;\,Q}\rightleftarrows \Lambda \Bimod_{P'\,;\,Q'}:\phi^{*},\hspace{5pt} & \phi_{!}:\Lambda \Ibimod_{O}\rightleftarrows \Lambda \Ibimod_{O'}:\phi^{*}.
\end{array}
$$
(Also, see  Sections~\ref{E8} and~\ref{sss:reduced_main}  for a refinement of this result when we forget about the arity zero components of bimodules or, respectively, if such components are reduced to a point.)
\end{itemize}
%\end{main theorem}

%\vspace{5pt}
\noindent \textbf{Organization of the paper:} We review the background of our constructions in the first section of the paper. We review the definition of the projective model category of $\Sigma$-sequences and the definition of the Reedy model category of $\Lambda$-sequences. We also build an explicit fibrant coresolution in the category of $\Lambda$-sequences.  Then we recall the definition of the projective model category of operads and the Reedy model category of reduced operads together with their properties. Most of the new results in this section appear at the very end in Subsection~\ref{ss:useful}, where we study some natural properties of 
cofibrant operads. 
%Most of the results in this section are already well known. 

In the second section, we define the projective model category of $(P\text{-}Q)$-bimodules. First, we show that this category is equivalent to the category of algebras over a colored operad. Then we give combinatorial descriptions of the free bimodule functor and of pushouts. After that, we define the projective model category structure for bimodules and we prove that this model structure is relatively left proper and that the extension/restriction adjunctions along weak equivalences of operads form Quillen equivalences. 

In the third section, we study the Reedy model category of bimodules. We check that (almost all) the constructions introduced in the second section can be extended to the Reedy model category. Furthermore, we give an explicit Reedy fibrant coresolution as well as a characterization of cofibrations. As a consequence of this characterization, we show that this model category is relatively left proper and that the extension/restriction adjunctions along weak equivalences of reduced operads form  Quillen equivalences. Both model structures having the same set of weak equivalences, they produce the same homotopy category. Then we construct a functorial cofibrant resolution for bimodules in both (projective and Reedy) model structures. As an application, we explain how our Reedy fibrant coresolution can be expressed in terms of internal hom in the category of $\Sigma$-sequences. In the last subsection, 
assuming  that both operads $P$ and $Q$ are reduced, we study the subcategory $\Lambda_*\Bimod_{P\,;\,Q}$ of reduced bimodules  equipped with the Reedy model category structure. This subcategory enjoys slightly  better properties as we compare the Reedy model structures of bimodules and of reduced bimodules.

In the fourth section we adapt the results from the second section to the context of infinitesimal bimodules. In that case, the proofs are easier since pushouts of infinitesimal bimodules coincide with pushouts taken componentwise. Similarly, we show that this category is equivalent to the category of algebras over a colored operad. After that, we introduce the projective model category structure and we prove that the extension/restriction adjunctions along weak equivalences of operads form  Quillen equivalences.  

In the fifth section we adapt the results from the third section to the context of infinitesimal bimodules. In the same way, we build an explicit fibrant coresolution. We give a characterization of cofibrations and, as a consequence of it, we show that the extension/restriction adjunctions along weak equivalences of reduced operads form  Quillen equivalences. We then compare the projective and  Reedy model structures on infinitesimal bimodules. 
At the end we exhibit explicit cofibrant resolutions.   

The last sixth section is an appendix where several technical lemmas that we use from the equivariant homotopy theory are  formulated and proved. 

\vspace{5pt}
\noindent \textbf{Notation:}
In~\cite{Fre1,Fre2}, the notation $\Lambda_{\ast}\Operad$ actually refers to a category of $\Lambda$-operads,
which is defined by dropping the arity zero component
of reduced operads. But we do not use this convention in this paper. We therefore forget about the refined structure of a $\Lambda$-operad
and we use the notation $\Lambda_{\ast}\Operad$ for the category of reduced operads. We keep the letter $\Lambda$
in order to emphasize the underlying $\Lambda$-diagram structure of our objects,
but we forget about further reductions in the definition
of our structures. We write $\Lambda_*$ instead of $\Lambda$ to remind that the operads in this category are reduced.

In the paper we use many different sets of rooted trees.  As a general rule, we use letter $\mathbb{P}$ for sets of planar trees and letter
$\mathbb{T}$
for non-planar trees. Usually the set $\ell(T)$ of leaves  of a planar tree $T$ is labelled by a permutation in $\Sigma_{|T|}$, where $|T|$ is the number
of leaves in the tree. Internal vertices in these trees are usually allowed to have any arity unless we use the superscript $\geq 1$ or $\geq 2$, like
in $\mathbb{P}^{\geq 1}$ or $\mathbb{T}^{\geq 2}$  meaning that the arities of vertices are $\geq 1$ or $\geq 2$, respectively. We use different terms to deal with trees in the context of operads, such as set of leaves $\ell(T)$, set of vertices $V(T)$, set of edges $E(T)$, the arity $|v|$ of a vertex $v$, etc from \cite[Section~5.8]{BM}. 
Even though we formally do not define these terms, we show them on figures, so that the reader can easily guess the meaning of those words without referring to loc. cit.

For a subspace of a space $X$, we sometimes use notation $\partial X$. By $\partial\prod_{i\in I} X_i$ we understand a subspace in $\prod_{i\in I}
X_i$ consisting of points with at least one coordinate in $\partial X_i$.  A point in the product space $\prod_{i\in I} X_i$ is usually denoted by $\{x_{i}\}_{i\in I}$ or just  $\{x_{i}\}$ if there is no ambiguity about the set $I$.

%\vspace{5pt}

\section{Model category structures for operads}%\vspace{5pt}

In this section, we introduce the categories $\Sigma Seq$ and $\Lambda Seq$ as well as their model category structures called \textit{projective} and \textit{Reedy} model category structures. These are categories whose objects are sequences of topological spaces with some extra structures. We also define the categories of operads $\Sigma\Operad$ and reduced operads $\Lambda_{\ast}\Operad$. Both categories inherit model category structures from the following adjunctions in which the functors $\mathcal{F}^{\Sigma}$ and $\mathcal{F}^{\Lambda}$ are the left adjoints to the forgetful functors:
$$
\mathcal{F}^{\Sigma}:\Sigma Seq\rightleftarrows \Sigma\Operad:\mathcal{U}^{\Sigma}\hspace{15pt}\text{and}\hspace{15pt} \mathcal{F}^{\Lambda}:\Lambda_{>0} Seq\rightleftarrows \Lambda_{\ast}\Operad:\mathcal{U}^{\Lambda}.
$$
Here $\Lambda_{>0} Seq$ can be interpreted as the full model subcategory of $\Lambda Seq$ composed of objects whose arity zero component is the one point topological space (see Section \ref{Sect1.2}). Both model categories $\Sigma\Operad$ and $\Lambda_{\ast}\Operad$ have  intensively been studied by Berger-Moerdijk \cite{BM} and the second author \cite{Fre1,Fre2}. We list their properties in Section~\ref{E4}. Usually, in order to define a model category structure from an adjunction, we use the following statement also called the \textit{transfer principle}: 

\begin{thm}{\cite[Section 2.5]{BM}}\label{E3}
Let $\mathcal{D}$ be a cofibrantly generated model category with a set of generating cofibrations $\mathcal{D}_{c}$ and a set of generating acyclic cofibrations $\mathcal{D}_{ac}$. Let $L:\mathcal{D}\rightleftarrows \mathcal{C}:R$ be an adjunction with left adjoint $L$ and right adjoint $R$. Assume that $\mathcal{C}$ is bicomplete. Define a map $f$ in $\mathcal{C}$ to be a weak equivalence (respectively, a fibration) if $R(f)$ is a weak equivalence (respectively, a fibration) in $\mathcal{D}$. If the following conditions are satisfied:
\begin{itemize}
\item[$(i)$] both sets $L(\mathcal{D}_{c})$ and $L(\mathcal{D}_{ac})$ permit the small object argument;
\item[$(ii)$] $\mathcal{C}$ has a fibrant replacement functor for objects;
\item[$(iii)$] $\mathcal{C}$ has a functorial path object for  fibrant objects, i.e. for any fibrant object $X$ there is a functorial factorization of the diagonal map into a weak equivalence followed by a fibration 
$$
\xymatrix{
X\ar[r]^{\hspace{-15pt}\simeq} & Path(X) \ar@{->>}[r] & X\times X;
}
$$
\end{itemize} 
\noindent then we have a cofibrantly generated model category structure on $\mathcal{C}$ in which the set of generating cofibrations (respectively acyclic cofibrations) is given by $L(\mathcal{D}_{c})$ (respectively, $L(\mathcal{D}_{ac})$). Furthermore, this model category structure makes the adjunction $(L;R)$ into a Quillen adjunction.
\end{thm}

As explained in the following subsections, all objects in the category $\Sigma Seq$ are fibrant and the identity functor produces a functorial fibrant replacement in the category $\Sigma\Operad$. So, the transfer principle can easily be applied to the adjunction $(\mathcal{F}^{\Sigma};\mathcal{U}^{\Sigma})$. Unfortunately, the objects in $\Lambda_{>0} Seq$ are not necessarily fibrant and the second author proves in~\cite{Fre2} the existence of the model category structure for reduced operads without the transfer principle. In the present work, we build an explicit functorial fibrant replacement in both categories $\Lambda Seq$ and $\Lambda_{>0} Seq$. This resolution will be enhanced in the next sections in order to define Reedy model category structures for (infinitesimal) bimodules using the transfer principle.

Both categories $\Sigma Seq$ and $\Lambda Seq$ are obtained as categories of functors from $\Sigma$ and $\Lambda$ to topological spaces. So the following model category structures are particular cases of model categories of diagrams. We refer to \cite{GKR, HKRS, BHK} for a comprehensive study of projective model categories of diagrams over a discrete category (and of dual injective model categories of diagrams), to \cite{Mo, DRO} for a study of projective model categories of diagrams in the enriched setting. For Reedy model structures and applications to simplicial homotopy theory, we refer the reader to \cite{GJ,Ree} and to
 \cite{Ba,An,BM4,Re2,RV} for generalizations to enriched categories or extended Reedy categories.

\subsection{The projective model categories of $G$-spaces and of $\Sigma$-sequences}\label{MCSSpace}

\noindent \textit{$\bullet$ The model category of spaces.} In what follows, by spaces we mean  compactly generated, but not necessarily Hausdorff, topological 
spaces. Such spaces are often called  $k$-spaces~\cite{Hov1}.  One has 
a natural {\it kelleyfication functor} from the category of all topological spaces to $k$-spaces. The topology of mapping spaces, products, subspaces and more generally limits in this category are defined by taking kellyfication of  their usual compact-open, product and subspace topologies.  The coproducts, quotients and more generally any colimits of $k$-spaces are automatically $k$-spaces and kelleyfication is not necessary. The category of $k$-spaces has the advantage of being cartesian closed~\cite[Theorem 5.5]{Lew}. This statement implies that the cartesian products distribute over colimits, which is a prerequisite for the theory of operads and operadic objects. Moreover, it implies that the product of two quotient maps (in particular, of a quotient map and an identity one) is again a quotient map in the category of $k$-spaces \cite[Corollary~5.9.10]{Br}. We use this observation
 in our constructions of free objects and pushouts in the category of operads and in the categories of bimodules over operads.

%We denote by $Top$ the category of $k$-spaces which has the advantage to be cartesian closed~\cite[Theorem 5.5]{Lew}. % As explained in \cite[Lemma 2.4.1]{Hov1}, every topological space is small relative to the inclusions. 

The category of spaces, fixed in the previous paragraph (thus, the category of $k$-spaces), is denoted by $Top$ and is equipped with the Quillen model category structure (see \cite[Theorem 2.4.23]{Hov1})  in which a continuous map is a weak equivalence (respectively, a  fibration) if it is a weak homotopy equivalence (respectively, a Serre fibration). According to this definition, all spaces are fibrant and the model category $Top$ is cofibrantly generated. The set of  generating cofibrations $S_{c}$ and the set of generating acyclic cofibrations $S_{ac}$ are the following ones, where $S^{-1}$ denotes the empty set: 
$$
S_{c}= \big\{ \,\, S^{n-1}\hookrightarrow D^n,\text{ } n\geq 0\,\,
\big\} \hspace{15pt}\text{and} \hspace{15pt} 
S_{ac}=  \big\{\,\, 
D^n\times\{0\}\hookrightarrow D^n\times [0,1],
\text{ } n\geq 0
\,\,\big\}. \vspace{5pt}
$$

\noindent \textit{$\bullet$ Projective model category of $G$-spaces.} Let $G$ be a topological monoid. The category $G\text{-}Top$ of $G$-spaces consists of spaces equipped with a right action of $G$. There is an adjunction $G[-]:Top\rightleftarrows G\text{-}Top:\mathcal{U}$, where $\mathcal{U}$ is the forgetful functor and $G[-]$ is the functor sending a space $X$ to the $G$-space $G[X]=X\times G$. As a consequence of Theorem~\ref{E3}, the category $G\text{-}Top$ inherits a cofibrantly generated model category structure whose sets of generating cofibrations and acyclic cofibrations are $G[S_{c}]$ and $G[S_{ac}]$, respectively. Indeed, the identity functor provides a fibrant replacement functor while, for any $G$-space $X$, the functorial path object is given by the mapping space
$$
Path(X)= Map(\,[0\,,\,1]\,,\,X\,).
$$
Cofibrations and fibrations in this category will be called {\it $G$-cofibrations} and {\it $G$-fibrations}, respectively. We will be mostly using spaces with a right action
of a monoid (or a group). At few occasions we will need to deal with spaces endowed with a left action. Such spaces will be called {\it left} $G$-spaces and the category
of such will be denoted by $G^{op}\text{-}Top$.\vspace{5pt}

\noindent \textit{$\bullet$ Projective model category of $\Sigma$-sequences.} Let $\Sigma$ be the category whose objects are finite sets $[n]=\{1,\ldots,n\}$, with $n\geq 0$, and morphisms are bijections between them. By a $\Sigma$-sequence, we mean a contravariant functor from $\Sigma$ to the category of spaces. In practice, a $\Sigma$-sequence is given by a family of spaces $X(0),$ $X(1),\ldots$ together with an action of the symmetric group: for each permutation $\sigma \in \Sigma_{n}$, there is a map
\begin{equation}\label{A0}
\begin{array}{rcl}
\sigma^{\ast}:X(n) & \longrightarrow & X(n); \\ 
 x & \longmapsto & x\cdot \sigma,
\end{array} 
\end{equation}
satisfying the relations $(x\cdot\sigma)\cdot \tau=x\cdot(\sigma\tau)$, with $\tau\in \Sigma_{n}$, and $x\cdot e=x$. A morphism between $\Sigma$-sequences is a family of continuous maps that should preserve the right action of the symmetric groups. We denote by $\Sigma Seq$ the category of $\Sigma$-sequences and by $\Sigma_{>0} Seq$ its subcategory composed of $\Sigma$-sequences whose arity $0$ component is empty.

Given an integer $r\geq 0$, we also consider the category of $r$-truncated $\Sigma$-sequences $T_{r}\Sigma Seq$ which we define as follows. Let $T_{r}\Sigma$ be the category with objects $[n]=\{1,\ldots,n\}$, $0\leq n \leq r$, and bijections between them. An $r$-truncated $\Sigma$-sequence is a contravariant functor from $T_{r}\Sigma$ to the category of spaces. In practice, an $r$-truncated $\Sigma$-sequence is given by a family of spaces $X(0),\ldots, X(r)$ together with an action of the corresponding symmetric group $\Sigma_{n}$ for each $n\leq r$.

 A (possibly truncated) $\Sigma$-sequence is said to be \textit{pointed} if there is a distinguished element $\ast_{1}\in X(1)$ called \textit{unit}. It is said \textit{well-pointed} if the inclusion
$*_1\to X(1)$ is a cofibration. 
Recall that we deal with the Quillen model structure on topological spaces. Thus, we consider the cofibrations of the Quillen model structure in this definition, not the Hurewicz cofibrations of the usual notion of well-pointed space.
Note that a space $T$ is well-pointed in this sense as soon as it is cofibrant (with respect to the Quillen model structure). Indeed, in this case, $T$ occurs as a retract of a (generalized) CW-complex and we may just observe that the inclusion of any choice of base point in a (generalized) CW-complex is a cofibration to conclude that this is the same for $T$. From this observation, we deduce that a (truncated) $\Sigma$-sequence $X$ is well-pointed in the sense of our definition if and only if it is pointed and its  component $X(1)$ is cofibrant.

 There is an obvious functor called \textit{truncation functor} 
$$
T_{r}(-):\Sigma Seq \longrightarrow T_{r}\Sigma Seq.
$$

One has
$$
\Sigma Seq=\prod_{n\geq 0} \Sigma_{n}\text{-}Top, \hspace{20pt}\Sigma_{>0} Seq=\prod_{n\geq 1} \Sigma_{n}\text{-}Top \hspace{15pt}\text{and}\hspace{15pt} T_{r}\Sigma Seq= \prod_{0\leq n\leq r} \Sigma_{n}\text{-}Top.
$$%\vspace{5pt}

Since $\Sigma_{n}\text{-}Top$ is a cofibrantly generated model category for any $n\geq 0$, the categories $\Sigma Seq$, $\Sigma_{>0} Seq$ and $T_{r}\Sigma Seq$ are endowed with a cofibrantly generated model category structure, called \textit{the projective model category structure}, in which all objects are fibrant. More precisely, a map between (possibly truncated) $\Sigma$-sequences is a weak equivalence (respectively, a fibration) if the map is degreewise  a weak homotopy equivalence (respectively, a Serre fibration). The sets of generating cofibrations $S_{c}$ and acyclic cofibrations $S_{ac}$ of $\Sigma Seq$ (respectively, $\Sigma_{>0} Seq$ and $T_{r}\Sigma Seq$) are given by 
$$
S_{c}=\underset{\substack{n\geq 0 \\ (\text{resp. } n>0 \text{ and }\\  0\leq n \leq r )}}{\bigcup} \left( S_{c}^{n}\times \, \underset{m\neq n}{\prod}\, 1_{m}\right) \hspace{15pt}\text{and}\hspace{15pt}
S_{ac}=\underset{\substack{n\geq 0 \\ (\text{resp. } n>0 \text{ and }\\  0\leq n \leq r )}}{\bigcup} \left( S_{ac}^{n}\times \, \underset{m\neq n}{\prod}\, 1_{m}\right)
$$
where $S_{c}^{n}$ and $S_{ac}^{n}$ are the sets of generating cofibrations and  acyclic cofibrations, respectively, of $\Sigma_{n}\text{-}Top$ while $1_{m}\colon\emptyset\to\emptyset$ is the identity map of the initial object of $\Sigma_{m}\text{-}Top$ (see \cite[Proposition 11.1.10]{Hir}).\vspace{5pt}

\noindent \textit{$\bullet$ Notation for cofibrations.} Let $\mathcal{C}$ be a category together with a functor $\mathcal{U}$ from $\mathcal{C}$ to the category $\Sigma Seq$ (respectively the categories $\Sigma_{>0} Seq$ and $T_{r}\Sigma  Seq$). In the rest of the paper, an object $C$ in the category $\mathcal{C}$ is said to be \textit{$\Sigma$-cofibrant}  if the underlying $\Sigma$-sequence $\mathcal{U}(C)$ is cofibrant in the projective model category $\Sigma Seq$ (respectively, $\Sigma_{>0} Seq$ and $T_{r}\Sigma Seq$). It is called \textit{componentwise cofibrant} if every component  $U(C)(n)$ is cofibrant in $Top$. In case the
forgetful functor $U$ factors through the category of pointed (truncated) $\Sigma$-sequences, the object $C$ is called \textit{well-pointed} if $U(C)(1)$ is cofibrant. In the following the category $\mathcal{C}$ will be the category of operads $\Sigma\Operad$, of reduced operads $\Lambda_{\ast}\Operad$, of bimodules $\Sigma\Bimod_{P\,;\,Q}=\Lambda\Bimod_{P\,;\,Q}$, of
reduced bimodules $\Lambda_*\Bimod_{P\,;\,Q}$, or  of infinitesimal bimodules $\Sigma\Ibimod_O=
\Lambda\Ibimod_{O}$.

\subsection{The Reedy model categories of $\Lambda$- and $\Lambda_{>0}$-sequences}\label{Sect1.2}

\noindent \textit{$\bullet$ The category of $\Lambda$-sequences.} We refer the reader to \cite{Fre1,Fre2} for a detailed account on the categories introduced in this subsection. Let $\Lambda$ be the category whose objects are finite sets $[n]=\{1,\ldots,n\}$, with $n\geq 0$, and morphisms are injective maps between them. (Hence, $\Sigma$ is the subcategory of isomorphisms of $\Lambda$.) By a $\Lambda$-sequence, we understand a contravariant functor from $\Lambda$ to spaces and we denote the corresponding category by $\Lambda Seq$. In practice, such an object is given by a $\Sigma$-sequence $X(0),\,X(1),\ldots$ together with maps generated by applications of the form 
\begin{equation}\label{E5}
s^{\ast}_{i}:X(n)\longrightarrow X(n-1), \,\,\,\text{with } 1\leq i\leq n,
\end{equation}
associated to the injective maps 
$$
\begin{array}{ccccccc}
s_{i}:[n-1] & \longrightarrow & [n] & ; & \ell & \longmapsto & \left\{
 \begin{array}{ll}\vspace{5pt}
 \ell & \text{if } \ell< i, \\ 
 \ell+1 & \text{if } \ell\geq i.
 \end{array} 
 \right.
\end{array} 
$$

Given an integer $r\geq 0$, we also consider the full subcategory $T_{r}\,\Lambda$ whose objects are families of finite sets $[n]=\{1,\ldots,n\}$, with $0\leq n \leq r$. An $r$-truncated $\Lambda$-sequence is a contravariant functor from $T_{r}\,\Lambda$ to spaces and we denote by $T_{r}\,\Lambda Seq$ the associated category. We will also be using the categories $\Lambda_{>0}$ and $T_r\Lambda_{>0}$ which are full subcategories of non-empty objects of $\Lambda$ and $T_r\Lambda$. The categories  $\Lambda_{>0} Seq$ and $T_r\Lambda_{>0} Seq$ are similarly defined. There exist obvious truncation functors
$$
T_{r}(-):\Lambda Seq\longrightarrow T_{r}\,\Lambda Seq \hspace{20pt}\text{and} \hspace{20pt} T_{r}(-):\Lambda_{>0} Seq\longrightarrow T_{r}\,\Lambda_{>0} Seq.
$$ \vspace{5pt}

\noindent \textit{$\bullet$ Useful adjunctions.} The inclusions of categories $\Sigma\subset \Lambda$ and $T_{r}\Sigma\subset T_{r}\Lambda$ induce adjunctions between the categories of (possibly truncated) $\Sigma$-sequences and $\Lambda$-sequences 
$$
\Lambda[-]:\Sigma Seq \rightleftarrows \Lambda Seq:\mathcal{U} \hspace{15pt}\text{and}\hspace{15pt}\Lambda_{r}[-]:T_{r}\Sigma Seq \rightleftarrows T_{r}\Lambda Seq:\mathcal{U}
$$
where $\mathcal{U}$ is the obvious forgetful functor, which forgets about the operations generated by (\ref{E5}). The functor $\Lambda[-]$ sends a $\Sigma$-sequence $X$ to the $\Lambda$-sequence $\Lambda[X]$ given by 
$$
\Lambda[X](n):=\underset{\substack{\Lambda_{+}([n];[m])\\ m\geq n}}{\coprod}X(m),\hspace{15pt}\text{for all } n\geq 0,
$$
where $\Lambda_{+}$ is the subcategory of order preserving injective maps. A point in $\Lambda[X](n)$ is denoted by $(h;x)$ with $h:[n]\rightarrow [m]$ an order preserving injective map and $x\in X(m)$. For any permutation $\sigma \in \Sigma_{n}$ and any order preserving injective map $h:[n]\rightarrow [m]$, we denote by $\sigma_{h}\in \Sigma_{m}$ the permutation 
$$
\sigma_{h}(i):=\left\{
\begin{array}{cl}\vspace{5pt}
h(\sigma(j)), & \text{if } h(j)=i, \\ 
i, & \text{otherwise}.
\end{array} 
\right.
$$  
According to this notation, the $\Lambda$-structure on $\Lambda[X]$ is given by the following formulas:
$$
\begin{array}{cclcccc}%\vspace{9pt}
\sigma^{\ast}:\Lambda[X](n) & \longrightarrow & \Lambda[X](n) & ; & (h;x) & \longmapsto & (h;x\cdot \sigma_{h}), \\ 
s_{i}^{\ast}:\Lambda[X](n) & \longrightarrow & \Lambda[X](n-1) & ; & (h;x) & \longmapsto & (h\circ s_{i};x).
\end{array} \vspace{5pt}
$$

\noindent \textit{$\bullet$ The matching object.} For a (possibly truncated) $\Lambda$-sequence $X$, the \textit{matching object} of $X$, denoted by $\mathcal{M}(X)$, is the (possibly truncated) $\Sigma$-sequence defined as follows: 
\begin{equation}\label{eq:matching}
\mathcal{M}(X)(\,n\,)=\underset{\substack{h\in \Lambda_{+}([\ell]\,;\,[n])\\ \ell< n}}{\mathrm{lim}} X(\ell).
\end{equation}
Let $\sigma\in \Sigma_{n}$ be a permutation and $h:[\ell]\rightarrow [n]$ be an order preserving inclusion. We denote by $h\cdot \sigma:[\ell]\rightarrow [n]$ the unique order preserving inclusion whose image is $\mathrm{Im}(\sigma\circ h)$ and $\sigma[h]\in \Sigma_{\ell}$ is the unique permutation satisfying $\sigma[h](i)<\sigma[h](j)$ if and only if $\sigma(h(i))<\sigma(h(j))$. According to this notation, the action of the symmetric group on the matching object is the following one:
$$
\begin{array}{rcl}%\vspace{7pt}
\sigma^{\ast}:\mathcal{M}(X)(n) & \longrightarrow & \mathcal{M}(X)(n); \\ 
x=\{x_{h}\}_{h} & \longmapsto & \sigma^{\ast}(x)=\{x_{h\cdot \sigma}\cdot \sigma[h]\}_{h}.
\end{array} \vspace{5pt}
$$

\noindent \textit{$\bullet$ The Reedy model category of $\Lambda$-sequences.} According to \cite[Theorem 8.3.19]{Fre2}, the categories $\Lambda Seq$, $\Lambda_{>0} Seq$, $T_{r}\,\Lambda Seq$ and $T_{r}\,\Lambda_{>0} Seq$ are endowed with cofibrantly generated model category structures in which weak equivalences are objectwise weak homotopy equivalences. A morphism $f:X\rightarrow Y$ is a fibration if the corresponding maps
$$
X(n)\longrightarrow \mathcal{M}(X)(n)\times_{\mathcal{M}(Y)(n)}Y(n),
$$
whenever defined, are  Serre fibrations.  The set of generating cofibrations (respectively the set of generating acyclic cofibrations) consists of maps of the form
$$
(\Lambda[f],\iota):\Lambda[X] \underset{\partial \Lambda[X]}{\coprod} \partial \Lambda [Y]\longrightarrow \Lambda[Y],
$$
where $f:X\rightarrow Y$ is a generating cofibration (respectively a generating acyclic cofibration) in the projective model category of $\Sigma$-sequences. The map $\iota:\partial\Lambda[Y]\rightarrow \Lambda[Y]$ is the inclusion where $\partial \Lambda[-]$ is the functor from $\Sigma$-sequences to $\Lambda$-sequences left adjoint to the matching object functor $\mathcal M$ and expressed by the formula
$$
\partial \Lambda[Y](n)=\underset{\substack{h\in \Lambda_{+}([n]\,;\,[\ell])\\  \ell>n}}{\mathrm{colim}} \Lambda[Y](\ell) = 
\underset{\substack{\Lambda_{+}([n];[\ell])\\ \ell> n}}{\coprod}Y(\ell).
$$

%$\mathbb{T}$

\begin{rmk}\label{rmk:truncation}
The truncation functors 
$$
T_r\colon \Lambda Seq\to T_r\,\Lambda Seq;\quad
 T_r\colon \Lambda_{>0} Seq\to T_r\,\Lambda_{>0} Seq;\quad
  T_r\colon T_{r'}\, \Lambda Seq\to T_r\,\Lambda Seq, \,\,\,r'>r,
  $$
   preserve fibrations, cofibrations
and weak equivalences.
\end{rmk}

 For fibrations and weak equivalences, the statement follows from definition. For cofibrations we recall \cite[Theorem~8.3.20]{Fre2} that a morphism in the Reedy model structure is a cofibration if and only if it is a projective cofibration in the corresponding category of (truncated) $\Sigma$-sequences.

%$\mathbb{P}$

\subsection{A fibrant  replacement functor for $\Lambda$- and $\Lambda_{>0}$-sequences}\label{A9} 

%$\mathbb{P}$

As explained previously, in order to apply the transfer principle, described in Theorem \ref{E3}, we need a functorial fibrant replacement. Unfortunately, the objects in $\Lambda Seq$ and $\Lambda_{>0} Seq$ are not necessarily fibrant and therefore the identity functors can not be regarded as a fibrant replacement one. To solve this problem, we build explicit and functorial fibrant coresolutions
$$
(-)^{f}:\Lambda Seq \longrightarrow \Lambda Seq \hspace{20pt}\text{and}\hspace{20pt} (-)^{f}:\Lambda_{>0} Seq \longrightarrow \Lambda_{>0} Seq.
$$

For this purpose, we need some notation. For any map $h\in \Lambda_{+}([\ell];[n])$, we denote by $h_{c}\in \Lambda_{+}([n-\ell];[n])$ its complementary map which is the unique order preserving inclusion so that $Im(h)\cap Im(h_{c})=\emptyset$. Then, for any pair of order preserving inclusions of the form $s_{i}:[\ell]\rightarrow [\ell+1]$ and $h:[\ell+1]\rightarrow [n]$, we denote by $\varepsilon_{h;i}\in [n-\ell]$ the unique index such that the following diagram commute:
$$
\xymatrix{
[n-\ell-1] \ar[rr]^{s_{\varepsilon_{h;i}}}\ar[dr]_{h_{c}} & & [n-\ell]\ar[dl]^{(h\circ s_{i})_{c}} \\
& [n] & 
}
$$
Let $X$ be a $\Lambda$-sequence. The space $X^{f}(n)$ is the subspace  
$$
X^{f}(n)\subset \prod_{\substack{\Lambda_{+}([\ell]\,,\,[n])\\ \ell \leq n}} Map\big(\, [0\,,\,1]^{n-\ell}\,;\,X(\ell)\,\big),
$$
consisting of families of maps $\{f_{h}\}_{h\in \Lambda_{+}([\ell]\,,\,[n])}$ such that
\begin{equation}\label{A4}
\xymatrix@C=45pt@R=35pt{
[0\,,\,1]^{n-\ell-1}\ar[d]^{f_{h}} \ar[r]^{\tau_{1}[\varepsilon_{h;i}]} & [0\,,\,1]^{n-\ell}\ar[d]^{f_{h\circ s_{i}}} \\
X(\ell+1)\ar[r]^{s_{i}^{\ast}} & X(\ell)
}
\end{equation}
where $\tau_{t}[k]:[0\,,\,1]^{n-\ell-1}\rightarrow [0\,,\,1]^{n-\ell}$, with $t\in [0\,,\,1]$ and $k\in [n-\ell-1]$, inserts $t$ at the $k$-th position:%\vspace{5pt}
$$
\tau_{t}[k](t_{1},\ldots,t_{n-\ell-1})= (t'_{1},\ldots,t'_{n-\ell}) \hspace{15pt} \text{with}\hspace{15pt} 
t'_{j}=\left\{
\begin{array}{cc}\vspace{4pt}
t_{j} & \text{if } j< k, \\ \vspace{4pt}
t & \text{if } j= k,  \\ 
t_{j-1} & \text{if } j> k. 
\end{array} 
\right.
$$

\noindent \textit{$\bullet$ The $\Lambda$-structure on $X^{f}$.} In order to describe the $\Lambda$-structure we consider the following notation. For any order preserving inclusions $s_{i}:[n-1]\rightarrow [n]$ and $h:[\ell]\rightarrow [n-1]$, we denote by $\vartheta_{h;i}$ the unique index such that the following diagram commute:
$$
\xymatrix@C=35pt@R=35pt{
[n-\ell-1]\ar[r]^{h_{c}}\ar[d]_{s_{\vartheta_{h;i}}} & [n-1]\ar[d]^{s_{i}}\\
[n-\ell]\ar[r]_{(s_{i}\circ h)_{c}} & [n] 
}
$$
According to this notation, the $\Lambda$-structure operations are given by
$$
\begin{array}{ccllcccl}\vspace{9pt}
s^{\ast}_{i}: & X^{f}(n) & \longrightarrow & X^{f}(n-1) & ; & \{f_{h}\}_{h\in \Lambda_{+}([\ell]\,,\,[n])} & \longmapsto & \{(s^{\ast}_{i}\circ f)_{h}\}_{h\in \Lambda_{+}([\ell]\,,\,[n-1])}, \\ 
\sigma^{\ast}: & X^{f}(n) & \longrightarrow &  X^{f}(n) & ; & \{f_{h}\}_{h\in \Lambda_{+}([\ell]\,,\,[n])} & \longmapsto & \{(f\cdot \sigma)_{h}\}_{h\in \Lambda_{+}([\ell]\,,\,[n])},
\end{array} 
$$
where the continuous maps $(s^{\ast}_{i}\circ f)_{h}$ and $(f\cdot \sigma)_{h}$ are the following ones:
\begin{equation}\label{E1}
\begin{array}{rlcccrcl}\vspace{9pt}
(s^{\ast}_{i}\circ f)_{h}: & [0\,,\,1]^{n-\ell-1} & \longrightarrow & X(\ell) & ; & (t_{1},\ldots,t_{n-\ell-1}) & \longmapsto & f_{s_{i}\circ h}(\tau_{0}[\vartheta_{h;i}](t_{1},\ldots,t_{n-\ell-1})),\\ 
(f\cdot \sigma)_{h}: & [0\,,\,1]^{n-\ell} & \longrightarrow & X(\ell) & ; &  (t_{1},\ldots,t_{n-\ell}) & \longmapsto & f_{h\cdot\sigma}(t_{\sigma[h_{c}]^{-1}(1)},\ldots, t_{\sigma[h_{c}]^{-1}(n-\ell)})\cdot \sigma[h].
\end{array} 
\end{equation}

\noindent \textit{$\bullet$ The fibrant replacement functor.} The $\Lambda$-sequence $X^{f}$ is obviously functorial along the $\Lambda$-sequence $X$. Furthermore, there is a map $\varphi:X\rightarrow X^{f}$ sending a point $x\in X(n)$ to the family of constant maps $\{\varphi(x)_{h}\}_{h\in\Lambda_{+}([\ell]\,,\,[n])}$ obtained using the $\Lambda$-structure of $X$:
$$
\xymatrix@R=1pt{
\varphi(x)_{h}:\hspace{-20pt} & [0\,,\,1]^{n-\ell} \ar[r] & X(\ell);\\
& (t_{1},\ldots,t_{n-\ell}) \ar@{|->}[r] & h^{\ast}(x). 
}
$$
The map is well defined and preserves the $\Lambda$-structures. Indeed, one has the following equalities:
$$
\begin{array}{ccl}\vspace{7pt}
s^{\ast}_{i}(\varphi(x)) & = & s^{\ast}_{i}\left(\left\{
\begin{minipage}{10pt}
\xymatrix@R=-2pt{
\varphi(x)_{h}:\hspace{-30pt} & [0\,,\,1]^{n-\ell} \ar[r] & X(\ell);\\
& (t_{1},\ldots,t_{n-\ell}) \ar@{|->}[r] & h^{\ast}(x). 
}
\end{minipage}
 \right\}_{h\in \Lambda_{+}([\ell]\,,\,[n])}\right) \\ \vspace{7pt}
 
& = & \left\{
\begin{minipage}{10pt}
\xymatrix@R=-2pt{
s^{\ast}_{i}\circ\varphi(x)_{h}:\hspace{-30pt} & [0\,,\,1]^{n-\ell-1} \ar[r] & X(\ell);\\
& (t_{1},\ldots,t_{n-\ell-1}) \ar@{|->}[r] & \varphi(x)_{s_{i}\circ h}(t_{1},\ldots,t_{\vartheta_{h;i}-1},0,t_{\vartheta_{h;i}},\ldots,t_{n-\ell-1}). 
}
\end{minipage}
 \right\}_{h\in \Lambda_{+}([\ell]\,,\,[n-1])}\\  \vspace{5pt}

\varphi(s^{\ast}_{i}(x))  & = & \left\{
\begin{minipage}{10pt}
\xymatrix@R=-2pt{
\varphi(x)_{h}:\hspace{-30pt} & [0\,,\,1]^{n-\ell-1} \ar[r] & X(\ell);\\
& (t_{1},\ldots,t_{n-\ell-1}) \ar@{|->}[r] & h^{\ast}(s^{\ast}_{i}(x)). 
}
\end{minipage}
 \right\}_{h\in \Lambda_{+}([\ell]\,,\,[n-1])}

\end{array} 
$$

\begin{pro}\label{Z7}
The map $\phi:X\rightarrow X^{f}$ is a weak equivalence of $\Lambda$-sequences.
\end{pro}

\begin{proof}
More precisely, we show that the map of $\Lambda$-sequences $\varphi_{n}:X(n)\rightarrow X^{f}(n)$ is a homotopy equivalence of $\Sigma$-sequences. For this purpose, we introduce a map of $\Sigma$-sequences (which is not a map of $\Lambda$-sequences) $\psi:X^{f}\rightarrow X$ given by
$$
\xymatrix@R=-2pt{
\psi_{n}:\hspace{-30pt} &X^{f}(n) \ar[r] & X(n);\\
& \{f_{h}\}_{h\in \Lambda_{+}([\ell]\,,\,[n])} \ar@{|->}[r] & f_{[n]\rightarrow [n]}(\ast), 
}
$$
which makes $\varphi$ into a deformation retract. The homotopy consists in bringing the parameters to $1$:
$$
\begin{minipage}{10pt}
\xymatrix@R=-1pt{
H_{n}: [0\,,\,1]\times X^{f}(n)\ar[r] & X^{f}(n);\\
t\,;\,\{f_{h}\} \ar@{|->}[r] & \{H_{n}(t\,;\,f_{h})\},
}
\end{minipage}
$$
with $H_{n}(t\,;\,f_{h})(t_{1},\ldots, t_{n-\ell})=f_{h}\big( (1-t)t_{1}+t,\ldots,(1-t)t_{n-\ell}+t\big)$.
\end{proof}

\begin{pro}\label{E2}
The $\Lambda$-sequence $X^{f}$ is Reedy fibrant. 
\end{pro}

\begin{lmm}[Fiber product version of Reedy's patching lemma, see Lemma $1.3$ in  \cite{Re1}]\label{E6}
 For any commutative diagram of spaces of the form
$$
\xymatrix@R=18pt{
A\ar[r]^{f} \ar[d]_{v_{A}} & B \ar[d]_{v_{B}} & \ar[l]_{g} C  \ar[d]_{v_{C}}\\
A'\ar[r]_{f'}  & B'  & \ar[l]^{g'} C' 
}
$$
the induced map between the limits of the horizontal diagrams 
$$
v\colon\mathrm{lim}\big( A\rightarrow B \leftarrow C\big) \longrightarrow \mathrm{lim}\big( A'\rightarrow B' \leftarrow C'\big)
$$
is a fibration if the map $v_{C}$ as well as the map
\begin{equation}\label{H0}
(v_{A};f):A\longrightarrow B\times_{B'}A'
\end{equation}
are fibrations. Furthermore, the map $v$ is an acyclic fibration if the map $v_{C}$ and (\ref{H0}) are acyclic fibrations.  
\end{lmm}

\begin{proof}[Proof of Proposition \ref{E2}]
We show that the map from $X^{f}(n)$ to the matching object $\mathcal{M}(X^{f})(n)$ is a Serre fibration. Let us remark that the spaces $X^{f}(n)$ and $M(X^{f})(n)$ can be expressed in terms of pullback diagrams. More precisely, one has
$$
X^{f}(n)=\mathrm{lim} \left(
A\longrightarrow \underset{s_{i}:[n-1]\rightarrow [n]}{\prod} X(n-1) \longleftarrow X(n)
\right)
\hspace{15pt}\text{where}\hspace{15pt}
A\subset \underset{\substack{\Lambda_{+}([\ell]\,,\,[n])\\ \ell<n}}{\prod} Map \big( [0\,,\,1]^{n-\ell}\,;\,X(\ell)\big)
$$
is the subspace satisfying the condition (\ref{A4}). The map from $A$ to the product $\prod_{s_{i}}X(n-1)$ sends a family of maps $\{f_{h}\}$ to the family of points $\{f_{s_{i}}(1)\}$ by taking the evaluation at the point $1$. Furthermore, one has the commutative diagram
$$
\xymatrix{
A\ar[r]\ar[d] &  \underset{s_{i}:[n-1]\rightarrow [n]}{\displaystyle\prod} X(n-1) \ar[d] & X(n)\ar[l]\ar[d] \\
\mathcal{M}(X^{f})(n)\ar[r] & \ast & \ast \ar[l]
}
$$
According to Lemma \ref{E6} and since $X(n)$ is fibrant, we only need to check that the map
\begin{equation}\label{A6}
A\longrightarrow A'=\mathcal{M}(X^{f})(n)\times \underset{s_{i}:[n-1]\rightarrow [n]}{\prod} X(n-1).
\end{equation}
is a Serre fibration. In other words, if we denote by $\partial' [0\,,\,1]^{n-\ell}$ the subspace of $[0\,,\,1]^{n-\ell}$ composed of the point $(1,\ldots,1)$ and the points having at least one coordinate equal to $0$, then $A'$ is the following subspace satisfying the relation (\ref{A4}):
$$
A' \subset  \underset{\substack{\Lambda_{+}([\ell]\,,\,[n])\\ l<n}}{\prod} Map \big( \partial' [0\,,\,1]^{n-\ell}\,;\,X(\ell)\big).
$$

Let us notice that the inclusion from $\partial' [0\,,\,1]^{n-\ell}$ into $[0\,,\,1]^{n-\ell}$ is a cofibration as an inclusion of CW-complexes. Unfortunately, we can not deduce directly the result due to condition (\ref{A4}). To solve this problem, we introduce a cofiltration of the map (\ref{A6}) according to the dimension of the cubes. Let us consider the following subspaces:
$$
A_{k}\subset\underset{\substack{\Lambda_{+}([\ell]\,,\,[n])\\ n-k\leq \ell<n}}{\prod} Map \big(  [0\,,\,1]^{n-\ell}\,;\,X(\ell)\big)
\hspace{15pt}\text{and}\hspace{15pt}
A'_{k}\subset\underset{\substack{\Lambda_{+}([\ell]\,,\,[n])\\ n-k\leq \ell<n}}{\prod} Map \big( \partial' [0\,,\,1]^{n-\ell}\,;\,X(\ell)\big)
$$
satisfying condition (\ref{A4}). In particular, one has $A_{n}=A$ and $A_{n}'=A'$. Furthermore, the spaces $A_{k}$ and $A'_{k}$ can be obtained from $A_{k-1}$ and $A'_{k-1}$, respectively, using the following pullback diagrams:
$$
\begin{minipage}{10pt}
\xymatrix@C=15pt{
A_{k} \ar[r] \ar[d] & \underset{[n-k]\rightarrow [n]}{\displaystyle \prod} Map\big( [0\,,\,1]^{k}\,;\,X(n-k)\big) \ar[d]\\
A_{k-1} \ar[r] & \underset{[n-k]\rightarrow [n-k+1]\rightarrow [n]}{\displaystyle\prod} Map\big( [0\,,\,1]^{k-1}\,;\,X(n-k)\big)
}
\end{minipage}
\hspace{15pt}%\text{and} \hspace{5pt}
\begin{minipage}{10pt}
\xymatrix@C=15pt{
A'_{k} \ar[r] \ar[d] & \underset{[n-k]\rightarrow [n]}{\displaystyle\prod} Map\big( \partial' [0\,,\,1]^{k}\,;\,X(n-k)\big) \ar[d]\\
A'_{k-1} \ar[r] & \underset{[n-k]\rightarrow [n-k+1]\rightarrow [n]}{\displaystyle\prod} Map\big( \partial' [0\,,\,1]^{k-1}\,;\,X(n-k)\big)
}
\end{minipage}
$$
We prove by induction that the maps $A_{k}\rightarrow A'_{k}$ are Serre fibrations. First, the map
$$
A_{1}=\underset{[n-1]\rightarrow [n]}{\prod} Map\big( [0\,,\,1]\,;\,X(n)\big)\longrightarrow \underset{[n-1]\rightarrow [n]}{\prod} Map\big( \partial [0\,,\,1]\,;\,X(n)\big) =A'_{1}
$$
is obviously a Serre fibration since the inclusion from $\partial [0\,,\,1]=\{0\,,\,1\}$ into the interval $[0\,,\,1]$ is a cofibration. From now on, we assume that the map $A_{k-1}\rightarrow A'_{k-1}$ is a Serre fibration. Then we consider the commutative diagram
$$
\xymatrix{
\underset{[n-k]\rightarrow [n]}{\displaystyle\prod} Map\big(  [0\,,\,1]^{k}\,;\,X(n-k)\big) \ar[r] \ar[d] &  \underset{[n-k]\rightarrow [n-k+1]\rightarrow [n]}{\displaystyle\prod} Map\big(  [0\,,\,1]^{k-1}\,;\,X(n-k)\big) \ar[d] &  A_{k-1}\ar[l] \ar[d]\\
\underset{[n-k]\rightarrow [n]}{\displaystyle\prod} Map\big( \partial' [0\,,\,1]^{k}\,;\,X(n-k)\big) \ar[r] &  \underset{[n-k]\rightarrow [n-k+1]\rightarrow [n]}{\displaystyle\prod} Map\big( \partial' [0\,,\,1]^{k-1}\,;\,X(n-k)\big) &  \ar[l]A'_{k-1}
}
$$
According to Lemma \ref{E6}, one has to check that the map from the space
$$
\underset{[n-k]\rightarrow [n]}{\prod} Map\big(  [0\,,\,1]^{k}\,;\,X(n-k)\big)
$$
to the limit of the diagram
$$
\xymatrix{
 &  \underset{[n-k]\rightarrow [n-k+1]\rightarrow [n]}{\displaystyle\prod} Map\big(  [0\,,\,1]^{k-1}\,;\,X(n-k)\big) \ar[d] \\
\underset{[n-k]\rightarrow [n]}{\displaystyle\prod} Map\big( \partial' [0\,,\,1]^{k}\,;\,X(n-k)\big) \ar[r] &  \underset{[n-k]\rightarrow [n-k+1]\rightarrow [n]}{\displaystyle\prod} Map\big( \partial' [0\,,\,1]^{k-1}\,;\,X(n-k)\big)
}
$$
is a Serre fibration. The limit corresponds to the space 
$$
\underset{[n-k]\rightarrow [n]}{\prod} Map\big(  \partial [0\,,\,1]^{k}\,;\,X(n-k)\big)
$$  
and the map 
$$
\underset{[n-k]\rightarrow [n]}{\prod} Map\big(  [0\,,\,1]^{k}\,;\,X(n-k)\big) \longrightarrow \underset{[n-k]\rightarrow [n]}{\prod} Map\big(  \partial [0\,,\,1]^{k}\,;\,X(n-k)\big)
$$
is obviously a Serre fibration since the inclusion from $\partial [0\,,\,1]^{k}$ to $[0\,,\,1]^{k}$ is a cofibration.
\end{proof}

\begin{rmk}
The same strategy can be used in order to get a fibrant replacement functor for $r$-truncated $\Lambda$-sequences. In that case, we only need to restrict our construction to order preserving inclusions $h:[\ell]\rightarrow [n]$ with $n\leq r$. Similarly, we get fibrant replacement functors for the categories $\Lambda_{>0}Seq$ and $T_{r}\Lambda_{>0}Seq$.
\end{rmk}

\subsection{The projective/Reedy model category of operads}\label{E4}

An \textit{operad} is a pointed $\Sigma$-sequence $O$ together with operations called \textit{operadic compositions}
\begin{equation}\label{OperadicComposition}
\circ_{i}:O(n)\times O(m)\longrightarrow O(n+m-1),\hspace{15pt} \text{with } 1\leq i \leq n.
\end{equation}
These operations satisfy  associativity and unit axioms as well as compatibility relations with the symmetric group action. More precisely, for any integers $i\in \{1,\ldots, n\}$, $j\in\{ i+1,\ldots , n\}$, $k\in \{1,\ldots, m\}$ and any permutations $\sigma\in \Sigma_{n}$ and $\tau\in \Sigma_{m}$, one has the following commutative diagrams:
$$
\underset{\text{Linear associativity axiom}}{\xymatrix@C=35pt{
O(n)\times O(m)\times O(\ell) \ar[r]^{\circ_{i}\times id} \ar[d]_{id\times \circ_{k}} & O(n+m-1)\times O(\ell)\ar[d]^{\circ_{i+k-1}} \\
O(n)\times O(m+\ell-1) \ar[r]_{\circ_{i}} & O(n+m+\ell-2)
}}\hspace{30pt}
\underset{\text{Ramified associativity axiom}}{\xymatrix@C=35pt{
O(n)\times O(m)\times O(\ell) \ar[r]^{\circ_{i}\times id} \ar[d]_{\circ_{j}\times id} & O(n+m-1)\times O(\ell)\ar[d]^{\circ_{i+m-1}} \\
O(n+\ell-1)\times O(m) \ar[r]_{\circ_{i}} & O(n+m+\ell-2)
}}\vspace{8pt}
$$
$$
\underset{\text{unit axiom}}{\xymatrix{
O(n)\times O(1)\ar[dr]_{\circ_{i}} & O(n) \ar[r]^{\hspace{-20pt}\ast_{1}\times id} \ar[l]_{\hspace{20pt}id\times \ast_{1}} \ar@{=}[d] & O(1)\times O(n) \ar[dl]^{\circ_{1}}\\
& O(n) & 
}}\hspace{60pt}
\underset{\text{compatibility with the symmetric group action}}{\xymatrix{
O(n)\times O(m) \ar[r]^{\circ_{i}} \ar[d]_{\sigma^{\ast}\times \tau^{\ast}} & O(n+m-1) \ar[d]^{(\sigma\circ_{\sigma(i)}\tau)^{\ast}}\\
O(n)\times O(m)\ar[r]_{\circ_{\sigma(i)}} & O(n+m-1) 
}}\vspace{8pt}
$$
where the permutation $\sigma\circ_{\sigma(i)}\tau$ is obtained from the well known operadic compositions on the symmetric groups (see \cite[Proposition 1.1.9]{Fre1}).

A map between operads should preserve the operadic compositions. We denote by $\Sigma\Operad$ the category of topological operads.  The category of operads is obviously endowed with a forgetful functor to the category of $\Sigma$-sequences by forgetting the operadic composition \eqref{OperadicComposition}:
\begin{equation}\label{A3}
\mathcal{U}^\Sigma:\Sigma\Operad\longrightarrow \Sigma Seq.\vspace{2pt}
\end{equation}

\noindent \textit{$\bullet$ The category of reduced operads and their underlying $\Lambda$-structures.} An operad~$O$ is said to be \textit{reduced} if $O(0)$ is the one point topological space. This point is denoted by $\ast_{0}$. We denote by $\Lambda_{\ast}\Operad$ the category of reduced operads. This category is equipped with a forgetful functor to the category of $\Lambda_{> 0}$-sequences,  which consists in forgetting  the arity zero component and the operadic compositions \eqref{OperadicComposition} for $m\geq 1$:
\begin{equation}\label{A5}
\mathcal{U}^\Lambda:\Lambda_{\ast}\Operad\longrightarrow \Lambda_{>0} Seq.
\end{equation}
(Note also that the category of $\Lambda_{>0}$-sequences is equivalent to the category of reduced $\Lambda$-sequences, i.e $\Lambda$-sequences $X$ so that $X(0)=\ast$.) Indeed, if $O$ is a reduced operad, then the $\Lambda_{>0}$-structure on $\mathcal{U}^\Lambda(O)$ is generated by the operations of the form
$$
\begin{array}{rcl}\vspace{3pt}
s_{i}^{\ast}:\mathcal{U}^\Lambda(O)(n) = O(n)& \longrightarrow & \mathcal{U}^\Lambda(O)(n-1)=O(n-1);  \\ 
\theta & \longmapsto & \theta\circ_{i}\ast_{0}.
\end{array} 
$$

\noindent \textit{$\bullet$ The model category of algebras over an operad.} An algebra over a (possibly reduced) operad~$O$, or $O$\textit{-algebra}, is a topological space $X$ together with operations of the form 
$$
\alpha_{n}:O(n)\times X^{\times n}\longrightarrow X, \hspace{20pt}\text{with } n\geq 0,
$$
compatible with the operadic structure (see \cite[Section 1.1.13 and Figure 1.9]{Fre1}). The category of $O$-algebras is denoted by $Alg_{O}$. It has been proved in \cite[Theorem 2.1]{BM3} that the category of algebras over any operad~$O$ in $Top$ inherits a cofibrantly generated model category structure by using the transfer principle applied to the adjunction $\mathcal{F}_{O}:Top\rightleftarrows Alg_{O}:\mathcal{U}$ where the free algebra functor $\mathcal{F}_{O}$ is the left adjoint to the forgetful functor $\mathcal{U}$.\vspace{5pt}

\noindent \textit{$\bullet$ The projective and Reedy model categories of operads and reduced operads.} Both forgetful functors (\ref{A3}) and (\ref{A5}) have left adjoints, which we respectively denote by $\mathcal{F}^\Sigma$ and $\mathcal{F}^\Lambda$, and which, given a $\Sigma$- or a $\Lambda_{>0}$-sequence $X$, produce the free operad generated by $X$. Explicitly, elements of $\mathcal{F}^\Sigma(X)$  and $\mathcal{F}^\Lambda(X)$ are described as rooted trees (without univalent vertices for reduced operads) with internal vertices labelled by elements of the  sequence. We refer the reader to \cite{Fre1} for a detailed account on these adjunctions:
$$
\mathcal{F}^\Sigma:\Sigma Seq\rightleftarrows \Sigma\Operad:\mathcal{U}^\Sigma\hspace{15pt}\text{and}\hspace{15pt} \mathcal{F}^{\Lambda}:\Lambda_{>0} Seq \rightleftarrows \Lambda_{\ast}\Operad:\mathcal{U}^\Lambda.
$$

As a consequence of the transfer principle \ref{E3}, the category $\Sigma\Operad$ inherits a cofibrantly generated model category structure which we call the \textit{projective} model structure (we refer the reader to \cite{BM} for more details). Similarly to the usual category of $\Lambda$-sequences, the category $\Lambda_{>0} Seq$ is also endowed with a (cofibrantly generated) Reedy model category structure. The second author in \cite{Fre2}  proves that $\Lambda_{\ast}\Operad$ has a cofibrantly generated model category structure called the \textit{Reedy} model structure. In both cases, a map of (possibly reduced) operads $f:P\rightarrow Q$ is a weak equivalence (respectively, a fibration) if the map $\mathcal{U}^\Sigma(f)$  or $\mathcal{U}^\Lambda(f)$ is a weak equivalence (respectively, a fibration) in the appropriate category. 
 
%  By convention, an operad is $\Sigma$-cofibrant if the corresponding $\Sigma$-sequence is cofibrant in the projective model category $\Sigma Seq$. Furthermore, the operad~$O$ is said to be \textit{well-pointed} if the inclusion $\ast_{1}\rightarrow O(1)$ from the unit of the operad is a cofibration in the category $Top$. In what follows, we list the main properties of these two model category structures.
%  

\begin{thm}\label{E9}
The projective and the Reedy model structures have the following properties:
\begin{itemize}
\item[$\blacktriangleright$] \cite[Section 2.5]{BM}: All operads are fibrant in $\Sigma\Operad$.
\item[$\blacktriangleright$] \cite[Theorem 3.1.10]{HRY}: The category of operads (respectively of reduced operads) is left proper (see Section \ref{C0}) relative to the class of $\Sigma$-cofibrant operads (respectively of reduced $\Sigma$-cofibrant operads).
\item[$\blacktriangleright$] \cite[Theorem 8.4.12]{Fre2}: A map of reduced operads $\phi:P\rightarrow Q$ is a cofibration in $\Lambda_{\ast}\Operad$ if and only if the corresponding map $\phi_{>0}:P_{>0}\rightarrow Q_{>0}$ is a cofibration in $\Sigma\Operad$ where $P_{>0}$ and $Q_{>0}$ are the sub-operads obtained from $P$ and $Q$, respectively, by redefining the arity zero components to be empty.
\item[$\blacktriangleright$] \cite[Theorem 1]{FTW}: If $P$  and $Q$ are reduced operads, then one has 
a weak equivalence between the derived mapping spaces
$$
\Sigma\Operad^{h}(P\,;\,Q)\simeq \Lambda_{\ast}\Operad^{h}(P\,;\,Q).
$$ 
\item[$\blacktriangleright$] \cite[Theorem 4.4]{BM}, \cite[Theorem 15.A]{Fre3}: If $\phi:P\rightarrow Q$ is a weak equivalence between  $\Sigma$-cofibrant operads, then the extension $\phi_{!}$ and restriction $\phi^\ast$ functors (see Section \ref{E8}) form a Quillen equivalence 
$$
\phi_{!}:Alg_{P}\rightleftarrows Alg_{Q}:\phi^{\ast}.
$$
\end{itemize}
\end{thm}

%\noindent \textit{$\bullet$ Useful propositions.}
\subsection{Properties of cofibrant operads}\label{ss:useful}

The results of this section are used in Section~\ref{C0} in order to prove that the category of bimodules is relatively left proper.

\begin{pro}\label{p:useful3}
If $P$ is a cofibrant operad, then its suboperad $P_{>0}$, obtained by forgetting its arity zero component, is also cofibrant.
\end{pro}

\begin{proof}
%The result must be well-known. We briefly sketch its proof.
Since an operad is cofibrant if and only if it is a retract of a cellular one, we can assume that $P$ is cellular: $P=\mathrm{colim}_{\alpha<\lambda} P_\alpha$, where $P_{\alpha}$ is obtained from $P_{<\alpha}:=\mathrm{colim}_{\beta<\alpha}P_\beta$ using the pushout  
\begin{equation}\label{pushouteq1}
%\vcenter{
\xymatrix{
\mathcal{F}(\partial X_{\alpha}) \ar[r] \ar[d] & \mathcal{F}(X_{\alpha}) \ar[d]\\
P_{<\alpha} \ar[r] & P_{\alpha},
}%}
%\hspace{15pt} \text{with} \hspace{15pt} 
%\left\{
%\begin{array}{lc}\vspace{9pt}
%\partial X''_{\alpha}=X''_{\alpha}=X_{\alpha} & \text{if } ar_{\alpha}=1, \\ 
%\partial X''_{\alpha}= \partial X_{\alpha} \text{ and }  X''_{\alpha}=  X_{\alpha} & \text{if } ar_{\alpha}\geq 2.
%\end{array} 
%\right.
\end{equation}
where $\mathcal{F}(-)$ is the free operadic functor (see \cite{BM} for a combinatorial description of $\mathcal{F}$ in terms of trees) and each $\partial X_\alpha\to X_\alpha$ is a generating $\Sigma$-cofibration. 
%
%
%Each $\partial X_\alpha\to X_\alpha$ is a generating $\Sigma$-cofibration.
We need to show that $P_{>0}$ is also cellular. Note that as an operad in sets, $P$ is a free operad generated by the $\Sigma$-sequence $X=\coprod_{\alpha\in\lambda} X_\alpha\setminus
\partial X_\alpha$. We claim that $P_{>0}$ is also free as an operad in sets being generated by its $\Sigma$-subsequence represented by trees, whose vertices are labelled by $X$,
with the property that only their root vertex can have leaves and in fact must have at least one leaf attached,
see Figure~\ref{fig:new_decomp}. %Indeed, if $X_{\alpha}$ is concentrated in arity $0$, then we remove $X_{\alpha}$ of the construction of $P_{>0}$ but we need to keep track of the points obtained from the composition between $X_{\alpha}$ and  $P_{\leq \alpha}$ living in arity bigger than $1$. 

\begin{figure}[!h]
%\begin{center}
\includegraphics[scale=0.4]{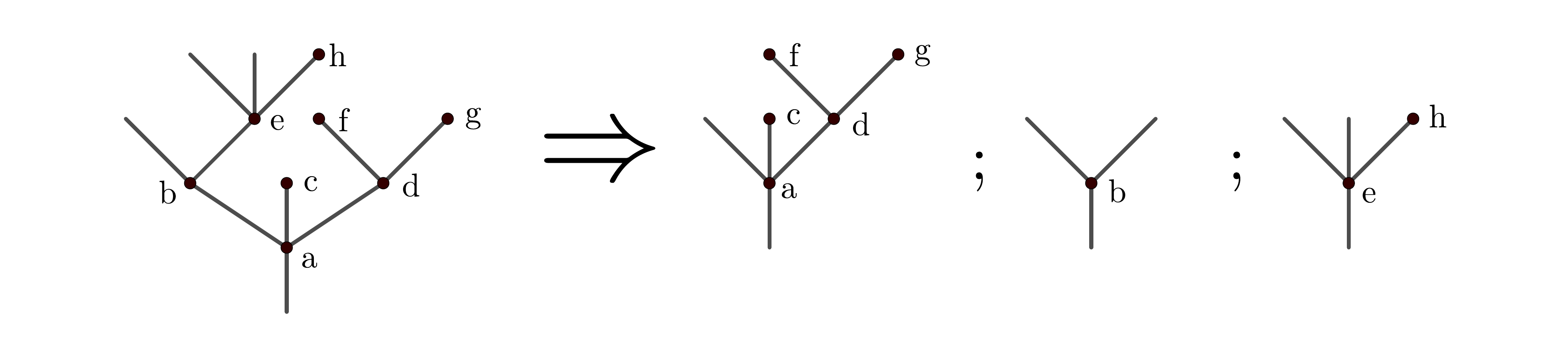}%\vspace{-15pt}
\caption{An element in a free operad and its decomposition in the new positive arity generators.}\label{fig:new_decomp}
%\end{center}
\end{figure}

This set of generating trees splits into cells of $P_{>0}$, namely by the
way from which $X_{\alpha}$'s, $\alpha\in\lambda$, the labels come from. So the set of cells can be described as the set of trees as above with vertices labelled by elements $\alpha$'s
from $\lambda$. One can define a total ordering of this set that prescribes in which order the new cells are attached. Given two such trees, we compare first the maximal elements from $\lambda$ 
they have as labels (including their arity zero vertices). If their maximal elements are the same, we compare which one has more such maximal labels. If the numbers of such labels 
are the same, we compare their next to maximal labels. And so on. If they have exactly the same sets of  labels, we put any random order between them,
or we put them together in one bigger cell.
\end{proof}
%\hspace{-65pt}

 Let $O$ be an operad. By $\Sigma_{k}\wr O(1)$, we understand the monoid that acts on $O(k)$ and is given by the following extension
$$
\xymatrix{
1 \ar[r] & O(1)^{\times k} \ar[r] & \Sigma_{k}\wr O(1) \ar[r] & \Sigma_{k} \ar[r] & 1.
}
$$

%For a monoid $\Gamma$ in $Top$ we denote by $\Gamma\text{-}Top$ the category of right $\Gamma$-modules (in $Top$). This category can be thought of as the category of algebras over an operad with only unary operations. As such it is endowed 
%with the projective model structure (see above), in which morphisms are fibrations, respectively weak equivalences, if and only if they are fibrations, 
%repectively equivalences in $Top$. Cofibrations, respectively fibrations, in this category will be called $\Gamma$-cofibrations, 
%respectively $\Gamma$-fibrations.

\begin{pro}\label{J1}
If $O$ is a cofibrant operad, then each space $O(n)$ is $\Sigma_{n}\wr O(1)$-cofibrant.
\end{pro}

%To formulate our second proposition

For each $1< k < n$, we denote by $O(n\,;\,k)$ the subspace of $O(n)$ which consists of points of the form 
$$
(\theta_{1}\circ_{1}\theta_{2})\cdot \sigma,
$$
with $\theta_{2}\in O(i)$, $2\leq i \leq k$, $\theta_{1}\in O(n-i+1)$ while $\sigma$  preserves the position of  $k+1,\ldots, n$ and shuffles $\{1,\ldots,i\}$ with $\{i+1,\ldots,k\}$. Both spaces $O(n)$ and $O(n\,;\,k)$ inherit an action of the monoid $( \Sigma_{k}\times\Sigma_{n-k})\wr O(1):=(\Sigma_k\wr O(1))\times (\Sigma_{n-k}
\wr O(1))$ and the inclusion from $O(n\,;\,k)$ into $O(n)$ is a $( \Sigma_{k}\times\Sigma_{n-k})\wr O(1)$-equivariant map.
%
% For shortness, we denote by $( \Sigma_{k}\wr O(1))\times \Sigma_{n-k}$ the monoid
%$$
%( \Sigma_{k}\wr O(1))\times \Sigma_{n-k}=( \Sigma_{k}\wr O(1))\times \Sigma_{n-k}.
%$$
%$( \Sigma_{k}\wr O(1))\times \Sigma_{n-k}$
%\newpage

\begin{pro}\label{J2}
If $O$ is a cofibrant operad, then $O(n\,;\,k)$ is  $( \Sigma_{k}\times\Sigma_{n-k})\wr O(1)$-cofibrant and the map
 \begin{equation}\label{eq:Onk}
 O(n\,;\,k)\rightarrow O(n)
 \end{equation}
  is a $( \Sigma_{k}\times\Sigma_{n-k})\wr O(1)$-cofibration. As a consequence, $O(n\,;\,k)$ is $( \Sigma_{k}\wr O(1))\times \Sigma_{n-k}$-cofibrant and
  the map~\eqref{eq:Onk} is a $( \Sigma_{k}\wr O(1))\times \Sigma_{n-k}$-cofibration.
\end{pro}

These propositions are proved by similar arguments which are both adaptation of Berger-Moerdijk's proof of \cite[Proposition~4.3]{BM} 
stating that cofibrant operads are always $\Sigma$-cofibrant. 
The latter result is obtained by  iteratively using \cite[Lemma~5.10]{BM}. Its slightly stronger version
 \cite[Lemma~2.5.3]{BM2}  is  Lemma~\ref{D2}, both being equivariant pushout-product type statements with respect to a discrete group action.
 In the Appendix we formulate and prove their %\todo{Question Benoit: en quoi ce lemma est analogue?}
  analogue -- Lemma~\ref{l:push_prod1} for topological monoids that applies to our case
 of action by monoids $\Sigma_n\wr O(1)$, $( \Sigma_{k}\times\Sigma_{n-k})\wr O(1)$,  $( \Sigma_{k}\wr O(1))\times \Sigma_{n-k}$ or alike. 

%
%
%\begin{lmm}{\cite[Proposition 2.6]{CS}}\label{D1}
%Dual to Lemma \ref{E6}, let $\mathcal{D}$ be the small category of shape $\{\xymatrix@C=15pt{\ast_{1} &\ar[r]^{c_{3}} \ast_{2} \ar[l]_{c_{1}}& \ast_{3}  }\}$ and $F_{1},F_{2}\in Func(\mathcal{D}; \mathcal{C})$ be two functors from $\mathcal{D}$ to a model category $\mathcal{C}$. Let $t:F_{1}\Rightarrow F_{2}$ be a natural transformation. The map $f$ induced by the natural transformation $t$ between the colimits:
%$$
%\xymatrix{
%\operatorname{colim}_{\mathcal{D}}F_{1}\ar[d]^{f} \ar@{=}[r] & \operatorname{colim}\big(\hspace{-31pt} & F_{1}(\ast_{1})\ar[d]_{t(\ast_{1})} &\ar[r]\ar[d]_{t(\ast_{2})} F_{1}(\ast_{2})\ar[l] & \ar[d]_{t(\ast_{3})} F_{1}(\ast_{3})\big)\\
%\operatorname{colim}_{\mathcal{D}}F_{2} \ar@{=}[r] &  \operatorname{colim}\big(\hspace{-31pt} & F_{2}(\ast_{1}) &\ar[r] F_{2}(\ast_{2})\ar[l] & F_{2}(\ast_{3})\big)
%}
%$$
%is a cofibration if the map $t(\ast_{3})$ and the map
%\begin{equation}\label{H1}
%t(\ast_{1})\cup F_{2}(c_{1}): F_{1}(\ast_{1})\underset{F_{1}(\ast_{2})}{\coprod}F_{2}(\ast_{2})\longrightarrow F_{2}(\ast_{1})
%\end{equation}
%are cofibrations. Furthermore, the map $f$ is an acyclic cofibration if the maps $t(\ast_{3})$ and (\ref{H1}) are also weak equivalences.
%\end{lmm}
%

\begin{proof}[Proof of Proposition \ref{J1}]
Let $O$ be a cofibrant operad. By Proposition~\ref{p:useful3}, without loss of generality, we can assume that $O(0)=\emptyset$. Since $O$ is cofibrant, this operad is a retract of a cellular operad $P$.   
% that is an operad obtained by a possibly transfinite composition $P=\mathrm{colim}_{\alpha< \lambda}P_{\alpha}$  of pushouts of the form 
%$$
%\xymatrix{
%\mathcal{F}(\partial X_{\alpha}) \ar[r] \ar[d] & \mathcal{F}(X_{\alpha}) \ar[d]\\
%P_{\alpha} \ar[r] & P_{\alpha +1}
%}
%$$
%where $\mathcal{F}(-)$ is the free operadic functor (see \cite{BM} for a combinatorial description of $\mathcal{F}$ using trees) and $\partial X_{\alpha}\rightarrow X_{\alpha}$ is a generating $\Sigma$-cofibration concentrated in arity $\geq 1$. In particular the operad $P$ satisfies also the condition $P(0)=\emptyset$.
 In what follows, we denote by $P_1$ the sub-operad of $P$ obtained by taking the restriction to the arity $1$: 
$$
%P'(n)
P_1(n)=\left\{
\begin{array}{cl}%\vspace{5pt}
P(1) & \text{if } n=1, \\ 
\emptyset & \text{otherwise}.
\end{array} 
\right.
$$
In the absence of arity zero operations, the cellular attachments can be reordered so that first we attach arity one cells, then arity two cells and so on. 
In particular, the map of operads $P_1\rightarrow P$ can also be seen as a cellular extension:% \vspace{5pt}
\begin{equation}\label{J6}
\xymatrix{
%P'=P_{0}\ar[r] & 
P_{1} \ar[r] &P_{2}\ar[r] &\cdots \ar[r] & P_{\alpha} \ar[r] & P_{\alpha +1} \ar[r] & \cdots \ar[r] & P,
}
\end{equation}% \vspace{5pt}
where $P_{\alpha}$, $\alpha>1$, is obtained from $P_{<\alpha}:=\operatorname{colim}_{\beta<\alpha}P_{\beta}$ using a pushout of the form~\eqref{pushouteq1}, 
%\begin{equation}\label{pushouteq1}
%%\vcenter{
%\xymatrix{
%\mathcal{F}(\partial X_{\alpha}) \ar[r] \ar[d] & \mathcal{F}(X_{\alpha}) \ar[d]\\
%P_{<\alpha} \ar[r] & P_{\alpha}
%}%}
%%\hspace{15pt} \text{with} \hspace{15pt} 
%%\left\{
%%\begin{array}{lc}\vspace{9pt}
%%\partial X''_{\alpha}=X''_{\alpha}=X_{\alpha} & \text{if } ar_{\alpha}=1, \\ 
%%\partial X''_{\alpha}= \partial X_{\alpha} \text{ and }  X''_{\alpha}=  X_{\alpha} & \text{if } ar_{\alpha}\geq 2.
%%\end{array} 
%%\right.
%\end{equation}
 where  each $\partial X_\alpha\to X_\alpha$ is a generating $\Sigma$-cofibration concentrated in arity $ar_\alpha\geq 2$.

%\noindent \textbf{Combinatorial description of the pushout \eqref{pushouteq1}.} 
For $n\geq 2$, one has $P_1(n)=\emptyset$, which is $\Sigma_n\wr P(1)$-cofibrant. 
In what follows, we will show that each map $P_{<\alpha}(n)\rightarrow P_{\alpha}(n)$ is a $\Sigma_{n}\wr P(1)$-cofibration. For this purpose, we need a combinatorial description of the pushout \eqref{pushouteq1} using the language of trees. Let $\mathbb{P}_{n}^{\geq 1}$ be the set of
planar rooted trees having exactly $n$ leaves indexed by an element of the symmetric group $\Sigma_{n}$ with internal vertices of arity $\geq 1$. According to this notation, $P_{\alpha}(n)$ is obtained from the set of trees $\mathbb{P}_{n}^{\geq 1}$ by indexing the vertices by points in $P_{<\alpha}(n)$ and $X_{\alpha}$. More precisely, one has 
\begin{equation}\label{eqpushout1}
P_{\alpha}(n)=\left.\underset{T\in \mathbb{P}_{n}^{\geq 1}}{\coprod}\,\, \underset{v\in V(T)}{\prod} \left[ P_{<\alpha}(|v|)\underset{\partial X_{\alpha}(|v|)}{\coprod}X_{\alpha}(|v|) \right]\right/\sim,
\end{equation}
where the equivalence relation is generated by the relation contracting two consecutive vertices indexed by points in $P_{<\alpha}(n)$ using its operadic structure, the compatibility with the symmetric group action and the removal of vertices indexed by the unit $\ast_{1}\in P_{<\alpha}(1)=P(1)$. Let us remark that the arity one vertices are necessarily indexed by points in $P(1)$.% since the $\Sigma$-sequences $\partial X''_{\alpha}\rightarrow  X''_{\alpha}$ are concentrated in arity $ar_{\alpha}\geq 2$ (otherwise the pushout is trivial by construction).  

%In order to remove the equivalence relation in \eqref{eqpushout1}, 
We equip $P_{\alpha}(n)$ with the filtration~\eqref{J5} given by the number of vertices indexed by $X_{\alpha}$. Let $\mathbb{T}_{n}^{\geq 2}[m]$,  $n\geq 2$ and $m\geq 1$, be the set of non-planar rooted trees with internal vertices of arity $\geq 2$, having $n$ leaves and two kinds of vertices called auxiliary and primary, respectively. The corresponding sets of vertices of a tree~$T$ are denoted by $V_{aux}(T)$ and $V_{pri}(T)$. Furthermore, we assume that  there are no consecutive primary vertices and each tree $T\in \mathbb{T}^{\geq 2}_{n}[m]$ has exactly $m$ auxiliary vertices. For any $T\in \mathbb{T}^{\geq 2}_{n}[m]$, we denote by $E_{1}(T)$ the subset of the set $E(T)$ of edges which consists of: the root edge of $T$ if it is adjacent to
an auxiliary vertex; the leaf edges of $T$ connected to auxiliary vertices; the inner edges of $T$ connecting two auxiliary vertices. According to this notation, we set
\begin{equation}\label{J8}
\overline{X}_{\alpha}(T)= % \left(
 \underset{v\in V_{pri}(T)}{\prod} P_{<\alpha}(n)(|v|) \times \underset{v\in V_{aux}(T)}{\prod} X_{\alpha}(|v|)  \times \underset{e\in E_{1}(T)}{\prod} P(1).
% \right)\underset{Aut(T)}{\times} \Sigma_{n},
\end{equation}
 The subspace $\partial \overline{X}_{\alpha}(T)\subset \overline{X}_{\alpha}(T)$ is defined as the one composed of elements having at least one auxiliary vertex indexed by $\partial X_{\alpha}$. 
The automorphism group $Aut(T)$ of the tree $T$ acts on these spaces  $\partial \overline{X}_{\alpha}(T)$ and $\overline{X}_{\alpha}(T)$
 by permuting the incoming edges of the vertices. We choose a bijection of the set $\ell(T)$ of leaves of~$T$ with the set~$[n]$.
This allows us to consider the group $Aut(T)$  as a subgroup of  the symmetric group $\Sigma_n$ observing how it permutes the $n$ leaves of $T$. 
 Denote by $Aut(T)\wr P(1)$ the submonoid of $\Sigma_n\wr P(1)$ defined as the pullback
\begin{equation}\label{eq:pb_aut}
\xymatrix@R=35pt{
Aut(T)\wr P(1)  \ar[r] \ar[d] & Aut(T) \ar[d] \\
\Sigma_n\wr P(1) \ar[r] & \Sigma_n.
}
\end{equation}
%Monoid $Aut(T)\wr P(1)$  acts on~$\partial \overline{X}_{\alpha}(T)$ and $\overline{X}_{\alpha}(T)$. Moreover,
%\begin{equation}\label{J8_wr}
%\overline{X}_{\alpha}(T)=  \left( \underset{v\in V_{pri}(T)}{\prod} P_{<\alpha}(n)(|v|) \times \underset{v\in V_{aux}(T)}{\prod} X_{\alpha}(|v|)  \times \underset{e\in E_{1}(T)}{\prod} P(1)\right)\underset{Aut(T)\wr P(1)}{\times} \Sigma_{n}\wr P(1).
%\end{equation}

One has the filtration by the number $m$ of vertices labelled by $X$:
\begin{equation}\label{J5}
\xymatrix{
P_{<\alpha}(n)= P_{\alpha}(n)_{0} \ar[r] & \cdots \ar[r] & P_{\alpha}(n)_{m-1}\ar[r] & P_{\alpha}(n)_{m} \ar[r] & \cdots \ar[r] & P_{\alpha}(n).
}
\end{equation} 
The inclusion $P_{\alpha}(n)_{m-1}\to P_{\alpha}(n)_{m}$ fits in the following pushout diagram of $\Sigma_{n}\wr P(1)$-spaces:
\begin{equation}\label{J7}
\xymatrix@R=35pt{
\underset{T\in \mathbb{T}_{n}^{\geq 2}[m]}{\displaystyle\coprod}\left( \partial \overline{X}_{\alpha}(T) \underset{Aut(T)\wr P(1)}{\times} \Sigma_{n}\wr P(1)\right)
  \ar[r] \ar[d] & \underset{T\in \mathbb{T}_{n}^{\geq 2}[m]}{\displaystyle\coprod} \left( \overline{X}_{\alpha}(T)\underset{Aut(T)\wr P(1)}{\times} \Sigma_{n}\wr P(1)\right)
  \ar[d] \\
P_{\alpha}(n)_{m-1} \ar[r] & P_{\alpha}(n)_{m}.
}
\end{equation}
%where the coproduct is taken along the set of isomorphism classes of trees with $n$ leaves. 
%In particular, $P'_{0}(n)=P'(n)=\emptyset$, $n\geq 2$, is  obviously $\Sigma_{n}\wr P(1)$-cofibrant  and 
The map $P_{\alpha}(n)_{m-1}\rightarrow P_{\alpha}(n)_{m}$ is a $\Sigma_{n}\wr P(1)$-cofibration if the upper horizontal arrow in the above diagram is one. % a $\Sigma_{n}\wr P(1)$-cofibration.
According to Lemma~\ref{l:adj_mon}, the extension functor preserves cofibrations and we are only left to showing that
every inclusion $\partial \overline{X}_{\alpha}(T)\rightarrow \overline{X}_{\alpha}(T)$ is an $Aut(T)\wr P(1)$-cofibration.\hspace{7pt}

%\begin{equation}\label{J4}
%\underset{T\in \mathbb{T}_{n}^{\geq 2}[m]}{\displaystyle\coprod} \partial \overline{X}_{\alpha}(T)\longrightarrow \underset{T\in \mathbb{T}_{n}[m]^{\geq 2}}{\displaystyle\coprod}  \overline{X}_{\alpha}(T).
%\end{equation}
%\vspace{.2cm}

\vspace{5pt}
\noindent \textbf{The map  $\partial \overline{X}_{\alpha}(T)\rightarrow \overline{X}_{\alpha}(T)$ is an $Aut(T)\wr P(1)$-cofibration.} 
%We argue by induction over the number of  vertices in $T$. 
We induct over the number of  vertices in $T$.
The base of induction is when $T$ has 0 vertices for which the statement is vacuously true. 
% is an $n$-corolla, whose only vertex must be auxiliary. 
%One has $\partial \overline{X}_{\alpha}(T)=\partial X_\alpha(n)\times P(1)^{\times n}$ and $\overline{X}_{\alpha}(T)= X_\alpha(m)\times P(1)^{\times n}$. The statement is obtained applying Lemma~\ref{l:push_prod1} to the case $\Gamma=\Sigma_n\wr P(1)$, $\Gamma_1= P(1)^m$, $\Gamma_2=G_2=
%\Sigma_n$, with $A\to B$ being $\partial X_\alpha\to X_\alpha$ and $X\to Y$ being $\emptyset\to P(1)^{\times m}$.
Now assume that $T$ has $\geq 1$ vertices and the statement holds for any other rooted tree with less vertices.
  We also assume that 
for any~$i$, $P_{<\alpha}(i)$ is $\Sigma_i\wr P(1)$-cofibrant as our argument follows a double induction. Let $r$ denote  the root vertex. Let us assume that its first $\ell$ incoming edges are grafted to trees $T_1,\ldots,T_\ell$, for some $1\leq \ell \leq |r|$, and 
the remaining $\ell-|r|$ edges connect to leaves. Consider the following short exact sequence of monoids:
\[
1\rightarrow \prod_{i=1}^\ell \bigl(Aut(T_i)\wr P(1)\bigr) \rightarrow Aut(T)\wr P(1)\rightarrow Aut(r)\times \left(\Sigma_{|r|-\ell}\wr P(1)\right)\rightarrow 1,
\]
where $Aut(r)\subset\Sigma_\ell$ is the subgroup of permutation of the first $\ell$ incoming edges of $r$ as a result of $Aut(T)$-action.
It is easy to see that this sequence is split-surjective (see Definition~\ref{d:mon_ex_seq}) and thus Lemma~\ref{l:push_prod1} can be applied.  There are two cases to consider: the vertex $r$  is auxiliary or it is primary. In both cases, $G_2=Aut(r)\times
\Sigma_{|r|-\ell}$ and we explain how to apply Lemma~\ref{l:push_prod1}.

In the first case we take $A\to B$ to be $P(1)\times \partial X_\alpha(|r|)
\times P(1)^{\times (|r|-\ell)}\to P(1)\times  X_\alpha(|r|)\times P(1)^{\times (|r|-\ell)}$, where the first factor $P(1)$ corresponds to the root edge,
while the other factors $P(1)$ correspond to the leaf edges connected to $r$.
It is an $Aut(r)\times (\Sigma_{|r|-\ell}\wr P(1))$-cofibration  by restriction (Lemma~\ref{l:adj_mon}) and also Lemma~\ref{l:push_prod1}.
For $X\to Y$ we take $\partial \prod_{i=1}^\ell  \overline{X}_{\alpha}(T_i)\rightarrow  \prod_{i=1}^\ell  \overline{X}_{\alpha}(T_i)$.
%, where 
% $\partial \prod_{i=1}^\ell  \overline{X}_{\alpha}(T_i)$ is the subset of the product with at least one coordinate in $\partial  \overline{X}_{\alpha}(T_i)$.
  It is a $\prod_{i=1}^\ell (Aut(T_i)\wr P(1))$-cofibration according to Example~\ref{ex:push_prod1}.
  Denote in this case by 
   $_{P(1)\backslash}\overline{X}_{\alpha}(T)$
  % $\overline{X}_{\alpha}(T)_{/P(1)}$
    the product~\eqref{J8} with one factor
 $P(1)$ missing, which corresponds to the root edge.  We also denote by
 $_{P(1)\backslash}\partial\overline{X}_{\alpha}(T)$
 % $\partial\overline{X}_{\alpha}(T)_{/P(1)}$
   the image of $\partial\overline{X}_{\alpha}(T)$ under the projection  
   $\overline{X}_{\alpha}(T)\to {_{P(1)\backslash}\overline{X}_{\alpha}(T)}$.
   % $\overline{X}_{\alpha}(T)\to \overline{X}_{\alpha}(T)_{/P(1)}$.
    Note that the induction
 step also allows us to conclude that the inclusion  
 $_{P(1)\backslash}\partial \overline{X}_{\alpha}(T)\rightarrow {_{P(1)\backslash}\overline{X}_{\alpha}(T)}$
% $\partial \overline{X}_{\alpha}(T)_{/P(1)}\rightarrow \overline{X}_{\alpha}(T)_{/P(1)}$
  is an $Aut(T)\wr P(1)$-cofibration.
 
 In the second case, if $r$ is primary, then we take for $A\to B$ the inclusion $\emptyset\rightarrow  P_{<\alpha}(|r|)$. It is an $Aut(r)\times \left(\Sigma_{|r|-\ell}\wr P(1)\right)$-cofibration  by restricting from $\Sigma_{|r|}\wr P(1)$ and applying Lemma~\ref{l:adj_mon}. 
 For $X\to Y$ we take the map 
 $\partial \prod_{i=1}^\ell  (_{P(1)\backslash}\overline{X}_{\alpha}(T_i))\rightarrow  \prod_{i=1}^\ell ( _{P(1)\backslash}\overline{X}_{\alpha}(T_i)
)$.

\medskip

\noindent \textbf{The space $O(n)$ is $\Sigma_n\wr O(1)$-cofibrant.} Since the operad~$O$ is a retract of the operad $P$, the monoid $\Sigma_{n}\wr O(1)$ is a retract of the monoid $\Sigma_{n}\wr P(1)$. Let $f:\Sigma_{n}\wr O(1)\rightarrow \Sigma_{n}\wr P(1)$ and $g:\Sigma_{n}\wr P(1)\rightarrow \Sigma_{n}\wr O(1)$ be these maps of monoids such that $g\circ f=id$. These maps give rise to Quillen adjunctions:
$$
\xymatrix{
\Sigma_{n}\wr O(1)\text{-}Top \ar@<0.5ex>[r]^{i_{1}} & \Sigma_{n}\wr P(1)\text{-}Top \ar@<0.5ex>[l]^{i_{2}} \ar@<0.5ex>[r]^{j_{1}} & \Sigma_{n}\wr O(1)\text{-}Top \ar@<0.5ex>[l]^{j_{2}}
}
$$
For any $A\in \Sigma_{n}\wr O(1)$-$Top$ and $B\in \Sigma_{n}\wr P(1)$-$Top$, the objects $j_{2}(A)$ and $i_{2}(B)$ are defined by $j_{2}(A)=A$ and $i_{2}(B)=B$ as spaces, with the structure operations such that:
$$
\begin{array}{ccccc}\vspace{7pt}
j_{2}(A) \times \Sigma_{n}\wr P(1) & \underset{id\times g}{\longrightarrow} & A\times \Sigma_{n}\wr O(1) & \longrightarrow & A, \\ 
i_{2}(B) \times \Sigma_{n}\wr O(1) & \underset{id\times f}{\longrightarrow} & B\times \Sigma_{n}\wr P(1) & \longrightarrow & B.
\end{array} 
$$

In order to show that $O(n)$ is $\Sigma_{n}\wr O(1)$-cofibrant, we consider the lifting problem
\begin{equation}\label{eq:J2}
\xymatrix{
\emptyset \ar[r] \ar[d] & X \ar@{->>}[d]^{\simeq} \\
O(n) \ar[r] & Y,
}
\end{equation}
where $X\rightarrow Y$ is an acyclic fibration in $\Sigma_{n}\wr O(1)$-$Top$. Then we apply the functor $j_{2}$ to Diagram \eqref{eq:J2}. The map $j_{2}(X)\rightarrow j_{2}(Y)$ is still an acyclic fibration in $\Sigma_{n}\wr P(1)$-$Top$ because the restriction functor $j_{2}$ creates fibrations and weak equivalences
(in the sense that a map $\beta$ in $ \Sigma_{n}\wr O(1)$-$Top$ is a fibration or a weak equivalence precisely if $j_2(\beta)$ is so in  $ \Sigma_{n}\wr P(1)$-$Top$). Furthermore, $O$ being a retract of $P$, the space $j_{2}(O(n))$ is a retract of $P(n)$ in $\Sigma_{n}\wr P(1)$-$Top$. Since the latter is a  $\Sigma_{n}\wr P(1)$-cofibration, the space  $j_{2}(O(n))$ is a  $\Sigma_{n}\wr P(1)$-cofibrant object and there is a map $h:j_{2}(O(n))\rightarrow j_{2}(X)$ solution of the lifting problem \eqref{eq:J2} in $\Sigma_{n}\wr P(1)$-$Top$. Finally, the map $i_{2}(h)$ provides a solution to the lifting problem \eqref{J2} due to the relation $g\circ f=id$. 
\end{proof}

\begin{proof}[Proof of Proposition~\ref{J2}]
In what follows, we overuse the notation introduced in the proof of Proposition~\ref{J1}. By Proposition~\ref{p:useful3}, without loss of generality we can assume that $O(0)=\emptyset$.
Arguing in the same way as at the end of the proof of Proposition~\ref{J1}, we can assume that $O$ is a cellular operad. Indeed, $O$ being a retract of a cellular operad $P$, the inclusion $O(n\,;\,k)\rightarrow O(n)$ is also a retract of $P(n\,;\,k)\rightarrow P(n)$.  For simplicity and to agree with the notation Proposition \ref{J1}, we assume that $O=P$. 

\medskip

Recall filtration~\eqref{J6} in $P(n)$. We consider a similar filtration in $P(n\,;\,k)$ by taking $P_{\alpha}(n\,;\,k):=P(n\,;\,k)\cap P_{\alpha}(n)$.
Similarly, we filter the inclusion $P(n\,;\,k)\to P(n)$ 
taking $P'_\alpha(n):=P(n\,;\,k)\cup P_{\alpha}(n)$. Denote by $P_{<\alpha}(n\,;\,k):=\mathrm{colim}_{\beta<\alpha}P_{\beta}(n\,;\,k)$
and $P'_{<\alpha}(n):=\mathrm{colim}_{\beta<\alpha}P'_{\beta}(n)$.  In order to prove the proposition, one has to show that the inclusions 
$P_{<\alpha}(n\,;\,k)\rightarrow P_{\alpha}(n\,;\,k)$ and $P'_{<\alpha}(n)\rightarrow P'_{\alpha}(n)$ are $(\Sigma_k\times\Sigma_{n-k})\wr P(1)$-cofibrations. 

\medskip

This is done by filtering these inclusions similarly to~\eqref{J5}. Namely, we define $P_{\alpha}(n\,;\,k)_m:=P(n\,;\,k)\cap P_{\alpha}(n)_m$ and
$P'_\alpha(n)_m:=P(n\,;\,k)\cup P_{\alpha}(n)_m$. We are left to showing that the inclusions 
$P_{\alpha}(n\,;\,k)_{m-1}\rightarrow P_{\alpha}(n\,;\,k)_m$ and $P'_\alpha(n)_{m-1}\rightarrow P'_{\alpha}(n)_m$ are $(\Sigma_k\times \Sigma_{n-k})\wr P(1)$-cofibrations. 

\medskip

Let $\mathbb{T}_{n,k}^{\geq 2}[m]$, $1<k<n$,  denote the set of exactly the same trees as in $\mathbb{T}_{n}^{\geq 2}[m]$ in which in addition 
 $k$ (out of~$n$) leaves are marked as special. For $T\in \mathbb{T}_{n,k}^{\geq 2}[m]$, we denote by $U(T)\in\mathbb{T}_{n}^{\geq 2}[m]$
 the tree obtained by forgetting which leaves are special. Let $Aut(T)$ denote the group of automorphisms of $T$. For each tree $T\in \mathbb{T}_{n,k}^{\geq 2}[m]$ we choose  an ordering of its leaves so that we count special leaves first. This ordering
 gives us an inclusion $Aut(T)\to\Sigma_n$ (as well as an inclusion $Aut(U(T))\to\Sigma_n$).  One obviously has
  $Aut(T)=Aut(U(T))\cap(\Sigma_k\times\Sigma_{n-k})$. We similarly define 
 $$
 Aut(T)\wr P(1):=\bigl(Aut(U(T))\wr 
 P(1)\bigr)\cap\bigl((\Sigma_k\times\Sigma_{n-k})\wr P(1)\bigr).
 $$
Consider the subset $ \mathbb{T}_{n,k}^{\geq 2}[m]_I\subset  \mathbb{T}_{n,k}^{\geq 2}[m]$ composed of trees that contain a vertex
whose all incoming edges are special leaf edges. Its complement is denoted by $ \mathbb{T}_{n,k}^{\geq 2}[m]_{II}:=
\mathbb{T}_{n,k}^{\geq 2}[m]\setminus \mathbb{T}_{n,k}^{\geq 2}[m]_I$.

One has similar to~\eqref{J7} pushout diagrams of $\bigl(\Sigma_{k}\wr P(1)\times\Sigma_{n-k}\wr P(1)\bigr)$-spaces.
\begin{equation}\label{eq:push_filt_I}
\xymatrix{
\underset{T\in \mathbb{T}_{n,k}^{\geq 2}[m]_I}{\displaystyle\coprod} \partial \overline{X}_{\alpha}(T) \underset{Aut(T)\wr P(1)}{\times} \bigl((\Sigma_{k}\times\Sigma_{n-k})\wr P(1)\bigr)
  \ar[r] \ar[d] & \underset{T\in \mathbb{T}_{n,k}^{\geq 2}[m]_I}{\displaystyle\coprod}  \overline{X}_{\alpha}(T)\underset{Aut(T)\wr P(1)}{\times} \bigl((\Sigma_{k}\times\Sigma_{n-k})\wr P(1)\bigr)
  \ar[d] \\
P_{\alpha}(n\,;\,k)_{m-1} \ar[r] & P_{\alpha}(n\,;\,k)_{m}.
}
\end{equation}
\begin{equation}\label{eq:push_filt_II}
\xymatrix{
\underset{T\in \mathbb{T}_{n,k}^{\geq 2}[m]_{II}}{\displaystyle\coprod} \partial \overline{X}_{\alpha}(T) \underset{Aut(T)\wr P(1)}{\times} \bigl((\Sigma_{k}\times\Sigma_{n-k})\wr P(1)\bigr)
  \ar[r] \ar[d] & \underset{T\in \mathbb{T}_{n,k}^{\geq 2}[m]_{II}}{\displaystyle\coprod}  \overline{X}_{\alpha}(T)\underset{Aut(T)\wr P(1)}{\times} \bigl((\Sigma_{k}\times\Sigma_{n-k})\wr P(1)\bigr)
  \ar[d] \\
P'_{\alpha}(n)_{m-1} \ar[r] & P'_{\alpha}(n)_{m}.
}
\end{equation}
In the above $ \overline{X}_{\alpha}(T)$ is defined by~\eqref{J8}.

The upper horizontal arrows in diagrams~\eqref{eq:push_filt_I} and~\eqref{eq:push_filt_II} are $\bigl(\Sigma_{k}\wr P(1)\times\Sigma_{n-k}\wr P(1)\bigr)$-cofibrations. Indeed, from the proof of Proposition~\ref{J1} we know that each inclusion $\partial\overline{X}_{\alpha}(T)
\rightarrow \overline{X}_{\alpha}(T)$ is an $Aut(U(T))\wr P(1)$-cofibration. Applying Lemma~\ref{l:adj_mon} to 
the restriction along $Aut(T)\wr P(1)\to Aut(U(T))\wr P(1)$, we get that such inclusion is an $Aut(T)\wr P(1)$-cofibration. 
Applying again this lemma to the induction along $Aut(T)\wr P(1)\to (\Sigma_k\times\Sigma_{n-k})\wr P(1)$ and using the fact that 
the diagrams~\eqref{eq:push_filt_I} and~\eqref{eq:push_filt_II}  are pushout squares, we conclude that  
$P_{\alpha}(n\,;\,k)_{m-1}\rightarrow P_{\alpha}(n\,;\,k)_m$ and $P'_\alpha(n)_{m-1}\rightarrow P'_{\alpha}(n)_m$ are $(\Sigma_k\times
\Sigma_{n-k})\wr P(1)$-cofibrations. Applying again Lemma~\ref{l:adj_mon} to the restriction along $(\Sigma_k\wr P(1))\times\Sigma_{n-k}
\rightarrow  (\Sigma_k\times\Sigma_{n-k})\wr P(1)$, these maps are also $(\Sigma_k\wr P(1))\times\Sigma_{n-k}$-cofibrations.
\end{proof}

\section{The projective model category of $(P\text{-}Q)$-bimodules}\vspace{5pt}

Let $P$ and $Q$ be two operads. A $(P\text{-}Q)$-bimodule is a $\Sigma$-sequence $M\in \Sigma Seq$ together with operations 
\begin{equation}\label{A2}
\begin{array}{llr}\vspace{7pt}
\gamma_{r}: & M(n)\times \underset{1\leq i \leq n}{\prod} Q(m_{i})\longrightarrow  M\big(\,\sum_{i}\,m_{i}\,\big), & \text{right operations},\\ \vspace{7pt}
\gamma_{\ell}: & P(n)\times \underset{1\leq i \leq n}{\prod} M(m_{i})\longrightarrow  M\big(\,\sum_{i}\,m_{i}\,\big),& \text{left operations},\vspace{-0pt}
\end{array}
\end{equation}
satisfying the following relations, with $\sigma\in \Sigma_{n}$ and $\tau_{i}\in \Sigma_{m_{i}}$:
$$
\underset{\text{Associativity for the right operations}}{\xymatrix@C=12pt{
M(n)\hspace{-1pt}\times\hspace{-5pt} \underset{1\leq i\leq n}{\prod}\hspace{-5pt}Q(m_{i})\hspace{-1pt}\times\hspace{-5pt} \underset{\substack{1\leq i \leq n\\ 1\leq j\leq m_{i}}}{\prod}\hspace{-8pt}Q(k_{i,j})\ar[r] \ar[d] & M(n)\hspace{-1pt}\times\hspace{-5pt} \underset{1\leq i\leq n}{\prod}\hspace{-5pt} Q(\sum_{j}\, k_{i,j}) \ar[d]\\
M(\sum_{i}\, m_{i})\times \hspace{-8pt}\underset{\substack{1\leq i \leq n\\ 1\leq j\leq m_{i}}}{\prod}\hspace{-8pt} Q(k_{i,j}) \ar[r] & M(\sum_{i,j}\, k_{i,j})
}}\hspace{20pt}
\underset{\text{Associativity for the left operations}}{\xymatrix@C=12pt{
P(n)\hspace{-1pt}\times\hspace{-5pt} \underset{1\leq i\leq n}{\prod}\hspace{-5pt}P(m_{i})\hspace{-1pt}\times\hspace{-5pt} \underset{\substack{1\leq i \leq n\\ 1\leq j\leq m_{i}}}{\prod}\hspace{-8pt}M(k_{i,j})\ar[r] \ar[d] & P(n)\hspace{-1pt}\times\hspace{-5pt} \underset{1\leq i\leq n}{\prod}\hspace{-5pt} M(\sum_{j}\, k_{i,j}) \ar[d]\\
P(\sum_{i}\, m_{i})\times \hspace{-8pt}\underset{\substack{1\leq i \leq n\\ 1\leq j\leq m_{i}}}{\prod}\hspace{-8pt} M(k_{i,j}) \ar[r] & M(\sum_{i,j}\, k_{i,j})
}}\vspace{5pt}
$$
$$
\underset{\text{Compatibility between the left and right operations}}{\xymatrix@C=12pt{
P(n)\hspace{-1pt}\times\hspace{-5pt} \underset{1\leq i\leq n}{\prod}\hspace{-5pt}M(m_{i})\hspace{-1pt}\times\hspace{-5pt} \underset{\substack{1\leq i \leq n\\ 1\leq j\leq m_{i}}}{\prod}\hspace{-8pt}Q(k_{i,j})\ar[r] \ar[d] & P(n)\hspace{-1pt}\times\hspace{-5pt} \underset{1\leq i\leq n}{\prod}\hspace{-5pt} M(\sum_{j}\, k_{i,j}) \ar[d]\\
M(\sum_{i}\, m_{i})\times \hspace{-8pt}\underset{\substack{1\leq i \leq n\\ 1\leq j\leq m_{i}}}{\prod}\hspace{-8pt} Q(k_{i,j}) \ar[r] & M(\sum_{i,j}\, k_{i,j})
}}\hspace{40pt}
\underset{\text{Compatibility with the unit of the operad}}{\xymatrix@R=43pt{
M(n)\times Q(1)\ar[dr]_{\gamma_{\ell}} & M(n) \ar[r] \ar[l] \ar@{=}[d] & P(1)\times M(n) \ar[dl]^{\gamma_{r}}\\
& M(n) & 
}}\vspace{15pt}
$$
$$
\underset{\text{Right compatibility with the symmetric group action}}{\xymatrix{
M(n)\hspace{-1pt}\times\hspace{-5pt} \underset{1\leq i\leq n}{\prod}\hspace{-5pt}Q(m_{i})\ar[r] \ar[d]_{\sigma^{\ast}\times \prod_{i} (\tau_{i})^{\ast}} & M(m_{1}+\cdots +m_{n}) \ar[d]^{(\sigma(\tau_{1},\ldots,\tau_{n}))^{\ast}}\\
M(n)\hspace{-1pt}\times\hspace{-5pt} \underset{1\leq i\leq n}{\prod}\hspace{-5pt}Q(m_{\sigma(i)}) \ar[r] & M(m_{1}+\cdots +m_{n})
}}\hspace{60pt}
\underset{\text{Left compatibility with the symmetric group action}}{\xymatrix{
P(n)\hspace{-1pt}\times\hspace{-5pt} \underset{1\leq i\leq n}{\prod}\hspace{-5pt}M(m_{i})\ar[r] \ar[d]_{\sigma^{\ast}\times \prod_{i} (\tau_{i})^{\ast}} & M(m_{1}+\cdots +m_{n}) \ar[d]^{(\sigma(\tau_{1},\ldots,\tau_{n}))^{\ast}}\\
P(n)\hspace{-1pt}\times\hspace{-5pt} \underset{1\leq i\leq n}{\prod}\hspace{-5pt}M(m_{\sigma(i)}) \ar[r] & M(m_{1}+\cdots +m_{n})
}}\vspace{15pt}
$$
where the permutation $\sigma(\tau_{1},\ldots,\tau_{n})$ is obtained from the well known operadic compositions on the symmetric groups (see \cite[Proposition 1.1.9]{Fre1}).

As part of the left operations, there is a map $\gamma_{0}:P(0)\rightarrow M(0)$ in arity $0$. A map between $(P\text{-}Q)$-bimodules should preserve these operations. We denote by  $\Sigma\Bimod_{P\,;\,Q}$ the category of $(P\text{-}Q)$-bimodules.  
%In the special case $P=Q=O$, we denote by  $\Sigma\Bimod_{O}$ the category of $(O\text{-}O)$-bimodules.
 Thanks to the unit in $Q(1)$, the right operations $\gamma_{r}$ can equivalently be defined as a family of continuous maps
$$
\circ^{i}:M(n)\times Q(m)\longrightarrow M(n+m-1),\hspace{15pt}\text{with }1\leq i\leq n.
$$

Given an integer $r\geq 0$, we also consider the category of $r$-truncated bimodules $T_{r}\Sigma\Bimod_{P\,;\,Q}$. An object is an $r$-truncated $\Sigma$-sequence endowed with left and right operations (\ref{A2}) under the conditions $n\leq r$ and $\sum m_{i}\leq r$ for $\gamma_{r}$ and the condition $\sum m_{i}\leq r$ for $\gamma_{\ell}$. One has an obvious truncation functor 
$$
T_{r}(-):\Sigma\Bimod_{P\,;\,Q}\longrightarrow T_{r}\Sigma\Bimod_{P\,;\,Q}.
$$
In the rest of the paper, we use the notation
$$
\begin{array}{rl}\vspace{5pt}
x\circ^{i}q=\circ^{i}(x;q), & \text{for } x\in M(n) \text{ and } q\in Q(m), \\ 
p(x_{1},\ldots,x_{n})=\gamma_{\ell}(p,x_{1},\ldots,x_{n}), & \text{for } p\in P(n) \text{ and } x_{i}\in M(m_{i}).
\end{array} 
$$

\begin{expl}
If $\eta:P\rightarrow Q$ is a map of operads, then $\eta$ is also a map of $P$-bimodules. Indeed, any operad is a bimodule over itself while the $P$-bimodule structure on $Q$ is given by the following formulas:
$$
\begin{array}{lcl}\vspace{7pt}
\circ^{i}: Q(n)\times P(m)\longrightarrow  Q(m+n-1)& ; & (q;p)\longmapsto q\circ_{i}\eta(p), \\ 
\gamma_{\ell}:  P(n)\times Q(m_{1})\times \cdots \times Q(m_{n}) \longrightarrow  Q(m_{1}+\cdots + m_{n})& ; & (p,q_{1},\ldots,q_{n})\longmapsto (\cdots ((\eta(p)\circ_{n}q_{n})\circ_{n-1}q_{n-1})\cdots)\circ_{1}q_{1}.
\end{array} 
$$
\end{expl}

\subsection{Properties of the category of bimodules}\label{Z6}

In this subsection we introduce some basic properties related to the category of $(P\text{-}Q)$-bimodules where $P$ and $Q$ are two fixed operads. First, we show that the category of $(P\text{-}Q)$-bimodules is equivalent to the category of algebras over an explicit colored operad denoted by $P{+}Q$. Thereafter,  we build the free bimodule functor using the language of trees. Using this explicit construction of the free functor, we are able to give a combinatorial description of the pushout for bimodules.

\subsubsection{Bimodules as algebras over a colored operad}\label{D8}

Given two operads $P$ and $Q$, we build a colored operad $P{+}Q$ such that the category of $(P\text{-}Q)$-bimodules is equivalent to the category of $(P{+}Q)$-algebras. By a colored operad $C$, with set of colors $S$, we understand a family of spaces $C=\{C(s_{1},\ldots,s_{n};s_{n+1}),\, n\geq 0,\, s_{i}\in S\}$,
equipped with an action of permutations such that $\sigma^*: C(s_{1},\ldots,s_{n};s_{n+1})\rightarrow C(s_{\sigma(1)},\ldots,s_{\sigma(n)};s_{n+1})$, for $\sigma\in\Sigma_n$,
with units $\ast_{s}\in C(s;s)$ for $s\in S$, and with operadic compositions
$$
\circ_{i}:C(s_{1},\ldots,s_{n};s_{n+1})\times C(s'_{1},\ldots,s'_{m};s_{i})\longrightarrow C(s_{1},\ldots,s_{i-1},s'_{1},\ldots,s'_{m},s_{i+1},\ldots, s_{n};s_{n+1}),
$$
that satisfy associativity, unit and equivariance relations which are similar to the ones introduced in 
Section~\ref{E4}.
An algebra over $C$ is a family of spaces $\{X_{s},\, s\in S\}$ together with operations of the form
$$
\alpha[s_{1},\ldots,s_{n};s_{n+1}]:C(s_{1},\ldots,s_{n};s_{n+1})\times X_{s_{1}}\times \cdots \times X_{s_{n}}\longrightarrow X_{s_{n+1}},
$$
compatible with the operadic structure.

The operad $P+Q$ has the set of non-negative integers $S = \mathbb{N}$ as colors so that an element $\theta\in(P+Q)(n_1,...,n_k;m)$
governs an operation of the form $\theta: M(n_1)\times\dots\times M(n_k)\rightarrow M(m)$
on a $(P\text{-}Q)$-bimodule $M$.
We define this operad $P+Q$ by a presentation by generators and relations. We take two kinds of generating operations,
which respectively encode the left $P$-action and the right $Q$-action
of a $(P\text{-}Q)$-bimodule structure.
We actually integrate the first kind of generators in a colored operad $P_1$ (encoding the left operations) and the second kind of generators in another colored operad $Q_1$ (encoding the right operations), which is concentrated in arity one.
We explain the definition of these operads $P_1$ and $Q_1$ and their action on a $(P\text{-}Q)$-bimodule $M$ in the next paragraph.
We shape the composite elements of the operad $P+Q$ on trees equipped with two sets of vertices, left vertices, represented by diamonds $\blacklozenge$, which correspond to the $P_1$-factors,
and right vertices, represented by circles $\bullet$, which correspond to the $Q_1$-factors.
We describe the structure of these trees in a second step and we explain afterwards the definition of the operad $P+Q$
with elements shaped on such trees, moded out by relations which reflect the structure relations
of left and right actions on bimodules. (We can actually identify $P+Q$ with a coproduct of the colored operads $P_1$
and $Q_1$ moded out by extra relations which reflect
the commutation of left and right actions
on a bimodule.)

\begin{defi}\label{D5-bis}
\textbf{The colored operads $P_{1}$ and $Q_{1}$}
\begin{itemize}[leftmargin=12pt]
\item[$\blacktriangleright$]
To any collection of non-negative integers $n_{1},\ldots, n_{k}$, we associate the space $P_{1}(n_{1},\ldots,n_{k};n_{1}+\cdots + n_{k})$
such that:
$$
P_{1}(n_{1},\ldots,n_{k};n_{1}+\cdots + n_{k}):=P(k).
$$
We have an action of permutations $\sigma^*: P_{1}(n_{1},\ldots,n_{k};n_{1}+\cdots + n_{k})\rightarrow P_{1}(n_{\sigma(1)},\ldots,n_{\sigma(k)};n_{1}+\cdots + n_{k})$,
induced by the symmetric structure of the operad, unit elements $*_n\in P_1(n,n)$, given by the unit of $P$,
and ``colored'' composition operations
\begin{equation}\label{D4}
\begin{array}{rcl}\vspace{5pt}
\circ_{i}:P_{1}(n_{1},\ldots,n_{k};m)\times P_{1}(n'_{1},\ldots,n'_{k'};n_{i}) & \longrightarrow &  P_{1}(n_{1},\ldots,n_{i-1},n'_{1},\ldots,n'_{k'},n_{i+1},\ldots,n_{k};m), \\
p\,;\,p' & \longmapsto & p\circ_{i}p',
\end{array}
\end{equation}
induced by the composition operations of $P$, for $n_{i}=n'_{1}+\cdots + n'_{k'}$ and $m=n_{1}+\cdots+ n_{k}$.
We immediately see that these structure operations fulfill the unit, associativity and equivariance axioms of colored operads (as a consequence of the operad
axioms in $P$).
We also see that the left $P$-action on a $(P$-$Q)$-bimodule $M$ gives an operation
$$
P_{1}(n_{1},\ldots,n_{k};n_{1}+\cdots + n_{k})\times M(n_1)\times\dots\times M(n_k)\rightarrow M(n_1+\dots+n_k),
$$
which makes $M$ and algebra over the colored operad $P_1$.
\item[$\blacktriangleright$]
To any integers $n$ and $m$, we associate the space $Q_{1}(n\,;\,m)$ such that:
$$
Q_{1}(n\,;\,m):=\underset{\alpha:[m]\rightarrow [n]}{\coprod}\,\,\, \underset{i\in [n]}{\prod} \, Q(|\alpha^{-1}(i)|).
$$
A point in the above space is denoted by $(\alpha\,;\, \{q_{i}\}_{i\in [n]})$.
We have a unit element $*_n\in Q_{1}(n\,;\,n)$ given by collection of operadic units $1\in Q(1)$
in the factor of $Q_{1}(n\,;\,n)$ indexed by the identity map $id: [n]\rightarrow [n]$.
We have compositions of the form
\begin{equation}\label{C9}
\begin{array}{rcl}\vspace{5pt}
 \mu: Q_{1}(n\,;\,m)\times Q_{1}(m\,;\, \ell) & \longrightarrow & Q_{1}(n\,;\,\ell); \\
(\alpha\,;\, \{q_{i}\})\,;\, (\alpha'\,;\, \{q'_{j}\}) & \longmapsto & (\alpha\circ\alpha'\,;\, \{z_{i}\}_{i\in [n]}),
\end{array}
\end{equation}
where $z_{i}$, with $i\in [n]$, is obtained using the operadic structure of $Q$:
$$
z_{i}:=\sigma_i(\alpha,\alpha)^*\left[\big(\cdots \big(\,\big(\, q_{i}\circ_{l_{i}}q'_{b_{\ell_{i}}}\,\big)\circ_{\ell_{i}-1} q'_{b_{l_{i-1}}}\,\big) \cdots \big)\circ_{1} q'_{b_{1}}\right],\hspace{15pt} \text{with } \alpha^{-1}(i)=\{b_{1}<\cdots < b_{\ell_{i}}\}.
$$
Here $\sigma_i(\alpha,\alpha)^*$ is the inverse of the shuffle permutation of the set $(\alpha\circ\alpha')^{-1}(i)$, reordering this set as $(\alpha')^{-1}(b_1),\ldots,(\alpha')^{-1}(b_{\ell_i})$.
The operadic axioms imply that these operations are associativity and unital, so that $Q_{1}$
forms a colored operad (concentrated in arity one).
We also see that the right $Q$-action on a $(P-Q)$-bimodule $M$ gives an operation
$$
\zeta: Q_{1}(n\,;\,m)\times M(n)\rightarrow M(m)
$$
by
$$
\zeta\big( (\alpha\,;\,\{q_{i}\}),x\big) = \sigma_\alpha^*\left(\big(\cdots\big(\,x\circ^{n} q_{n}\big)\cdots\big)\circ^{1}q_{1}\right),
$$
for any $(\alpha\,;\,\{q_{i}\})\in Q_{1}(n\,;\,m)$, $x\in M(n)$, where $\sigma_\alpha\in\Sigma_m$ is the inverse of the corresponding shuffle
of the blocks of size $|\alpha^{-1}(1)|,\ldots,|\alpha^{-1}(n)|$
in~$[m]$.
We readily deduce from the bimodule axioms that this operation is unital and associative with respect to the composition operation (\ref{C9}),
so that $M$ forms an algebra over the colored operad $Q_1$.
\end{itemize}
\end{defi}

\begin{defi}\textbf{The set of trees $\mathbb{P}[n_{1},\ldots,n_{k};m]$}

\noindent Let $n_{1},\ldots, n_{k}$ and $m$ be non-negative integers. An element in $\mathbb{P}[n_{1},\ldots,n_{k};m]$ is a tuple $T=(T,V_{\ell}(T),V_{r}(T),f)$, where $T$ is a planar rooted tree having $k$ leaves indexed by a permutation from $\Sigma_{k}$ and having two kinds of vertices called \textit{left} and \textit{right} vertices, respectively. The sets $V_{\ell}(T)$ and $V_{r}(T)$ consist of left vertices and right vertices, respectively. In particular, right vertices are necessarily of arity one and are represented by circles~$\bullet$ in the tree while the left vertices are represented by diamonds~$\blacklozenge$. Left vertices are allowed to be of any arity $\geq 0$.

\begin{figure}[!h]
\begin{center}
\includegraphics[scale=0.4]{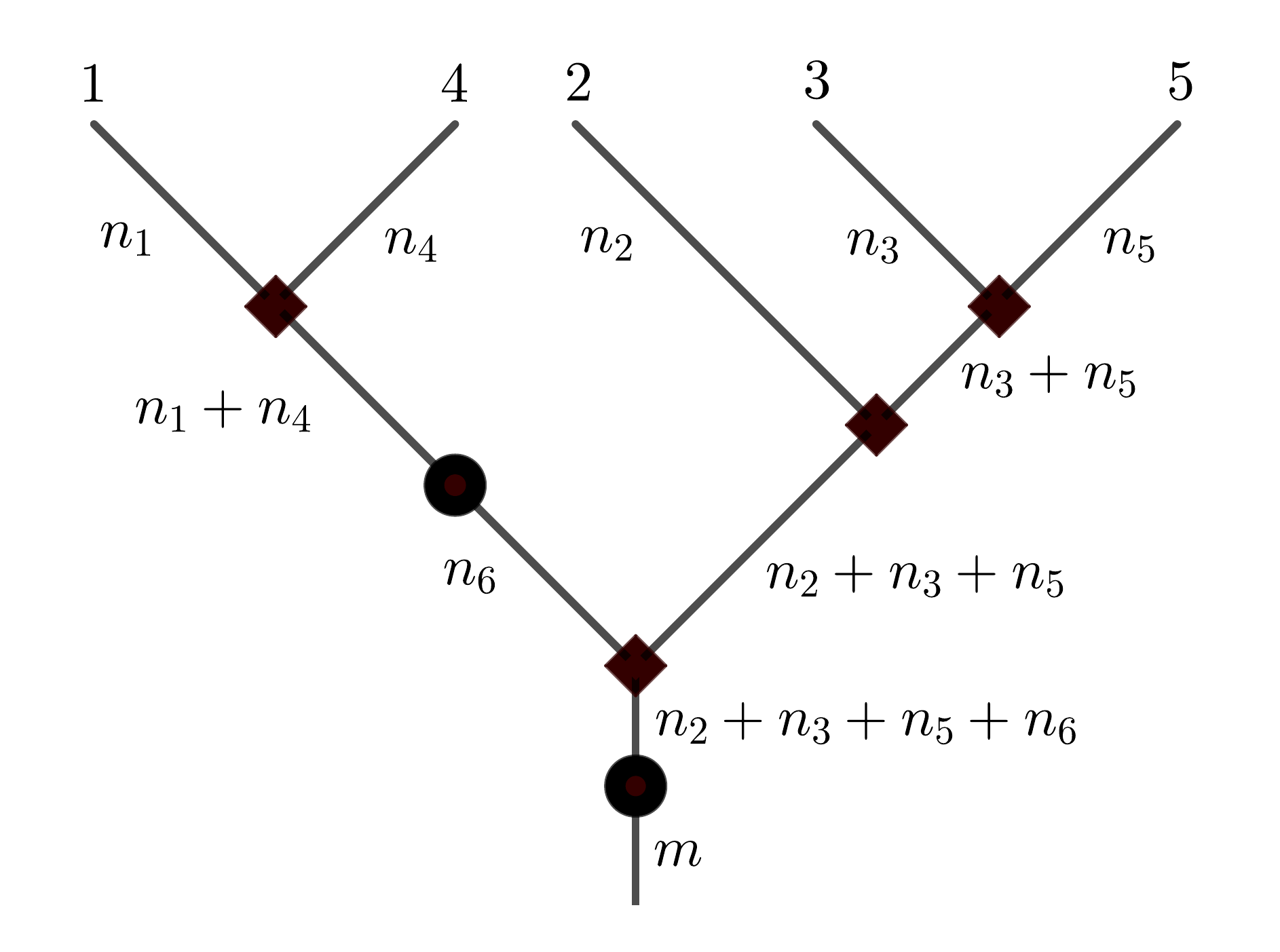}%\vspace{-10pt}
\caption{Illustration of an element in $\mathbb{P}[n_{1},n_{2},n_{3},n_{4},n_{5};m]$.}
\end{center}
\end{figure}

The assignment $f:E(T)\rightarrow \mathbb{N}$ labels the edges of the planar tree by integers. In particular, the outgoing edge of the tree is labelled by $m$ while the $k$ leaves are labelled by $n_{1},\ldots, n_{k}$, according to the permutation. We denote by $n_{1}^{v},\ldots, n_{|v|}^{v}$ the integers labelling the incoming edges of a vertex $v$ and by $n_{0}^{v}$ its output edge according to the orientation toward the root. Furthermore, if $v$ is a left vertex, then one has the relation  $n_{0}^{v}=n_{1}^{v}+\cdots + n_{|v|}^{v}$.\end{defi}

\begin{const}\label{D5}
The colored operad $P{+}Q$, with set of colors $\mathbb{N}$ (non-negative integers), is obtained from the sets of trees $\mathbb{T}[n_{1},\ldots,n_{k};m]$ by indexing the left vertices by points in $P_{1}$ and the right vertices by points in $Q_{1}$. More precisely, one has
$$
(P{+}Q)(n_{1},\ldots,n_{k};m):= \left. \underset{T\in \mathbb{T}[n_{1},\ldots,n_{k};m]}{\coprod}\, \left[\underset{v\in V_{r}(T)}{\prod} \, Q_{1}(n_{1}^{v}\,;\,n_{0}^{v}) \,\,\times \underset{v\in V_{\ell}(T)}{\prod}\, P_{1}(n_{1}^{v},\ldots, n_{|v|}^{v}\,;\, n_{0}^{v})\right]\right/ \!\!\!\sim \vspace{10pt}
$$
where the equivalence relation is generated by the following axioms:

\newpage

\begin{itemize}[leftmargin=12pt]
\item[$\blacktriangleright$] If a vertex is indexed by the unit of the operad $P_1$ or $Q_1$, then we remove it:

\begin{center}
\includegraphics[scale=0.3]{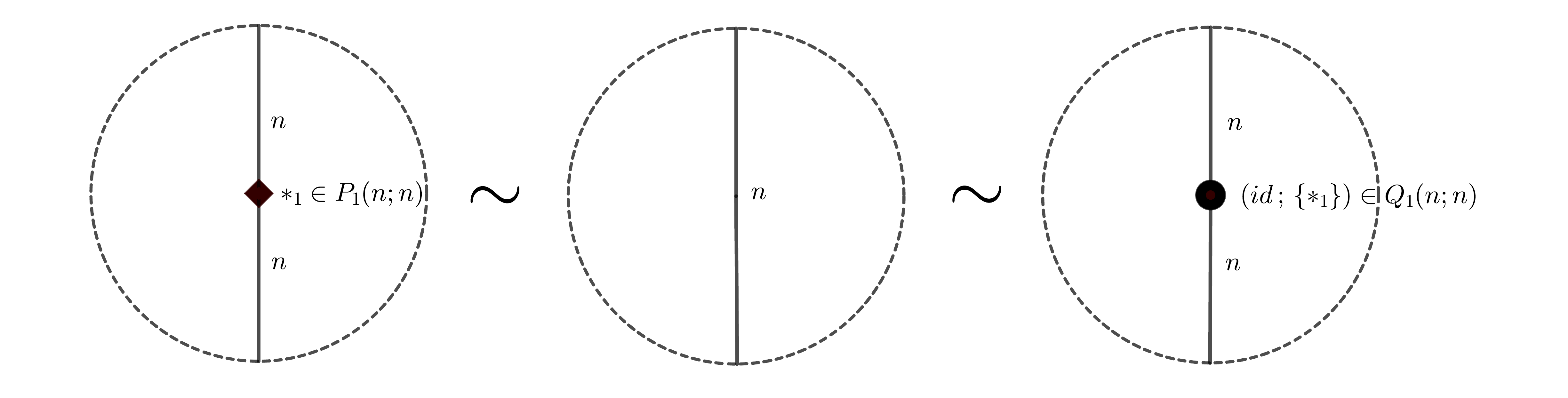}
\end{center}

\item[$\blacktriangleright$] If a left vertex is indexed by a point of the form $p\cdot \sigma$, with $p\in P(k)$ and $\sigma\in \Sigma_{k}$, then one has the following identity in which $\tau=\sigma(id_{\Sigma_{n_{1}}},\ldots,id_{\Sigma_{n_{k}}})\in \Sigma_{n_{1}+\cdots+n_{k}}$ is the element permuting the blocks $\{1,\ldots,n_{1}\},\ldots,\{n_{1}+\cdots +n_{k-1}+1,\ldots,n_{1}+\cdots+n_{k}\}$ in $[n_{1}+\cdots + n_{k}]$:

\begin{figure}[!h]
\begin{center}
\includegraphics[scale=0.32]{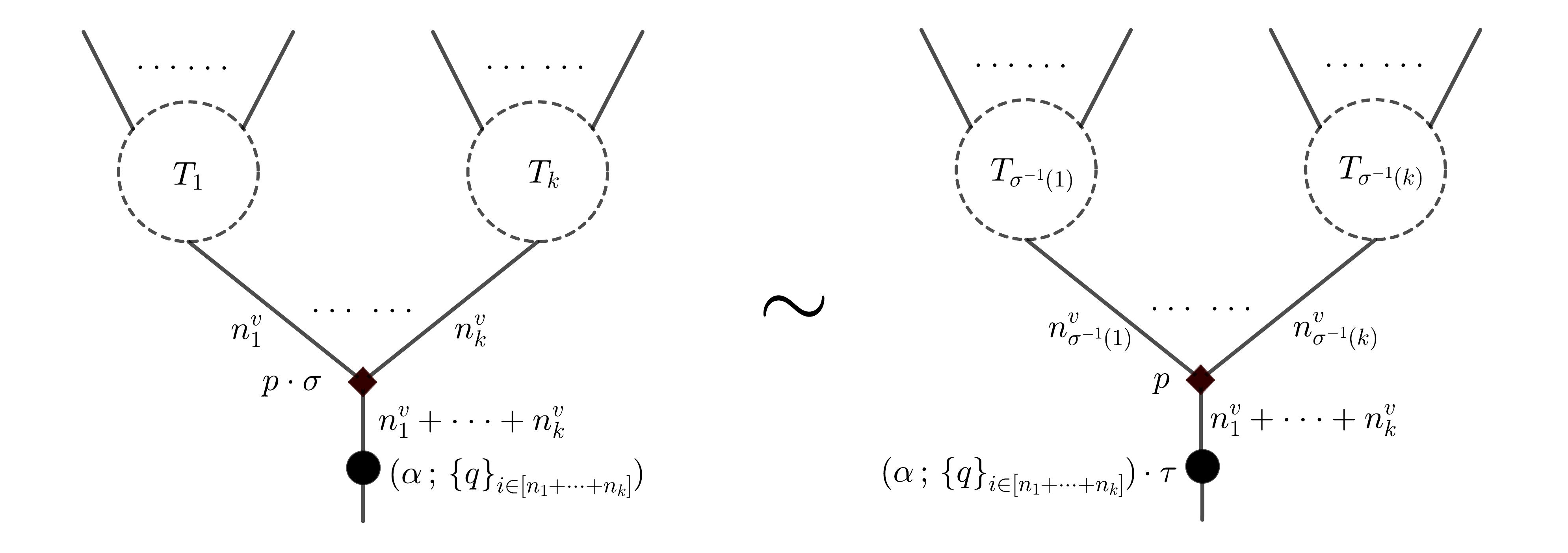}
\end{center}
\end{figure}

\item[$\blacktriangleright$] If there are two consecutive left vertices or two consecutive right vertices, then we contract the edge connecting them using the operation (\ref{C9}) or (\ref{D4}):

%\begin{figure}[!h]
\begin{center}
\includegraphics[scale=0.32]{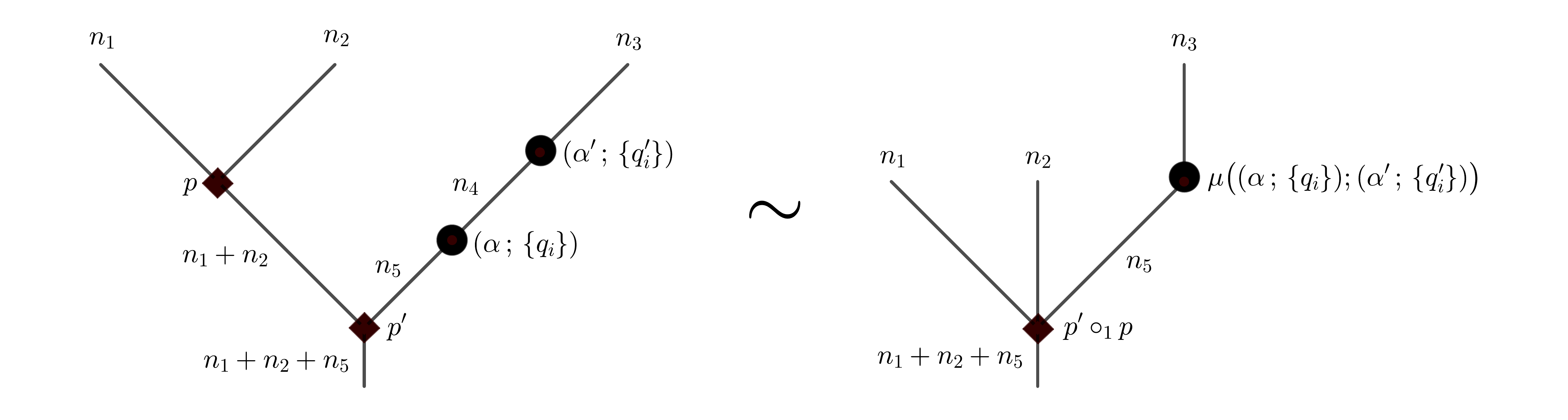}
\end{center}
%\end{figure}

\item[$\blacktriangleright$] If the incoming edges of a left vertex are connected to right vertices, then we can permute them:

\begin{figure}[!h]
\begin{center}
\includegraphics[scale=0.32]{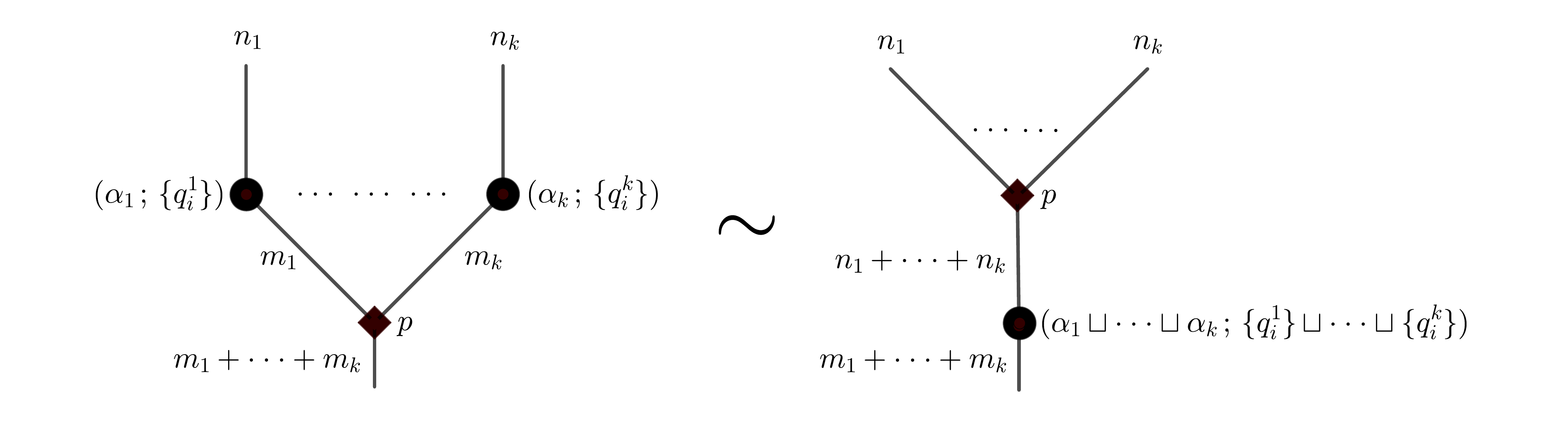}
\end{center}
\end{figure}
\end{itemize}

Let $x$ and $y$ be two points in $(P{+}Q)(n_{1},\ldots,n_{k};m)$ and $(P{+}Q)(n'_{1},\ldots,n'_{k};n_{i})$, respectively. The operadic composition $x\circ_{i}y$ consists in grafting the tree indexing $y$ into the $i$-th leaf of the tree indexing $x$ and keeping the labels of each tree in order to decorate the new one.\vspace{15pt}

\hspace{-52pt}
\includegraphics[scale=0.34]{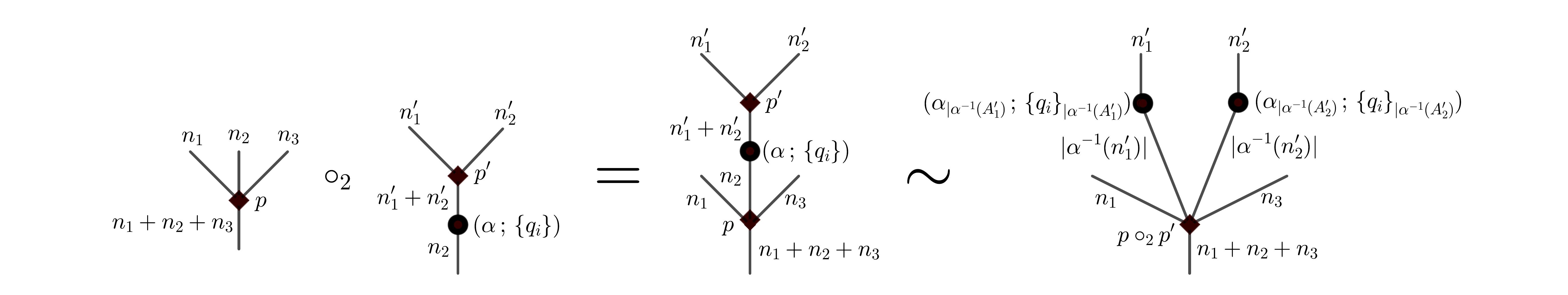}

\begin{figure}[!h]
%\begin{center}
\caption{Illustration of the operadic composition $\circ_{2}$.}
%\end{center}
\end{figure}
\end{const}

%\newpage

\begin{rmk}\label{D6}
According to the relations introduced in the previous construction, for any integers $n_{1},\ldots, n_{k}$ and $m$, each point in the component $(P{+}Q)(n_{1},\ldots,n_{k};m)$ has a representative element having exactly one left vertex and one right vertex such that the right vertex is the root of the tree,
see Figure~\ref{fig:P+Q}. Thanks to the second relation of Construction \ref{D5}, one can order the leaves from $1$ to $k$. Therefore,
as a space,
\begin{equation}\label{eq:P+Q}
(P{+}Q)(n_{1},\ldots,n_{k};m)=P(k)\times\left[ \underset{\alpha:[m]\rightarrow [n_{1}+ \cdots+n_{k}]}{\coprod}\,\, \underset{1\leq i \leq n_{1}+\cdots + n_{k}}{\prod} \,\, Q(|\alpha^{-1}(i)|)\right].
\end{equation}

\begin{figure}[!h]
\begin{center}
\includegraphics[scale=0.37]{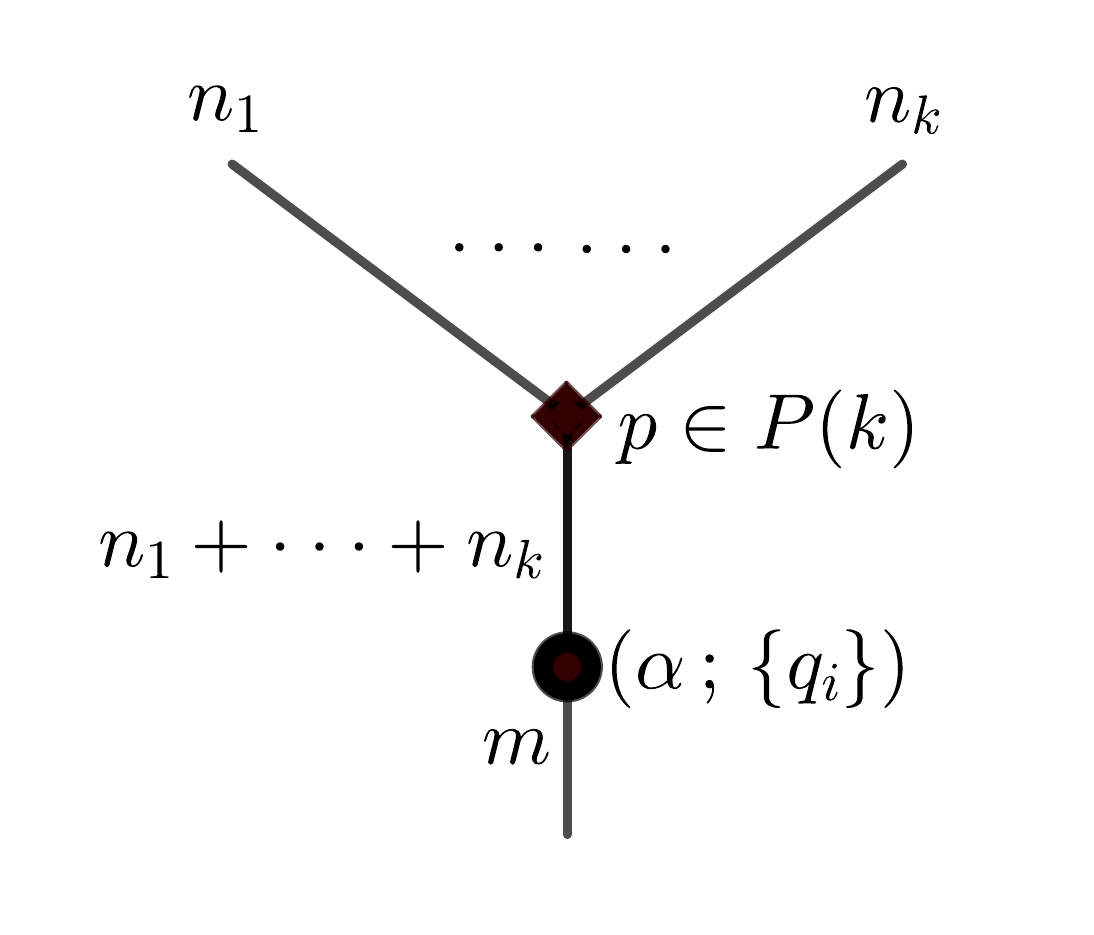}
\caption{Illustration of a representative element.}\label{fig:P+Q}%\vspace{-20pt}
\end{center}
\end{figure}
\end{rmk}

\begin{pro}
The category of $(P\text{-}Q)$-bimodules is equivalent to the category of $(P{+}Q)$-algebras.
\end{pro}

\begin{proof}
In one direction, we already explained that a $(P\text{-}Q)$-bimodule $M$ inherits an action of the operads $P_1$ and $Q_1$ that generate $P+Q$.
We just check that these actions are compatible with the relations of Construction~\ref{D5}
when we compose them to get an action of $P+Q$ on $M$ (the compatibility with the first three relations follows from the unit, associativity and equivariance
of the action of the operads $P_1$ and $Q_1$ on $M$, the compatibility with the fourth relation
follows from the commutation of the left and right actions).

In the converse direction, we show that a $(P{+}Q)$-algebra $M'$ inherits a bimodule structure.
First, $M'$ is a $\Sigma$-sequence thanks to the action of the left vertices indexed by $(\sigma\,;\,\{*_1\})$, where $\sigma$
is a permutation from some $\Sigma_k$ and $\{ *_1\}$ is a collection of $k$ identity elements from $Q(1)$.
Then we just inverse the constructions of Definition~\ref{D5-bis}
in order to retrieve left and right operations on $M'$.
To be explicit, in order to define the left operations $\gamma_{\ell}$,
we consider the inclusion $\iota_{\ell}: P(k)\rightarrow (P{+}Q)(n_{1},\ldots,n_{k};n_{1}+ \cdots + n_{k})$, for any integers $n_{1},\ldots,n_{k}$,
sending a point $p\in P(k)$ to the $k$-corolla whose root is a left vertex indexed by $p$.
The map $\gamma_{\ell}$ is defined as the following composite:
$$
\xymatrix{
P(k)\times M'(n_{1})\times \cdots \times M'(n_{k}) \ar[r]^{\gamma_{\ell}} \ar[d]_{\iota_{\ell}\times id \times \cdots \times id} & M'(n_{1}+ \cdots + n_{k})\\
(P{+}Q)(n_{1},\ldots,n_{k};n_{1}+ \cdots + n_{k})\times M'(n_{1})\times \cdots \times M'(n_{k}) \ar[ru] &
}
$$
Similarly, for any pair of integers $(n\,;\,m)$ and $1\leq i\leq n$, there is a map $\iota_{r}:Q(m)\rightarrow (P{+}Q)(n\,;\, n+m-1)$ sending a point $q\in Q(m)$ to the element $(\alpha\,;\,\{q_{j}\})$ where $\alpha$ is given by
$$
\alpha:[n+m-1]\longrightarrow [n]\,\,;\,\,j\longmapsto \left\{
\begin{array}{ll}
j & \text{if } j\leq i, \\
i & \text{if } i<j<i+m, \\
j-m+1 & \text{if } j\geq i+m.
\end{array}
\right.
$$
One has $q_{i}=q$ and $q_{j}=\ast_{1}$ for any $j\neq i$. The right operation $\circ^{i}$, with $1\leq i \leq n$, is defined as the following composite map:
$$
\xymatrix{
M'(n)\times Q(m) \ar[r]^{\circ^{i}} \ar[d]_{id\times \iota_{r}} & M'(n+m-1)\\
M'(n)\times (P{+}Q)(n\,;\, n+m-1) \ar[r]_{\cong} & (P{+}Q)(n\,;\, n+m-1)\times M'(n) \ar[u]
}
$$
The relations introduced in Construction \ref{D5} readily imply that these operations satisfy the bimodule axioms.
\end{proof}

\subsubsection{The free bimodule functor}\label{C2} 

We denote by $\Sigma Seq_{P}$ and $T_{r}\Sigma Seq_{P}$ the categories of $\Sigma$-sequences and $r$-truncated $\Sigma$-sequences $M$ equipped with a map $\gamma_{0}:P(0)\rightarrow M(0)$. In other words, if $P_{0}$ is the $\Sigma$-sequence given by $P_{0}(0)=P(0)$ and the empty set otherwise, then one has the following identities:
\begin{equation}\label{Z3}
\Sigma Seq_{P}:= P_{0}\downarrow\Sigma Seq \hspace{15pt}\text{and} \hspace{15pt} T_{r}\Sigma Seq_{P}:= P_{0}\downarrow T_{r}\Sigma Seq.
\end{equation}
Furthermore, there is a forgetful functor from the category of (possibly truncated) bimodules to the category of $\Sigma$-sequences endowed with a map from $P_{0}$:
\begin{equation}\label{A8}
\mathcal{U}^\Sigma:\Sigma\Bimod_{P\,;\,Q} \longrightarrow\Sigma Seq_{P}:\hspace{15pt}\text{and} \hspace{15pt} \mathcal{U}^{T_{r}\Sigma}:T_{r}\Sigma\Bimod_{P\,;\,Q} \longrightarrow T_{r}\Sigma Seq_{P}.
\end{equation}

Their left adjoints denoted $\mathcal{F}_{P\,;\,Q}^\Sigma$ and  $\mathcal{F}_{P\,;\,Q}^{T_{r}\Sigma}$, respectively, are some versions of free functors. As usual in the operadic theory, the free functor can be described as a coproduct indexed by a particular set of trees. In that case, we use the set of \textit{trees with section} which are pairs $T=(T\,;\,V^{p}(T))$, where $T$ is a planar rooted tree (whose leaves are labelled by some permutation) and $V^{p}(T)$ is a subset of vertices, called \textit{pearls}, satisfying the following condition: each path from a leaf or a univalent vertex to the root passes through a unique pearl.

\begin{figure}[!h]
\begin{center}
\includegraphics[scale=0.35]{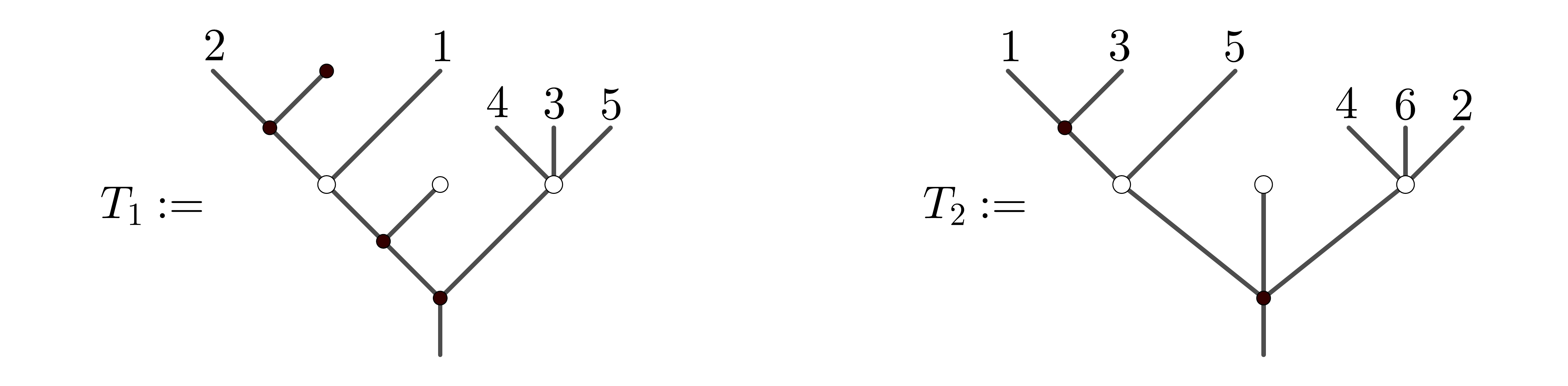}\vspace{-10pt}
\caption{Examples of a tree with section $T_{1}\in s\mathbb{P}_{5}$ and a reduced tree with section $T_{2}\in rs\mathbb{P}_{6}$.}
\end{center}
\end{figure}

The set of pearls forms a section cutting the tree into two parts. 
Note that the arity zero vertices can be either above or on the section, but never below.
We denote by $V^{u}(T)$ (respectively, $V^{d}(T)$) the vertices above the section (respectively below the section). A tree with section is said to be \textit{reduced} if each non-pearl  vertex is connected to a pearl by an inner edge.  We denote by $s\mathbb{P}_{n}$ and $rs\mathbb{P}_{n}$ the sets of trees with section and reduced trees with section, respectively, having exactly $n$ leaves labelled by a permutation of $\{1\ldots n\}$.

\begin{const}\label{F9}
Let $M=\{M(n)\}$ be a $\Sigma$-sequence equipped with a map $\gamma_{0}:P(0)\rightarrow M(0)$. The space $\mathcal{F}_{P\,;\,Q}^\Sigma(M)(n)$ is obtained from the set of reduced trees with section by indexing the pearls by points in $M$ whereas the  vertices above the section (respectively below the section) are indexed by points in the operad $Q$ (respectively the operad $P$). More precisely, one has 
\begin{equation}\label{G4}
\mathcal{F}_{P\,;\,Q}^\Sigma(M)(n)=\left. \left(
\underset{T\in rs\mathbb{P}_{n}}{\coprod}\hspace{5pt} \underset{p\in V^{p}(T)}{\prod} M(|p|)\times \underset{v\in V^{d}(T)}{\prod} P(|v|) \times \underset{v\in V^{u}(T)}{\prod} Q(|v|)  \right)\,\,
\right/\!\!\sim.
\end{equation}
A point in $\mathcal{F}_{P\,;\,Q}^\Sigma(M)$ is denoted by $[T\,;\,\{m_{p}\}\,;\,\{p_{v}\}\,;\,\{q_{v}\}]$ where $T$ is a reduced tree with section while $\{m_{p}\}_{p\in V^{p}(T)}$, $\{p_{v}\}_{v\in V^{d}(T)}$ and $\{q_{v}\}_{v\in V^{u}(T)}$ are points in $M$, $P$ and $Q$, respectively. The equivalence relation is generated by the following relations:
%\newpage

\begin{itemize}
\item[$i)$] \textit{The unit relation}: if a vertex is indexed by the unit of the operad $P$ or $Q$, then we can remove it: \vspace{-5pt}
\begin{figure}[!h]
\begin{center}
\includegraphics[scale=0.4]{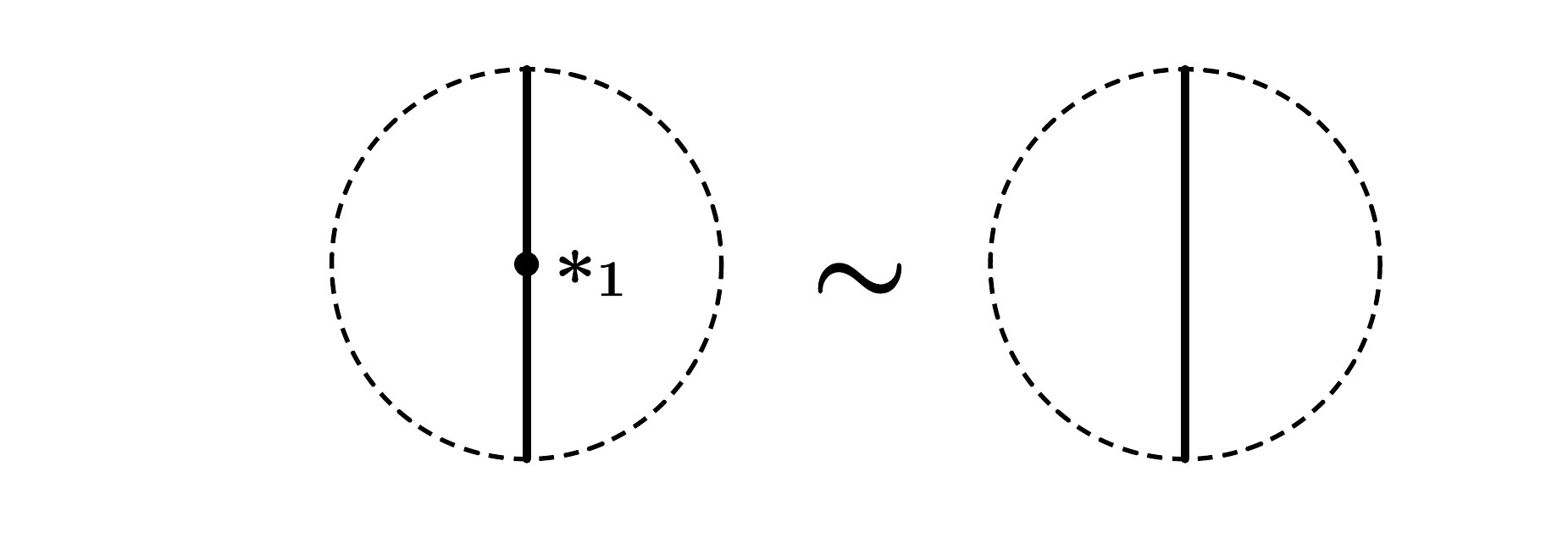}\vspace{-15pt}
%\caption{Illustration of the unit relation.}%\vspace{-20pt}
\end{center}
\end{figure}

\item[$ii)$] \textit{The compatibility with the symmetric group action}: if a vertex is labelled by $x\cdot \sigma$, with $x$ a point in  $P(n)$, $Q(n)$ or $M(n)$ and $\sigma\in \Sigma_{n}$, then we can remove $\sigma$ by permuting the incoming edges. \vspace{-5pt}
\begin{figure}[!h]
\begin{center}
\includegraphics[scale=0.48]{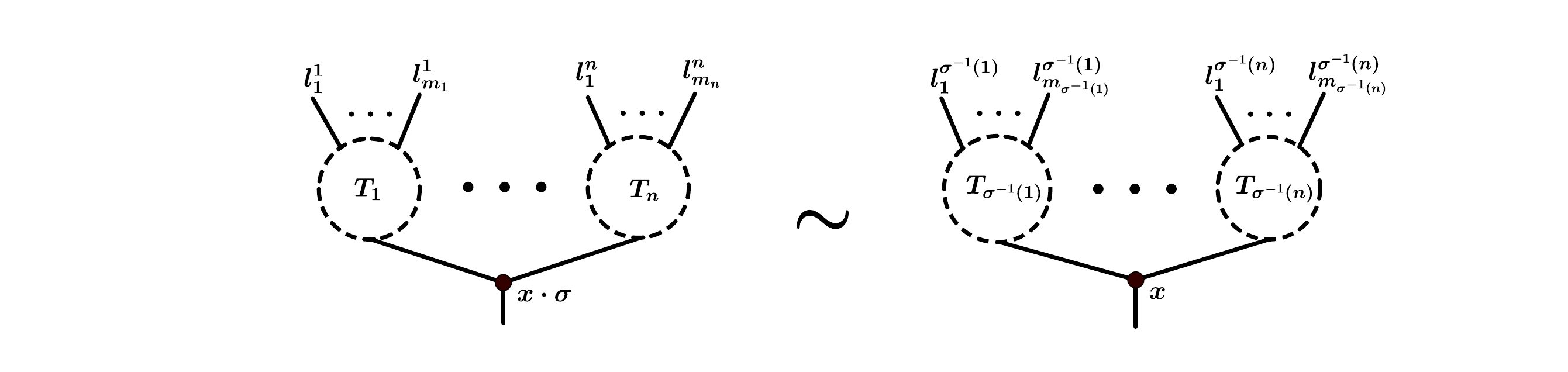}\vspace{-10pt}
\caption{Illustration of the compatibility with the symmetric group.}\vspace{-5pt}
\end{center}
\end{figure}

\item[$iii)$] \textit{The $\gamma_0$-relation}: if a pearl is indexed by a point of the form $\gamma_{0}(x)$, with $\gamma_0:P(0)\rightarrow M(0)$ and $x\in P(0)$, then we contract its output edge using the operadic structures of $P$. In particular, if the vertex below the section indexed by $p\in P(n)$ is connected only to univalent pearls indexed by $\gamma_{0}(p_{1}),\ldots,\gamma_{0}(p_{n})$, respectively, then we can contract all the incoming edges. The new vertex so obtained is a pearl indexed by $\gamma_{0}((\cdots ((p\circ_{n}p_{n})\circ_{n-1}p_{n-1})\cdots )\circ_{1}p_{1})$. \vspace{-5pt}
\begin{figure}[!h]
\begin{center}
\includegraphics[scale=0.3]{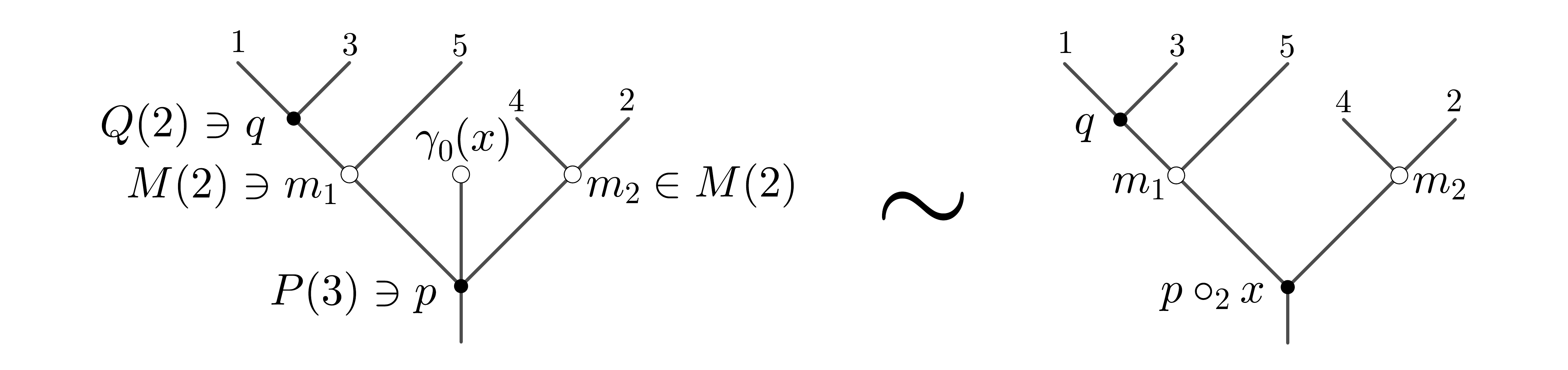}\vspace{-10pt}
\caption{Illustration of the $\gamma_0$-relation.}%\vspace{-20pt}
\end{center}
\end{figure}
\end{itemize}

The right operation $\circ^{i}$ with an element $q\in Q(m)$ consists in grafting the $m$-corolla indexed by $q$ into the $i$-th leaf of the reduced tree with section $T$. If the so obtained element contains an inner edge joining two consecutive vertices other than a pearl, then we contract it using the operadic structure of $Q$.

The left operation between an element $p\in P(n)$ and a family of points
$[T_{i}\,;\,\{m_{p}^{i}\}\,;\,\{p_{v}^{i}\}\,;\,\{q_{v}^{i}\}]\in \mathcal{F}_{P\,;\,Q}^\Sigma(M)$, with $1\leq i\leq n$, is defined as follows: each tree $T_{i}$, with $1\leq i\leq n$, is grafted from left to right to a leaf of the $n$-corolla whose vertex is indexed by $p$. If the so obtained element contains inner edges joining two consecutive vertices other than pearls, then we contract them using operadic structure of $P$. 

\hspace{-83pt}\includegraphics[scale=0.57]{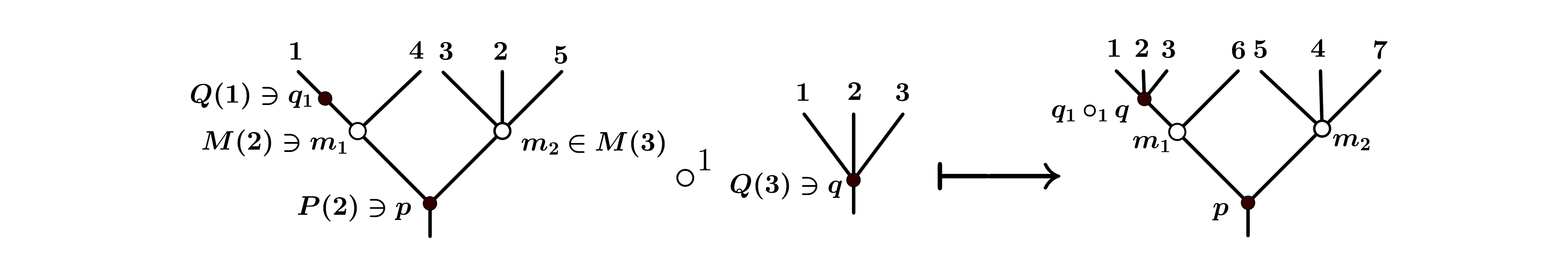}\vspace{-20pt}

\begin{figure}[!h]
\caption{Illustration of the right operation $\circ^{1}:\mathcal{F}_{P\,;\,Q}^\Sigma(M)(5)\times Q(3)\rightarrow \mathcal{F}_{P\,;\,Q}^\Sigma(M)(7)$.}\label{G7}%\vspace{-30pt}
\end{figure}

%\newpage

Finally, one has the following map sending $p\in P(0)$ to the element $[T\,;\,\{\gamma_{0}(p)\}\,;\,\emptyset\,;\,\emptyset]$ where $T$ is the pearled $0$-corolla whose root is indexed by $\gamma_{0}(p)$: 
$$
\gamma'_{0}:P(0) \longrightarrow \mathcal{F}_{P\,;\,Q}^\Sigma(M)(0).
$$

Similarly, the free $r$-truncated bimodule functor $\mathcal{F}_{P;Q,}^{T_{r}\Sigma}$ is obtained from the formula (\ref{G4}) by taking the restriction of the coproduct to the reduced trees with section  having at most $r$ leaves and such that each pearl has at most $r$ incoming edges. The equivalence relation, the left and right operations and the map $\gamma'_{0}$ are defined in the same way.  Finally, one has two functors:
$$
\mathcal{F}_{P\,;\,Q}^\Sigma:\Sigma Seq_{P}\longrightarrow\Sigma \Bimod_{P\,;\,Q}\hspace{15pt}\text{and} \hspace{15pt} \mathcal{F}_{P\,;\,Q}^{T_{r}\Sigma}:T_{r}\Sigma Seq_{P}\longrightarrow T_{r}\Sigma\Bimod_{P\,;\,Q}.
$$
\end{const}

\begin{thm}\label{THMadj}
One has the following adjunctions:
\begin{equation}\label{G5}
\mathcal{F}_{P\,;\,Q}^\Sigma:\Sigma Seq_{P}\rightleftarrows \Sigma \Bimod_{P\,;\,Q}:\mathcal{U}^\Sigma\hspace{15pt}\text{and} \hspace{15pt} \mathcal{F}_{P\,;\,Q}^{T_{r}\Sigma}:T_{r}\Sigma Seq_{P}\rightleftarrows T_{r}\Sigma\Bimod_{P\,;\,Q}:\mathcal{U}^{T_{r}\Sigma}.
\end{equation}
\end{thm}

\begin{proof}
Let $M'$ be a $(P\text{-}Q)$-bimodule and $f:M\rightarrow M'$ be a morphism in the category $\Sigma Seq_{P}$. One has to prove that there exists a unique map of $(P\text{-}Q)$-bimodules $\tilde{f}:\mathcal{F}_{P\,;\,Q}^\Sigma(M)\rightarrow M'$ such that the following diagram commutes:
\begin{equation}\label{b1}
\xymatrix{
M \ar[r]^{f} \ar[d]_{i} & M' \\
\mathcal{F}_{P\,;\,Q}^\Sigma(M)\ar@{-->}[ru]_{\exists\, !\, \tilde{f}} & 
}
\end{equation}

We build the map $\tilde{f}$ by induction on the number of vertices in the set $np(T)=V(T)\setminus V^{p}(T)$. Let $[(T\,;\,\sigma)\,;\,\{m_{p}\}\,;\,\{p_{v}\}\,;\,\{q_{v}\}]$ be a point in $\mathcal{F}_{P\,;\,Q}^\Sigma(M)$ such that $|nb(T)|=0$ and $\sigma$ is the permutation indexing the leaves of $T$. By construction, $T$ is necessarily a pearl corolla with only one vertex labelled by $m_{r}\in M$. Due to the commutativity of Diagram (\ref{b1}), the following equality has to be satisfied:
$$
\tilde{f}([(T\,;\,\sigma)\,;\,\{m_{p}\}\,;\,\{p_{v}\}\,;\,\{q_{v}\}])=f(m_{r})\cdot \sigma.
$$

Let $[(T\,;\,\sigma)\,;\,\{m_{p}\}\,;\,\{p_{v}\}\,;\,\{q_{v}\}]$ be a point in $\mathcal{F}_{P\,;\,Q}^\Sigma(M)$ where $T$ has only one vertex $v$ which is not a pearl. There are two cases to consider. If $v$ is the root of the tree $T$, then the root is labelled by a point $p_{v}\in P$ and $[(T\,;\,\sigma)\,;\,\{m_{p}\}\,;\,\{p_{v}\}\,;\,\{q_{v}\}]$ has a decomposition of the form 
$$
p_{v}(\,[(T_{1}\,;\,id)\,;\,\{m_{1}\}\,;\,\emptyset\,;\,\emptyset],\ldots, [(T_{|v|}\,;\,id)\,;\,\{m_{|v|}\}\,;\,\emptyset\,;\,\emptyset]\,)\cdot\sigma,
$$ 
where $T_{i}$ is a pearl corolla labelled by $m_{i}\in M$. Since $\tilde{f}$ has to be a $(P\text{-}Q)$-bimodule map, one has the equality
$$
\tilde{f}([(T\,;\,\sigma)\,;\,\{m_{p}\}\,;\,\{p_{v}\}\,;\,\{q_{v}\}])=p_{v}\big(\, f(m_{1}),\ldots, f(m_{|v|})\,\big)\cdot \sigma.
$$

If the root is a pearl, then there exists a unique inner edge $e$ such that $s(e)=v$ and $t(e)=r$. So, the point $[(T\,;\,\sigma)\,;\,\{m_{p}\}\,;\,\{p_{v}\}\,;\,\{q_{v}\}]$ has a decomposition on the form $([(T_{1}\,;\,id)\,;\,\{m_{p}\}\,;\,\emptyset\,;\,\emptyset]\circ^{i}q_{s(e)}) \cdot \sigma$ with $q_{s(e)}\in Q$ and $m_{t(e)}\in M$. Since $\tilde{f}$ has to be an $(P\text{-}Q)$-bimodule map, there is the equality
$$
\tilde{f}([(T\,;\,\sigma)\,;\,\{m_{p}\}\,;\,\{p_{v}\}\,;\,\{q_{v}\}])= \big(\,f(m_{t(e)})\circ^{i}q_{s(e)}\,\big)\cdot\sigma.
$$

Assume $\tilde{f}$ has been defined for $|np(T)|\leq n$. Let $[(T\,;\,\sigma)\,;\,\{m_{p}\}\,;\,\{p_{v}\}\,;\,\{q_{v}\}]$ be a point in $\mathcal{F}_{P\,;\,Q}^\Sigma(M)$ such that $|np(T)|= n+1$. By definition, there is an inner edge $e$ whose target vertex is a pearl. So, the point $[(T\,;\,\sigma)\,;\,\{m_{p}\}\,;\,\{p_{v}\}\,;\,\{q_{v}\}]$ has a decomposition of the form $([(T_{1}\,;\,id)\,;\,\{m_{p}\}\,;\,\{p_{v}\}\,;\,\{q_{v}\}\setminus \{q_{s(e)}\} ]\circ^{i} q_{s(e)})\cdot\sigma$ where $T_{1}$ is a planar tree with section such that $|np(T_{1})|= n$. Since $\tilde{f}$ has to be a $(P\text{-}Q)$-bimodule map, there is the equality
$$
\tilde{f}([(T\,;\,\sigma)\,;\,\{m_{p}\}\,;\,\{p_{v}\}\,;\,\{q_{v}\}])=\big(\,\tilde{f}([(T_{1}\,;\,id)\,;\,\{m_{p}\}\,;\,\{p_{v}\}\,;\,\{q_{v}\}\setminus \{q_{s(e)}\} ])\circ^{i}q_{s(e)}\,\big)\cdot\sigma.
$$ 
Due to the $(P\text{-}Q)$-bimodule axioms, $\tilde{f}$ does not depend on the choice of the decomposition and $\tilde{f}$ is a $(P\text{-}Q)$-bimodule map. The uniqueness follows from the construction. Similarly, we can prove that the functor $\mathcal{F}_{P\,;\,Q}^{T_{r}\Sigma}$ is the left adjoint to the forgetful functor.
\end{proof}

\subsubsection{Combinatorial description of the pushout}\label{B7}

Let $P$ and $Q$ be two topological operads. In the following, we use the notation introduced for the free bimodule functor in order to give an explicit description of the pushout in the category of $(P\text{-}Q)$-bimodules. This description will be used in the next subsections. We fix the following diagram in the category of $(P\text{-}Q)$-bimodules
\begin{equation}\label{G2}
\xymatrix{
A\ar[r]^{f_{1}} \ar[d]_{f_{2}} & C\\
B &
}
\end{equation}
Then we consider the $\Sigma$-sequence $D$ obtained from $A$, $B$, $C$ and the sets $s\mathbb{P}_n$ of trees with section (see Section \ref{C2}) by indexing the pearls of such trees by points in $B$ or $C$ whereas the other vertices below the section (respectively above the section) are indexed by points in the operad $P$ (respectively the operad $Q$). More precisely, one has 
\begin{equation}\label{G1}
D(n)=\left.\left(\underset{T\in s\mathbb{P}_{n}}{\coprod} \,\,\,\underset{p\in V^{p}(T)}{\prod} \left( B(|p|)\underset{A(|p|)}{\bigsqcup} C(|p|)\right) \times  \underset{v\in V^{d}(T)}{\prod} P(|v|)\times  \underset{v\in V^{u}(T)}{\prod} Q(|v|) \right)\right/\sim.
\end{equation}
By abuse of notation, we denote by $[T\,;\,\{m_{p}\}\,;\,\{p_{v}\}\,;\,\{q_{v}\}]$ a point in $D(n)$. The equivalence relation is generated
by relations (i), (ii), (iii) in Construction \ref{F9}  (unit relation, compatibility with the symmetric group action and $\gamma_0$-relation). Furthermore, one has also the following relations:
\begin{itemize}
\item[$iv)$] \textit{Pushout relation 1}: each inner edge, which is not connected to a pearl, is contracted using the operadic structures of $P$ and $Q$. 
\item[$v)$] \textit{Pushout relation 2}:  every inner edge above the section connected to a pearl is contracted using the right $Q$-module structures of $B$ and $C$.
\item[$vi)$] \textit{Pushout relation 3}: if a vertex $v$ below the section is connected to pearls indexed by points in $B$ (respectively, $C$), then we contract the incoming edges of $v$ using the $P$-module structure of $B$ (respectively, the left $P$-module structure of $C$).
\end{itemize}
\hspace{-35pt}\includegraphics[scale=0.38]{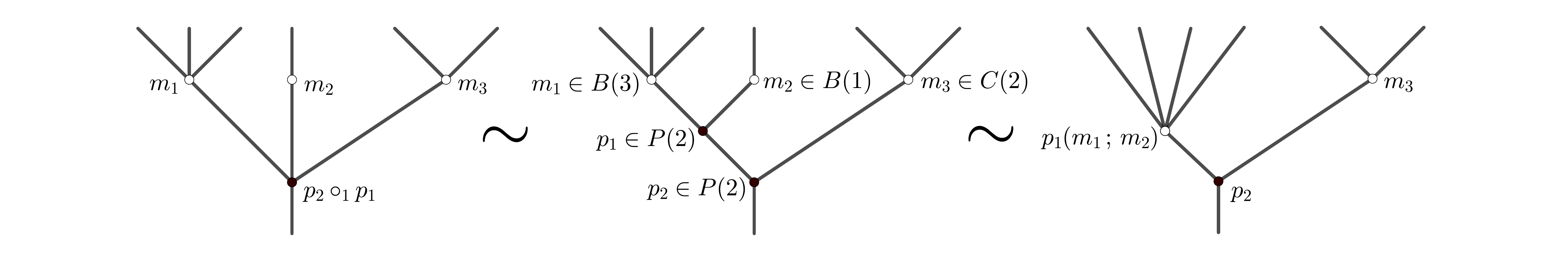}\vspace{-40pt}
\begin{figure}[!h]
\caption{Illustration of the pushout relations $(iv)$ and $(vi)$.}\vspace{-5pt}
\end{figure}

The $\Sigma$-sequence $D$ inherits a $(P\text{-}Q)$-bimodule structure from  $B$ and $C$. The right operations are defined using the right module structures of $B$ and $C$. Similarly, the left operation between an element $p\in P(n)$ and a family of points
$[T_{i}\,;\,\{m_{p}^{i}\}\,;\,\{p_{v}^{i}\}\,;\,\{q_{v}^{i}\}]\in D$, with $1\leq i\leq n$, is defined as follows: each tree $T_{i}$, with $1\leq i\leq n$, is grafted from left to right to a leaf of the $n$-corolla whose vertex is indexed by $p$. Moreover, there is a map 
$$
\gamma'_{0}:P(0)\longrightarrow D(0)
$$
sending a point $p\in P(0)$ to the $0$-corolla labelled by $\gamma_{0}(p)\in A(0)$. 

The reader can check that the so obtained bimodule is well defined and that this construction works in the context of truncated bimodules. If $A_{r}$, $B_{r}$, $C_{r}$ are $r$-truncated bimodules and $f_{1}$, $f_{2}$ are $r$-truncated bimodule maps, then the pushout in the category of $r$-truncated bimodules $D_{r}$ is obtained from the formula (\ref{G1}) by taking the restriction of the coproduct to the  trees with section  having at most $r$ leaves and such that each pearl has at most $r$ incoming edges. The equivalence relation, the left and right operations and the map $\gamma'_{0}$ are defined in the same way.

\begin{pro}
One has the following identities:
$$
D=\underset{\hspace{40pt}\Sigma\Bimod_{P\,;\,Q}}{\mathrm{colim}}\big( B \longleftarrow A \longrightarrow C\big)\hspace{15pt}\text{and}\hspace{15pt} D_{r}=\underset{\hspace{40pt}T_{r}\Sigma\Bimod_{P\,;\,Q}}{\mathrm{colim}}\big( B_{r} \longleftarrow A_{r} \longrightarrow C_{r}\big).
$$
\end{pro}

\begin{proof}
We need to check the universal property of the pushout in the category of $(P\text{-}Q)$-bimodules. Let $D'$ be a $(P\text{-}Q)$-bimodule together with $(P\text{-}Q)$-bimodule maps $g_{1}:B\rightarrow D'$ and $g_{2}:C\rightarrow D'$ such that $g_{1}\circ f_{1}=g_{2}\circ f_{2}$. One has to show that there is a unique $(P\text{-}Q)$-bimodule map $\delta:D\rightarrow D'$ such that the following diagram commutes:
\begin{equation}\label{G3}
\xymatrix@R=15pt{
A \ar[r]^{f_{2}} \ar[d]_{f_{1}} & C \ar[d]\ar[rdd]^{g_{2}} & \\
B \ar[r] \ar[rrd]_{g_{1}} & D \ar@{-->}[rd]^{\hspace{-4pt}\delta} & \\
  &   & D'
}
\end{equation}

Let $[(T\,;\,\sigma)\,;\,\{m_{p}\}\,;\,\{p_{v}\}\,;\,\{q_{v}\}]$ be a point $D$ where $\sigma$ is the permutation labelling the leaves. Due to the pushout relations, we can assume that $T$ is a reduced tree with section without vertices above the section. If the tree with section $T$ has only one vertex (which is necessarily a pearl) indexed by $m_{r}$ in $B$ or $C$, then, due to the commutative diagram (\ref{G3}), $\delta$ must be defined as follows:
$$
\delta([(T\,;\,\sigma)\,;\,\{m_{p}\}\,;\,\{p_{v}\}\,;\,\{q_{v}\}])=\left\{
\begin{array}{ll}\vspace{5pt}
g_{1}(m_{r})\cdot\sigma & \text{if } m_{r}\in B, \\ 
g_{2}(m_{r})\cdot\sigma & \text{if } m_{r}\in C.
\end{array} 
\right.
$$

If the tree with section $T$ has more than $2$ vertices, then the root of $T$ is indexed by a point $p_{r}$ in the operad $P$ and the point $[(T\,;\,\sigma)\,;\,\{m_{p}\}\,;\,\{p_{v}\}\,;\,\{q_{v}\}]$ has a decomposition of the form 
$$
(p_{r}([(T_{1}\,;\,id)\,;\,\{m_{1}\}\,;\,\emptyset\,;\,\emptyset],\cdots, [(T_{|v|}\,;\,id)\,;\,\{m_{|v|}\}\,;\,\emptyset\,;\,\emptyset]))\cdot \sigma
$$ 
where $T_{1},\ldots, T_{|v|}$ are corollas.  Since $\delta$ is a bimodule map, one has
$$
\delta([(T\,;\,\sigma)\,;\,\{m_{p}\}\,;\,\{p_{v}\}\,;\,\{q_{v}\}])= \big( p_{r}(\delta( [(T_{1}\,;\,id)\,;\,\{m_{1}\}\,;\,\emptyset]\,;\,\emptyset]),\ldots, \delta( [(T_{|v|}\,;\,id)\,;\,\{m_{|v|}\}\,;\,\emptyset\,;\,\emptyset]])\big)\cdot \sigma.
$$
Thanks to the $(P\text{-}Q)$-bimodule axioms, $\delta$ does not depend on the choice of the representative element and $\delta$ is a $(P\text{-}Q)$-bimodule map. The uniqueness follows from the construction. The same arguments work for the truncated case.
\end{proof}

\begin{lmm}\label{Final3}
Let  $\partial X\rightarrow X$ be a morphism of $\Sigma Seq_P$ which defines a closed inclusion of topological spaces objectwise. For every pushout diagram of the form 
\begin{equation}\label{eq:push_out_incl}
\xymatrix@R=15pt{
\mathcal{F}^{\Sigma}_{P;Q}(\partial X ) \ar[r] \ar[d] & \mathcal{F}^{\Sigma}_{P;Q}(X) \ar[d] \\
B\ar[r] & D,
}
\end{equation}
the $(P\text{-}Q)$-bimodule map $B\rightarrow D$ is also a closed inclusion of topological spaces objectwise. 
\end{lmm}

\begin{proof}% \todo{PB}
%A map is a topological embedding with closed image if and only if it is injective and closed.

Let $s\mathbb{P}'_n$ denote the set of planar trees with section and with $n$ leaves labelled by a permutation
from $\Sigma_n$, whose set $V^p$ of pearls is partitioned into two subsets $V^p=V_{pri}^p\sqcup V_{aux}^p$
of primary and auxiliary pearls, respectively. One has,
\[
D(n)=\left.\left(\underset{T\in s\mathbb{P}'_{n}}{\coprod} \, X(T)\right)\right/\sim, 
\quad \text{where}\quad X(T)=\,\underset{p\in V^{p}_{pri}(T)}{\prod}  B(|p|) \times 
\underset{p\in V^{p}_{aux}(T)}{\prod}  X(|p|) \times
  \underset{v\in V^{d}(T)}{\prod} P(|v|)\times  \underset{v\in V^{u}(T)}{\prod} Q(|v|).
\]
The relations are (i)-(vi), see above. We denote by $\pi_T$ the map $\pi_T\colon X(T)\to D(n)$.

A closed injective map is always a topological inclusion. 
We prove first that each map $i\colon B(n)\to D(n)$, $n\geq 0$, is injective and then that it is closed.

 As a set, $D(n)$ does not depend on the topology of $X$. So,
we can choose a topology such that each $X(i)\setminus\partial X(i)$ is discrete and open in $X(i)$.  In the latter case, one has  $D=B\coprod
 \mathcal{F}^{\Sigma}_{P;Q}\bigl(P_0\sqcup(X\setminus\partial X)\bigr)$ and the inclusion of $B$ in this coproduct is obviously injective. Indeed, a tree with a vertex labelled
 by $x\in X\setminus\partial X$ in the coproduct can never loose such vertex by means of relations (i)-(vi) of the coproduct, and therefore
 can not produce a relation in~$B$.

Now we check that for any closed $C\subset B(n)$, the set $i(C)$ is closed in $D(n)$. Consider
  \[
X^\partial(T):= \underset{p\in V^{p}_{pri}(T)}{\prod}  B(|p|) \times 
\underset{p\in V^{p}_{aux}(T)}{\prod}\partial  X(|p|) \times
  \underset{v\in V^{d}(T)}{\prod} P(|v|)\times  \underset{v\in V^{u}(T)}{\prod} Q(|v|).
\]
Since each composition map $\pi^\partial_T\colon X^\partial(T)\to B(n)$ is continuous, the set $(\pi^\partial_T)^{-1}(C)$ is closed
in $X^\partial(T)$. On the other hand, since each map $\partial  X(|p|) \to X(|p|)$ is a closed inclusion, the map
$X^\partial(T)\to X(T)$ is also one. (Here, we use the fact that kelleyfication preserves closed inclusions.) 
%\footnote{Here one must be careful as  products are taken in $Top$ and therefore are endowed with the kelleyfication of
%the usual product topology. Nonetheless, it is easy to check that if $\partial S\to S$ is a closed inclusion of $k$-spaces and $T$ is a $k$-space, then
%$\partial S\times T\to S\times T$ is again a closed inclusion even when the products are taken in our category $Top$ of $k$-spaces.}.
Thus $(\pi^\partial_T)^{-1}(C)=(\pi_T)^{-1}(i(C))$ is closed in $X(T)$ for every $T$ and we conclude that so is $i(C)$ in $D(n)$.

\end{proof}

\subsection{The model category structure}\label{H2}

By using the identifications (\ref{Z3}), the categories $\Sigma Seq_{P}$ and $T_{r}\Sigma Seq_{P}$ inherit model category structures from the categories of $\Sigma$-sequences and truncated $\Sigma$-sequences, respectively. More precisely, a map is a weak equivalence, a fibration or a cofibration if the corresponding map is a weak equivalence, a fibration or a cofibration in the category of (truncated) $\Sigma$-sequences. In particular, $\Sigma Seq_{P}$ and $T_{r}\Sigma Seq_{P}$ are cofibrantly generated and all their objects are fibrant. 
Their (acyclic) generating cofibrations are $\{P_0\sqcup\partial X\to P_0\sqcup X\}$, where $\{\partial X\to X\}$ is the set of (acyclic) generating cofibrations of $\Sigma Seq$ or $T_r\Sigma Seq$,
respectively.
By applying the transfer principle \ref{E3} to the adjunctions
\begin{equation}\label{A8}
\mathcal{F}_{P\,;\,Q}^\Sigma:\Sigma Seq_{P}\rightleftarrows\Sigma \Bimod_{P\,;\,Q}:\mathcal{U}^\Sigma\hspace{15pt}\text{and} \hspace{15pt} \mathcal{F}_{P\,;\,Q}^{T_{r}\Sigma}:T_{r}\Sigma Seq_{P}\rightleftarrows T_{r}\Sigma\Bimod_{P\,;\,Q}:\mathcal{U}^{T_{r}\Sigma},
\end{equation}
we get the following statement:

\begin{thm}\label{ProjectBimod}
For any pair $(P,Q)$ of topological operads, the category of (truncated) $(P\text{-}Q)$-bimodules
$\Sigma \Bimod_{P\,;\,Q}$ (respectively, $T_r\Sigma \Bimod_{P\,;\,Q}$, $r\geq 0$) inherits a cofibrantly generated model category structure, called the projective model category structure, in which all objects are fibrant.  The model structure in question makes the adjunctions (\ref{A8}) into  Quillen adjunctions. More precisely, a bimodule map $f$  is a weak equivalence (respectively, a fibration) if and only if the induced map $\mathcal{U}^\Sigma(f)$ or $\mathcal{U}^{T_{r}\Sigma}(f)$ is a weak equivalence (respectively, a fibration) in the category of (possibly truncated) $\Sigma$-sequences. 
\end{thm}

\begin{proof}
According to the transfer principle \ref{E3}, we have to check the small object argument as well as the existence of a functorial fibrant replacement and a functorial factorization of the diagonal map in the category  $\Sigma\Bimod_{P\,;\,Q}$.

We check first the small object argument. %We say that a $\Sigma$-sequence $X$ is {\it compact} if $\coprod_{n\geq 0} X(n)$ is compact. We say that  a sequence $X\in\Sigma Seq_P$ is
%compact if the quotient sequence $X/P_0$, defined as $X(0)/P(0)$ at $n=0$, and as $X(n)$ for $n>0$, is compact. 
% Compact sequences (similarly to compact spaces in $Top$)
%are $\aleph_0$-small relative to small inclusions. This means that for any ordinal $\lambda$ and any $\lambda$ sequence $A_\alpha$ of closed inclusions in $\Sigma Seq_P$,
%one has 
%\[
%\Sigma Seq_P(X\,;\,\mathrm{colim}_{\alpha<\lambda}A_\alpha)\cong \mathrm{colim}_{\alpha<\lambda}\Sigma Seq_P(X\, ;\, A_\alpha).
%\]
 Let $\mathcal{F}^{\Sigma}_{P;Q}(X)$ be a domain of an element in the set of generating (acyclic) cofibrations in $\Sigma \Bimod_{P\,;\,Q}$.
 % Let $\lambda=\mathfrak{c}^+$ be the cardinal following continuum $\mathfrak{c}$.
  Let $\lambda=\aleph_1$ be the first uncountable ordinal. Assume that $\{M_{\alpha}\}_{\alpha < \lambda}$ is a $\lambda$-sequence of   $(P\text{-}Q)$-bimodules, such that 
 each map $M_{<\alpha}:=\mathrm{colim}_{\beta<\alpha}M_\beta\to M_\alpha$  fits into a pushout square
\begin{equation}\label{eq:push_out_incl}
\xymatrix@R=15pt{
\mathcal{F}^{\Sigma}_{P;Q}(\partial Y ) \ar[r] \ar[d] & \mathcal{F}^{\Sigma}_{P;Q}(Y) \ar[d] \\
M_{<\alpha}\ar[r] & M_\alpha,
}
\end{equation}
where $\partial Y\to Y$ is a possibly infinite coproduct in $\Sigma Seq_P$ of generating (acyclic) cofibrations.
  Due to the adjunction \eqref{A8}, one has the identity 
$$
\Sigma\Bimod_{P\,;\,Q}\big(\,\mathcal{F}^{\Sigma}_{P;Q}(X)\,;\, \mathrm{colim}_{\alpha < \lambda} M_{\alpha}\,\big) \cong \Sigma Seq_{P}\big(\,X\,;\, \mathcal{U}(\mathrm{colim}_{\alpha < \lambda} M_{\alpha})\,\big).
$$ 
%We denote by $M_{<\alpha}:=\mathrm{colim}_{\beta<\alpha}M_\beta$.
Since the forgetful functor is monadic, it preserves filtered colimits. So, one has 
$$
\Sigma Seq_{P}\big(\,X\,;\, \mathcal{U}(\mathrm{colim}_{\alpha < \lambda} M_{\alpha})\,\big) \cong \Sigma Seq_{P}\big(\,X\,;\,\mathrm{colim}_{\alpha < \lambda}  \mathcal{U}(M_{\alpha})\,\big).
$$
It follows from Lemma~\ref{Final3}  that $M_{<\alpha}\rightarrow M_{\alpha}$ is an objectwise closed inclusion. The sequence $X\setminus P_0$ is concentrated in only one arity where it is a finite union of spheres (or discs), thus a separable space.
% whose set cardinality is $\leq\mathfrak{c}$. 
Therefore, $X$ is $\aleph_1$-small
%$\mathfrak{c}^+$-small
 relative to componentwise closed inclusions\footnote{In fact  $X$ is $\aleph_0$-small relative to componentwise closed inclusions of $T_1$-spaces, as $X$ being compact can not
 contain a discrete countable closed subspace. Note, however, that in our category $Top$ non-$T_1$ spaces are allowed.}
 and one has the identities
$$
\begin{array}{rclcc}
\Sigma Seq_{P}\big(\,X\,;\,\mathrm{colim}_{\alpha < \lambda}  \mathcal{U}(M_{\alpha})\,\big) & \cong & \mathrm{colim}_{\alpha < \lambda}\, \Sigma Seq_{P}\big(\,X\,;\,  \mathcal{U}(M_{\alpha})\,\big) &  \cong & \mathrm{colim}_{\alpha < \lambda}\, \Sigma \Bimod_{P\,;\,Q}\big(\,\mathcal{F}^{\Sigma}_{P;Q}(X)\,;\,  M_{\alpha}\,\big).
\end{array} 
$$
This proves the small object argument.

Since all objects are fibrant in the category of $\Sigma$-sequences, the identity functor provides a functorial fibrant replacement. For any $M\in \Sigma\Bimod_{P\,;\,Q}$, one needs to prove the existence of an element $Path(M)\in \Sigma\Bimod_{P\,;\,Q}$ inducing a factorization of the diagonal map
$$
\xymatrix{
\Delta: M \ar[r]^{\hspace{-5pt}\simeq}_{\hspace{-5pt}f_{1}} & Path(M) \ar@{->>}[r]_{f_{2}} & M\times M,
}
$$
where $f_{1}$ is a weak equivalence and $f_{2}$ is a Serre fibration. Let us consider
$$
Path(M)(n)=Map\big( [0\,,\,1]\,;\,M(n)\big).
$$
This $\Sigma$-sequence inherits a bimodule structure from $M$. The map from $M$ to $Path(M)$, sending a point to the constant path, is clearly a homotopy equivalence. Furthermore, the map 
$$
f_{2}:Map\big( [0\,,\,1]\,;\,M(n)\big)\longrightarrow Map\big( \partial [0\,,\,1]\,;\,M(n)\big)=(M\times M)(n)
$$
induced by the inclusion $i:\partial [0\,,\,1]\rightarrow [0\,,\,1]$ is a Serre fibration since the map $i$ is a cofibration.
\end{proof}

\begin{proof}[Alternative proof]
As explained in Section \ref{D8}, the category of $(P\text{-}Q)$-bimodules is equivalent to the category of algebras over a colored operad $P{+}Q$.
 According to 
the general result of \cite[Theorem 2.1]{BM3}, the category of algebras over any topological operad has a projective model structure.
 To be precise, let $Seq= Top^\mathbb{N}$ be the projective model category of sequences of topological spaces. % colored $\Sigma$-sequences with the set of colors $S=\mathbb{N}$.
   The adjunction between the forgetful functor and the free $(P{+}Q)$-algebra functor
$$
\mathcal{F}_{P{+}Q}:Seq\rightleftarrows Alg_{P{+}Q}:\mathcal{U},
$$
induces a cofibrantly generated model category structure on the category of $(P{+}Q)$-algebras. In this model structure, a map of $(P{+}Q)$-algebras $f$ is a weak equivalence (respectively, a fibration) if the corresponding map $\mathcal{U}(f)$ is a weak equivalence (respectively, a fibration) in the category of sequences. This model structure coincides with the model structure described in the theorem
because the projective model structure on $\Sigma Seq$ is itself transferred from $Seq$.   
\end{proof}

\subsubsection{Relative left properness of the projective model category}\label{C0}

\noindent \textit{$\bullet$ Definitions of relative left and right properness.} First, we recall the definition in a general setting. Let $\mathcal{C}$ be a model category and let $\mathcal{S}$ be a class of objects of $\mathcal{C}$. The model category $\mathcal{C}$ is said to be left proper (respectively right proper) relative to $\mathcal{S}$ if for each pushout diagram (respectively pullback diagram) of the form
$$
\xymatrix{
A \ar[r]^{f}_{\simeq} \ar@{^{(}->}[d]_{g} & B \ar[d]^{j} \\
C \ar[r]_{i} & D
}\hspace{30pt}  \hspace{10pt}
\xymatrix@R=5pt{
&A \ar[r]^{i} \ar[dd]_{j} & B \ar@{->>}[dd]^{g}  &\\
(\text{respectively,}& & & )\\
& C \ar[r]_{f}^{\simeq} & D &
}
$$
with $g$ a cofibration (respectively, a fibration)  and $f$ a weak equivalence between objects in $\mathcal{S}$, the morphism $i$ is also a weak equivalence. In particular, the category $\mathcal{C}$ is said to be left proper (respectively right proper) if $\mathcal{C}$ is left proper (respectively right proper) relative to all the objects. Furthermore, a model category $\mathcal{C}$ is said to be proper (relative to $\mathcal{S}$) if the category is both left and right proper (relative to $\mathcal{S}$). The advantage of such a category is that we have a criterion allowing to identify homotopy invariant colimits or/and limits.

\begin{pro}{\cite[Proposition A.2.4.4]{Lur}}\label{D0}
Let $\mathcal{C}$ be a model category which is left proper relative to a class of objects $\mathcal{S}$. For any commutative diagram in the category $\mathcal{C}$ of the form
$$
\xymatrix{
A \ar[d]_{v_{A}} &\ar[r]^{g} B \ar[d]_{v_{B}} \ar[l]_{f}&  C  \ar[d]_{v_{C}}\\
A'  &\ar[r]_{g'} B'  \ar[l]^{f'}&  C' 
}
$$
the induced map between the colimits of the horizontal diagrams 
$$
\mathrm{colim}\big( A\leftarrow B \rightarrow C\big) \longrightarrow \mathrm{colim}\big( A'\leftarrow B' \rightarrow C'\big)
$$
is a weak equivalence if the vertical morphisms are weak equivalences; one of the pairs $(f,C)$ or $(g,A)$ consists of a cofibration and an object in $\mathcal{S}$; one of the pairs $(f',C')$ or $(g',A')$ consists of a cofibration and an object in $\mathcal{S}$.

Dually, let $\mathcal{C}$ be a model category which is right proper relative to a class $\mathcal{S}$. For any commutative diagram in the category $\mathcal{C}$ of the form
$$
\xymatrix{
A \ar[d]_{v_{A}}\ar[r]^{f} & B \ar[d]_{v_{B}} &\ar[l]_{g}  C  \ar[d]_{v_{C}}\\
A'  \ar[r]_{f'}& B' &  \ar[l]^{g'} C' 
}
$$
the induced map between the limits of the horizontal diagrams 
$$
\mathrm{lim}\big( A\rightarrow B \leftarrow C\big) \longrightarrow \mathrm{lim}\big( A'\rightarrow B' \leftarrow C'\big)
$$
is a weak equivalence if the vertical morphisms are weak equivalences; one of the pairs $(f,C)$ or $(g,A)$ consists of a fibration and an object in $\mathcal{S}$; one of the pairs $(f',C')$ or $(g',A')$ consists of a fibration and an object in $\mathcal{S}$.

\end{pro}

\begin{rmk}\label{B6}
Even when  the category $\mathcal{C}$ is not left or right proper, the statement of Proposition~\ref{D0} still holds
provided  one of the pairs $(f\,;\,C)$ or $(g\,;\,A)$ consists of a (co)fibration and a (co)fibrant object whereas one of the pairs $(f'\,;\,C')$ or $(g'\,;\,A')$ consists of a (co)fibration and a (co)fibrant object.
%Proposition \ref{D0} is still true if the category $\mathcal{C}$ is not left or right proper under additional assumptions on the commutative diagram (see \cite[Proposition 13.1.2]{Hir}). In that case, we need to assume that one of the pairs $(f\,;\,C)$ or $(g\,;\,A)$ consists of a (co)fibration and a (co)fibrant object whereas one of the pairs $(f'\,;\,C')$ or $(g'\,;\,A')$ consists of a (co)fibration and a (co)fibrant object. 
\end{rmk}

\noindent \textit{$\bullet$ Application to the projective model category of bimodules.} Let $P$ and $Q$ be two topological operads. From now on, we focus our attention on the projective model category of $(P\text{-}Q)$-bimodules. More precisely, we show that this category is right proper relative to all the objects and is relatively  left proper.

\begin{thm}\label{th:right_proper}
The projective model category  $\Sigma\Bimod_{P\,;\,Q}$ is right proper.    
\end{thm}

\begin{proof}
We consider the following pullback diagram in which $f$ is a weak equivalence and $g$ is a fibration:
$$
\xymatrix{
A = \mathrm{lim}\,\big(\hspace{-31pt} & C \ar[r]^{f}_{\simeq}& D &\ar@{->>}[l]_{g}  B\big).
}
$$
The map $i:A\rightarrow B$ is equivalent to the following map between pullback diagrams:
$$
\xymatrix{
A\ar[d]^{i} \ar@{=}[r] & \mathrm{lim}\,\big(\hspace{-31pt} & C\ar[d]^{\simeq} \ar[r]^{\simeq}&\ar@{=}[d] D &\ar@{->>}[l] \ar@{=}[d] B\,\,\big)\\
B \ar@{=}[r] &  \mathrm{lim}\,\big(\hspace{-31pt} & D \ar@{=}[r]& D &\ar@{->>}[l] B\,\,\big)
}
$$
According to Remark \ref{B6} and since all the objects in the projective model category of bimodules are fibrant, the map $i:A\rightarrow B$ is also a weak equivalence. Consequently, the category $\Sigma\Bimod_{P\,;\,Q}$ is right proper relative to all the objects.
\end{proof}

\begin{lmm}\label{Z0}
Let $C\leftarrow A \rightarrow B$ be a diagram in a cocomplete category $\mathcal{C}$. If $A\rightarrow C$ is a retract of $A'\rightarrow C'$, then the morphism $C\rightarrow \mathrm{colim}(C\leftarrow A \rightarrow B)$ is a retract of $C'\rightarrow \mathrm{colim}(C'\leftarrow A' \rightarrow A \rightarrow B)$. 
\end{lmm}

\begin{proof}
Since $A\rightarrow C$ is a retract of $A'\rightarrow C'$, one has a commutative diagram 
\begin{equation}\label{Z4}
\xymatrix{
C\ar[r]^{g_{1}} & C'\ar[r]^{g_{2}} & C \\
A \ar[r]^{h_{1}} \ar[u] \ar[d] & A' \ar[r]^{h_{2}} \ar[u] \ar[d] & A  \ar[u] \ar[d] \\
B \ar@{=}[r] & B \ar@{=}[r] & B
}
\end{equation}
such that $g_{2}\circ g_{1}=id$ and $h_{2}\circ h_{1}=id$. By taking the pushout of the vertical diagrams in (\ref{Z4}), we get the commutative diagram
$$
\xymatrix@C=50pt{
C\ar[r]^{g_{1}} \ar[d] & C' \ar[d] \ar[r]^{g_{2}} & C\ar[d]\\
C\underset{A}{\bigsqcup} B \ar[r]_{k_{1}=g_{1}\underset{h_{1}}{\bigsqcup} id} & C'\underset{A'}{\bigsqcup} B \ar[r]_{k_{2}=g_{2}\underset{h_{2}}{\bigsqcup} id} & C\underset{A}{\bigsqcup} B 
}
$$
in which $g_{2}\circ g_{1}=id$ and $k_{2}\circ k_{1}=id$. 
\end{proof}

\begin{thm} \label{C4}
If $P$ and $Q$ are operads, such that  $P(0)\in Top$ is cofibrant, $P_{>0}\in\Sigma\Operad$ is a cofibrant operad,  and $Q$ is a componentwise cofibrant operad, then the projective model category  $\Sigma\Bimod_{P\,;\,Q}$ is left proper relative to
the class $\mathcal{S}$  of componentwise cofibrant bimodules $M$ for which the arity zero left action map $\gamma_0\colon P(0)\to M(0)$ is a cofibration.    
\end{thm} 

\begin{cor}\label{cor:left_prop}
If $Q$ is a componentwise cofibrant operad and $P$ is either a Reedy cofibrant or projectively cofibrant operad, then the projective model category 
 $\Sigma\Bimod_{P\,;\,Q}$ is relatively left proper.
 \end{cor}
 
 \begin{proof}
 An operad is Reedy cofibrant if and only if it is reduced and its positive arity part is cofibrant. This proves the first case. The second case follows from Proposition~\ref{p:useful3}.
 \end{proof}
 
\begin{proof}[Proof of Theorem~\ref{C4}]
We first prove this theorem assuming $P(0)=\emptyset$ and at the end we explain how the argument has to be adjusted to the general case $P(0)\neq\emptyset$.

In order to prove that the category is left proper relative to $\mathcal{S}$, we consider the following pushout diagram in the category of $(P\text{-}Q)$-bimodules in which the map $f$ is a weak equivalence between componentwise cofibrant objects and the map $g$ is a cofibration:
$$
\xymatrix{
D = \operatorname{colim}\,\big(\hspace{-31pt} & C & \ar@{_{(}->}[l]_{g} A \ar[r]_{\simeq}^{f}&  B\big).
}\vspace{-0pt}
$$  
Since any cofibration is a retract of a cellular extension, we can assume that $g$ is a cellular extension due to Lemma \ref{Z0}. On the other hand, any cellular extension is a possibly transfinite sequence of cell attachments. By Theorem~\ref{th:sigma_bimod_cof}a,  cofibrations with domain in  $\mathcal{S}$ are componentwise cofibrations. 
 Thus, without loss of generality, we can assume that $g$ is a cellular attachment and we restrict our study to diagrams of the form 
$$
\xymatrix@R=15pt{
\mathcal{F}_{P\,;\,Q}^\Sigma(\partial X) \ar[r] \ar[d] & A\ar[r] \ar[d] & B \ar[d] \\
\mathcal{F}_{P\,;\,Q}^\Sigma(X) \ar[r] & C\ar[r] & D,
}
$$
where $\partial X \rightarrow X$ is a generating cofibration in the category of $\Sigma$-sequences and both squares are
pushout diagrams. The strategy is to use the explicit description of the pushout from Section \ref{B7} and to introduce a filtration in $\Sigma$-sequences $C$ and $D$ according to the number of vertices in the trees indexed by $X$: 
\begin{equation}\label{B9}
\xymatrix@R=20pt{
A=C_{0} \ar[d] \ar[r] & C_{1} \ar[d] \ar[r] & \cdots \ar[r] & C_{i-1} \ar[d] \ar[r] & C_{i} \ar[d] \ar[r] & \cdots \ar[r] & C \ar[d]\\
B=D_{0} \ar[r] & D_{1} \ar[r] & \cdots \ar[r] & D_{i-1} \ar[r] & D_{i} \ar[r] & \cdots \ar[r] & D.
}
\end{equation}
We will prove that  in~\eqref{B9},  each horizontal map is a componentwise cofibration and each vertical map is a weak equivalence.

Let $rs\mathbb{P}_n[i]$ denote the set of reduced planar trees with section with $n$ leaves labelled by a permutation in $\Sigma_n$. In addition
we assume that it has two types of pearls: $i$ auxiliary ones and some number of primary ones, primary ones coming first. Moreover, 
it is required that all incoming edges of any primary pearl are leaf edges, while all incoming edges of any auxiliary pearl connects it to an internal vertex.

It follows from  Construction~\ref{F9} of a free bimodule and the combinatorial description of a pushout (Subsection~\ref{B7}) 
that any element in $C(n)$ or $D(n)$ can be obtained as a tree $T\in rs\mathbb{P}_n[i]$ (for some $i\geq 0$), whose root,
primary pearls, auxiliary pearls, and vertices above the section are labelled by $P$,  $A$ (or $B$), $X$, and $Q$, respectively. The filtration~\eqref{B9}
is by the number $i$ of auxiliary pearls in such elements. To show that maps $C_{i-1}(n)\to C_i(n)$ and $D_{i-1}(n)\to D_i(n)$ are
cofibrations, we additionally filter this inclusion by the total number $\ell$ of pearls:
\begin{equation}\label{eq_additional_filtr}
\xymatrix@R=20pt{
C_{i-1}(n)=:C_{i}(n)_{i-1} \ar[d] \ar[r] & C_{i}(n)_{i} \ar[d] \ar[r] & \cdots \ar[r] & C_{i}(n)_{\ell-1} \ar[d] \ar[r] & C_{i}(n)_\ell \ar[d] \ar[r] & \cdots \ar[r] & C_i(n) \ar[d]\\
D_{i-1}(n)=:D_{i}(n)_{i-1} \ar[r] & D_{i}(n)_{i} \ar[r] & \cdots \ar[r] & D_{i}(n)_{\ell-1} \ar[r] & D_{i}(n)_\ell \ar[r] & \cdots \ar[r] & D_i(n)
}
\end{equation}
We will show that each horizontal inclusion is a cofibration and each vertical map is an equivalence, which would prove the theorem.

\vspace{10pt}
\noindent \textbf{The maps $C_i(n)_{\ell-1}\to C_i(n)_\ell$ and $D_i(n)_{\ell-1}\to D_i(n)_\ell$  are 
cofibrations.}
For $\ell\geq i$, let $rs\mathbb{P}_n[i\,;\,\ell]\subset rs\mathbb{P}_n[i]$ be the subset of planar trees with section with exactly $\ell$ pearls. 
Let $rs\mathbb{T}_n[i\,;\,\ell]:= rs\mathbb{P}_n[i\,;\,\ell]/{\sim}$ be the set of equivalence classes, where two trees are equivalent if they are isomorphic as non-planar trees by an isomorphism that forgets the labels of leaves and that sends primary pearls to primary ones and same with auxiliary ones.\vspace{-10pt}

\begin{figure}[!h]
\begin{center}
\includegraphics[scale=0.35]{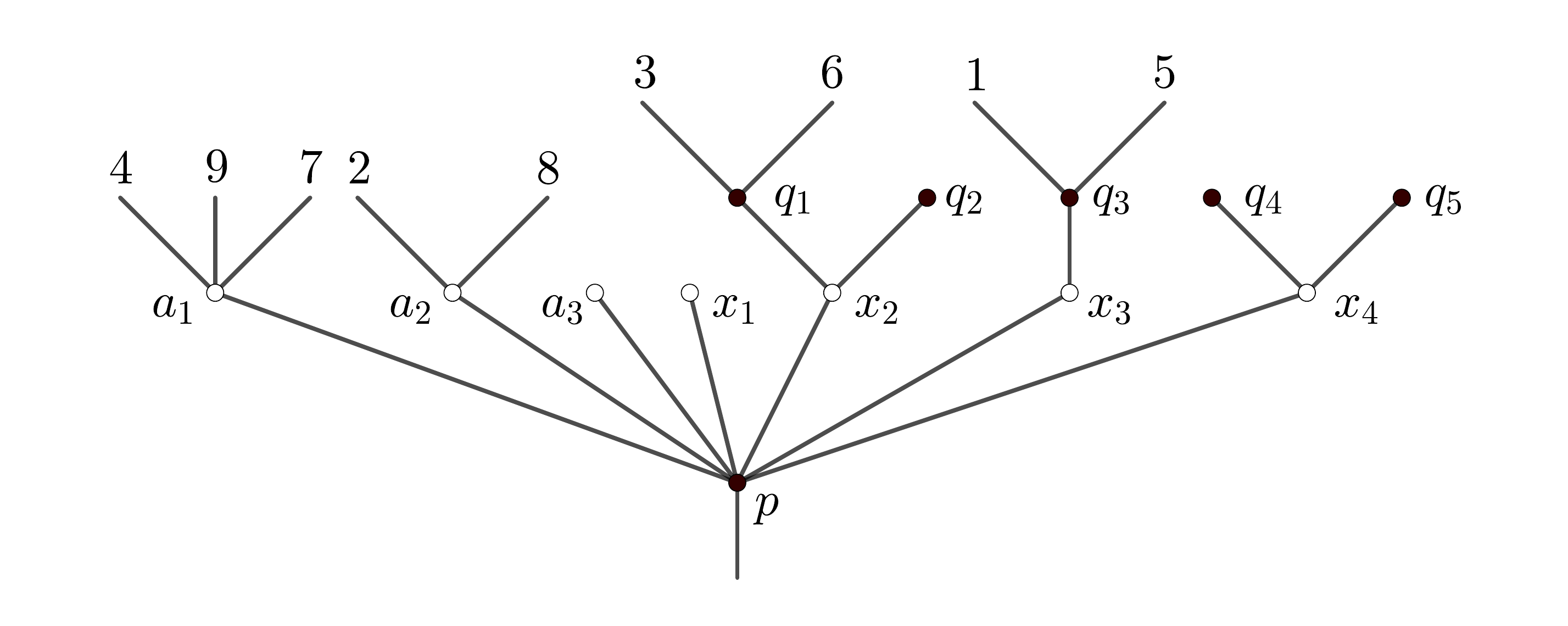}\vspace{-15pt}
\caption{Illustration of a tree in $rs\mathbb{P}_9[4\,;\,7]\subset rs\mathbb{P}_9[4]$ labelled by $P$, $A$, $X$ and $Q$.}\vspace{-5pt}
\end{center}
\end{figure}

For $T\in rs\mathbb{T}_n[i\,;\,\ell]$, denote by $Aut(T)$ the group of automorphisms of $T$ that does not mix the auxiliary pearls
with the primary ones (only arity zero ones can potentially be permuted, but we do not allow it). A choice of representative
of $T$ in $rs\mathbb{P}_n[i\,;\,\ell]$ gives a labeling of leaves by the set $[n]$ and therefore induces a homomorphism
$Aut(T)\to\Sigma_n$. Denote by $Aut_0(T)$ the kernel of this homomorphism which is the subgroup of $Aut(T)$ permuting  the univalent vertices and the auxiliary pearls that have only univalent vertices above them.  Let $p_1,\ldots,p_\ell$ be the pearls of the chosen planar 
representative of~$T$ labelled in the planar order they appear. 
 Then we consider the spaces 
$$
M_{1}(T)=\underset{\ell-i+1\leq j\leq \ell}{\prod} X(|p_{j}|)\times \underset{v\in V^{u}(T)}{\prod} Q(|v|),
%\hspace{15pt} 
$$
where $V^u(T)$ is the set of vertices of $T$ above the section,
and % \hspace{15pt}
\begin{multline*}
M_{2}(T\,;\,A)=P(\ell)\times_{P(1)^{\ell-i}} \underset{1\leq j\leq \ell-i}{\prod} A(|p_{j}|)=%\\
{\mathrm{coeq}}\left(P(\ell)\times{P(1)^{\ell-i}} \times\underset{1\leq j\leq \ell-i}{\prod} A(|p_{j}|) 
\rightrightarrows P(\ell)\times \underset{1\leq j\leq \ell-i}{\prod} A(|p_{j}|)\right).
\end{multline*}
The two arrows in the coequalizer correspond respectively to the action of $P(1)^{\ell-i}$ on the first $(\ell-i)$
inputs of $P(\ell)$ (by operadic composition) and on $\underset{1\leq j\leq \ell-i}{\prod} A(|p_{j}|)$ (by
the left $P$-action on $A$).

In other words, $M_{1}(T)$ is the space of indexations of the $i$ last pearls by points in $X$ and the vertices above the section by points in the operad $Q$. Similarly, $M_{2}(T\,;\,A)$ is the space of indexations of the other pearls by points in $A$ and the root by a point in $P(\ell)$. This presentation is not unique,
that is why we take the coequalizer.

Moreover, we denote by $M_{1}^{-}(T)$ the subspace of $M_{1}(T)$ formed by points having at least one pearl indexed by a point in $\partial X$. The space $M_{2}^{-}(T\,;\,A)$ consists of points in $M_{2}(T\,;\,A)$ for which the root is indexed by a point $p\in P(\ell\,;\,\ell-i)\subset P(\ell)$ (see
notation  before Proposition~\ref{J2}). We define
\[
M(T\,;\,A):=M_1(T)\times M_2(T\,;\,A)\quad \text{ and }\quad M^{-}(T\,;\,A):=M_1(T)\times M_2^{-}(T\,;\,A)\coprod_{M_1^{-}(T)\times M^{-}_2(T\,;\,A)} M^{-}_1(T)\times M_2(T\,;\,A).
\]
The spaces $M_2(T\,;\,B)$, $M_2^{-}(T\,;\,B)$, $M(T\,;\,B)$, $M^-(T\,;\,B)$ are defined similarly.

%of the form
%\begin{multline*}
%p=(p_{1}\circ_{1}p_{2})\cdot\sigma, \hspace{5pt}\text{with } p_{2}\in P(j),\, 2\leq j\leq \ell-i,\,\, p_{1}\in P(\ell-j+1),\, \\
%\text{ and } \sigma \text{ is a shuffle of } \{1,\ldots,j\} \text{ with } \{j+1,\ldots,\ell-i\}
%\text{ fixing } \{\ell-i+1,\ldots,\ell\}.
%\end{multline*}

One has the following pushout squares
\begin{equation}\label{eq:squares}
\xymatrix{
\underset{T\in rs\mathbb{T}_n[i;\ell]}{\displaystyle\coprod} M^{-}(T\,;\,A)\hspace{-5pt}\underset{Aut(T)}{\times} \hspace{-5pt}\Sigma_{n}\ar[r] \ar[d] &C_i(n)_{\ell-1} \ar[d] \\
\underset{T\in rs\mathbb{T}_n[i;\ell]}{\displaystyle\coprod} M(T\,;\,A)\hspace{-5pt}\underset{Aut(T)}{\times} \hspace{-5pt}\Sigma_{n}\ar[r] &
C_i(n)_{\ell}
}
\hspace{40pt} 
\xymatrix{
\underset{T\in rs\mathbb{T}_n[i;\ell]}{\displaystyle\coprod} M^{-}(T\,;\,B)\hspace{-5pt}\underset{Aut(T)}{\times} \hspace{-5pt}\Sigma_{n}\ar[r] \ar[d] &D_i(n)_{\ell-1} \ar[d] \\
\underset{T\in rs\mathbb{T}_n[i;\ell]}{\displaystyle\coprod} M(T\,;\,B)\hspace{-5pt}\underset{Aut(T)}{\times} \hspace{-5pt}\Sigma_{n}\ar[r] &
D_i(n)_{\ell}
}
\end{equation}

To prove that the right vertical maps are cofibrations one needs to show that the left ones are such. 
Recall that $Aut_0(T)$ denotes the kernel of $Aut(T)\to \Sigma_n$. Equivalently, for every $T\in rs\mathbb{T}_n[i;\ell]$ we need to show that the maps
\begin{equation}\label{eq:cof_M}
M^-(T\,;\,A)/Aut_0(T) \to M(T\,;\, A)/Aut_0(T)\quad \text{ and }\quad M^-(T\,;\,B)/Aut_0(T) \to M(T\,;\, B)/Aut_0(T)
\end{equation}
are cofibrations. We also need later that the sources of the maps are cofibrant.

Let $G_1$ be the subgroup of $Aut_0(T)$ that fixes all the pearls. It permutes the arity zero vertices above the section. The subgroup $G_1$ is normal and the quotient group $G_2=Aut_0(T)/G_1$ is responsible for the permutations of the arity zero primary pearls and the auxiliary pearls that 
have only arity zero vertices above them (in particular they can be of arity zero themselves). Note that $G_2\subset\Sigma_{\ell-i}\times\Sigma_i$.
By construction, the spaces $M_{1}^{-}(T)$, $M_{1}(T)$, $M_{2}^{-}(T\,;\,A)$ and $M_{2}(T\,;\,A)$ are equipped with an action 
of $Aut(T)$ and of $Aut_0(T)$ by taking restriction. Moreover, for the last two, the $Aut_0(T)$-action factors through $G_2$.
%the automorphism group $G_{2}$. All of them inherits an action of the automorphism group $Aut'(T)$. 

%Since the operad $Q$ is $\Sigma$-cofibrant and the map $\partial X \rightarrow X$ is a $\Sigma$-cofibration between $\Sigma$-cofibrant objects,
 The inclusion $M_{1}^{-}(T)\rightarrow M_{1}(T)$ is an $Aut_0(T)$-equivariant $G_{1}$-cofibration.   % between $G_{1}$-cofibrant objects.
 Indeed, $G_1=\prod_{j=\ell-i+1}^\ell\Sigma_{d_j}$, where $d_j$ is the number of univalent vertices above the $j$-th  pearl. Applying iteratively
 Lemma~\ref{D2},   we get that the map
 \begin{equation}\label{eq:cof_X}
 \partial\underset{\ell-i+1\leq j\leq \ell}{\prod} X(|p_{j}|)\rightarrow \underset{\ell-i+1\leq j\leq \ell}{\prod} X(|p_{j}|)
 \end{equation}
 is a $G_1$-cofibration. On the other hand, since $Q$ is componentwise cofibrant, the space 
 \[
  \underset{v\in V^{u}(T)}{\prod} Q(|v|)
  \]
  is a cofibrant (in $Top$) $G_1$-space. Applying again Lemma~\ref{D2}, we get that  $M_{1}^{-}(T)\rightarrow M_{1}(T)$ is a  $G_{1}$-cofibration. 
 
 The source of the map~\eqref{eq:cof_X} is also $G_1$-cofibrant as it can be obtained as a sequence of $G_1$-cofibrations
 \[
 Y_1\to Y_2\to\ldots\to Y_{i},
 \]
 where 
 \[
 Y_m=\partial\left(\underset{\ell-i+1\leq j\leq \ell-i+m}{\prod} X(|p_{j}|)\right)\times \underset{\ell-i+m+1\leq j\leq \ell}{\prod} \partial X(|p_{j}|).
 \]
 Note that $Y_1=\underset{\ell-i+1\leq j\leq \ell}{\prod} \partial X(|p_{j}|)$ is $G_1$-cofibrant being a product 
 of $\Sigma_{d_j}$-cofibrant spaces (see Example~\ref{ex:push_prod1}). We finally conclude that $M^-_1(T)$ is
 $Aut_0(T)$-cofibrant.

 %\todo{Demontrer que c'est une $G_{1}$-cofibration (idem pour le cas $G_{2}$.} 
   Besides, we claim that
 %   since the operad $P$ is cofibrant and the $\Sigma$-sequence $A$ is $\Sigma$-cofibrant, Proposition~\ref{J2} implies that
  the inclusion from $M_{2}^{-}(T\,;\,A)$ into $M_{2}(T\,;\,A)$ is a  $G_{2}$-cofibrant map  between $G_2$-cofibrant spaces. The assignment
  \[
  R\mapsto R\times_{P(1)^{\ell-i}} \underset{1\leq j\leq \ell-i}{\prod} A(|p_{j}|)
  \]
  can be viewed as a functor from $(\Sigma_{\ell-i}\wr P(1))\times\Sigma_i\text{-}Top$ to $G_2\text{-}Top$. This functor preserves cofibrations. Indeed,  it preserves
  colimits and sends any generating cofibration $S^{k-1}\times (\Sigma_{\ell-i}\wr P(1))\times\Sigma_i\rightarrow D^k\times (\Sigma_{\ell-i}\wr P(1))\times\Sigma_i$ to a $G_2$-cofibration:
  \[
\left( S^{k-1}\times \Sigma_{\ell-i}\times \underset{1\leq j\leq \ell-i}{\prod} A(|p_{j}|)\right)\times\Sigma_i \rightarrow \left(D^k
 \times \Sigma_{\ell-i}\times \underset{1\leq j\leq \ell-i}{\prod} A(|p_{j}|)\right)\times\Sigma_i.
 \]
 The map above is a $G_2$-cofibration by  Lemmas \ref{D2} and~\ref{l:adj_mon} (recall that $G_2\subset\Sigma_{\ell-i}\times\Sigma_i$).  On the other hand,
 since $P$ is a cofibrant operad, by Proposition~\ref{J2}, the inclusion $P(\ell\,;\,\ell-i)\to P(\ell)$ is a $(\Sigma_{\ell-i}\wr P(1))\times\Sigma_i$-cofibration
 with a $(\Sigma_{\ell-i}\wr P(1))\times\Sigma_i$-cofibrant source. We conclude that  $M_{2}^{-}(T\,;\,A)\to M_{2}(T\,;\,A)$ is a  $G_{2}$-cofibrant map  between $G_2$-cofibrant spaces.

   Now combining that $M_1^-(T)\to M_1(T)$ is an $Aut_0(T)$-equivariant $G_1$-cofibration and that $M_{2}^{-}(T\,;\,A)\to M_{2}(T\,;\,A)$
   is a $G_2$-cofibration and applying Lemma~\ref{D2}, we get that $M^-(T\,;\,A)\to M(T\,;\, A)$ is an $Aut_0(T)$-cofibration.
   (We get a similar statement for~$B$ as well by replacing $A$ with $B$ in all formulas.)
   
   We can similarly conclude that $M_1^-(T)\times M_2^-(T\,;\,A)$ is $Aut_0(T)$-cofibrant and the inclusions
   \[
   M_1^-(T)\times M_2^-(T\,;\,A)\to M_1^-(T)\times M_2(T\,;\,A)\to M^-(T\,;\,A)
   \]
   are $Aut_0(T)$-cofibrations. As a result, $M^-(T\,;\,A)$ and (arguing similarly) $M^-(T\,;\,B)$ are $Aut_0(T)$-cofibrant.
   
   \vspace{10pt}
\noindent \textbf{The map $ C_i(n)_\ell\to D_i(n)_\ell$  is a weak equivalence.}
   By induction we assume that the map $C_i(n)_{\ell-1}\to D_i(n)_{\ell-1}$ is a weak equivalence of cofibrant spaces.
   The spaces $C_i(n)_\ell$ and $D_i(n)_\ell$ are the colimits of the first and second lines, respectively, in the following diagram.
     $$
\xymatrix{
\underset{T\in rs\mathbb{T}_n[i;\ell]}{\displaystyle\coprod} M(T\,;\,A)\hspace{-5pt}\underset{Aut(T)}{\times} \hspace{-5pt}\Sigma_{n}\ar[d] & \ar[l] \underset{T\in rs\mathbb{T}_n[i;\ell]}{\displaystyle\coprod} M^-(T\,;\,A)\hspace{-5pt}\underset{Aut(T)}{\times} \hspace{-5pt}\Sigma_{n}\ar[r] \ar[d] & C_i(n)_{\ell-1} \ar[d]\\  
\underset{T\in rs\mathbb{T}_n[i;\ell]}{\displaystyle\coprod} M(T\,;\,B)\hspace{-5pt}\underset{Aut(T)}{\times} \hspace{-5pt}\Sigma_{n}& \ar[l]\underset{T\in rs\mathbb{T}_n[i;\ell]}{\displaystyle\coprod} M^-(T\,;\,B)\hspace{-5pt}\underset{Aut(T)}{\times} \hspace{-5pt}\Sigma_{n}\ar[r]  &  D_i(n)_{\ell-1}
}
$$
  As it follows from Proposition~\ref{D0}, the induced map of colimits is an equivalence provided all vertical maps are equivalences. 
  We are left to showing that the left two arrows are such. In other words, it only remains  to prove that for every tree $T\in rs\mathbb{T}_n[i;\ell]$,
  the maps
  \[
  M^-(T\,;\,A)\to M^-(T\,;\,B)\quad \text{ and } \quad M(T\,;\,A)\to M(T\,;\,B)
  \]
  are weak equivalences (as we already proved that all the four spaces above are $Aut_0(T)$-cofibrant). The latter fact
  is deduced from the fact that the induced maps
  \[
  M^-_2(T\,;\,A)\to M^-_2(T\,;\,B)\quad \text{ and } \quad M_2(T\,;\,A)\to M_2(T\,;\,B)
  \]
 are equivalences of cofibrant spaces. The latter is proved by Lemma~\ref{l:left_mod} 
 for which we take $\Gamma=P(1)^{\ell-i}$,  $Z=P(\ell)$ or $P(\ell\,;\,\ell-i)$, and the
 corresponding $\Gamma^{op}$-equivariant equivalence between cofibrant (in $Top$)
 objects is the map $\prod_{j=1}^{\ell-i}A(|p_j|)\to \prod_{j=1}^{\ell-i}B(|p_j|)$.
 
 \vspace{10pt}
 \noindent \textbf{Proving the case $P(0)\neq\emptyset$.}
 The only difference with  the case $P(0)=\emptyset$ is that we need to take into account the $\gamma_{0}$-relation contracting the univalent pearls indexed by points coming from the image of $P(0)$. For this purpose, we change slightly the definition of the spaces $M_2^-(T\,;\,A)$ and $M_2^-(T\,;\,B)$. Thus one also needs to check the three statements below.
 \begin{itemize}
 \item $M_2^-(T\,;\,A)$ and $M_2^-(T\,;\,B)$ are $G_2$-cofibrant.
 \item $M_2^-(T\,;\,A)\to M_2(T\,;\,A)$ and $M_2^-(T\,;\,B)\to M_2(T\,;\,B)$  are  $G_2$-cofibrations.
 \item $M_2^-(T\,;\,A)\to M_2^-(T\,;\,B)$ is a weak equivalence.
 \end{itemize}
 Denote by $\vec{A}:=\prod_{j=1}^{\ell-i}A(|p_j|)$ and by $\partial\vec{A}\subset\vec{A}$ the subset consisting of points with at least one coordinate in $\gamma_0(P(0))\subset A(0)$.
 One has
 \[
 M_2^-(T\,;\,A)=\left( P(\ell)\times_{P(1)^{\ell-i}}\partial\vec{A}\right) \underset{P(\ell\,;\,\ell-i)\times_{P(1)^{\ell-i}}\partial\vec{A}}{\coprod}
 \left(P(\ell\,;\,\ell-i)\times_{P(1)^{\ell-i}}\vec{A}\right).
 \]
 By the same argument as above, the functor
 \[
 (-)\times_{P(1)^{\ell-i}}\partial\vec{A}\colon (\Sigma_{\ell-i}\wr P(1))\times\Sigma_i\text{-}Top \to   G_2\text{-}Top
 \]
 preserves cofibrations. As a consequence $P(\ell\,;\,\ell-i)\times_{P(1)^{\ell-i}}\partial\vec{A}$ is $G_2$-cofibrant and the inclusion
 $P(\ell\,;\,\ell-i)\times_{P(1)^{\ell-i}}\partial\vec{A}\to P(\ell)\times_{P(1)^{\ell-i}}\partial\vec{A}$ is a $G_2$-cofibration. Moreover, we know that $P(\ell\,;\,\ell-i)\times_{P(1)^{\ell-i}}\vec{A}$ is $G_2$-cofibrant. We conclude that $M_2^-(T\,;\,A)$ is so.
 
 The fact that $M_2^-(T\,;\,A)\to M_2(T\,;\,A)$ is a $G_2$-cofibration follows from Lemma~\ref{l:push_prod2} applied to $\Gamma=P(1)^{\ell-i}$, $K_1=1$, $K=K_2=G_2$, $A\to B$ being
 $P(\ell\,;\,\ell-i)\to P(\ell)$, and $X\to Y$ being $\partial\vec{A}\to\vec{A}$.
 
 Finally, to prove that $M_2^-(T\,;\,A)\to M_2^-(T\,;\,B)$ is a weak equivalence, we apply Lemma~\ref{l:left_mod} again to show that 
$$
\begin{array}{l}\vspace{4pt}
P(\ell\,;\,\ell-i)\times_{P(1)^{\ell-i}}\partial\vec{A} \to P(\ell\,;\,\ell-i)\times_{P(1)^{\ell-i}}\partial\vec{B}, \\ \vspace{4pt}
P(\ell)\times_{P(1)^{\ell-i}}\partial\vec{A} \to P(\ell)\times_{P(1)^{\ell-i}}\partial\vec{B}, \\ 
P(\ell\,;\,\ell-i)\times_{P(1)^{\ell-i}}\vec{A} \to P(\ell\,;\,\ell-i)\times_{P(1)^{\ell-i}}\vec{B},
\end{array} 
$$
are all the three weak equivalences. The statement follows.
 \end{proof}

\begin{thm}\label{th:sigma_bimod_cof}
(a) Let $P$, $Q$ and $\mathcal{S}$ be as in Theorem~\ref{C4}. In the category $\Sigma\Bimod_{P\,;\,Q}$, cofibrations with domain in  $\mathcal{S}$ are componentwise cofibrations. In particular,
the class $\mathcal{S}$   of  objects
is closed under cofibrations and  cofibrant $(P\text{-}Q)$-bimodules are always componentwise cofibrant.

(b) Assume additionally that $Q$ is $\Sigma$-cofibrant.  In the category $\Sigma\Bimod_{P\,;\,Q}$, cofibrations with domain in the subclass $\mathcal{S}^\Sigma\subset\mathcal{S}$ of $\Sigma$-cofibrant objects, are $\Sigma$-cofibrations. The class $\mathcal{S}^\Sigma$
is closed under cofibrations and  cofibrant $(P\text{-}Q)$-bimodules are always 
$\Sigma$-cofibrant.
\end{thm}

\begin{proof}
Let $A$ be a conponentwise cofibrant  $(P\text{-}Q)$-bimodule. It is enough to check that any cellular extension $A\to C$ 
%, where
%\[
%C:=A\coprod_{\mathcal{F}_{P\,;\,Q}^\Sigma(\partial X) }\mathcal{F}_{P\,;\,Q}^\Sigma(X), 
%\]
is a componentwise cofibration (respectively, $\Sigma$-cofibration). In the proof of Theorem~\ref{C4} we define 
a filtration~\eqref{B9} in~$C$ in which, as we proved, every inclusion $C_{i-1}\to C_i$ is a componentwise cofibration. 

Additionally
assuming that $Q$ and $A$ are $\Sigma$-cofibrant, one can show  that any such inclusion is
a $\Sigma$-cofibration. The strategy is the same to filter this inclusion as in~\eqref{eq_additional_filtr} and then to show that the left
vertical map in the first square~\eqref{eq:squares} is a $\Sigma_n$-cofibration. The latter is proved by showing
that each inclusion $M^-(T\,;\,A)\to M(T\,;\,A)$ is an $Aut(T)$-cofibration. Let $H_1\subset Aut(T)$ be the subgroup of elements 
that fix all the pearls and the leaves above the  primary pearls.  This group permutes vertices above the section and their leaves. It is normal and the quotient group $H_2=Aut(T)/H_1$ is responsible for permuting  the pearls and leaves above 
the primary pearls. By essentially the same argument, the inclusion $M_1^-(T)\to M_1(T)$ is an $Aut(T)$-equivariant 
$H_1$-cofibration and the inclusion $M_2^-(T\,;\,A)\to M_2(T\,;\,A)$ is an $H_2$-cofibration (by Lemma~\ref{l:push_prod2} applied for $\Gamma=P(1)^{\ell-i}$,
$K=H_2$, $K_1\subset H_2$ -- the subgroup of permutations of leaves above primary pearls). The proof
is completed by revoking Lemma~\ref{D2}.
\end{proof}
%\todo{Finish the proof. The idea is to analize the }
%The proof is similar to that of the fact \cite[Corollary 5.2]{BM} that the class of $\Sigma$-cofibrant operads is closed under cofibrations.
%One needs to show that for any generating cofibration $\partial X\rightarrow X$ of $\Sigma Seq$, the inclusion of free bimodules 
%$\mathcal{F}_{P\,;\,Q}^\Sigma(\partial X)\rightarrow \mathcal{F}_{P\,;\,Q}^\Sigma(X)$ is a $\Sigma$-cofibration. 

\subsubsection{Extension/restriction adjunction for the projective model category of bimodules} \label{E8}

Let $\phi_{1}:P\rightarrow P'$ and $\phi_{2}:Q\rightarrow Q'$ be  maps of operads. Similarly to the category of algebras (see Theorem~\ref{E9}), we show that the projective model categories of $(P\text{-}Q)$-bimodules and of $(P'\text{-}Q')$-bimodules are Quillen equivalent under some conditions on the maps $\phi_{1}$ and $\phi_{2}$. For this purpose, we recall the constructions of the \textit{restriction} functor $\phi^{\ast}$ and the \textit{extension} functor $\phi_{!}$ in the context of bimodules:
$$
\phi_{!}:\Sigma\Bimod_{P\,;\,Q}\rightleftarrows \Sigma\Bimod_{P'\,;\,Q'}:\phi^{\ast}.
$$

\noindent \textit{$\bullet$ The restriction functor.} The restriction functor $\phi^{\ast}$ sends a $(P'\text{-}Q')$-bimodule $M$ to the $(P\text{-}Q)$-bimodule $\phi^{\ast}(M)=\{\phi^{\ast}(M)(n)=M(n),\,n\geq 0\}$ in which the $(P\text{-}Q)$-bimodule structure is defined as follows using the $(P'\text{-}Q')$-bimodule structure of $M$:
$$
\begin{array}{rcl}\vspace{3pt}
\circ^{i}: \phi^{\ast}(M)(n)\times Q(m) & \longrightarrow & \phi^{\ast}(M)(n+m-1);\\ \vspace{10pt}
 x\,;\,q & \longmapsto & x\circ^{i}\phi_{2}(q), \\ \vspace{3pt}
\gamma_{\ell}:P(n)\times \phi^{\ast}(M)(m_{1})\times \cdots \times \phi^{\ast}(M)(m_{n}) & \longrightarrow & \phi^{\ast}(M)(m_{1}+\cdots + m_{n}); \\
p\,;\,x_{1},\ldots, x_{n} & \longmapsto & \phi_{1}(p)(x_{1},\ldots, x_{n}).
\end{array} 
$$

\noindent \textit{$\bullet$ The extension functor.} The extension functor $\phi_{!}$ is obtained as a quotient of the free $(P'\text{-}Q')$-bimodule functor introduced in Section \ref{C2}. More precisely, if $M$ is a $(P\text{-}Q)$-bimodule, then the extension functor is given by the formula 
$$
\phi_{!}(M)(n)=\mathcal{F}^{\Sigma}_{P'\,;\,Q'}(\mathcal{U}^{\Sigma}(M))(n)/\sim
$$ 
where the equivalence relation is generated by the axiom which consists in contracting inner edges having a vertex $v$ below the section (respectively above the section) indexed by a point of the form $\phi_{1}(p)$ (respectively a point of the form $\phi_{2}(q)$) using the left $P$-module structure (respectively the right $Q$-module structure) of $M$ as illustrated in the following picture:

\hspace{-30pt}\includegraphics[scale=0.33]{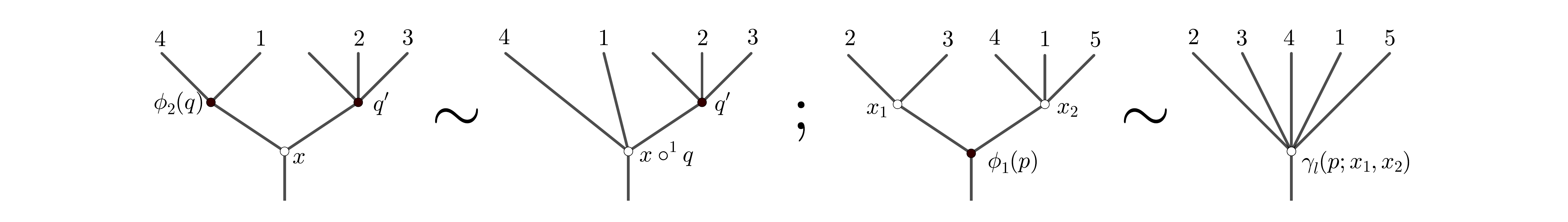}

The $(P'\text{-}Q')$-bimodule structure on the free object is compatible with the equivalence relation and provides a $(P'\text{-}Q')$-bimodule structure on $\phi_{!}(M)$. Let us remark that the $(P\text{-}Q)$-bimodule map $M\rightarrow \phi^{\ast}(\phi_{!}(M))$, sending a point $x\in M(n)$ to the $n$-corolla indexed by $x$, is not necessarily injective. For instance, if there are $q_{1}\neq q_{2}$ in $Q(m)$ and $x\in M(n)$ such that $\phi_{2}(q_{1})=\phi_{2}(q_{2})$ and $x\circ^{i}q_{1}\neq x\circ^{i}q_{2}$, then $x\circ^{i}q_{1}$ and $x\circ^{i}q_{2}$ have the same image in $\phi^{\ast}(\phi_{!}(M))$ as illustrated in the following picture:\vspace{-5pt}

\begin{figure}[!h]
\begin{center}
\includegraphics[scale=0.33]{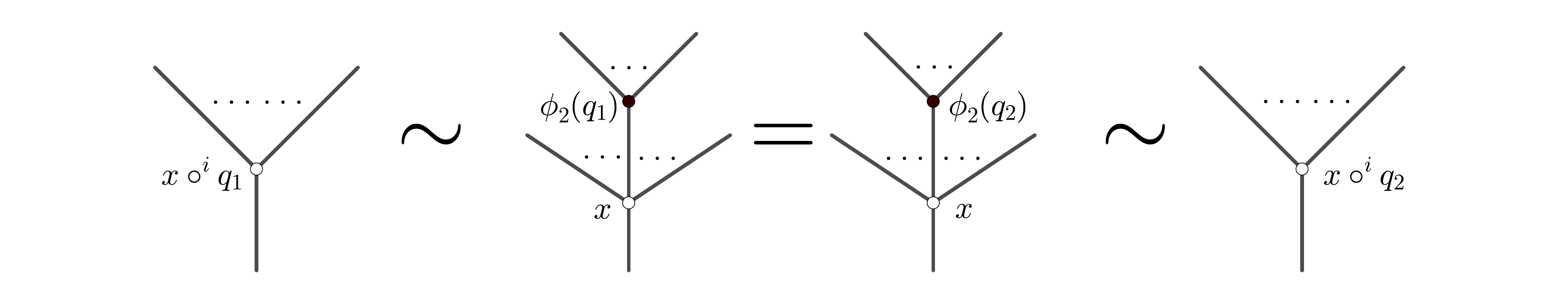}\vspace{-15pt}
\end{center}
\end{figure}

\begin{thm}\label{G6}
Let $\phi_{1}:P\rightarrow P'$ be a weak equivalence between  $\Sigma$-cofibrant operads and $\phi_{2}:Q\rightarrow Q'$ be a weak equivalence between componentwise cofibrant  operads. The extension and restriction functors, as well as their truncated versions, give rise to Quillen equivalences: 
\begin{equation}\label{eq:bimod_ind_restr1}
\phi_!\colon\Sigma\Bimod_{P\,;\,Q}\rightleftarrows\Sigma\Bimod_{P'\,;\,Q'}\colon\phi^*,
\end{equation}
\begin{equation}\label{eq:tr_bimod_ind_restr}
\phi_!\colon\TT_r \Sigma\Bimod_{P\,;\,Q}\rightleftarrows\TT_r \Sigma\Bimod_{P'\,;\,Q'}\colon\phi^*.
\end{equation}
\end{thm}\vspace{5pt}

\begin{lmm}\label{D9}
Any pair of operadic maps $(\phi_{1}\,,\,\phi_{2})$, with  $\phi_{1}:P\rightarrow P'$  a weak equivalence between  $\Sigma$-cofibrant operads and $\phi_{2}:Q\rightarrow Q'$  a weak equivalence between componentwise cofibrant operads, induces a weak equivalence $\phi_{+}:P{+}Q\rightarrow P'{+}Q'$ between $\Sigma$-cofibrant colored operads where $P{+}Q$ and $P'{+}Q'$ are colored operads obtained from Construction \ref{D5}.
\end{lmm}

\begin{proof}
According to Remark \ref{D6}, for any family of integers $n_{1},\ldots,n_{k}$ and $m$, the spaces $(P{+}Q)(n_{1},\ldots,n_{k};m)$ and $(P'{+}Q')(n_{1},\ldots,n_{k};m)$ have the following description: 
$$
\begin{array}{ccccc}\vspace{9pt}
(P{+}Q)(n_{1},\ldots,n_{k};m) & \cong & P(k)\times Q_{1}(n_{1}+ \ldots + n_{k};m) & \cong & P(k)\times \underset{\alpha:[m]\rightarrow [n_{1}+\cdots + n_{k}]}{\displaystyle\coprod}\hspace{6pt}\underset{i\in [n_{1}+\cdots + n_{k}]}{\displaystyle\prod} Q(|\alpha^{-1}(i)|), \\ 
(P'{+}Q')(n_{1},\ldots,n_{k};m) & \cong & P'(k)\times Q'_{1}(n_{1}+ \ldots + n_{k};m) & \cong & P'(k)\times \underset{\alpha:[m]\rightarrow [n_{1}+\cdots + n_{k}]}{\displaystyle\coprod}\hspace{6pt}\underset{i\in [n_{1}+\cdots + n_{k}]}{\displaystyle\prod} Q'(|\alpha^{-1}(i)|).
\end{array} 
$$
Consequently, the weak equivalences $\phi_{1}:P\rightarrow P'$ and $\phi_{2}:Q\rightarrow Q'$  induce a weak equivalence of colored operads $\phi_{+}:P{+}Q\rightarrow P'{+}Q'$. Let $\Sigma'$ be the subgroup $\Sigma_k$ which can send $i$ to $j$ if and only if $n_i=n_j$.
This group   acts as a subgroup of $\Sigma_k$ on the factor $P(k)$ (respectively the factor $P'(k)$) and it 
acts on the factor $Q_1(n_{1}+ \ldots + n_{k};m)$ (respectively, the factor $Q'_1(n_{1}+ \ldots +n_{k};m)$) by reordering the summands in the disjoint union labelled by maps $\alpha\colon [m]\to [n_1+\ldots+n_k]$. 
The reordering is induced by permutation of the blocks $\{1,\ldots,n_{1}\},\ldots,\{n_{1}+\cdots + n_{k-1}+1,\ldots, n_{1}+\cdots + n_{k}\}$ in
$[n_1+\ldots+n_k]$.   
Consequently, the colored operads $P{+}Q$ and $P'{+}Q'$ are $\Sigma$-cofibrant as soon as the operads $P$ and $P'$ are $\Sigma$-cofibrant
and the components of $Q$ and $Q'$ are cofibrant.
\end{proof}

\begin{proof}[Proof of Theorem \ref{G6}]
As explained in Section \ref{D8}, the projective model category of bimodules over an operad is equivalent to the projective model category of algebras over a specific colored operad. If we denote by $P{+}Q$ and $P'{+}Q'$ the colored operads obtained from Construction \ref{D5}, then one has 
$$
\phi_! :\Sigma\Bimod_{P\,;\,Q}= Alg_{P{+}Q} \rightleftarrows Alg_{P'{+}Q'} = \Sigma\Bimod_{P'\,;\,Q'}:\phi^*,
$$
induced by the extension/restriction adjunction between the categories of algebras. Due to Lemma \ref{D9}, the induced map $\phi_{+}:P{+}Q\rightarrow P'{+}Q'$ is a weak equivalence between $\Sigma$-cofibrant colored operads. Consequently, according to~\cite[Theorem~4.4]{BM} the extension/restriction adjunction between the categories of algebras is a Quillen equivalence. 
\end{proof}\vspace{10pt}

\noindent \textit{$\bullet$ Bimodules with the empty set in arity zero.} 
Consider the case where the acting operads $P$ and $Q$ are trivial in arity zero $P(0)=Q(0)=\emptyset$
and consider the full subcategory  $\Sigma_{>0}\Bimod_{P\,;\,Q}$ of $(P\text{-}Q)$-bimodules $M$ also satisfying $M(0)=
\emptyset$. One can similarly to Construction~\ref{D5} define a colored operad $(P{+}Q)_{>0}$ that governs this
algebraic structure. Its set of colors is the set $\NN_{>0}$ of positive integers. Its components can be similarly described:
\begin{equation}\label{eq:P+Q>0}
(P{+}Q)_{>0}(n_1,\ldots,n_k;m) = P(k)\times \underset{\alpha:[m]\twoheadrightarrow [n_{1}+\cdots + n_{k}]}{\displaystyle\coprod}\hspace{6pt}\underset{i\in [n_{1}+\cdots + n_{k}]}{\displaystyle\prod} Q(|\alpha^{-1}(i)|).
\end{equation}
The crucial difference is that the coproduct is taken over surjective maps $\alpha:[m]\twoheadrightarrow [n_{1}+\cdots + n_{k}]$.

\begin{pro}\label{p:P+Q>0}
Let $P$ and $Q$ be componentwise cofibrant operads   satisfying $P(0)=Q(0)=\emptyset$, then
the colored operad $(P{+}Q)_{>0}$ is $\Sigma$-cofibrant.
\end{pro}

\begin{proof}
Consider the component~\eqref{eq:P+Q>0} of $(P{+}Q)_{>0}$.
 Let $\Sigma'$ be the subgroup of $\Sigma_k$ which can send $i$ to $j$ if and only if $n_i=n_j$. One has that $P(k)$ is a cofibrant $\Sigma'$-space, while the second factor
\[
 \underset{\alpha:[m]\twoheadrightarrow [n_{1}+\cdots + n_{k}]}{\displaystyle\coprod}\hspace{6pt}\underset{i\in [n_{1}+\cdots + n_{k}]}{\displaystyle\prod} Q(|\alpha^{-1}(i)|)
 \]
 is $\Sigma'$-cofibrant. Indeed, since all $\alpha$'s are surjective in the coproduct, the group $\Sigma'$ acts freely on this disjoint union. Applying Lemma~\ref{D2} for $G_1=1$, $G=G_2=\Sigma'$,
 we get that the component~\eqref{eq:P+Q>0} is $\Sigma'$-cofibrant. 
\end{proof}

\begin{thm}\label{th:>0}
Let $\phi_{1}:P\rightarrow P'$  and $\phi_{2}:Q\rightarrow Q'$ be  weak equivalences between componentwise cofibrant  operads 
 satisfying $P(0)=P'(0)=Q(0)=Q'(0)=\emptyset$. The extension and restriction functors, as well as their truncated versions, give rise to Quillen equivalences: 
\begin{equation}\label{eq:bimod_ind_restr2}
\phi_!\colon\Sigma_{>0}\Bimod_{P\,;\,Q}\rightleftarrows\Sigma_{>0}\Bimod_{P'\,;\,Q'}\colon\phi^*,
\end{equation}
\begin{equation}\label{eq:tr_bimod_ind_restr}
\phi_!\colon\TT_r \Sigma_{>0}\Bimod_{P\,;\,Q}\rightleftarrows\TT_r \Sigma_{>0}\Bimod_{P'\,;\,Q'}\colon\phi^*.
\end{equation}
\end{thm}\vspace{5pt}

\begin{proof}
The proof is similar to that of Theorem~\ref{G6}. We use Proposition~\ref{p:P+Q>0} that the operads
$(P{+}Q)_{>0}$ and $(P'{+}Q')_{>0}$ governing these structures are $\Sigma$-cofibrant.
\end{proof}

%
%\begin{rmk}\label{r:comp_proj_bim}
%If $P$ and $Q$ are operads with $P(0)=Q(0)=\emptyset$ and if we consider the categoris of $(P\text{-}Q)$-bimodules
% also satisfying that their arity zero component is empty, then the corresponding operad $P{+}Q$ that encodes this structure is 
% $\Sigma$-cofibrant provided $P$ and $Q$ have cofibrant components. The $\Sigma$-cofibrancy of $P$ is not necessary because the summands
% in the components  $Q_1(n_1,\ldots,n_k;n)$ are labelled by surjective maps
%$M(0)=\varempty$, then 
%\end{rmk}
%\newpage

\section{The Reedy model category of $(P\text{-}Q)$-bimodules}\vspace{5pt}

Let $P$ be any topological operad and $Q$ be a reduced operad. From now on, we denote by $\Lambda \Bimod_{P\,;\,Q}$  and $T_{r}\, \Lambda \Bimod_{P\,;\,Q}$ the categories of $(P\text{-}Q)$-bimodules and $r$-truncated $(P\text{-}Q)$-bimodules, respectively, equipped with the Reedy model category structures. This structure is transferred from the categories $\Lambda Seq_P:=
P_0\downarrow\Lambda Seq$ and $T_{r}\,\Lambda Seq_P:=P_0\downarrow T_{r}\,\Lambda Seq $, respectively, along the adjunctions
\begin{equation}\label{F7}
\begin{array}{rcl}\vspace{8pt}
\mathcal{F}_{P\,;\,Q}^{\Lambda}:\Lambda Seq_P & \rightleftarrows & \Lambda \Bimod_{P\,;\,Q}:\mathcal{U}^{\Lambda}, \\ 
\mathcal{F}_{P\,;\,Q}^{T_{r}\Lambda}:T_{r}\Lambda Seq_P& \rightleftarrows & T_{r}\Lambda \Bimod_{P\,;\,Q}:\mathcal{U}^{\Lambda},
\end{array} 
\end{equation}
where both free functors are obtained from the functors $\mathcal{F}_{P\,;\,Q}^\Sigma$ and $\mathcal{F}_{P\,;\,Q}^{T_{r}\Sigma}$ by taking the restriction of the coproduct (\ref{G4}) to the reduced trees with section without univalent vertices other than the pearls. The (acyclic) generating cofibrations in $\Lambda Seq_P$ and $T_{r}\,\Lambda Seq_P$ are $\{P_0\sqcup\partial X\to P_0\sqcup X\}$, where $\{\partial X\to X\}$ is the set of (acyclic) generating cofibrations of $\Lambda Seq$ and  $T_r\Lambda\, Seq$,
respectively.

If we denote by $Q_{>0}$ the operad obtained from $Q$ by changing the arity $0$ component to the empty set (i.e. $Q_{>0}(0)=\emptyset$ and $Q_{>0}(n)=Q(n)$ for $n\geq 1$), then for any (possibly truncated) $\Lambda$-sequence $M$ and $n\geq 0$, one has
\begin{equation}\label{eq:free}
\mathcal{F}_{P\,;\,Q}^{\Lambda}(M)(n)\coloneqq 
\mathcal{F}^{\Sigma}_{P\,;\,Q_{>0}}(M)(n) \hspace{20pt} \text{and}\hspace{20pt} \mathcal{F}_{P\,;\,Q}^{T_{r}\Lambda}(M)(n)\coloneqq 
\mathcal{F}^{T_{r}\Sigma}_{P\,;\,Q_{>0}}(M)(n).
\end{equation}
By construction, the above $\Sigma$-sequences are equipped with a  (possibly truncated) right module structures over $Q_{>0}$. We can extend this structure in order to get a (possibly truncated) right $Q$-module structures using the operadic structure of $Q$ and the $\Lambda$ structure of $M$.

\hspace{-55pt}\includegraphics[scale=0.32]{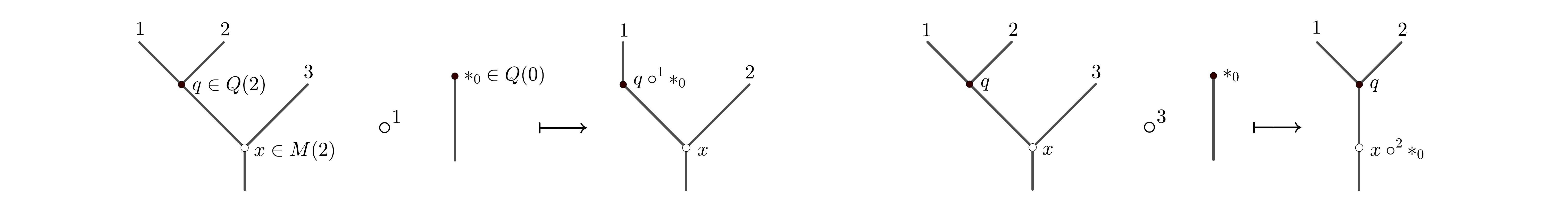}\vspace{-10pt}
 
\begin{figure}[!h]
\caption{Illustration of the right action by $\ast_{0}$.}
\end{figure}

\begin{thm}\label{Z5}
Let $P$ be an operad and $Q$ be a reduced  well-pointed operad. The categories $\Lambda \Bimod_{P\,;\,Q}$ and $T_{r}\Lambda \Bimod_{P\,;\,Q}$, with $r\geq 0$, admit cofibrantly generated model category structures, called Reedy model category structures, transferred from $\Lambda Seq_P$ and $T_{r}\Lambda Seq_P$, respectively, along the adjunctions $(\ref{F7})$. In particular, these model category structures make the pairs of functors $(\ref{F7})$ into Quillen adjunctions.
\end{thm}

\begin{proof}
According to the transfer principle \ref{E3}, we have to check the small object argument as well as the existence of a functorial fibrant replacement and a functorial factorization of the diagonal map in the category $\Lambda \Bimod_{P\,;\,Q}$. 

For the small object argument, let us remark that the pushout in the category of $(P\text{-}Q)$-bimodules is defined exactly as in the  category of $(P\text{-}Q_{>0})$-bimodules. More precisely, for any $n\geq 0$, one has\footnote{The forgetful functor from the category of $(P\text{-}Q)$-bimodules
to the category of $(P\text{-}Q_{>0})$-bimodules preserves colimits as it admits a right adjoint by Proposition~\ref{p:tens_hom}b for $Q_1=Q$, $Q_2=Q_{>0}$, $Y=Q$.}: 
\begin{equation}\label{eq:colim_PQ}
\underset{\hspace{30pt}\Lambda \Bimod_{P;Q}}{\mathrm{colim}}\big(\, B \leftarrow A \rightarrow C\,\big)(n)=
\hspace{-30pt}\underset{\hspace{40pt}\Sigma \Bimod_{P;Q_{>0}}}{\mathrm{colim}}\big(\, B \leftarrow A \rightarrow C\,\big)(n).
\end{equation}

 Let $\mathcal{F}^{\Lambda}_{P;Q}(X)$ be a domain of an element in the set of generating (acyclic) cofibrations in $\Lambda \Bimod_{P\,;\,Q}$. Let $\lambda=\aleph_1$ %=\mathfrak{c}^+$
 and assume that $\{M_{\alpha}\}_{\alpha < \lambda}$ is a $\lambda$-sequence of   $(P\text{-}Q)$-bimodules, such that 
 each map $M_{<\alpha}:=\mathrm{colim}_{\beta<\alpha}M_\beta\to M_\alpha$  fits into a pushout square
\begin{equation}\label{eq:push_out_incl}
\xymatrix@R=15pt{
\mathcal{F}^{\Lambda}_{P;Q}(\partial Y ) \ar[r] \ar[d] & \mathcal{F}^{\Lambda}_{P;Q}(Y) \ar[d] \\
M_{<\alpha}\ar[r] & M_\alpha,
}
\end{equation}
where $\partial Y\to Y$ is a possibly infinite coproduct in $\Lambda Seq_P$ of generating (acyclic) cofibrations. It follows from Lemma~\ref{Final3} and equations~\eqref{eq:free} and~\eqref{eq:colim_PQ}, that each map $M_{<\alpha}\to M_\alpha$
is an objectwise closed inclusion. Moreover, $\left(\coprod_{n\geq 0}X(n)\right)\setminus P(0)$ is a finite union of spheres (or discs), thus separable,
and therefore, $X$ is $\aleph_1$-small relative to componentwise closed inclusions.
% $\mathfrak{c}^+$-small.  
  So, the same argument used for the proof of the small objects argument in Theorem~\ref{ProjectBimod} works.

%
%and for every ordinal $\lambda$ and every $\lambda$-sequence of cofibrations of reduced $(P\text{-}Q)$-bimodules is also a  $\lambda$-sequence of closed inclusions, while
%the the domains of generating (acyclic) cofibrations of $\Lambda Seq_P$ are again compact. So, the same argument used for the proof of the small objects argument in Theorem~\ref{ProjectBimod} works for reduced bimodules.

Contrary to the category of $\Sigma \Bimod_{P;Q}$, the objects in the category $\Lambda \Bimod_{P;Q}$ are not necessarily fibrant and the identity functor is not a fibrant replacement functor. The aim of Section \ref{Z2} is to introduce such a fibrant replacement functor if $Q$ is well-pointed. This fibrant replacement is different from the fibrant coresolution functor for $\Lambda$-sequences defined in Section \ref{A9}.

From now on, we introduce a functorial path object in the Reedy model category of $(P\text{-}Q)$-bimodules. In other words, for any $M\in \Lambda\Bimod_{P\,;\,Q}$ which is fibrant in the category of $\Lambda$-sequences, we build an element $Path(M)\in \Lambda\Bimod_{P\,;\,Q}$ such that there is a factorization of the diagonal map
$$
\xymatrix{
\Delta: M \ar[r]^{\hspace{-5pt}\simeq}_{\hspace{-5pt}f_{1}} & Path(M) \ar@{->>}[r]_{f_{2}} & M\times M,
}
$$
where $f_{1}$ is a weak equivalence and $f_{2}$ is a fibration. Let us consider
\begin{equation}\label{PathObject}
Path(M)(n)=Map\big( [0\,,\,1]\,;\,M(n)\big).
\end{equation}
The object so obtained inherits a bimodule structure from $M$. The map from $M$ to $Path(M)$, sending a point to the constant path, is clearly a homotopy equivalence. Furthermore, let us remark that one has the following identities:
$$
\mathcal{M}(Path(M))(n)=Map\big( [0\,,\,1]\,;\,\mathcal{M}(M)(n)\big)\hspace{15pt}\text{and}\hspace{15pt}\mathcal{M}(M\times M)(n)=\mathcal{M}(M)(n)\times \mathcal{M}(M)(n).
$$
So, the map $f_{2}$ is a fibration if the map between the limits induced by the natural transformation 
$$
\xymatrix{
Path(M)(n)\ar@{=}[r]\ar[d] & Path(M)(n) \ar[d] & Path(M)(n)\ar@{=}[l]\ar[d]\\
Map\big( [0\,,\,1]\,;\,\mathcal{M}(M)(n)\big) \ar[r] & \mathcal{M}(M)(n)\times \mathcal{M}(M)(n) & M(n)\times M(n)\ar[l]
}
$$
is a Serre fibration. The right vertical map is obviously a Serre fibration because the inclusion from $\partial [0\,,\,1]$ into $[0\,,\,1]$ is a cofibration (see the proof of Theorem \ref{ProjectBimod}). Moreover, since the inclusion $\partial [0\,,\,1]\rightarrow [0\,,\,1]$ is a cofibration and the map $M(n)\rightarrow \mathcal{M}(M)(n)$ is a fibration, an alternative version of the pushout product lemma  (see \cite[Section 9.1.5]{Hir}) implies that the map
$$
Map\big([0\,,\,1]\,;\, M(n)\big)\longrightarrow Map\big([0\,,\,1]\,;\, \mathcal{M}(M)(n)\big) \underset{Map\big(\partial [0\,,\,1]\,;\, \mathcal{M}(M)(n)\big)}{\times} Map\big(\partial [0\,,\,1]\,;\, M(n)\big)
$$
is also a Serre fibration. As a consequence of Lemma \ref{E6}, the map $f_{2}$ is a Serre fibration.  
\end{proof}

\subsection{Properties of the Reedy model category of bimodules}\label{SectPropReedyBimod}

This subsection is divided into three parts. The first one is devoted to the construction of an explicit fibrant coresolution functor for reduced bimodules. In the second part, we  characterize (acyclic) cofibrations in the Reedy model category of bimodules as (acyclic) cofibrations in the usual projective model category of bimodules. The last part consists in extending the properties introduced in Section \ref{H2} to the Reedy model category. 

\subsubsection{A Reedy fibrant replacement functor for bimodules}\label{Z2}

Let $P$ be an operad and $Q$ be a reduced operad. The goal of this section is to give an explicit Reedy fibrant replacement in the category of bimodules if the operad $Q$ is well-pointed.
 A  conceptual description of this fibrant coresolution in terms of internal hom is given in Section~\ref{ss:int_hom}.\vspace{5pt}
 % In what follows, the constructions are explicit and more convenient to verify the Reedy fibrant conditions. Furthermore, we expect to extend this method to a more general context or different settings.

\noindent \textit{$\bullet$ The set of trees $\mathbb{P}[n]$.} Let $\mathbb{P}[n]$ be the set 
 of planar
trees~$T$ whose roots have exactly $n$ incoming edges. We label their leaves with the identity permutation in $\Sigma_{|T|}$,
where $|T|$ is the number of leaves in~$T$. We also label the $n$ incoming edges bijectively by the set $[n]$ in the planar order
from left to right. Such a tree $T$ is equipped with an orientation towards the root and we say that $v<v'$ if the path joining the vertex $v'$ with the root passes through the vertex $v$. It makes the set of vertices $V(T)$ into a partially ordered set. Moreover, we consider the operations
$$
\begin{array}{rcll}\vspace{8pt}
\delta_{i,m}:\mathbb{P}[n+m-1] & \longrightarrow &  \mathbb{P}[n], & \text{with } n,m\in \mathbb{N} \text{ and } 1\leq i\leq n, \\ 
\Gamma^{m}_{k}:\mathbb{P}[n] & \longrightarrow & \mathbb{P}[m], & \text{with } n,m,k\in \mathbb{N} \text{ and } m+k\leq n, 
\end{array}  
$$

The operation $\delta_{i,m}(T)$ is defined as follows. If $m=0$, then $\delta_{i,m}$ consists in adding an incoming edge to the root of $T$ at the $i$-th position. The new incoming edge is connected to a univalent vertex. If $m>0$, then $\delta_{i,m}(T)$ is obtained from $T$ by gluing together the incoming edges $i,i+1,\ldots,i+m-1$, of the root counted according to the planar order. In both cases, $\delta_{i,m}(T)$ has one additional vertex to those of $T$. The new vertex has exactly $m$ incoming edges.\vspace{20pt}

\hspace{-40pt}\includegraphics[scale=0.4]{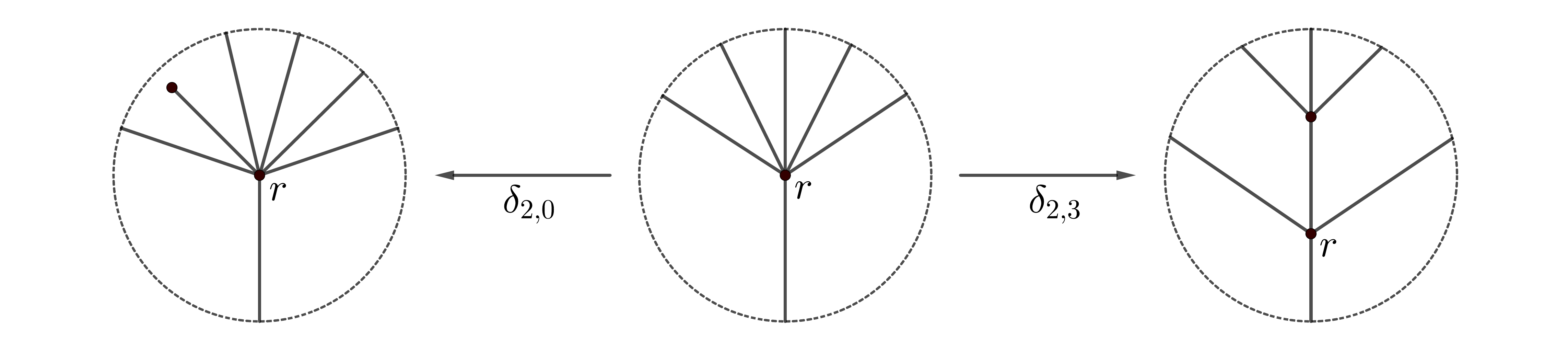}\vspace{-20pt}
\begin{figure}[!h]
\begin{center}
\caption{Illustration of the applications $\delta_{2,0}$ and $\delta_{2,3}$.}
\end{center}
\end{figure}

 For any $n$, $m$ and $k$ such that $m+k\leq n$, the map $\Gamma_{k}^{m}:\mathbb{P}[n]\rightarrow \mathbb{P}[m]$  consists in removing $n-m$ incoming edges together with the trees attached to them. The 
 removed edges  are  those labelled by $[n]\setminus \{k+1,\ldots,k+m\}$.%\vspace{-20pt}
\begin{figure}[!h]
\begin{center}
\includegraphics[scale=0.4]{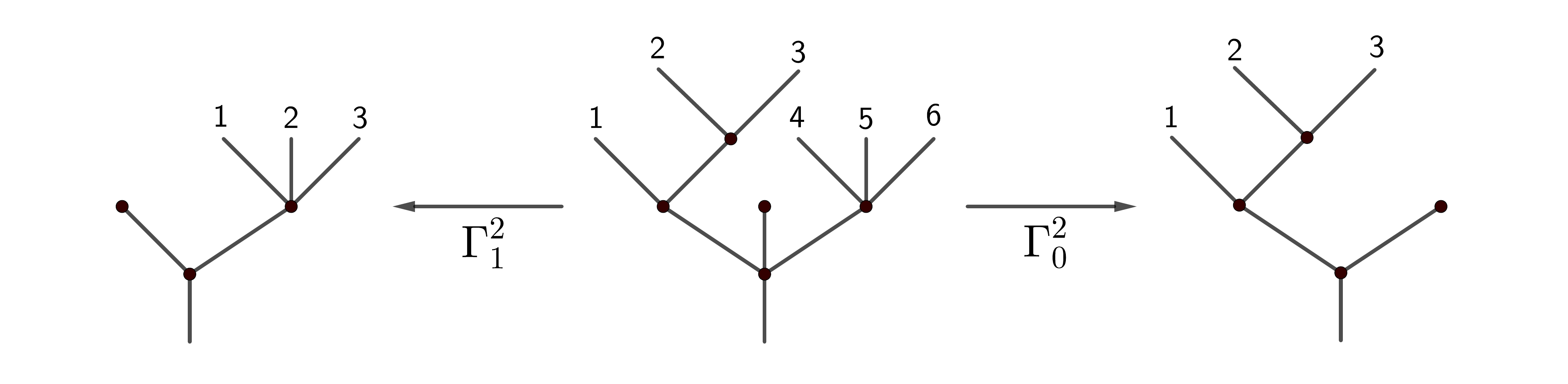}%\vspace{-15pt}
\caption{Illustration of the applications $\Gamma_{1}^{2}$ and $\Gamma_{0}^{2}$.}
\end{center}
\end{figure}
\\{}\\

\noindent \textit{$\bullet$ Construction of the fibrant replacement functor.} From now on, we fix a $(P\text{-}Q)$-bimodule $M$. Let $T$ be an element in $\mathbb{P}[n]$. We consider the spaces $H(T)$ of indexing the vertices of $T$ other than the root~$r$ by real numbers and points in the operad $Q$, respectively. More precisely, one has   
$$
\begin{array}{lcl}\vspace{4pt}
D(T) & = & \underset{v\in V(T)\setminus \{r\}}{\displaystyle\prod} Q(|v|), \\ 
H(T) & = & \left\{\left. \{t_{v}\}\in [0\,,\,1]^{|V(T)\setminus \{r\}|}\,\,\right| \forall v<v', t_{v}\leq t_{v'}
\right\}.
\end{array} 
$$
In other words, $H(T)$ is the space of order preserving maps
\[
V(T)\setminus\{r\} \to [0,1].
\]
Finally, we denote by $M^{f}(n)$ the subspace
\begin{equation}\label{eq:Mf}
M^{f}(n)\subset \underset{T\in \mathbb{P}[n]}{\prod} Map\big(\, H(T)\times D(T)\,;\,M(|T|)\,\big)
\end{equation}
composed of families of continuous maps $\{f_{T}\}_{T\in \mathbb{P}[n]}$ satisfying the following conditions:

\begin{itemize}[itemsep=7pt]

\item[$1.$] Let $T$ be an element in $\mathbb{P}[n]$ having a bivalent vertex $v$ other than the root. Then one has 
\begin{equation}\label{rel1}
\xymatrix{
H(T)\times D(T\setminus \{v\}) \ar[r] \ar[d] & H(T)\times D(T) \ar[d]^{f_{T}} \\
H(T\setminus \{v\})\times D(T\setminus \{v\}) \ar[r]_{f_{T\setminus \{v\}}} & M(|T\setminus \{v\}|)=M(|T|)
}
\end{equation}
where $T\setminus \{v\}$ is the tree obtained from $T$ by removing the bivalent vertex $v$ (i.e. by replacing the incoming and outgoing edges of $v$ by a single edge). The upper horizontal map indexes the vertex $v$ by the unit of the operad $Q$ while the left vertical map is obtained by forgetting the the real number indexing the bivalent vertex $v$. 

\item[$2.$] For any non-root vertex $v$ of $T$ and any permutation $\sigma$ of the incoming edges of $v$, one has 
\begin{equation}\label{rel0}
\xymatrix{
H(T)\times D(T) \ar[r] \ar[d]_{f_{T}} & H(T\cdot\sigma)\times D(T\cdot\sigma) \ar[d]^{f_{T\cdot\sigma}} \\
M(|T|) \ar[r]_{\hspace{-20pt}\sigma_{L}[T]^{\ast}} & M(|T\cdot\sigma|)=M(|T|)
}
\end{equation}
where $T\cdot\sigma$ is the tree obtained from $T$ by permuting the incoming edges of $v$ according to the permutation $\sigma$ and $\sigma_{L}[T]\in\Sigma_{|T|}$ is the induced permutation of the leaves of $T$. The upper horizontal map sends the decorations of the tree $T$ to the corresponding decorations of $T\cdot\sigma$  and acts on the $Q$-decoration of the vertex $v$ using the $\Sigma$-structure of the operad $Q$.

\item[$3.$] For any inner edge $e$, which is not connected to the root, one has
\begin{equation}\label{rel2}
\xymatrix{
H(T/e)\times D(T) \ar[r] \ar[d] & H(T)\times D(T) \ar[d]^{f_{T}} \\
H(T/e)\times D(T/e) \ar[r]_{f_{T/e}} & M(|T/e|)=M(|T|)
}
\end{equation}
where $T/e$ is the tree obtained from $T$ by contracting the edge $e$. The upper horizontal map indexes the source and the target vertices of $e$ by the real number in $H(T/e)$ corresponding to the vertex resulting from the contraction of $e$. The left vertical map is defined using the operadic structure of~$Q$.

\item[$4.$] Any tree $T$ has a unique decomposition of the form $T=T_{e}\circ_{i(e)}T_{e}'$ along any edge $e$. Then one has
\begin{equation}\label{rel3}
\xymatrix{
H(T_{e})\times D(T_{e})\times \underset{v\in V(T_{e}')}{\prod}Q(|v|) \ar[r] \ar[d]_{f_{T_{e}}\times \eta} & H(T)\times D(T) \ar[d]^{f_{T}}\\
M(|T_{e}|)\times Q(|T_{e}'|)\ar[r]_{\circ^{i(e)}} & M(|T|)
}
\end{equation}
The upper horizontal map consists in indexing the vertices associated to the tree $T_{e}'$ by $1$. The map $\eta:\prod_{v\in V(T_{e}')}Q(|v|)\rightarrow Q(|T_{e}'|)$ is defined using the operadic structure of $Q$ while the lower horizontal map is obtained using the right $Q$-module structure of $M$. 
\end{itemize}

\begin{rmk}
Let us notice that, as a special case of the fourth condition, for any univalent vertex $v$, one has 
\begin{equation*}
\xymatrix{
H(T\setminus \{v\})\times D(T\setminus \{v\})\ar[r] \ar[d]_{f_{T\setminus \{v\}}} & H(T)\times D(T) \ar[d]^{f_{T}}\\
M(|T\setminus \{v\}|)=M(|T|+1) \ar[r]_{\hspace{40pt}\circ^{v}\ast_{0}} & M(|T|)
}
\end{equation*}
where $T\setminus \{v\}$ is the tree obtained from $T$ by removing the univalent vertex $v$ (and thus producing one more leaf). The upper horizontal map consists in indexing the vertex $v$ by the real number $1$ and the unique point $\ast_{0}\in Q(0)$. The lower horizontal map composes the input of $M(|T\setminus \{v\}|)$, corresponding to the vertex $v$, with the point $\ast_{0}\in Q(0)$ using the right $Q$-module structure of $M$.
\end{rmk}\vspace{3pt}

\noindent \textit{$\bullet$ The $\Sigma$-structure on the fibrant coresolution.} The space $M^{f}(n)$ inherits an action of the permutation group $\Sigma_{n}$. More precisely, for any $\sigma\in \Sigma_{n}$, we denote by $T_{\sigma}$ the tree obtained from $T\in \mathbb{P}[n]$ by permuting the incoming edges associated to the root of $T$ according to the permutation $\sigma$. Such a permutation induces the following two bijections:
\begin{equation}\label{Finish1}
\begin{array}{rrlll}\vspace{8pt}
\sigma_{V}[T]: & V(T\setminus \{r\}) & \longrightarrow &  V(T_{\sigma}\setminus \{r\}) & \in \Sigma_{|V(T)\setminus \{r\}|},  \\ 
\sigma_{L}[T]:  & \ell(T) & \longrightarrow & \ell(T_{\sigma}) &\in \Sigma_{|T|}. 
\end{array} 
\end{equation}
Here, $\ell(T)$ denotes the set $[|T|]$ of leaves. 
So, the action of the permutation group $\sigma^{\ast}:M^{f}(n)\rightarrow M^{f}(n)$ sends a family of continuous maps $\{f_{T}\}_{T\in \mathbb{P}[n]}$ to the family $\{(f\cdot\sigma)_{T}\}_{T\in \mathbb{P}[n]}$ given by the formula
$$
\begin{array}{rcl}\vspace{8pt}
(f\cdot\sigma)_{T}: H(T)\times D(T) & \longrightarrow & M(|T|); \\ 
\{t_{v}\}\,\,,\,\,\{q_{v}\} & \longmapsto & f_{T_{\sigma}}\big(\, \{t_{\sigma_{V}[T](v)}\}\,\,,\,\,\{q_{\sigma_{V}[T](v)}\}\,\big)\cdot \sigma_{L}[T].
\end{array} 
$$\vspace{3pt}

\noindent \textit{$\bullet$ The bimodule structure on the fibrant coresolution.} Since one has the identity $M^{f}(0)=M(0)$, one has a map $\gamma_{0}:P(0)\rightarrow M^{f}(0)$ and the $\Sigma$-sequence $M^{f}$ inherits a $(P\text{-}Q)$-bimodule structure whose right operations are given by
\begin{equation}\label{Finish2}
\begin{array}{rcl}\vspace{5pt}
\circ^{i}:M^{f}(n)\times Q(m) & \longrightarrow & M^{f}(n+m-1); \\ 
\{f_{T}\}_{T\in \mathbb{P}[n]}\,\,,\,\,q & \longmapsto & \{(f\circ^{i}q)_{T}\}_{T\in \mathbb{P}[n+m-1]},
\end{array} 
\end{equation}
where $(f\circ^{i}q)_{T}$ is the composite map:
$$
\xymatrix{
(f\circ^{i}q)_{T}:H(T)\times D(T) \ar[r] & H(\delta_{i,m}(T))\times  D(\delta_{i,m}(T)) \ar[r]_{f_{\delta_{i,m}(T)}} & M(|\delta_{i,m}(T)|)=M(|T|).
}
$$
The left hand side map consists in indexing the new vertex by the real number $0$ and the point $q\in Q(m)$. Similarly, the left $P$-module structure on $M^{f}$ is given by the operations 
$$
\begin{array}{rlcl}\vspace{5pt}
\gamma_{\ell}:&P(k)\times M^{f}(n_{1})\times\cdots \times M^{f}(n_{k}) & \longrightarrow & M^{f}(n_{1}+\cdots+n_{k}); \\ 
&p\,\,,\,\, \{f_{T}^{1}\}_{T\in \mathbb{P}[n_{1}]},\ldots, \{f_{T}^{k}\}_{T\in \mathbb{P}[n_{k}]} & \longmapsto & \{p(f^{1},\ldots,f^{k})_{T}\}_{T\in \mathbb{P}[n_{1}+\cdots+n_{k}]},
\end{array} 
$$
where $p(f^{1},\ldots,f^{k})_{T}$ is the composite map
$$
\xymatrix@C=75pt{
H(T)\times D(T) \ar[d]_{\hspace{-60pt}\cong}\ar[r]^{\hspace{30pt}p(f^{1},\ldots,f^{k})_{T}} &  M(|T|)  \\
\underset{1\leq i \leq k}{\prod} H(\Gamma_{n_{1}+\cdots + n_{i-1}}^{n_{i}}(T))\times D(\Gamma_{n_{1}+\cdots + n_{i-1}}^{n_{i}}(T)) \ar[r]^{\hspace{30pt}\times_{i}f_{\Gamma_{n_{1}+\cdots + n_{i-1}}^{n_{i}}(T)}} &  \underset{1\leq i \leq k}{\prod} M(|\Gamma_{n_{1}+\cdots + n_{i-1}}^{n_{i}}(T)|)\ar[u]_{\gamma_{\ell}(p;-,\cdots, -)}.
}
$$ 

Finally, there is a map of $(P\text{-}Q)$-bimodules $\eta:M\rightarrow M^{f}$ sending a point $m\in M(n)$ to the family of continuous maps $\{\eta(m)_{T}\}_{T\in \mathbb{P}[n]}$ given by the formula
$$
\begin{array}{rcl}\vspace{5pt}
\eta(m)_{T}:H(T)\times D(T) & \longrightarrow & M(|T|); \\ 
\{t_{v}\}\,\,,\,\,\{q_{v}\} & \longmapsto & m\circ\{q_{v}\},
\end{array} 
$$
using the right $Q$-bimodule structure of $M$. It means that $m$ is taken for a root and we compose with $\{q_{v}\}$ using the operadic structure of $Q$ and the right module structure of $M$.

%\newpage

\begin{pro}\label{ProFibrantEq}
The map $\eta:M\rightarrow M^{f}$ is a weak equivalence.
\end{pro}

\begin{proof}
The proof is similar to the proof of Proposition \ref{Z7}. We show that the map of $\Lambda$-sequences $\eta_{n}:M(n)\rightarrow M^{f}(n)$ is a homotopy equivalence of $\Sigma$-sequences. For this purpose, we introduce a map of $\Sigma$-sequences $\psi:M^{f}\rightarrow M$ given by
$$
\xymatrix@R=-2pt{
\psi_{n}:\hspace{-30pt} &M^{f}(n) \ar[r] & M(n);\\
& \{f_{T}\}_{T\in \mathbb{P}[n]} \ar@{|->}[r] & f_{C_{n}}(\ast), 
}
$$
where $C_{n}$ is the $n$-corolla whose corresponding space $H(C_{n})\times D(C_{n})$ is necessarily the one point topological space. The map $\psi$ so obtained makes $\eta$ into a deformation retract and the homotopy consists in bringing the real numbers indexing the vertices other than the root to $1$. 
\end{proof}

\begin{pro}\label{MF}
If the operad $Q$ is well-pointed, then the $(P\text{-}Q)$-bimodule $M^{f}$ is Reedy fibrant. 
\end{pro}

\begin{proof}
The proof is divided  into two parts. First, we identify the matching object $\mathcal{M}(M^{f})(n)$ with a space ${\mathcal M}_{0}(n)$ defined in the same way as $M^{f}(n)$ by changing slightly the construction of the space $H(T)$. Thereafter, we build a tower of fibrations related to $ M^{f}(n)$ and ${\mathcal M}_{0}(n)$ according to the number of vertices of the trees in $\mathbb{P}[n]$. In the
proof we  use equivariant homotopy
theory  techniques
from Appendix~\ref{ss:ap_cell}.\vspace{7pt}

\noindent \textbf{Simplification of the matching object: }
We say that two trees $T_1$ and $T_2$ from $\mathbb{P}[n]$ are equivalent if they are isomorphic as non-planar trees and moreover the isomorphism between $T_1$ and $T_2$ preserves the order of the incoming edges of the root~$r$.
We denote by $\mathbb{T}[n]$ the so obtained set of equivalence classes. For $T\in\mathbb{T}[n]$ we denote by $Aut(T)$
the set of automorphisms of~$T$ which preserve the order of the incoming edges to the root.  Let $|T|$ denote the number of leaves of $T$. 
 For each $T\in\mathbb{T}[n]$ we choose a planar representative and  we label the leaves of $T$ in the corresponding planar order they appear. This gives us a homomorphism $Aut(T)\to \Sigma_{|T|}$, which is not always injective due to
 the presence of arity zero vertices.

Because of the relation~\eqref{rel0}, the space $M^f(n)$ can be equivalently described as the subspace
\[
M^{f}(n)\subset \underset{T\in \mathbb{T}[n]}{\prod} Map_{Aut(T)}\big(\, H(T)\times D(T)\,;\,M(|T|)\,\big)
\]
composed of families of continuous maps $\{f_{T}\}_{T\in \mathbb{T}[n]}$ satisfying the conditions~\eqref{rel1},
\eqref{rel2}, \eqref{rel3} properly understood. 

 For any tree $T\in \mathbb{T}[n]$, we denote by $H_{0}(T)$ the subspace of $H(T)$ which consists of points having at least one univalent vertex labelled by $0$ and connected to the root of $T$. Let ${\mathcal M}_{0}(n)$ be the subspace 
$$
{\mathcal M}_{0}(n)\subset \underset{T\in \mathbb{T}[n]}{\prod} Map_{Aut(T)}\big(\, H_{0}(T)\times D(T)\,;\,M(|T|)\,\big),
$$
satisfying the relations \eqref{rel1},  \eqref{rel2} and \eqref{rel3}. In order to show that the space ${\mathcal M}_{0}(n)$ is homeomorphic to the matching object $\mathcal{M}(M^{f})(n)$, we recall that a point in $\mathcal{M}(M^{f})(n)$ is a family of maps $\{\phi_{T,h}\}$, indexed by $h\in \Lambda_{+}([\ell],[n])$, $T\in \mathbb{T}[\ell]$ and $\ell<n$, satisfying some conditions related to the limit~\eqref{eq:matching} as well as the relations \eqref{rel1},   \eqref{rel2} and \eqref{rel3}.

For each pair $(T,h)$, with $h\in \Lambda_{+}([\ell],[n])$ and $T\in \mathbb{T}[\ell]$, we denote by $T[h]$ the tree in $\mathbb{T}[n]$ obtained from $T$ by adding univalent vertices connected to the root according to the order preserving map $h$ (see Figure \ref{FigureSimplification}). So the map 
$$
\begin{array}{rcl}\vspace{5pt}
\alpha:{\mathcal M}_{0}(n) & \longrightarrow & \mathcal{M}(M^{f})(n); \\ 
\phi=\{\phi_{T}\} & \longmapsto & \{(\alpha\circ\phi)_{T,h}\},
\end{array} 
$$ 
is given by the composite maps
$$
\xymatrix{
(\alpha\circ\phi)_{T,h}:H(T)\times D(T) \ar[r] & H_{0}(T[h])\times D(T[h]) \ar[r]^{\phi_{T[h]}} & M(|T[h]|)=M(|T|),
} 
$$
where the left hand side map consists in indexing the new univalent vertices by the unique point in $Q(0)$ and the real number $0$.

\begin{figure}[!h]
\begin{center}
\includegraphics[scale=0.32]{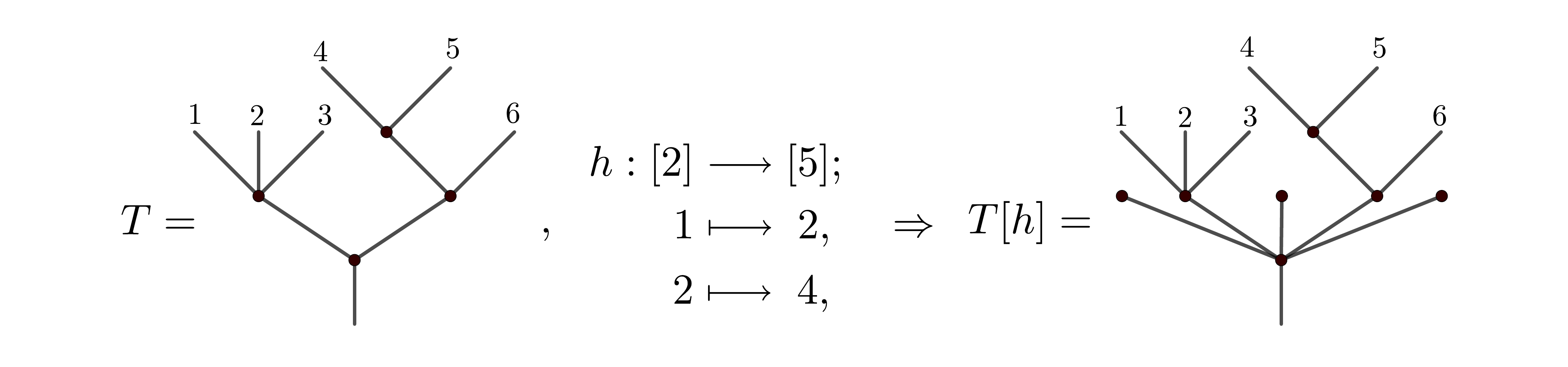}\vspace{-10pt}
\caption{Illustration of the construction of $T[h]$.}\label{FigureSimplification}\vspace{-10pt}
\end{center}
\end{figure}

Conversely, for any tree $T\in \mathbb{T}[n]$ and $\{t_{v}\}\in H_{0}(T)$, we denote by $\ell_{\{t_{v}\}}$ the number of incoming edges of the root of $T$ which are not connected to univalent vertices indexed by $0$. Furthermore, we denote by $h[\{t_{v}\}]:[\ell_{\{t_{v}\}}]\rightarrow [n]$ the order preserving map which keeps in mind the position %(according to the planar order of the tree)
 of the incoming edges of the root which are not connected to a univalent vertex indexed by $0$. Finally, $T[\{t_{v}\}]\in \mathbb{T}[\ell_{\{t_{v}\}}]$ is the tree obtained from $T$ by removing the incoming edges of the root  which are  connected to a univalent vertex indexed by $0$. So, one has the map 
$$
\begin{array}{rcl}\vspace{5pt}
\beta:\mathcal{M}(M^{f})(n) & \longrightarrow &{\mathcal M}_{0}(n) ; \\ 
\phi=\{\phi_{T,h}\} & \longmapsto & \{(\beta\circ\phi)_{T}\},
\end{array} 
$$ 
given by 
$$
\begin{array}{rcl}\vspace{5pt}
(\beta\circ\phi)_{T}:H_{0}(T)\times D(T) & \longrightarrow & M(|T|) ; \\ 
\{t_{v}\}\,,\,\{q_{v}\} & \longmapsto & \phi_{T[\{t_{v}\}],h[\{t_{v}\}]}(\{\tilde{t}_{v}\}\,,\,\{\tilde{q}_{v}\}),
\end{array} 
$$
where the families $\{\tilde{t}_{v}\}$ and $\{\tilde{q}_{v}\}$ are obtained from $\{t_{v}\}$ and $\{q_{v}\}$, respectively, by removing the parameters corresponding to the univalent vertices indexed by $0$ and connected to the root. The map $\beta$ is well defined due to the relations induced by the limit and, together with the map $\alpha$, induces a homeomorphism between ${\mathcal M}_{0}(n)$ and the matching object $\mathcal{M}(M^{f})(n)$. Furthermore, the map from $M^{f}(n)\rightarrow \mathcal{M}(M^{f})(n)$ is equivalent to the restriction map 
$$
r:M^{f}(n)\longrightarrow {\mathcal M}_{0}(n),
$$ 
induced by the inclusion $H_{0}(T)\rightarrow H(T)$ for any tree $T\in \mathbb{T}[n]$.\vspace{7pt}

\noindent \textbf{Construction of the tower of fibrations:} % For both spaces $M^{f}(n)$ and ${\mathcal M}_{0}(n)$, 
 We construct a tower of fibrations according to the number of vertices of the trees in $\mathbb{T}[n]$:
$$
\xymatrix{
{\mathcal M}_{0}(n)=A_0&
A_{1}\ar[l]& \ar[l] \cdots & \ar[l]  A_{k-1} & \ar[l] A_{k} & \ar[l] \cdots & \ar[l] M^{f}(n). 
}
$$

For this purpose, we introduce the set $\mathbb{T}[n,k]$ of  trees in $\mathbb{T}[n]$ having exactly $k$ vertices. In particular, $\mathbb{T}[n,1]$ has only one element which is the $n$-corolla $C_{n}$. The space $A_k$ is defined as the subspace 
\[
A_k\subset \left(\underset{\underset{i\leq k}{T\in \mathbb{T}[n,i]}}{\prod} Map_{Aut(T)}\big(\, H(T)\times D(T)\,;\,M(|T|)\,\big)
\right)\times 
 \left(\underset{\underset{i> k}{T\in \mathbb{T}[n,i]}}{\prod} Map_{Aut(T)}\big(\, H_{0}(T)\times D(T)\,;\,M(|T|)\,\big)
\right)
\]
composed of families of continuous maps $\{f_{T}\}_{T\in \mathbb{T}[n]}$ satisfying the conditions~\eqref{rel1},
\eqref{rel2}, \eqref{rel3}.

% Consequently, one has the identifications $A_{1}=Map( H(C_{n})\times D(C_{n}) \,;\, M(n)))=M(n)$ and $B_{1}=\ast$. 
The space $A_{k}$ is defined by induction from $A_{k-1}$  using the pullback diagram:%\todo{Prendre arbres planaires et Aut(T)-morphismes}
\begin{equation}\label{S1}
\xymatrix@C=15pt{
A_{k} \ar[r]\ar[d] & \underset{[T]\in \mathbb{T}[n,k]}{\displaystyle \prod} Map_{Aut(T)}\big( H(T)\times D(T) \,;\, M(|T|)\,\big) \ar[d]\\
A_{k-1}\ar[r] &  \underset{[T]\in \mathbb{T}[n,k]}{\displaystyle \prod} Map_{Aut(T)}\big( (H\times D)^{-}(T) \,;\, M(|T|)\,\big),
}
%\hspace{15pt}
%\xymatrix@C=15pt{
%B_{k} \ar[r]\ar[d] & \underset{[T]\in \mathbb{T}[n,k]}{\displaystyle \prod} Map_{Aut(T)}\big( H_{0}(T)\times D(T) \,;\, M(|T|)\,\big) \ar[d]\\
%B_{k-1}\ar[r] &  \underset{[T]\in \mathbb{T}[n,k]}{\displaystyle \prod} Map_{Aut(T)}\big( (H_{0}\times D)^{-}(T) \,;\, M(|T|)\,\big).
%}
\end{equation}
 In the diagram above, for any tree $T\in \mathbb{T}[n,k]$, $D^{-}(T)$ is the subspace of $D(T)$ formed by points having at least one bivalent vertex labelled by the unit of the operad $Q$. On the other hand, the space $H^{-}(T)\subset H(T)$ consists of points in $H_{0}(T)$ or having two consecutive vertices indexed by the same real number or having a vertex indexed by $1$. Then we set 
$$
(H\times D)^{-}(T):=\big( H^{-}(T)\times D(T)\big) \underset{H^{-}(T)\times D^{-}(T)}{\coprod}\big( H(T)\times D^{-}(T)\big).
$$ 
%$$
%(H_{0}\times D)^{-}(T):=\big( H_{0}^{-}(T)\times D(T)\big) \underset{H_{0}^{-}(T)\times D^{-}(T)}{\coprod}\big( H_{0}(T)\times D^{-}(T)\big).
%$$ 

%Since the operad $Q$ is well-pointed and $\Sigma$-cofibrant, the inclusion from $D^{-}(T)$ into $D(T)$ is a cofibration. Furthermore, the inclusions $H^{-}(T)\rightarrow H(T)$ and $H^{-}_{2}(T)\rightarrow H_{0}(T)$ are cofibrations as inclusions of $CW$-complexes. As a consequence of Lemma \ref{D2}, the maps 
%$$
%(H\times D)^{-}(T)\longrightarrow H(T)\times D(T)\hspace{15pt} \text{ and } \hspace{15pt} (H_{0}\times D)^{-}(T)\longrightarrow H_{0}(T)\times D(T)
%$$
%are cofibrations and the vertical maps in the diagrams \eqref{S1} are Serre fibrations. \vspace{7pt}

To show that the map $A_k\to A_{k-1}$ is a Serre fibration, one has to prove that the right
vertical map in~\eqref{S1} is one. Let $t$ be the number of arity one (non-root) vertices of~$T$. Denote by
\[
D_1(T) : = Q(1)^{\times t}\,\,,\quad\quad D_{\neq 1}(T) : =   \underset{\underset{|v|\neq 1}{v\in V(T)\setminus \{r\}}}{\prod} Q(|v|).
\]
One has, $D(T)=D_1(T)\times D_{\neq 1}(T)$. We similarly denote by $D_1^{-}(T)$ the subset  consisting of points
having at least one coordinate equal to the unit $\ast_1\in Q(1)$ and define 
$$
(H\times D_1)^{-}(T):=\big( H^{-}(T)\times D_1(T)\big) \underset{H^{-}(T)\times D_1^{-}(T)}{\coprod}\big( H(T)\times D_1^{-}(T)\big). 
$$ 
We have
\begin{equation}\label{eq:HD}
(H\times D)^{-}(T)=(H\times D_1)^{-}(T)\times D_{\neq 1}(T) \quad \text{and} \quad
(H\times D)(T)=(H\times D_1)(T)\times D_{\neq 1}(T).
\end{equation}
Recall Definition~\ref{d:ap_cell} of  a {\it cellularly} equivariant cofibration.  Let $\underline{n}$ denote the ordered set $\{0,1,\ldots,n+1\}$. The dual
 to the simplicial 
indexing category $\Delta$ can be defined as the category having $\{\underline{n}\,,\,\, n\in\NN\}$ as the set of objects with morphisms
order preserving maps that also preserve both extrema. The space $H(T)$ can be described as the realization 
of the simplicial set $H_T(\bullet)$ with $H_T(\underline{n})$ being the set of order preserving maps 
\begin{equation}\label{eq:T_ss}
V(T)\setminus\{r\}\to
\underline{n}.
\end{equation}
 Note that  $H^{-}(T)$ is the realization of a simplicial subset $H^{-}_T(\bullet)\subset H_T(\bullet)$.
It consists of order-preserving maps~\eqref{eq:T_ss} which are either sending an arity zero vertex of $T$ to~0 or non-injective or sending
some vertex of~$T$ to the maximal element. %For $H^{-}(T)$ one should take maps~\eqref{eq:T_ss}
%. For $H_2^-(T)$ one takes the intersection of the latter two. 
 By Lemma~\ref{l:app_ss}, the inclusion $H^{-}(T)\to H(T)$ is a  cellularly
$Aut(T)$-equivariant cofibration. Using the fact that the operad~$Q$ is well-pointed and also applying Lemmas~\ref{l:app_ind_restr}, \ref{l:app_push_prod}, \ref{l:app_sigma}, we get that the inclusion
$(H\times D_1)^{-}(T)\to H(T)\times D_1(T)$ %and $(H_2\times D_1)^{-}(T)\to H_2(T)\times D_1(T)$ are
is a cellularly $Aut(T)$-equivariant cofibration. 

Using the fact that the category $Top$ is cartesian closed, for any $X\in Aut(T)\text{-}Top$, one has a homeomorphism
\begin{equation}\label{eq:Aut_adj}
Map_{Aut(T)}\bigl(X\times D_{\neq 1}(T)\,;\, M(|T|)\bigr) = Map_{Aut(T)}\bigl(X\,;\, Map\bigl(D_{\neq 1}(T)\,,\,M(|T|)
\bigr)\bigr),
\end{equation}
where the $Aut(T)$-action on $ Map\bigl(D_{\neq 1}(T)\,,\,M(|T|)
\bigr)$ is defined by $(f\cdot g)(x)=f(x\cdot g^{-1})\cdot g$, $g\in Aut(T)$. Applying this homeomorphism  to~\eqref{eq:HD} and then 
Lemma~\ref{l:app_restr} to  the right vertical map in~\eqref{S1}, we get that the latter is a  Serre fibration.

\end{proof}

\begin{rmk}\label{r:trunc_bim_Reedy}
The same strategy can be used in order to get a fibrant replacement functor for the Reedy model category of $r$-truncated $(P\text{-}Q)$-bimodules.
The fibrant replacement should be defined as a subspace of the product~\eqref{eq:Mf} with an additional restriction
$|T|\leq r$. The  constraints are the same.
\end{rmk}

\subsubsection{Characterization of Reedy cofibrations for bimodules}\label{Fin1}

It has been shown by the second author \cite[Theorem 8.3.20]{Fre2} that a map of $\Lambda$-sequences is a Reedy cofibration if and only if it
is a projective $\Sigma$-cofibration.
In the context of operads,  he also proves \cite[Theorem 8.4.12]{Fre2} that a map $\phi:P\rightarrow Q$ between reduced operads is cofibrant in the Reedy model category $\Lambda_{\ast}\Operad$ if and only if the corresponding map $\phi_{>0}:P_{>0}\rightarrow Q_{>0}$  is a cofibration in the projective model structure of (not necessarily reduced) operads. In what follows, we prove an analogous version in the context of operadic bimodules. 

\begin{thm}\label{C1}
Let $P$ and $Q$ be as in Theorem~\ref{Z5}. A morphism $\phi:M\rightarrow N$ in the category of (possibly truncated) $(P\text{-}Q)$-bimodules is a Reedy cofibration if and only if the corresponding map $\phi$ is a cofibration in the projective model category of (possibly truncated) $(P\text{-}Q_{>0})$-bimodules.
\end{thm}

\begin{proof}
First, we show that if the induced map is a cofibration in the projective model category of $(P\text{-}Q_{>0})$-bimodules then the map $\phi:M\rightarrow N$ is a Reedy cofibration in the category of $(P\text{-}Q)$-bimodules. For this purpose, we consider the following lifting problem in the category of $(P\text{-}Q)$-bimodules:
\begin{equation}\label{C5}
\xymatrix{
M\ar[r]^{i}\ar[d]_{\phi} & A\ar@{->>}[d]^{p}_{\simeq}\\
N\ar[r]_{j} \ar@{-->}[ru]^{\exists \varphi ?} & B
}
\end{equation}
where $p:A\rightarrow B$ is an acyclic Reedy fibration. The strategy is to build the map $\varphi$ by induction using an adjunction between $(P\text{-}Q)$-bimodules 
$$
ar_{s}:\Lambda \Bimod_{P\,;\,Q}\rightleftarrows \Lambda \Bimod_{P\,;\,Q}: cosk_{s}.
$$
The functor $ar_{s}$, called the arity filtration functor, sends a $(P\text{-}Q)$-bimodule $M$ to the bimodule $ar_{s}(M)$ defined as a quotient of the free $(P\text{-}Q_{>0})$-bimodule generated by the first $s$ components of $M$ where the equivalence relation is determined by the restriction of the bimodule structure on the first $s$ components of $M$. In other words, if $L_{s}$ denotes the left adjoint to the truncation functor $\TT_{s}$, then the arity filtration functor is given by 
$
ar_{s}=L_{s}\circ \TT_{s}.
$
By construction, $ar_{s}(M)$ is a $(P\text{-}Q)$-bimodule and  one has the identities 
$$
ar_{s-1}\circ ar_{s}=ar_{s-1}, \hspace{20pt}
 \mathrm{colim}_{s}ar_{s}(M)=M \hspace{15pt}\text{and}\hspace{15pt}ar_s(M)(k)=M(k),\text{ for } k\leq s.
$$
Let us notice that, even though this functor is called filtration, the natural map $ar_s(M)\rightarrow M$ might not be an inclusion.
  First, we give an explicit description of the right adjoint of the arity filtration functor. \vspace{7pt}

\noindent \textbf{The $\Lambda$-coskeletons associated to a bimodule.} According to the notation introduced by the second author \cite{Fre2}, the $\Sigma$-sequence $cosk_{s}(M)=\{cosk_{s}(M)(n),n\geq 0\}$, called the $s$-th coskeleton associated to the bimodule $M$, is given by the formula
\begin{equation}\label{DefCoskeleton}
cosk_{s}(M)(n)=\underset{\substack{h\in \Lambda_{+}([i]\,;\,[n])\\ 0\leq i\leq  s}}{\mathrm{lim}} M(i), \hspace{30pt}\text{with } n\geq 0.
\end{equation}
As a matter of fact, the right adjoint $R_s$, to the $s$-truncation functor $T_s$,  is defined by the same formula. Similarly to the arity filtration functor, the functor $cosk_{s}$ is also given by 
$$
cosk_{s}=R_{s}\circ \TT_{s}.
$$

By construction, a point $x\in cosk_{s}(M)(n)$ is a family of elements $x=\{x_{h}\in M(i), h\in \Lambda_{+}([i]\,;\,[n])\text{ and } i\leq s\}$ satisfying the relation of the limit: for any $f\in \Lambda_{+}([i]\,;\,[j])$ and  $g\in \Lambda_{+}([j]\,;\,[n])$, one has $f^{\ast}(x_{g})=x_{g\circ f}$. In particular, $cosk_{s}(M)(0)=M(0)$ is obviously endowed with a map from $P(0)$. Furthermore, the $\Sigma$-sequence $cosk_{s}(M)$ inherits a $\Lambda$-structure. For any $f\in \Lambda_{+}([n_{1}]\,;\,[n_{2}])$, one has 
\begin{equation}\label{CoskeletonLambda}
\begin{array}{rcl}\vspace{7pt}
f^{\ast}:cosk_{s}(M)(n_{2}) & \longrightarrow & cosk_{s}(M)(n_{1});\\ 
\{x_{u}\}_{h\in \Lambda_{+}([i]\,;\,[n_{2}])}^{0\leq i\leq s} & \longmapsto & \{x_{f\circ u}\}_{u\in \Lambda_{+}([i]\,;\,[n_{1}])}^{0\leq i\leq s}.
\end{array} 
\end{equation}

The $\Lambda$-sequence $cosk_{s}(M)$ is also a $(P\text{-}Q)$-bimodule. In order to define the right operations, we introduce some notation. Let $n,m>0$, $l\in \{1,\ldots,n\}$ and $h\in \Lambda_{+}([i]\,;\,[n+m-1])$. If we denote by $\ell_{1}\in \Lambda_{+}([m]\,;\,[n+m-1])$ and $\ell_{2}\in \Lambda_{+}([n]\,;\,[n+m-1])$ the order preserving maps
$$
\begin{array}{rclcrcl}\vspace{3pt}
\ell_{1}:[m] & \longrightarrow & [n+m-1]; & \hspace{10pt}\text{and}\hspace{10pt} & \ell_{2}:[n] & \longrightarrow & [n+m-1]; \\ 
\alpha & \longmapsto & \alpha +\ell, &  & \alpha & \longmapsto & \left\{
 \begin{array}{ll}
 \alpha & \text{if } \alpha\leq \ell, \\ 
 \alpha+m & \text{if } \alpha > \ell,
 \end{array} 
 \right.
\end{array} 
$$

\noindent then there exist unique morphisms $h_{1}$ and $h_{2}$ making the following diagrams commute:
$$
\xymatrix{
[i]\ar[r]^{h} & [n+m-1] \\
[|\mathrm{Im}(\ell_{1})\cap Im(h)|]\ar[u] \ar[r]_{\hspace{30pt}h_{1}} & [m]\ar[u]_{\ell_{1}} 
}\hspace{30pt}
\xymatrix{
[i]\ar[r]^{h} & [n+m-1] \\
[|\mathrm{Im}(\ell_{2}\setminus \{\ell\})\cap Im(h)|]\ar[u] \ar[r]_{\hspace{40pt}h_{2}} & [n]\ar[u]_{\ell_{2}} 
}
$$
Finally, if we denote by $\overline{\ell}=\ell-|\{\alpha\in [i]\,|\, h(\alpha)<\ell\}|$, then the right operations are defined as follows:
\begin{equation}\label{CoskeletonStruct}
\begin{array}{rcl}\vspace{5pt}
\circ^{i}:cosk_{s}(M)(n)\times Q(m) & \longrightarrow & cosk_{s}(M)(n+m-1);  \\ 
\{x_{u}\}_{u\in \Lambda_{+}([i]\,;\,[n])}^{0\leq i\leq s}\,\,;\,\,q & \longmapsto & \{x_{h_{2}}\circ^{\overline{\ell}}h_{1}^{\ast}(q)\}_{h\in \Lambda_{+}([i]\,;\,[n+m-1])}^{0\leq i\leq s}
\end{array} 
\end{equation}

Let $k_{1},\ldots,k_{\ell}>0$ and $h\in \Lambda_{+}([i]\,;\,[k_{1}+\cdots+k_{\ell}])$. In order to define the left operation, we introduce the morphism $\ell_{i}\in \Lambda_{+}([k_{i}]\,;\,[k_{1}+\cdots+k_{\ell}])$ sending $\alpha$ to $\alpha +k_{1}+\cdots + k_{i-1}$. Then there exists a unique morphism $h_{i}$ such that the following diagram commutes:
$$
\xymatrix{
[i] \ar[r]^{h} & [k_{1}+\cdots +k_{\ell}]\\
[|Im(h)\cap \mathrm{Im}(\ell_{i})|] \ar[r]_{h_{i}}\ar[u] & [k_{i}]\ar[u]_{\ell_{i}} 
}
$$
Finally, the left operation is given by the formula
$$
\begin{array}{rcl}%\vspace{7pt}
\gamma:P(\ell)\times cosk_{s}(M)(k_{1})\times \cdots \times cosk_{s}(M)(k_{\ell}) & \longrightarrow  & cosk_{s}(M)(k_{1}+\cdots + k_{\ell}); \\ 
p\,;\,\{x_{u_{1}}^{1}\}\,,\,\ldots, \{x_{u_{\ell}}^{\ell}\} & \longmapsto & \{p(x_{h_{1}}^{1},\ldots, x_{h_{\ell}}^{\ell})\}_{h\in \Lambda_{+}([i]\,;\,[k_{1}+\cdots + k_{\ell}])}^{0\leq i\leq s}.
\end{array}% \vspace{7pt}
$$

\noindent \textbf{The arity filtration and the coskeleton functors form an adjunction.} For any pair of reduced $(P\text{-}Q)$-bimodules $M$ and $N$, we 
define below a homeomorphism between the  mapping spaces of $(P\text{-}Q)$-bimodules:
\begin{equation}\label{C6}
F:\Lambda\Bimod_{P\,;\,Q}(ar_{s}(M)\,;\,N)\rightleftarrows \Lambda\Bimod_{P\,;\,Q}(M\,;\,cosk_{s}(N)):G.
\end{equation}
Let $f:M\rightarrow cosk_{s}(N)$ be a bimodule map. The bimodule map $G(f)$ is defined by induction. If $n\leq s$ and $x\in ar_{s}(M)(n)=M(n)$, then $G(f)(x)=f(x)_{[n]\rightarrow [n]}$, the point indexed by the identity order preserving map. From now on, we assume that $G(f)$ is defined for any element $ar_{s}(M)$ until the arity $n\geq s$. Let $x\in ar_{s}(M)(n+1)$. Then one has
$$
G(f)(x)= \left\{
\begin{array}{ll}\vspace{5pt}
G(f)(x')\circ^{i}y & \text{if } x=x'\circ^{i}y, \\ 
y(G(f)(x_{1}),\ldots,G(f)(x_{\ell})) & \text{if } x=y(x_{1},\ldots,x_{\ell}).
\end{array} 
\right.
$$
Conversely, let $g:ar_{s}(M)\rightarrow N$ be a bimodule map. Then one has
$$
\begin{array}{rcl}\vspace{5pt}
F(g)_{n}:M(n) & \longrightarrow & cosk_{s}(N)(n);  \\ 
x  & \longmapsto & \{x_{h}=g\circ h^{\ast}(x)\}_{h\in \Lambda_{+}([i]\,;\,[n])}^{0\leq i\leq s} .
\end{array} 
$$

\noindent \textbf{The arity functor $ar_{s}$ preserves $\Sigma\Bimod_{P,Q_{>0}}$-cofibrations.} The functor $ar_s$ can be defined on the category $\Sigma\Bimod_{P,Q_{>0}}$ and it fits into a Quillen adjunction
\begin{equation}\label{eq:tr>0_adj}
ar_s\colon\Sigma\Bimod_{P,Q_{>0}}\rightleftarrows\Sigma\Bimod_{P,Q_{>0}}\colon tr_s,
\end{equation}
where $tr_s$ is the truncation functor
$
tr_s(L)(n)=
\begin{cases}
L(n),&n\leq s;\\
*,& n>s.
\end{cases}
$
 
One has a square of functors
\[
\xymatrix{
\Lambda\Bimod_{P\,;\,Q} \ar[r]^{ar_s} \ar[d]_{\mathcal{U}} &  \Lambda\Bimod_{P\,;\,Q} \ar[d]^{\mathcal{U}} \\
\Sigma\Bimod_{P,Q_{>0}} \ar[r]^{ar_s} & \Sigma\Bimod_{P,Q_{>0}}.
}
\]
The vertical arrows denote the forgetful functor.
This square commutes by the same argument as in the case of operads, see \cite[Theorem 8.4.12]{Fre2}. One checks it 
first for free bimodules: $ar_s\mathcal{F}_{P\,;\,Q}^\Lambda(L)=\mathcal{F}_{P\,;\,Q}^\Lambda(L_{\leq s})$ and 
$ar_s\mathcal{F}_{P\,;\,Q_{> 0}}^\Sigma(L)=\mathcal{F}_{P\,;\,Q_{>0}}^\Sigma(L_{\leq s})$.  
Here, 
$L_{\leq s}$
denotes the sequence
$$ 
\begin{cases}
L(n),&n\leq s;\\
\emptyset,& n>s.
\end{cases}
$$
It follows from~\eqref{eq:free} that  the square commutes for free
bimodules. The forgetful functor $\mathcal{U}$ preserves colimits as it has a right adjoint by Proposition~\ref{p:tens_hom}b for $Q_1=Q$, $Q_2=Q_{>0}$, $Y=Q$.  So do the functors $ar_s$.
On the other hand, any bimodule can naturally be seen as a coequalizer of free ones. The adjunction~\eqref{eq:tr>0_adj} is  a Quillen adjunction as $tr_s$ preserves fibrations and equivalences. As a consequence, $ar_s\colon \Sigma\Bimod_{P,Q_{>0}}\to \Sigma\Bimod_{P,Q_{>0}}$ preserves cofibrations.\vspace{7pt}

\noindent \textbf{Construction of the lift by induction.} We work out our lifting problem (\ref{C5}) by an inductive construction on arity. By definition, one has $ar_{0}(N)=N_{0}$ (that is $N(0)$ in arity $0$ and the empty set otherwise) and $ar_{0}(M)=M_{0}$. Consequently, the bimodule map $\varphi_{0}:ar_{0}(N)\rightarrow A$ is obtained as a lift of the following diagram of $P$-algebras: 
\[
\xymatrix{
M(0) \ar[r]^{i_{0}} \ar[d]_{\phi_{0}} & A(0) \ar[d]^{p_{0}} \\
N(0) \ar[r]_{j_{0}} & B(0)
}
\]
Such a lift exists. Indeed, the left vertical map is a cofibration of $P$-algebras since the arity functor preserves cofibrations. Moreover, since  $p$ is an acyclic Reedy fibration, the map $A(0)\rightarrow B(0)\times_{\mathcal{M}(B)(0)}\mathcal{M}(A)(0)=B(0)$ is an acyclic Serre fibration.  

Then we assume that the bimodule map $\varphi_{s-1}:ar_{s-1}(N)\rightarrow A$ is well defined. We consider the following diagram in the category of reduced $(P\text{-}Q)$-bimodules: %\todo{explain}
\begin{equation}\label{C7}
\xymatrix@C=60pt{
ar_{s}(M)\underset{ar_{s-1}(M)}{\displaystyle\coprod}  ar_{s-1}(N) \ar[r]^{\hspace{40pt}(i\circ \iota\,;\, \varphi_{s-1})}\ar[d]_-{(ar_{s}(\phi)\,;\, \iota)} & A \ar@{->>}[d]^{p}_\simeq \\
ar_{s}(N) \ar[r]_{j\circ \iota} \ar@{-->}[ru]_{\exists \varphi_{s}?}& B
}
\end{equation}
where the upper horizontal map is defined using the map $i\circ \iota:ar_{s}(M)\rightarrow M \rightarrow A$ on the first summand and the map $\varphi_{s-1}:ar_{s-1}(N)\rightarrow A$ on the second summand. By applying the identifications $ar_{s-1}\circ ar_{s}=ar_{s-1}$ and $ar_{s}\circ ar_{s}=ar_{s}$ together with the adjunction (\ref{C6}), we get that the lifting problem (\ref{C7}) becomes equivalent to a lifting problem of the following form in the category of reduced $(P\text{-}Q)$-bimodules: %\todo{I am confused by your argument. Why lift in (38) is equivalent to a lift in (39)?}
\begin{equation}\label{C8}
\xymatrix@C=70pt{
ar_{s}(M) \ar[r]^{F(i\circ \iota)} \ar[d]_{ar_{s}(\phi)} & cosk_{s}(A)\ar[d]^{\pi\times cosk_{s}(p)}_\simeq \\
ar_{s}(N) \ar[r]_{\hspace{-45pt}F(\varphi_{s-1})\times F(ar_{s}(j))} \ar@{-->}[ru]^{\exists \tilde{\varphi}_{s}} & cosk_{s-1}(A) \underset{cosk_{s-1}(B)}{\times} cosk_{s}(B).
}
\end{equation}
The left vertical map is a cofibration in $\Sigma\Bimod_{P,Q_{>0}}$ since the arity functor preserves cofibrations. On the other hand, the right vertical map is an acyclic Serre fibration in both the projective model category of $\Sigma$-sequences and the Reedy model categories of $\Lambda$-sequences, see the proof of  \cite[Theorem 8.3.20]{Fre2}.  In that proof  the second author has a similar diagram (loc. cit. equation~(6)), but in the category of $\Lambda$-sequence.   The coskeleton functors are defined on the category $\Lambda Seq$ by the same formula~\eqref{DefCoskeleton}. Thus, the second author shows that for any Reedy fibration $A\to B$ of $\Lambda$-sequences, the induced map
$$
cosk_s(A)\to cosk_{s-1}(A) \underset{cosk_{s-1}(B)}{\times} cosk_{s}(B)
$$
is an acyclic both  Reedy and (therefore) Serre fibration. As
 a consequence the lift $\tilde{\varphi_{s}}$ of Diagram (\ref{C8}) exists in the category  of $(P\text{-}Q_{>0})$-bimodules. \vspace{7pt}
 %Consequently, the map $\varphi_{s}=G(\tilde{\varphi_{s}})$ provides a solution to the lifting problem of Diagram (\ref{C7}).

\noindent \textbf{The morphism $\varphi$ is compatible with the $\Lambda$-structure.} By definition, let us remark that one has the identifications $cosk_{s-1}(N)(s)=\mathcal{M}(N)(s)$ and $cosk_{s-1}(A)(s)=\mathcal{M}(A)(s)$. By construction of the map $F$ from~(\ref{C6}) applied to  $\varphi_{s-1}:ar_{s-1}\circ ar_{s}(N)(s)\rightarrow A(s)$, on has the factorization
$$
\xymatrix{
F(\varphi_{s-1}):ar_{s}(N)(s)=N(s) \ar[r] & \mathcal{M}(N)(s) \ar[r]^{\mathcal{M}(p)} & \mathcal{M}(A)(s).
}
$$ 
Consequently, the commutativity of diagram (\ref{C8}) implies the commutativity of the square:
$$
\xymatrix{
ar_{s}(N)(s)=N(s) \ar[r] \ar[d] & cosk_{s}(A)(s)=A(s) \ar[d]^{\pi} \\
\mathcal{M}(N)(s) \ar[r] & cosk_{s-1}(A)(s)=\mathcal{M}(A)(s).
}
$$
It follows that our morphism $\tilde{\varphi}_{s}$ intertwines the action of the restriction operators $h:[n_{1}]\rightarrow [n_{2}]$.
This proves that  $\tilde{\varphi}_s$ and, therefore, $\varphi_s$ are morphisms of reduced bimodules preserving the $\Lambda$-structure. Thus so is $\varphi=\mathrm{lim}_{s}\varphi_{s}$.\vspace{5pt}

\noindent \textbf{Conversely.} %\todo{voir commentaire Victor} 
The forgetful functor $\mathcal{U}\colon\Lambda\Bimod_{P\,;\,Q}\to\Sigma\Bimod_{P,Q_{>0}}$ preserves colimits. Therefore,
it is enough to check the statement for  generating  cofibrations in $\Lambda\Bimod_{P\,;\,Q}$. The latter are given by applying the free
functor $\mathcal{F}_{P\,;\,Q}^\Lambda$ to the generating  cofibrations of $\Lambda Seq_P$. On the other hand,
$\Lambda$-cofibrations are always $\Sigma$-cofibrations and the free functor $\mathcal{F}_{P\,;\,Q}^\Lambda$ 
agrees with $\mathcal{F}_{P\,;\,Q_{>0}}^\Sigma$, see~\eqref{eq:free}.
%
%  It has been established that each map in the set of generating (acyclic) cofibrations of $\Lambda$-sequences is an (acyclic) cofibration of $\Sigma$-sequences with the empty set in arity $0$. By adjunction, this result implies that the morphism of free bimodules induced by a generating (acyclic) cofibration of $\Lambda$-sequences has the left lifting property with respect to all acyclic fibration (just fibration) in the category of bimodules with the empty set in arity $0$. Consequently, the relative cell complexes of generating (acyclic) cofibrations of reduced bimodules define (acyclic) cofibrations in the category of bimodules with the empty set in arity $0$ and equipped with the projective model category structure. 
\end{proof}

\subsubsection{Left properness and extension/restriction functors for the Reedy model structure}\label{CH0}

As seen in Section \ref{H2}, under some conditions on the operads, the projective model category of bimodules is relatively left proper % relative to the class of componentwise cofibrant bimodules 
  and the extension/restriction adjunctions are Quillen equivalences. In what follows, we show that the Reedy model category inherits the same properties as a consequence of the characterization  of  Reedy cofibrations in the previous section. %We also show in Section \ref{reduced} that if we restrict our Reedy model category $\Lambda\Bimod_{P\,;\,Q}$ to subcategories, then we can slightly improve the assumptions on the operads $P$ and $Q$. 

\begin{thm} \label{ThmProperness} 
Let $P$ be an operad and $Q$ be a reduced and well-pointed operad. The Reedy model category  $\Lambda\Bimod_{P\,;\,Q}$
is right proper. If $P(0)\in Top$ is a cofibrant space, $P_{>0}\in\Sigma\Operad$ is a cofibrant operad and $Q$ is componentwise cofibrant, then 
 $\Lambda\Bimod_{P\,;\,Q}$ is left proper relative to the class $\mathcal{S}$ of componentwise cofibrant bimodules for which the  arity zero left action map $\gamma_0$ is a cofibration.
 In the latter case, cofibrations with domain in $\mathcal{S}$ are componentwise cofibrations, implying that the class of objects $\mathcal{S}$ is closed under cofibrations. If in addition $Q$ is $\Sigma$-cofibrant, then cofibrations with domain in the subclass $\mathcal{S}^\Sigma\subset\mathcal{S}$
 of $\Sigma$-cofibrant objects are $\Sigma$-cofibrations, and $\mathcal{S}^\Sigma$ is also closed under cofibrations.
\end{thm}

\begin{proof}
The first statement is proved in exactly the same way as Theorem~\ref{th:right_proper} using the facts that a Reedy fibration is always a projective (componentwise) fibration and that  pullbacks (and more generally limits) in the category of  bimodules are defined componentwise and the fact that the category of spaces is right proper.  
For the other statements, we first recall that the forgetful functor $\mathcal{U}\colon \Lambda\Bimod_{P\,;\,Q}\to \Sigma\Bimod_{P\,;\,Q_{>0}}$ preserves colimits\footnote{In fact
the forgetful functor from the category of $(P$-$Q)$-bimodules to left $P$-modules creates colimits. It preserves colimits as it has a right adjoint by Proposition~\ref{p:tens_hom}b for 
$Q_1=Q$, $Q_2= \mathbb{1}$, $Y=Q$. The fact that a cocone in bimodules is a colimit  if it is a colimit in the category of left modules is easily verified.}. In particular, pushouts are sent to pushouts.
By Theorem~\ref{C1}, it also creates cofibrations. The statements then follow from the analogous statements Theorems~\ref{C4} and~\ref{th:sigma_bimod_cof} for $\Sigma\Bimod_{P\,;\,Q_{>0}}$.
%
%For the second statement, we consider the following pushout diagram in the category $\Lambda \Bimod_{P\,;\,Q}$:
%\begin{equation}\label{H3}
%\xymatrix{
%A\ar[r]^{f}_{\simeq} \ar@{^{(}->}[d]_{g} & B \ar[d]^{j} \\
%C \ar[r]_{i} & D
%}
%\end{equation}
%where the map $f$ is a weak equivalence between componentwise cofibrant bimodules whereas the map $g$ is a Reedy cofibration. By construction of the pushout in the Reedy model category of bimodules, the above diagram \eqref{H3} is also a pushout diagram in the projective model category of $(P\text{-}Q_{>0})$-bimodules. Furthermore, the map $f$ induces a weak equivalence $f:A\rightarrow B$ between componentwise cofibrant $(P\text{-}Q_{>0})$-bimodules and, as a consequence of Theorem \ref{C1}, the map $g$ induces a cofibration $g:A\rightarrow C$ in the projective model category of $(P\text{-}Q_{>0})$-bimodules. Since the operad $P$ is cofibrant with $P(0)=\emptyset$, the projective model category of $(P\text{-}Q_{>0})$-bimodules is left proper with respect to componentwise cofibrant objects and diagram \eqref{H3} admits a lift.
\end{proof}

Let $\phi_{1}:P\rightarrow P'$ be a weak equivalence of operads and $\phi_{2}:Q\rightarrow Q'$ be a weak equivalence of reduced  operads. Similarly to Section \ref{E8}, we show that,
under some conditions, the Reedy model categories of $(P\text{-}Q)$-bimodules and $(P'\text{-}Q')$-bimodules are Quillen equivalent. By abuse of notation, we denote by $\phi^{\ast}$ and $\phi_{!}$ the restriction functor and the extension functor, respectively, between the Reedy model categories:
\begin{equation}
\phi_{!}:\Lambda \Bimod_{P\,;\,Q}\rightleftarrows \Lambda \Bimod_{P'\,;\,Q'}:\phi^{\ast}.
\end{equation}
In the same way as in Section \ref{E8}, for any $M\in \Lambda \Bimod_{P\,;\,Q}$ and $M'\in \Lambda \Bimod_{P'\,;\,Q'}$, one has
$$
\begin{array}{rcl}\vspace{7pt}
\phi_{!}(M) & = & \{\phi_{!}(M)(n)=\mathcal{F}_{P';Q'}^{\Lambda}(\mathcal{U}^\Lambda(M))(n)/\sim,\,\,\,n\geq 0 \}, \\ 
\phi^{\ast}(M') & = & \{\phi^{\ast}(M')(n)=M'(n),\,\,n\geq 0\}. 
\end{array} 
$$ 

\begin{thm}\label{ThmExt/rest}
Let $\phi_{1}\colon P\rightarrow P'$ be a  weak equivalences of $\Sigma$-cofibrant operads and $\phi_{2}\colon Q\rightarrow Q'$  be a weak equivalence between reduced and componentwise cofibrant operads. Then  one has Quillen equivalences
\begin{equation}\label{eq:bimod_ind_restr3}
\phi_!\colon\Lambda\Bimod_{P\,;\,Q}\rightleftarrows\Lambda\Bimod_{P'\,;\,Q'}\colon\phi^*,
\end{equation}
\begin{equation}\label{eq:tr_bimod_ind_restr}
\phi_!\colon\TT_r\Lambda\Bimod_{P\,;\,Q}\rightleftarrows\TT_r\Lambda\Bimod_{P'\,;\,Q'}\colon\phi^*.
\end{equation}
\end{thm}

\begin{proof}
Since the restriction functor creates weak equivalences, %(in the sense that a map $f$ in $\Lambda \Bimod_{P'\,;\,Q'}$ is a weak equivalence precisely if $\phi^{\ast}(f)$ is a weak equivalence in $\Lambda \Bimod_{P\,;\,Q}$),
  one has a Quillen equivalence if, for any Reedy cofibrant object $M$ in $\Lambda \Bimod_{P\,;\,Q}$, the adjunction unit 
\begin{equation}\label{H4}
M\longrightarrow \phi^{\ast}(\phi_{!}(M))
\end{equation}
is a weak equivalence. To prove it,  we consider the adjunction $\tilde{\phi}_!\colon\Sigma\Bimod_{P\,;\,Q_{>0}}\rightleftarrows\Sigma\Bimod_{P'\,;\,Q'_{>0}}\colon \tilde{\phi}^*$ induced by the weak equivalences of operads $\phi_{1}:P\rightarrow P'$ and $\phi_{2}^{>0}\colon Q_{>0}\rightarrow Q'_{>0}$. By construction, one has the identity 
$$
\phi^{\ast}(\phi_{!}(M))=\tilde{\phi}^{\ast}(\tilde{\phi}_{!}(M)).
$$
By Theorem \ref{G6}, the extension/restriction adjunction $(\tilde{\phi}_!,\tilde{\phi}^*)$ is a Quillen equivalence. Moreover, thanks to the characterization of Reedy cofibrations, $M$ is also cofibrant in the projective model category $\Sigma\Bimod_{P,Q_{>0}}$ and the map
$M\longrightarrow \tilde{\phi}^{\ast}(\tilde{\phi}_{!}(M))= \phi^{\ast}(\phi_{!}(M))$ is a weak equivalence.
%
%
%one has the identification
%$$
%\phi^{\ast}(\phi_{!}(M))_{>0}=(\phi_{>0})^{\ast}\big( (\phi_{>0})_{!}(M_{>0})\big).
%$$
%According to Theorem~\ref{th:>0}, since $\phi_{1}^{>0}:P_{>0}\rightarrow P'_{>0}$ and $\phi_{2}^{>0}:Q_{>0}\rightarrow Q'_{>0}$ are weak equivalences between  componentwise cofibrant operads, the pair of functors $((\phi_{>0})_{!}\,;\, (\phi_{>0})^{\ast})$ gives rise to a Quillen equivalence and the map $M(n)\rightarrow \phi^{\ast}(\phi_{!}(M))(n)$, with $n\geq 1$, is a weak equivalence. Moreover, the map $M(0)=\ast\rightarrow \phi^{\ast}(\phi_{!}(M))(0)=\ast$ is obviously a weak equivalence. %Finally, the map (\ref{H4}) is a weak equivalence too.
\end{proof}

\subsection{The connection between the model category structures on bimodules}\label{Fin3}

Similarly to the operadic case in \cite{FTW}, we build a Quillen adjunction between the projective model category of $(P\text{-}Q)$-bimodules and the Reedy model category of $(P\text{-}Q)$-bimodules where $P$ is an operad and $Q$ is a reduced operad. Furthermore, if $M$ and $N$ are two bimodules, then we show that there is a weak equivalence between the derived mapping spaces:
$$
\Sigma\Bimod_{P\,;\,Q}^{h}(M\,;\,N)\simeq \Lambda \Bimod_{P\,;\,Q}^{h}(M\,;\,N).
$$
For completeness of exposition, the last subsection is devoted to adapt the Boardman-Vogt resolution (well known for operads, see \cite{BM2}) to the context of  bimodules. We refer the reader to \cite{Duc2} for a detailed account of this construction.

\subsubsection{Quillen adjunction between the model category structures}\label{sss:adj_bim}

Let $P$ be an operad and $Q$ be a well-pointed reduced operad. The projective and the Reedy model categories of bimodules have the same class of weak equivalences and induce the same homotopy category. Consequently, one has the following statement about the adjunctions
\begin{equation}\label{eq:unitar_bim2b}
\begin{array}{rcl}\vspace{5pt}
id\colon\Sigma\Bimod_{P\,;\,Q} & \rightleftarrows & \Lambda\Bimod_{P\,;\,Q}\colon id, \\ 
id\colon\TT_r \Sigma \Bimod_{P\,;\,Q} & \rightleftarrows & \TT_r\Lambda\Bimod_{P\,;\,Q}\colon id.
\end{array} 
\end{equation}

\begin{thm}\label{th:S-L_equiv_bim}
For any operad $P$ and any well-pointed reduced operad~$Q$, the pairs of functors \eqref{eq:unitar_bim2b} form Quillen equivalences. Furthermore,  for any pair $M,\, N\in\Lambda\Bimod_{P\,;\,Q}$, one has
\begin{equation}\label{eq:der_bim_map2b}
\Sigma\Bimod_{P\,;\,Q}^h( M,N)\simeq \Lambda\Bimod_{P\,;\,Q}^h(M,N).
\end{equation}
Moreover, if $M,\, N\in\TT_r\Lambda\Bimod_{P\,;\,Q}$, with $r\geq 0$, then one has
\begin{equation}\label{eq:der_tr_bim_map2b}
\TT_r\Sigma\Bimod_{P\,;\,Q}^h( M,N)\simeq \TT_r\Lambda\Bimod_{P\,;\,Q}^h( M,N).
\end{equation}
\end{thm}

\begin{proof}
In order to prove that the pairs of functors \eqref{eq:unitar_bim2b} form Quillen adjunctions we check that the right adjoint functors preserve fibrations and acyclic fibrations. Let $f:M\rightarrow N$ be an (acyclic) fibration in the Reedy model category of bimodules. By definition, the map $f$ is a fibration if the corresponding map $\mathcal{U}^\Lambda(f)$ in the category of $\Lambda$-sequences is a fibration. In other words, it means that the maps 
\begin{equation}\label{B2b}
M(n)\longrightarrow \mathcal{M}(M)(n)\times_{\mathcal{M}(N)(n)}N(n),\hspace{15pt}\text{with } n\in \mathbb{N},
\end{equation}
are (acyclic) Serre fibrations. On the other hand, the map $id(f)$ is a fibration in the projective model category of bimodules if the maps $M(n)\rightarrow N(n)$, with $n\in \mathbb{N}$, are Serre fibrations. According to the notation introduced in Section \ref{Sect1.2}, the pair of functors 
$$
\Lambda[-]:\Sigma Seq \rightleftarrows \Lambda Seq :\mathcal{U}(-)
$$ 
forms a Quillen adjunction (see \cite[Theorem 8.3.20]{Fre2}). In particular, the forgetful functor $\mathcal{U}$ preserves (acyclic) fibrations. As a consequence,  if the maps (\ref{B2b}) are (acyclic) Serre fibrations, then the induced maps $M(n)\rightarrow N(n)$, with $n\in \mathbb{N}$, are (acyclic) Serre fibrations. Furthermore, the Quillen adjunction so obtained is obviously a Quillen equivalence since we consider identity functors. Finally, the identities \eqref{eq:der_bim_map2b} and \eqref{eq:der_tr_bim_map2b} are induced by taking a projective cofibrant replacement of $M$ and a Reedy fibrant replacement of~$N$
and using the fact that any projectively cofibrant object is Reedy cofibrant and any Reedy fibrant object is  projectively fibrant. 
\end{proof}

\subsubsection{Boardman-Vogt type resolution in the projective/Reedy model category}\label{3.2.1}

From a $(P\text{-}Q)$-bimodule $M$, we build a $(P\text{-}Q)$-bimodule $\mathcal{B}_{P\,;\,Q}(M)$. The points of $\mathcal{B}_{P\,;\,Q}(M)(n)$, $n\geq 0$, are equivalence classes $[T\,;\, \{t_{v}\}\,;\, \{p_{v}\}\,;\,\{m_{v}\}\,;\,\{q_{v}\}]$, where $T\in s\mathbb{P}_n$ (see Section \ref{C2}) is a tree with section while $\{p_{v}\}_{v\in V^{d}(T)}$, $\{m_{v}\}_{v\in V^{p}(T)}$ and $\{q_{v}\}_{v\in V^{u}(T)}$ are families of points labelling the vertices below the  section, on the section and above the section, respectively, by points in $P$, $M$ and $Q$, respectively. Furthermore, $\{t_{v}\}_{v\in V(T)\setminus V^{p}(T)}$ is a family of real numbers in the interval $[0\,,\,1]$ indexing the vertices which are not pearls. If $e$ is an inner edge above the section, then $t_{s(e)}\geq t_{t(e)}$. Similarly, if $e$ is an inner edge below the section, then $t_{s(e)}\leq t_{t(e)}$. In other words, closer to a pearl is a vertex, smaller is the corresponding number. The space $\mathcal{B}_{P\,;\,Q}(M)(n)$ is a quotient of the subspace of 
\begin{equation}\label{eq:union_stree}
\underset{T\in s\mathbb{P}_n}{\coprod}\,\,\,\underset{v\in V^{p}(T)}{\prod}\,M(|v|)\,\,\times\underset{v\in V^{d}(T)}{\prod}\,\big[\,P(|v|)\times [0\,,\,1]\big]\,\,\times\underset{v\in V^{u}(T)}{\prod}\,\big[\,Q(|v|)\times [0\,,\,1]\big]
\end{equation} 
determined by the restrictions on the families $\{t_{v}\}$. The equivalence relation is generated by the conditions:\vspace{10pt}

\begin{itemize}[topsep=0pt, leftmargin=*]
\item[$i)$] If a vertex is labelled by a unit in $P(1)$ or $Q(1)$, then locally one has the identity
%\begin{figure}[!h]
\begin{center}
\includegraphics[scale=0.35]{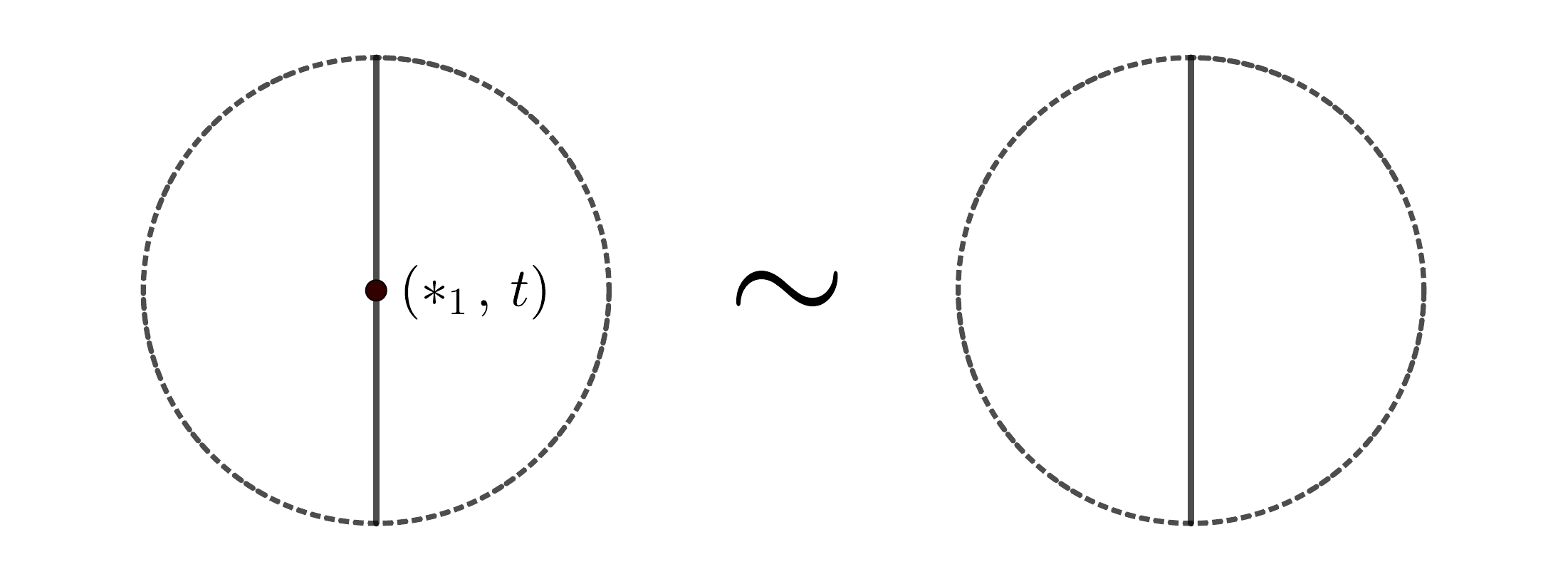}
\end{center}
%\end{figure}

\item[$ii)$] If a vertex is indexed by $a\cdot \sigma$, with $\sigma\in \Sigma$, then %\vspace{-10pt}
\begin{center}
\includegraphics[scale=0.35]{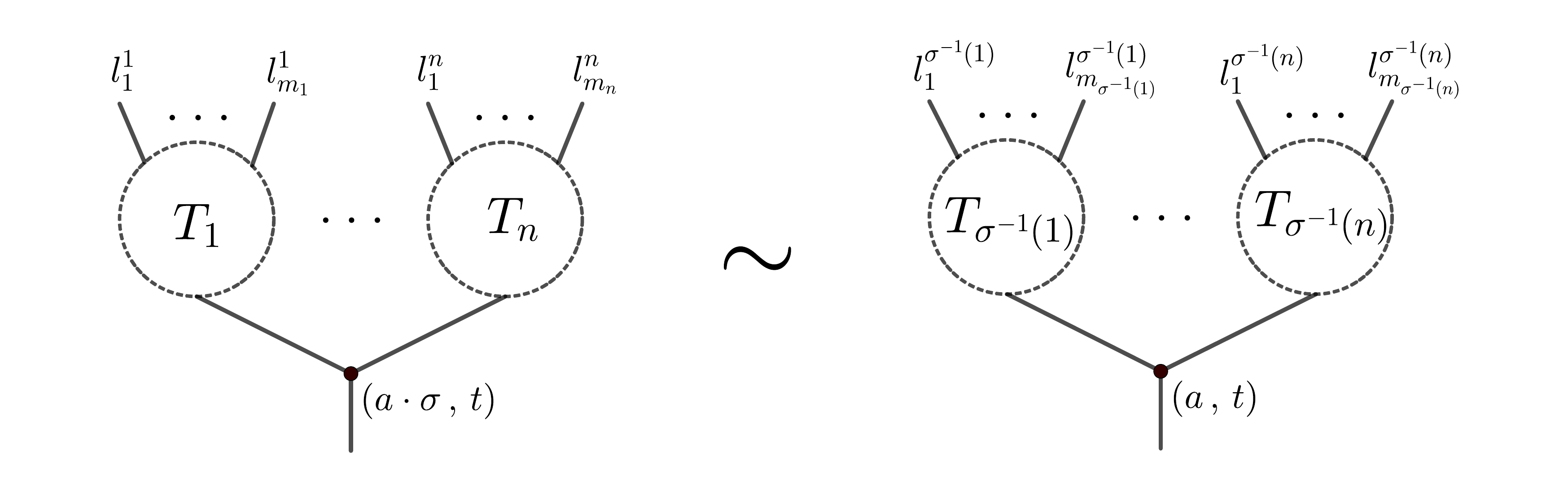}
\end{center}

\item[$iii)$] If two consecutive vertices, connected by an edge $e$, are indexed by the same real number $t\in [0\,,\,1]$, then~$e$ is contracted using the operadic structures of~$P$ and $Q$. The vertex so obtained is indexed by the real number~$t$.

\hspace{-55pt}\includegraphics[scale=0.33]{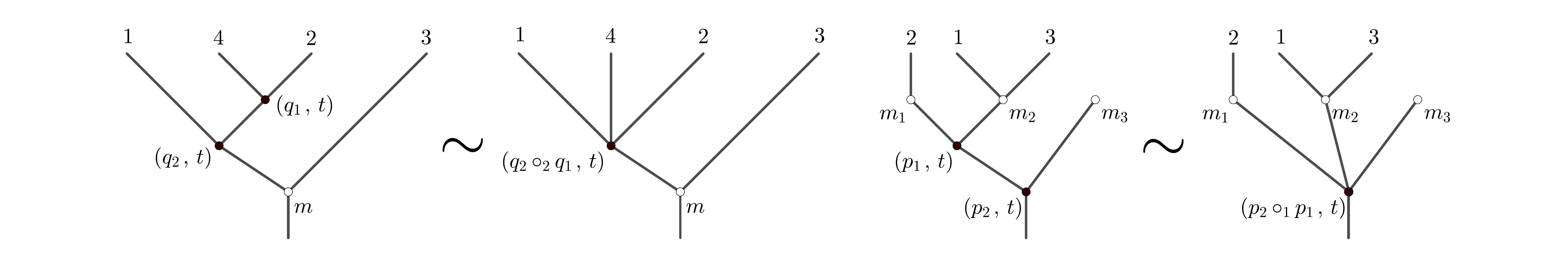}

\item[$iv)$]  If a vertex above the section is indexed by $0$, then its output edge is contracted by using the right module structure. Similarly, if a vertex below the section is indexed by $0$, then all its incoming edges are contracted by using the left module structure. In both cases the new vertex becomes a pearl.\vspace{5pt}

\hspace{-55pt}\includegraphics[scale=0.33]{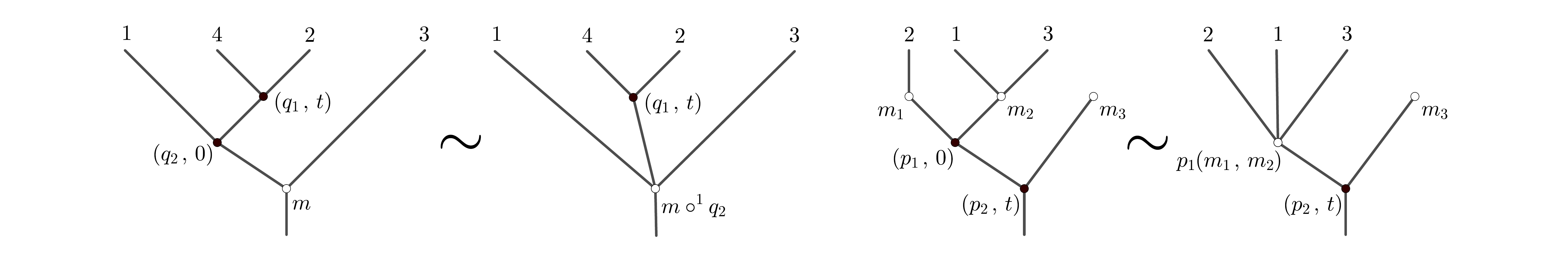}

\item[$v)$] If a univalent pearl is indexed by a point of the form $\gamma_{0}(x)$, with $x\in P(0)$, then we contract its output edge by using the operadic structure of $P$. In particular, if all the pearls connected to a vertex $v$ are univalent and of the form $\gamma_{0}(x_{v})$, then the vertex is identified to the pearled corolla with no input.%\vspace{-10pt}

\begin{center}
\includegraphics[scale=0.33]{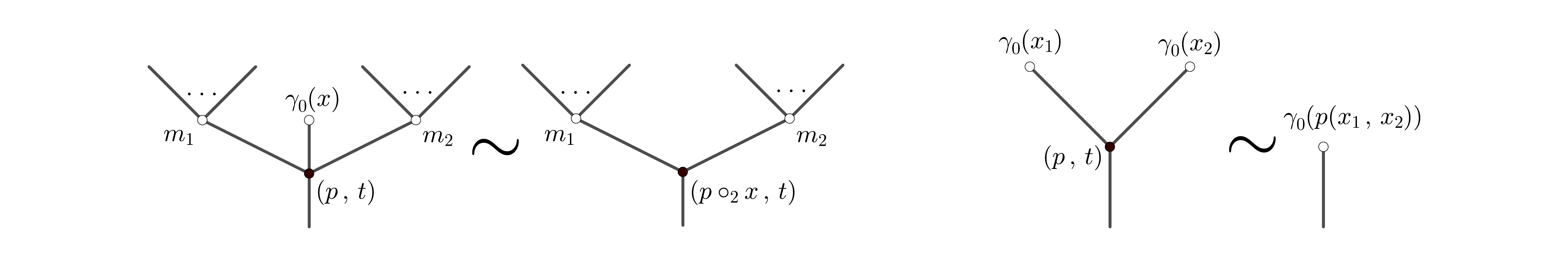}%\vspace{-5pt}
\end{center}

\end{itemize}

 Let us describe the $(P\text{-}Q)$-bimodule structure. Let $q\in Q(n)$ and $[T\,;\, \{t_{v}\}\,;\, \{p_{v}\}\,;\,\{m_{v}\}\,;\,\{q_{v}\}]$ be a point in $\mathcal{B}_{P\,;\,Q}(M)(m)$. The right operation $[T\,;\, \{t_{v}\}\,;\, \{p_{v}\}\,;\,\{m_{v}\}\,;\,\{q_{v}\}]\circ^{i}q$ consists in grafting the $n$-corolla labelled by $q$ to the $i$-th incoming edge of $T$ and indexing the new vertex by $1$. Similarly, let $[T_{i}\,;\, \{t_{v}^{i}\}\,;\, \{p_{v}^{i}\}\,;\,\{m_{v}^{i}\}\,;\,\{q_{v}^{i}\}]$ be a family of points in the spaces 
$\mathcal{B}_{P\,;\,Q}(M)(m_{i})$. The left module structure over $P$ is defined as follows: each tree of the family is grafted to a leaf of the $n$-corolla labelled by $p\in P(n)$ from left to right. The new vertex, coming from the $n$-corolla, is indexed by $1$.\vspace{5pt}

\hspace{-50pt}\includegraphics[scale=0.45]{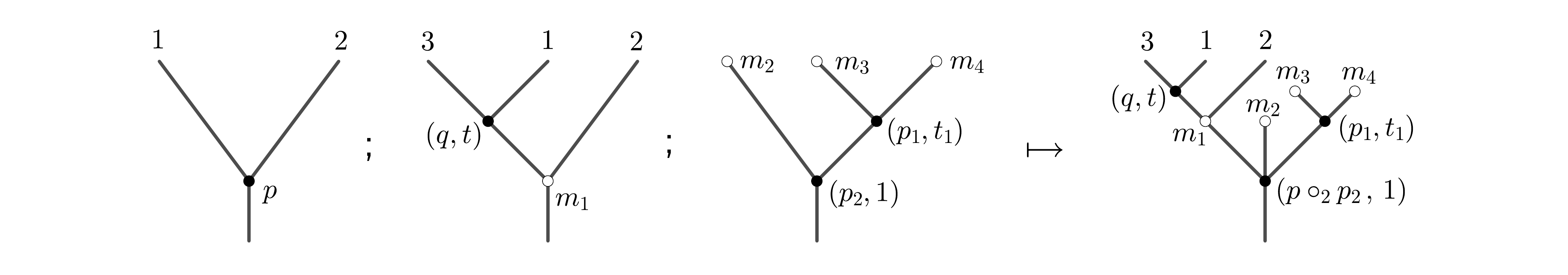}\vspace{-25pt}

\begin{figure}[!h]
\caption{Illustration of the left module structure $\gamma_{\ell}:P(2)\times \mathcal{B}_{P\,;\,Q}(M)(3) \times \mathcal{B}_{P\,;\,Q}(M)(0)\rightarrow \mathcal{B}_{P\,;\,Q}(M)(3)$.}
\end{figure}

One has an obvious inclusion of $\Sigma$-sequences $\iota \colon M\to \mathcal{B}_{P\,;\,Q}(M)$, where each element
$m\in M(n)$ is sent to an $n$-corolla labelled by $m$, whose only vertex is a pearl.  Furthermore, the map
\begin{equation}\label{D3}
\mu:\mathcal{B}_{P\,;\,Q}(M)\rightarrow M\,\,;\,\, [T\,;\, \{t_{v}\}\,;\, \{p_{v}\}\,;\,\{m_{v}\}\,;\,\{q_{v}\}]\mapsto [T\,;\, \{0\}\,;\, \{p_{v}\}\,;\,\{m_{v}\}\,;\,\{q_{v}\}],
\end{equation}
is defined by sending the real numbers indexing the vertices to $0$. The obtained element  is identified to the pearled corolla labelled by a point in $M$.  It is easy to see that $\mu$ is a $(P\text{-}Q)$-bimodule map.\vspace{10pt}

In order to get resolutions for truncated bimodules, one considers a filtration in $\mathcal{B}_{P\,;\,Q}(M)$ according to  the number of {\it geometrical inputs} which is the number of leaves plus the number of univalent vertices above the section. A point in $\mathcal{B}_{P\,;\,Q}(M)$ is said to be \textit{prime} if the real numbers indexing the vertices are strictly smaller than $1$. Otherwise, a point is said to be \textit{composite} and can be decomposed into \textit{prime components} as shown in Figure~\ref{B1}. More precisely, the prime components are obtained by removing the vertices indexed by $1$. 

\hspace{-40pt}\includegraphics[scale=0.38]{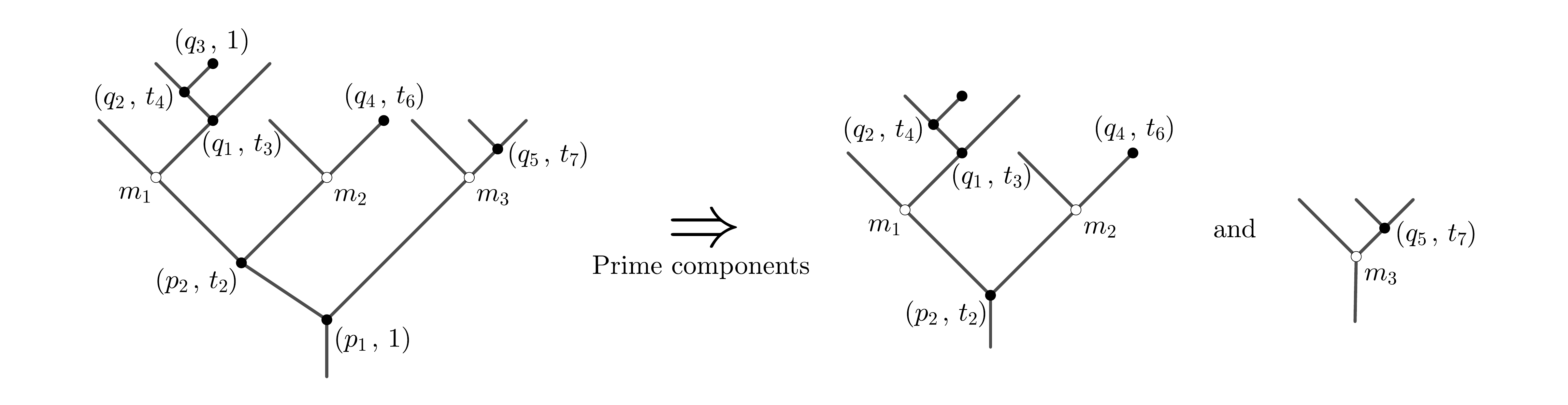}\vspace{-15pt}

\begin{figure}[!h]
\caption{A composite point and its prime components.}\label{B1}
\end{figure}

A prime point is in the $r$-th filtration layer $\mathcal{B}_{P\,;\,Q}(M)_{r}$ if the number of its geometrical inputs is at most~$r$. Similarly, a composite point is in the $r$-th filtration layer if its all prime components are in $\mathcal{B}_{P\,;\,Q}(M)_{r}$. For instance, the composite point in Figure~\ref{B1} is in the filtration layer $\mathcal{B}_{P\,;\,Q}(M)_{6}$. For each $r$, $\mathcal{B}_{P\,;\,Q}(M)_{r}$ is a $(P\text{-}Q)$-bimodule and one has the following filtration of $\mathcal{B}_{P\,;\,Q}(M)$:
\begin{equation}\label{B3}
\xymatrix{
 \mathcal{B}_{P\,;\,Q}(M)_{0} \ar[r] & \mathcal{B}_{P\,;\,Q}(M)_{1}\ar[r] & \cdots \ar[r] & \mathcal{B}_{P\,;\,Q}(M)_{r-1} \ar[r] & \mathcal{B}_{P\,;\,Q}(M)_{r} \ar[r] & \cdots \ar[r] &  \mathcal{B}_{P\,;\,Q}(M).
}
\end{equation}

\begin{thm}[Theorem 2.12 in \cite{Duc2}]\label{th:BV_proj_bimod}
Assume that  $P$ and $Q$ are  $\Sigma$-cofibrant operads, and $M$ is a $\Sigma$-cofibrant 
$(P\text{-}Q)$-bimodule for which the arity zero left action map $\gamma_0\colon P(0)\to M(0)$ is a cofibration. Then the objects $\mathcal{B}_{P\,;\,Q}(M)$ and $\TT_{r}\mathcal{B}_{P\,;\,Q}(M)_{r}$ are cofibrant replacements of $M$ and $\TT_{r}M$ in the categories  $\Sigma\Bimod_{P\,;\,Q}$ and $\TT_{r}\Sigma\Bimod_{P\,;\,Q}$, respectively. In particular, the maps $\mu$ and $\TT_{r}\mu|_{\TT_{r}\mathcal{B}_{P\,;\,Q}(M)_{r}}$ are weak equivalences.
\end{thm}

Now we change slightly the above construction in order to produce Reedy cofibrant replacements for $(P\text{-}Q)$-bimodules when $Q$ is a reduced operad. Let $M$ be a  $(P\text{-}Q)$-bimodule. We consider the $\Sigma$-sequence
\[
\calB_{P\,;\,Q}^\Lambda(M):=
\mathcal{B}_{P\,;\,Q_{>0}}(M).
\]
The superscript $\Lambda$
is to emphasize that we get a cofibrant replacement in the Reedy model category structure.  
% The map $\mu\colon \calB^\Lambda_{P\,;\,Q}(M)\to M$ is extended to  arity zero in the obvious way.
  The left action and the positive arity right action are defined in the same way as for
 $\calB_{P\,;\,Q_{>0}}(M)$. %The right action by the positive arity components is defined as it is on $\calB_{P\,;\,Q_{>0}}(M)$.
  The right action by $*_0\in Q(0)$ is defined in the obvious way as the right
 action by $*_0$ on $a$ in the vertex $(a,t)$ connected to the leaf labelled by $i$ as illustrated in the Figure \ref{A1}.\vspace{10pt}
 
 \begin{figure}[!h]
 \begin{center}
 \includegraphics[scale=0.25]{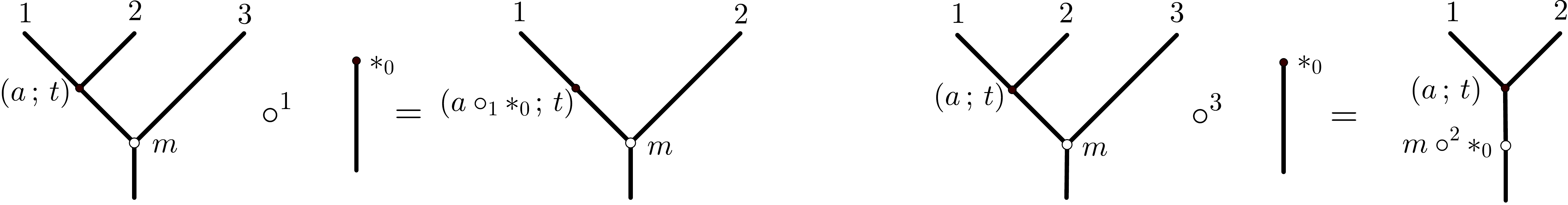}
 \caption{Illustration of the right action by $\ast_{0}$.}\label{A1}
 \end{center}
 \end{figure}

Note that since the arity zero component of $Q_{>0}$ is empty, in the union~\eqref{eq:union_stree} 
we can consider only trees whose all non-pearl vertices  have arities $\geq 1$. We denote this set by 
$s\mathbb{P}^{\geq 1}_n$. In other words, the space $\mathcal{B}^{\Lambda}_{P\,;\,Q}(M)$ can be obtained as the restriction of the coproduct \eqref{eq:union_stree} to this set. 

\begin{pro}\label{p:BV_reedy_bimod}
Assume that $P$ and $Q$ are $\Sigma$-cofibrant topological operads with $Q(0)=\ast$, and $M$ is a $\Sigma$-cofibrant $(P\text{-}Q)$-bimodule for which 
the arity zero left action map $\gamma_0\colon P(0)\to M(0)$ is a cofibration. Then the objects $\mathcal{B}_{P\,;\,Q}^\Lambda(M)$ and $\TT_{r}\mathcal{B}_{P\,;\,Q}^\Lambda(M)$ are cofibrant  replacements of $M$ and $\TT_{r}M$ in the categories $\Lambda\Bimod_{P\,;\,Q}$ and $\TT_{r}\Lambda\Bimod_{P\,;\,Q}$, respectively. In particular, the maps $\mu$ and $\TT_{r}\mu$ are weak equivalences.
\end{pro}

\begin{proof}
The map $\mu:\mathcal{B}^{\Lambda}_{P\,;\,Q}(M)=\mathcal{B}_{P\,;\,Q_{>0}}(M)\rightarrow M$, which sends the real numbers indexing the vertices to $0$, is a homotopy equivalence as shown by Theorem~\ref{th:BV_proj_bimod}. %More precisely, it is a homotopy equivalence in the category of $\Sigma$-sequences in which the homotopy consists in bringing the real numbers to $0$.
   Furthermore, by the same theorem, $\mathcal{B}^{\Lambda}_{P\,;\,Q}(M)=\mathcal{B}_{P\,;\,Q_{>0}}(M)$ is cofibrant in the projective model category of $(P\text{-}Q_{>0})$-bimodules. Due to~Theorem \ref{C1}, $\mathcal{B}_{P\,;\,Q}^{\Lambda}(M)$ is also Reedy cofibrant and it gives rise to a cofibrant resolution of $M$ in the Reedy model category $\Lambda \Bimod_{P\,;\,Q}$. The same arguments work for the truncated case. Note that  $\TT_{r}\mathcal{B}_{P\,;\,Q}^\Lambda(M)_r=\TT_{r}\mathcal{B}_{P\,;\,Q}^\Lambda(M)$
   since arity zero vertices above the section are not permitted.
\end{proof}

\subsubsection{A functorial cofibrant replacement in the projective/Reedy model category}\label{ss:cof_repl_bimod}

In the previous subsection we described a construction of projective and Reedy cofibrant replacements provided the bimodule and the acting operads are $\Sigma$-cofibrant. In this subsection we explain how that construction can be used to functorially define a cofibrant replacement without 
any assumption on the bimodule, while assuming that the right-acting operad~$Q$ has cofibrant components and the left-acting operad~$P$
is $\Sigma$-cofibrant. 
% in case of the projective model category $\Sigma\Bimod_{P\,;\,Q}$, and  has cofibrant components in case of the  Reedy model category $\Lambda\Bimod_{P\,;\,Q}$. 

Given a $(P\text{-}Q)$-bimodule $M$, we first replace it by the $\Sigma$-sequence $M':=|S_\bullet M|$, whose $n$-th space is the realization
of the simplicial set of singular simplices in~$M(n)$. The obtained object is a $(P'\text{-}Q')$-bimodule, where the operads $P'$ and $Q'$
are similarly defined as $P':=|S_\bullet P|$, $Q':=|S_\bullet Q|$. Let $E_\infty$ as usual denote a reduced $\Sigma$-cofibrant model of
the commutative operad. Define $\Sigma$-sequences $M'_\infty$, $P'_\infty$, $Q'_\infty$ as objectwise product
\[
M'_\infty(n):=M'(n)\times E_\infty(n),\quad P'_\infty(n):=P'(n)\times E_\infty(n),\quad Q'_\infty(n):=Q'(n)\times E_\infty(n).
\]
We get that $M'_\infty$ is a $\Sigma$-cofibrant bimodule over a pair $(P'_\infty,Q'_\infty)$ of $\Sigma$-cofibrant operads. We can
therefore apply 
the construction from the previous subsection. 

For the following theorem denote by
\[
\phi_1\colon P'_\infty \xrightarrow{\simeq} P'\xrightarrow{\simeq} P, \quad \phi_2\colon Q'_\infty\xrightarrow{\simeq} Q'\xrightarrow{\simeq} Q
\]
the natural induced equivalences of operads. Note that
\[
\phi_0\colon M'_\infty \xrightarrow{\simeq} M'\xrightarrow{\simeq} M
\]
is an equivalence of $(P'_\infty\text{-}Q'_\infty)$-bimodules.

\begin{pro}\label{p:cof_repl_bimod}
(a) Assume that $P$ is a $\Sigma$-cofibrant operad and $Q$ is a  componentwise cofibrant operad. Let
$M$ be any $(P\text{-}Q)$-bimodule for which the arity zero left action map $\gamma_0\colon P(0)\to M(0)$ is injective.
Then the objects $\phi_!\left(\mathcal{B}_{P'_\infty\,;\,Q'_\infty}(M'_\infty)\right)$ and $\phi_!\left(\TT_{r}\mathcal{B}_{P'_\infty\,;\,Q'_\infty}(M'_\infty)_r\right)$ are cofibrant
 replacements of $M$ and $\TT_{r}M$ in the categories $\Sigma\Bimod_{P\,;\,Q}$ and $\TT_{r}\Sigma\Bimod_{P\,;\,Q}$, respectively.

(b) Assume in addition that $Q$ is reduced. Then the objects $\phi_!\left(\mathcal{B}_{P'_\infty\,;\,Q'_\infty}^\Lambda(M'_\infty)\right)$ and $\phi_!\left(\TT_{r}\mathcal{B}_{P'_\infty\,;\,Q'_\infty}^\Lambda(M'_\infty)\right)$ are cofibrant
 replacements of $M$ and $\TT_{r}M$ in the categories $\Lambda\Bimod_{P\,;\,Q}$ and $\TT_{r}\Lambda\Bimod_{P\,;\,Q}$, respectively.
\end{pro}

\begin{proof}
The result is an immediate consequence of Theorems~\ref{G6},  \ref{ThmExt/rest}, \ref{th:BV_proj_bimod} and Proposition~\ref{p:BV_reedy_bimod}.
\end{proof}

\subsection{The Reedy fibrant replacement as an internal hom}\label{ss:int_hom}

In this subsection, we provide a more conceptual understanding of the Reedy fibrant replacement described in Subsection~\ref{Z2}. More precisely we explain this construction in terms of internal hom in the category of $\Sigma$-sequences, see Proposition~\ref{p:int_hom}.

\subsubsection{The right closed monoidal category of symmetric sequences}\label{sss:r_closed}

It is well known and appears in almost any textbook on the theory of operads that the category $\Sigma Seq$ of $\Sigma$-sequences has a monoidal structure $(\Sigma Seq,\circ,\mathbb{1})$ with the unit
\[
\mathbb{1}(k)=
\begin{cases}
*,  &k=1;\\
\emptyset, & k\neq 1;\\
\end{cases}
\]
and the composition product
\[
(X\circ Y)(k)=\coprod_{n\geq 0} X(n)\times_{\Sigma_n}\coprod_{\beta\colon [k]\to[n]}\,\,
\prod_{i=1}^n Y(|\beta^{-1}(i)|).
\]
Monoids in $(\Sigma Seq,\circ,\mathbb{1})$ are usual topological operads. 

It is  less known that $\Sigma Seq$ is {\it  closed} with respect to this monoidal structure. This fact is true for the category of $\Sigma$-sequences in any bicomplete closed symmetric monoidal category and was noticed by G.~M.~Kelly back in the 1970s~\cite{MaxKelly}.
More recently this also appeared in~\cite[Section~2.2]{Re1} and in~\cite[Section~3]{Ha}\footnote{We also refer to~\cite[Definition~1.20]{AroneChing}, where  this structure appears implicitly and from where our formula~\eqref{eq:int_hom} is borrowed.}. This means that $\Sigma Seq$ is endowed with an internal hom functor
 \[
 [-,-]\colon\Sigma Seq^{op}\times\Sigma Seq\to \Sigma Seq,
 \]
 such that for any $X\in\Sigma Seq$, the  functor $(-)\circ X$ is left adjoint to
 $[X,-]$. Sometimes this structure on a category is called {\it right} closed monoidal instead of just closed monoidal as a monoidal product with an object on the {\it right} has an adjoint. Explicitly,
 \begin{equation}\label{eq:int_hom}
 [X,Y](k)=\prod_{n\geq 0}\left[\prod_{\alpha\colon [n]\to [k]} Map\left(\prod_{i=1}^k
 X(|\alpha^{-1}(i)|), Y(n)\right)\right]^{\Sigma_n}.
 \end{equation}
 
 \subsubsection{The tensor-hom adjunction}\label{sss:tens_hom}
For concreteness all the statements in this subsection are made for the category $\Sigma Seq$. 
One should mention, however, that they hold true for any bicomplete right
closed monoidal category. (One only needs to replace the word \lq\lq{}operad\rq\rq{} by \lq\lq{}monoid\rq\rq{}.)

\begin{lmm}\label{l:int_hom_module}
Let $P,Q\in\Sigma\mathrm{Operad}$ and $X,Y\in\Sigma Seq$. Consider the internal hom object
$[X,Y]\in\Sigma Seq$.
\begin{itemize}
\item[$(a)$] If $X$ is a left $Q$-module, then $[X,Y]$ is a right $Q$-module.
\item[$(b)$] If $Y$ is a left $P$-module, then $[X,Y]$ is a left $P$-module.
\item[$(c)$] If $X$ is a left $Q$-module and $Y$ is a left $P$-module, then $[X,Y]$ is a $(P$-$Q)$-bimodule.
\end{itemize}
\end{lmm}

\begin{proof}[Sketch of the proof] 
For (a), the right action map $[X,Y]\circ Q\to [X,Y]$ is the adjoint to the composition
\[
[X,Y]\circ Q\circ X\xrightarrow{id_{[X,Y]}\circ\mu_X} [X,Y]\circ X\xrightarrow{ev_{X,Y}} Y.
\]
Here $\mu_X\colon Q\circ X\to X$ is the left $Q$-action on X, and $ev_{X,Y}$ is the adjoint
to the identity map $id_{[X,Y]}$.

For (b), the left action map $P\circ [X,Y]\to [X,Y]$ is adjoint to the composition
\[
P\circ[X,Y]\circ X\xrightarrow{id_P\circ ev_{X,Y}}P\circ Y\xrightarrow{\mu_Y} Y.
\]
Here $\mu_Y$ is the left $P$-action on $Y$. 

The facts that these formulae correctly define $P$ and $Q$ actions are easily checked as well as the fact
that these actions commute in case of (c).
\end{proof}

If $X$ is a right module over an operad $Q$ and $Y$ is a left $Q$-module, one defines $X\circ_Q Y\in
\Sigma Seq$ as the coequalizer
\[
X\circ_Q Y=\mathrm{coeq}\left(X\circ Q\circ Y\rightrightarrows
X\circ Y\right),
\]
where both arrows are $\mu_x\circ id_Y$ and $id_X\circ\mu_Y$.

In case $X$ and $Y$ are both right modules over an operad $Q$, one defines $[X,Y]_Q\in\Sigma Seq$ as the
equalizer
\[
[X,Y]_Q=\mathrm{eq}\left([X,Y]\rightrightarrows [X\circ Q,Y]\right),
\]
where the upper map is induced by the right $Q$-action $\mu_X\colon X\circ Q\to X$, and the lower map
is the adjoint to the composition
\[
[X,Y]\circ X\circ Q\xrightarrow{ev_{X,Y}\circ id_Q} Y\circ Q\xrightarrow{\mu_Y} Y.
\]

\begin{pro}\label{p:tens_hom}
{\normalfont (\cite[Proposition~5.22]{Ha})}
Let $Q_1,Q_2,P\in\Sigma\mathrm{Operad}$ and $Y\in\Sigma\mathrm{Bimod}_{Q_1;Q_2}$.
\begin{itemize}
\item[$(a)$]One has an adjunction between the categories of right $Q_1$ and $Q_2$ modules
\begin{equation}\label{eq:tens_hom1}
(-)\circ_{Q_1} Y\colon\Sigma\mathrm{RMod}_{Q_1}\rightleftarrows\Sigma\mathrm{RMod}_{Q_2}\colon
[Y,-]_{Q_2}.
\end{equation}

\item[$(b)$]One has an adjunction between the categories of $(P$-$Q_1)$ and $(P$-$Q_2)$ bimodules
\begin{equation}\label{eq:tens_hom2}
(-)\circ_{Q_1} Y\colon\Sigma\mathrm{Bimod}_{P;Q_1}\rightleftarrows\Sigma\mathrm{Bimod}_{P;Q_2}\colon
[Y,-]_{Q_2}.
\end{equation}

\end{itemize}
\end{pro}

\begin{proof}[Sketch of the proof]
The statements are general and hold in any bicomplete (right) closed monoidal category. The 
proof is essentially a categorification  of the tensor-hom adjunction between the categories of (bi)modules over rings.
\end{proof}

\subsubsection{The Reedy fibrant replacement}\label{sss:reedy}
Let $N$ be a $\Sigma$-cofibrant right module over a $\Sigma$-cofibrant operad $Q$. Viewed as a $(\mathbb{1}$-$Q)$-bimodule, we consider its resolution ${\mathcal B}_{\mathbb{1};Q}(N)$, see Subsection~\ref{3.2.1}, which is its cofibrant replacement as a right $Q$-module. It is easy to see that
in case $N$ is a $(P$-$Q)$-bimodule for some operad $P$, the sequence ${\mathcal B}_{\mathbb{1};Q}(N)$
has also a natural structure of a $(P$-$Q)$-bimodule. Denote by $Q^c:={\mathcal B}_{\mathbb{1};Q}(Q)$.
One can show that $Q^c$ is a cofibrant replacement of $Q$ as a $Q$-bimodule. Roughly speaking it is because even before taking its resolution, $Q$ is cofibrant as a left module over itself. 

We leave the following proposition as an exercise to the reader. For this proposition we do not assume that $Q$ is $\Sigma$-cofibrant.

\begin{pro}\label{p:int_hom}
Let $M\in \Lambda\mathrm{Bimod}_{P;Q}$, let $M^f$ be its Reedy fibrant replacement as defined in Subsection~\ref{Z2}, and let $Q^c:={\mathcal B}_{\mathbb{1};Q}(Q)$, see Subsection~\ref{3.2.1}. One has an isomorphism of  $(P$-$Q)$-bimodules:
\[
M^f=[Q^c,M]_Q.
\]
Moreover, the fibrant replacement map $M\to M^f$ is induced by the (cofibrant, in case $Q$ is $\Sigma$-cofibrant) 
replacement map $Q^c\to Q$:
\[
M=[Q,M]_Q\to [Q^c,M]_Q.
\]
\end{pro}

As an interesting consequence we have the following.

\begin{cor}\label{cor:fib_fib}
For any $M\in \Lambda\mathrm{Bimod}_{P;Q}$, one has a homeomorphism $(M^f)^f\cong M^f$ of fibrant replacements.
\end{cor}

\begin{proof}
It follows from Proposition~\ref{p:int_hom}  and  adjunction~\eqref{eq:tens_hom2} that
\[
(M^f)^f=[Q^c,[Q^c,M]_Q]_Q\cong[Q^c\circ_Q Q^c,M]_Q.
\]
On the other hand, it is easy to see that $Q^c\circ_Q Q^c\cong Q^c$. To define an explicit
homeomorphism $Q^c\xrightarrow{\cong} Q^c\circ_Q Q^c$, for each tree in $Q^c$ we draw a horizontal line $t=1/2$
and then replace all the real parameters in the vertices below the horizontal section $t\in[0,1/2]$ by $2t$ and all the real parameters of the vertices above the section $t\in(1/2,1]$ 
by $2t-1$.\footnote{In fact for any right $Q$-module $N$, one has 
${\mathcal B}_{\mathbb{1};Q}(N)= N\circ_Q Q^c$ and it is always true that 
${\mathcal B}_{\mathbb{1};Q}({\mathcal B}_{\mathbb{1};Q}(N))\cong {\mathcal B}_{\mathbb{1};Q}(N)$.}
\end{proof}\vspace{5pt}

\subsection{The subcategory of reduced bimodules}\label{reduced}

\subsubsection{Main properties}\label{sss:reduced_main}

Let $P$ and $Q$ be two reduced operads. We can then consider the category $\Lambda_*\Bimod_{P\,;\,Q}$ of reduced $(P$-$Q)$-bimodules.
% In the general setting of ($P$-$Q$)-bimodules, some technical issues arise from the map $\gamma_{0}:P(0)\rightarrow M(0)$. Thanks to the $\gamma$-relation that contracts points coming from $P(0)$ in must of the constructions done in the previous sections, we can slightly improve the theorems if we restrict our study to reduced bimodules that is the full subcategory $\Lambda_{\ast}\Bimod_{P\,;\,Q}$ whose objects are ($P$-$Q$)-bimodules $M$ with $M(0)=\ast$. 
  This category   has been used by the first and third authors \cite{DT} in order to get delooping theorems 
  for the Taylor tower approximations of mapping spaces avoiding multi-singularities, i.e. singularities depending on several points. 

The present section is devoted to adapt the constructions and theorems introduced in the previous sections to the category $\Lambda_{\ast}\Bimod_{P\,;\,Q}$. By considering only reduced bimodules, we can simplify some constructions.  % and we can improve slightly the main theorems (in particular the assumptions on the reduced operad $P$). 
  Since the proofs are almost the same, we list altogether the main statements. Note that the statements are slightly improved  for the functorial cofibrant resolution and for the extension/restriction adjunction.

\begin{thm}\label{MainThmRed}
Let $P$ and $Q$ be  reduced operads and, moreover, assume that  $Q$ is well-pointed.

\begin{itemize}[itemsep=10pt, leftmargin=20pt]
\item[$(i)$] \textbf{Reedy model structure.} The category $\Lambda_{\ast} \Bimod_{P\,;\,Q}$ admits a cofibrantly generated model category structure, called the Reedy model category structure,  transferred from the Reedy model category $\Lambda_{>0} Seq$ along the adjunction 
\begin{equation*}
\begin{array}{rcl}\vspace{8pt}
\mathcal{F}_{P\,;\,Q}^{\Lambda_{\ast}}:\Lambda_{>0} Seq & \rightleftarrows & \Lambda_{\ast} \Bimod_{P\,;\,Q}:\mathcal{U}^{\Lambda}, 
\end{array} 
\end{equation*}
where the free functor is given by 
\begin{equation}\label{eq:F_Lambda_*}
\mathcal{F}_{P\,;\,Q}^{\Lambda_{\ast}}(M)(n)\coloneqq \left\{ 
\begin{array}{cl}\vspace{5pt}
\mathcal{F}^{\Sigma}_{P_{>0}\,;\,Q_{>0}}(M)(n), &  n\geq 1, \\ 
\ast, &  n=0,
\end{array} 
\right.
\end{equation}
The model category so obtained makes the adjunction $(\mathcal{F}_{P\,;\,Q}^{\Lambda_{\ast}}, \mathcal{U}^{\Lambda})$ into a Quillen adjunction. Moreover, the fibrant coresolution functor introduced in Section \ref{Z2}   restricts to the category $\Lambda_{\ast} \Bimod_{P\,;\,Q}$ giving rise to a functorial fibrant replacement.
 In case $P$ and $Q$ are componentwise cofibrant,  the  functor $M\mapsto  \phi_!\left(\mathcal{B}_{P'_\infty\,;\,Q'_\infty}^\Lambda(M'_\infty)\right)$  of Proposition~\ref{p:cof_repl_bimod}  %Section \ref{ss:cof_repl_bimod} restricts to the category $\Lambda_{\ast} \Bimod_{P\,;\,Q}$ and it
  gives rise to a functorial  cofibrant resolution in  $\Lambda_{\ast} \Bimod_{P\,;\,Q}$.  

\item[$(ii)$] \textbf{Characterization of Reedy cofibrations.} A morphism $\phi:M\rightarrow N$ in the category $\Lambda_{\ast} \Bimod_{P\,;\,Q}$ is a Reedy cofibration if and only if the corresponding map $\phi_{>0}:M_{>0}\rightarrow N_{>0}$ is a cofibration in the projective model category of $(P_{>0}\text{-}Q_{>0})$-bimodules.

\item[$(iii)$] \textbf{Left and Right properness.} The Reedy model category  $\Lambda_{\ast}\Bimod_{P\,;\,Q}$ is right proper. If $P$ is Reedy cofibrant and $Q$ is componentwise cofibrant, then  $\Lambda_{\ast}\Bimod_{P\,;\,Q}$ is left proper relative to the class of componentwise cofibrant bimodules. In the latter case, the class of componentwise cofibrant bimodules is closed 
under cofibrations. In particular, cofibrant bimodules are componentwise cofibrant. If in addition $Q$ is $\Sigma$-cofibrant, the class of $\Sigma$-cofibrant bimodules is also closed under cofibrations
and cofibrant bimodules are $\Sigma$-cofibrant.

\item[$(iv)$] \textbf{Extension/restriction adjunctions.} Let $\phi_{1}\colon P\rightarrow P'$  and $\phi_{2}\colon Q\rightarrow Q'$  be  weak equivalences between  componentwise cofibrant  reduced operads. One has a Quillen equivalence
\begin{equation*}
\phi_!\colon\Lambda_{\ast}\Bimod_{P\,;\,Q}\rightleftarrows\Lambda_{\ast}\Bimod_{P'\,;\,Q'}\colon\phi^*,
\end{equation*}
\end{itemize}

\noindent All the above statements admit truncated versions for the categories  $T_{r}\Lambda_{\ast} \Bimod_{P\,;\,Q}$ with $r\geq 0$.

\end{thm}

\begin{proof}
The proof of $(i)$ is exactly the same as for the construction of the usual Reedy model category of ($P$-$Q$)-bimodules and the construction of the fibrant coresolution. Note that for the
cofibrant replacement functor we do not require $P$ to be $\Sigma$-cofibrant, but only that it is componentwise cofibrant. The reason is that the argument of Proposition~\ref{p:cof_repl_bimod}
uses the extension/restriction adjunction. The latter has also the same improved requirement that $P$ and $P'$ are  componentwise cofibrant, see statement $(iv)$ above.

Statement $(ii)$ follows from the fact that the forgetful functor $\iota\colon \Lambda_{\ast}\Bimod_{P\,;\,Q}\to \Lambda\Bimod_{P\,;\,Q}$ creates cofibrations, see 
Proposition~\ref{ProCompMCb}. Applying
the characterization of Reedy cofibrations for bimodules Theorem~\ref{C1}, a map $M\to N$ in $\Lambda_{\ast}\Bimod_{P\,;\,Q}$ is a cofibration if and only if it is
a cofibration in $\Sigma\Bimod_{P\,,\,Q_{>0}}$. Since $P(0)=M(0)=N(0)=\ast$, the arity zero component can be naturally ignored, and the map in question is a cofibration if and only if 
$M_{>0}\to N_{>0}$  is one in $\Sigma\Bimod_{P_{>0}\,,\,Q_{>0}}$ (or, equivalently, in $\Sigma_{>0}\Bimod_{P_{>0}\,,\,Q_{>0}}$).

Statement $(iii)$ follows from Theorem~\ref{ThmProperness} and the fact that the inclusion functor $\iota$ creates weak equivalences,  fibrations, pullbacks, and also
cofibrations and pushouts, see Proposition~\ref{ProCompMCb}. 

We now prove $(iv)$. Since the restriction functor creates weak equivalences, one has to check that, for any Reedy cofibrant object $M$ in $\Lambda_{\ast}\Bimod_{P\,;\,Q}$, the adjunction unit 
\begin{equation*}
M\longrightarrow \phi^{\ast}(\phi_{!}(M))
\end{equation*}
is a weak equivalence. Due to the characterization of Reedy cofibrations, $M_{>0}$ is also cofibrant in the projective model category $\Sigma_{>0}\Bimod_{P_{>0},Q_{>0}}$. 
Since $(\phi_1)_{>0}:P_{>0}\rightarrow P'_{>0}$ and $(\phi_2)_{>0}:Q_{>0}\rightarrow Q'_{>0}$ are still  weak equivalences between componentwise  cofibrant operads, by Theorem~\ref{th:>0}, the pair of functors $((\phi_{>0})_{!}\,;\, (\phi_{>0})^{\ast})$ gives rise to a Quillen equivalence
\begin{equation*}%\label{eq:bimod_ind_restr2}
(\phi_{>0})_!\colon\Sigma_{>0}\Bimod_{P\,;\,Q}\rightleftarrows\Sigma_{>0}\Bimod_{P'\,;\,Q'}\colon(\phi_{>0})^*.
\end{equation*}
Therefore, the map $M_{>0}\rightarrow (\phi_{>0})^{\ast}\bigl((\phi_{>0})_{!}(M_{>0})\bigr)(n)$ is a weak equivalence. The statement is a consequence of the identity 
$$
\phi^{\ast}(\phi_{!}(M))_{>0}=(\phi_{>0})^{\ast}\big( (\phi_{>0})_{!}(M_{>0})\big).
$$

\end{proof}

\subsubsection{Quillen adjunction between the Reedy model categories}\label{E7}

Let $P$ and $Q$ be two reduced operads with $Q$ well-pointed. The inclusion functor $\iota$ from the category of (possibly truncated) reduced bimodules into the category of (possibly truncated) bimodules has a left adjoint $\tau$ called the unitarization functor. This latter one consists in collapsing the arity zero component to a point and adjusting the other components according to the equivalence relation induced by this collapse:
\begin{equation}\label{eq:unitar_bimred}
\begin{array}{rcl}\vspace{5pt}
\tau\colon\Lambda\Bimod_{P\,;\,Q} & \rightleftarrows & \Lambda_{\ast}\Bimod_{P\,;\,Q}\colon\iota, \\ 
\tau\colon\TT_r \Lambda \Bimod_{P\,;\,Q} & \rightleftarrows & \TT_r\Lambda_{\ast}\Bimod_{P\,;\,Q}\colon\iota.
\end{array} 
\end{equation}

\begin{pro}\label{ProCompMCb}
The pairs of functors (\ref{eq:unitar_bimred}) form Quillen adjunctions. Moreover, the inclusion functor $\iota$ creates weak equivalences, fibrations, cofibrations, limits, and colimits.
\end{pro}

\begin{proof}
By construction, $\iota$ creates equivalences and fibrations. As a consequence, the adjunction~\eqref{eq:unitar_bimred} is a Quillen one. Since the limits are taken objectwise, and the limit of any point-constant diagram is a point, $\iota$ creates limits. Note that $\tau$, as a left adjoint, preserves colimits and cofibrations, and $\tau\circ\iota=id$. Therefore, one only needs to check that $\iota$ 
preserves colimits and cofibrations. 

For colimits, we again have  to show that the colimit of any diagram in $\Lambda\Bimod_{P\,;\,Q}$ with values in reduced bimodules
is a reduced bimodule. The truncation functor $T_r\colon \Lambda\Bimod_{P\,;\,Q}\to T_r\Lambda \Bimod_{P\,;\,Q}$ preserves colimits as it admits a right adjoint, see Section~\ref{Fin1}. 
In the case $r=0$, the category  $T_0\Lambda \Bimod_{P\,;\,Q}$ is equivalent to the category of $P$-algebras in $Top$. The one-point space is a free $P$-algebra generated by the empty set $F_P(\emptyset)$.
On the other hand, the free $P$-algebra functor $F_P\colon Top\to Alg_P$ preserves colimits and the colimit of any empty set constant diagram  is the empty set. The statement follows.
 
Finally, for cofibrations, since $\iota$ creates colimits, it is enough
to check that the generating cofibrations in $\Lambda_*\Bimod_{P\,;\,Q}$ are cofibrations in $\Lambda\Bimod_{P\,;\,Q}$. The generating cofibrations in the former
are the maps 
\begin{equation}\label{eq:gen_cof_L*}
\mathcal{F}_{P\,;\,Q}^{\Lambda_{\ast}}(\partial X) \to \mathcal{F}_{P\,;\,Q}^{\Lambda_{\ast}}(X),
\end{equation}
where $\partial X\to X$ is one of the generating cofibrations of $\Lambda_{>0}Seq$. For a $\Lambda_{>0}$-sequence $Y$, denote by $Y_+$ the $\Lambda$-sequence that agrees with $Y$ in positive arities and has one point in arity zero. From the formulas~\eqref{eq:F_Lambda_*} and~\eqref{eq:free}, it follows that $\mathcal{F}_{P\,;\,Q}^{\Lambda_{\ast}}(Y)=
\mathcal{F}_{P\,;\,Q}^{\Lambda}(Y_+)$. The inclusion $(\partial X)_+\to X_+$ is a $\Sigma$-cofibration and, as a consequence,
is  a $\Lambda Seq_P$-cofibration. Together with the fact that the functor $\mathcal{F}_{P\,;\,Q}^{\Lambda}$ preserves cofibrations, we conclude that the inclusion~\eqref{eq:gen_cof_L*}, or equivalently $\mathcal{F}_{P\,;\,Q}^{\Lambda}((\partial X)_+) \to \mathcal{F}_{P\,;\,Q}^{\Lambda}(X_+)$,  is a cofibration in $\Lambda\Bimod_{P\,;\,Q}$.

%We show that $\iota$ preserves fibrations and acyclic fibrations. Let $f:M\rightarrow N$ be an (acyclic) fibration in the Reedy model category of reduced bimodules. By definition, the map $f$ is an (acyclic) fibration if the corresponding map $\mathcal{U}^\Lambda(f)$ in the category of $\Lambda$-sequences is an (acyclic) fibration. In other words, it means that the maps 
%\begin{equation}\label{B2}
%M(n)\longrightarrow \mathcal{M}(M)(n)\times_{\mathcal{M}(N)(n)}N(n),\hspace{15pt}\text{with } n\in \mathbb{N},
%\end{equation}
%are (acyclic) Serre fibrations. On the over hand, the map $\iota(f)$ is a fibration in the in the Reedy model category of (non necessarily reduced) bimodules if the maps \eqref{B2} are (acyclic) Serre fibrations. Consequently, the functor $\iota$ obviously preserves (acyclic) fibrations and the adjunctions \eqref{eq:unitar_bimred} are Quillen adjunctions.  
\end{proof}

\begin{thm}\label{th:bimod_maps}
Let $P$ and $Q$ be as in Theorem~\ref{MainThmRed}. The inclusion functors $\iota$ 
from~\eqref{eq:unitar_bimred} % $\iota\colon \Lambda_{\ast}\Bimod_{P\,;\,Q}\to \Lambda\Bimod_{P\,;\,Q}$
induce  fully faithful inclusions of homotopy categories. Moreover, for any pair $M,\, N\in\Lambda_{\ast}\Bimod_{P\,;\,Q}$, one has
\begin{equation}\label{eq:der_bim_map1}
 \Lambda_{\ast}\Bimod_{P\,;\,Q}^h(M,N)\simeq\Lambda\Bimod_{P\,;\,Q}^h(\iota M,\iota N)\simeq \Sigma\Bimod_{P\,;\,Q}^h(\iota M,\iota N).
\end{equation}
Furthermore, if $M,\, N\in\TT_r\Lambda_{\ast}\Bimod_{P\,;\,Q}$, with $r\geq 0$, then one has
\begin{equation}\label{eq:der_tr_bim_map2}
\TT_r\Lambda_{\ast}\Bimod_{P\,;\,Q}^h(M,N)\simeq\TT_r\Lambda\Bimod_{P\,;\,Q}^h(\iota M,\iota N) \simeq\TT_r\Sigma\Bimod_{P\,;\,Q}^h(\iota M,\iota N).
\end{equation}
%
%(b) Let $P$ and $Q$ be reduced  componentwise cofibrant operads  and let $P_\infty$ be  a reduced $\Sigma$-cofibrant operad equivalent to
%$P$ due to an equivalence of operads $p\colon P_\infty\xrightarrow{\simeq}P$ (for example, it can be the 
%objectwise product $P_\infty:=P\times E_\infty$). For any  pair $M,\, N\in\Lambda_{\ast}\Bimod_{P\,;\,Q}$, one has
%\begin{equation}\label{eq:der_bim_map3}
% \Lambda_{\ast}\Bimod_{P\,;\,Q}^h(M,N)\simeq\Lambda\Bimod_{P_\infty\,;\,Q}^h(p^*\iota M,p^*\iota N).
%\end{equation}
%Moreover, if $M,\, N\in\TT_r\Lambda_{\ast}\Bimod_{P\,;\,Q}$, with $r\geq 0$, then one has
%\begin{equation}\label{eq:der_tr_bim_map4}
%\TT_r\Lambda_{\ast}\Bimod_{P\,;\,Q}^h(M,N)\simeq\TT_r\Lambda\Bimod_{P_\infty\,;\,Q}^h(p^*\iota M,p^*\iota N).
%\end{equation}
\end{thm}

\begin{proof}
The first statement follows from the fact that the functors $\iota$ themselves are fully faithful inclusions of model categories preserving equivalences, fibrations, and cofibrations,
see Proposition~\ref{ProCompMCb}. The equivalences~\eqref{eq:der_bim_map1} and~\eqref{eq:der_tr_bim_map2} follow from Theorem~\ref{th:S-L_equiv_bim}  and the fact
that   $\iota$ preserve cofibrant and fibrant
replacements and are fully faithful. 

\end{proof}

\section{The projective model category of $O$ infinitesimal bimodules}\label{SectProjIbimod}\vspace{5pt}

Let $O$ be an operad. An $O$ infinitesimal bimodule, or just an $O$-Ibimodule, is a $\Sigma$-sequence $M\in \Sigma Seq$ together with operations 
\begin{equation}\label{AA2}
\begin{array}{llr}\vspace{7pt}
\circ^{i}: & M(n)\times O(m)\longrightarrow  M(n+m-1), & \text{right infinitesimal operations with } i\in \{1,\ldots,n\},\\ \vspace{7pt}
\circ_{i}: & O(n)\times M(m)\longrightarrow  M(n+m-1), & \text{left infinitesimal operations with } i\in \{1,\ldots,n\},
\end{array}
\end{equation}
satisfying compatibility relations with the symmetric group action as well as associativity and unit axioms. More precisely, for any integers $i\in \{1,\ldots, n\}$, $j\in\{ i+1,\ldots , n\}$, $k\in \{1,\ldots, m\}$ and any permutation $\sigma\in \Sigma_{n}$ and $\tau\in \Sigma_{m}$, one has the following commutative diagrams: %\todo{Question Benoit: Il manque un axiom? There should also be a ramified compatibility where O comes before M}
$$
\underset{\text{Linear associativity for the right infinitesimal  operations }}{\xymatrix{
M(n)\times O(m)\times O(\ell)\ar[r]^{\circ_{k}} \ar[d]_{\circ^{i}} & M(n)\times O(m+\ell-1) \ar[d]^{\circ^{i}}\\
M(n+m-1)\times O(\ell) \ar[r]_{\circ^{k+i-1}} & M(n+m+\ell-2)
}}\hspace{20pt}
\underset{\text{Ramified associativity for the right infinitesimal operations}}{\xymatrix{
M(n)\times O(m)\times O(\ell)\ar[r]^{\circ^{i}} \ar[d]_{\circ^{j}} & M(n+m-1)\times O(\ell) \ar[d]^{\circ^{j+m-1}}\\
M(n+\ell-1)\times O(m) \ar[r]_{\circ^{i}} & M(n+m+\ell-2)
}}\vspace{10pt}
$$
$$
\underset{\text{Ramified compatibility between the left and right operations  1}}{\xymatrix{
O(n)\times M(m)\times O(\ell)\ar[r]^{\circ_{i}} \ar[d]_{\circ_{j}} & M(n+m-1)\times O(\ell) \ar[d]^{\circ^{j+m-1}}\\
O(n+\ell-1) \times M(m) \ar[r]_{\circ_{i}} & M(n+m+\ell-2)
}}\hspace{20pt}
\underset{\text{Ramified compatibility between the left and right operations  2}}{\xymatrix{
O(n)\times O(m)\times M(\ell)\ar[r]^{\circ_{i}} \ar[d]_{\circ_{j}} & O(n+m-1)\times M(\ell) \ar[d]^{\circ_{j+m-1}}\\
M(n+\ell-1) \times O(m) \ar[r]_{\circ^{i}} & M(n+m+\ell-2)
}}
\vspace{10pt}
$$
$$
\underset{\text{Linear compatibility between the left and right operations}}{\xymatrix{
O(n)\times M(m)\times O(\ell)\ar[r]^{\circ_{i}} \ar[d]_{\circ^{k}} & M(n+m-1)\times O(\ell) \ar[d]^{\circ^{k+i-1}}\\
O(n) \times M(m+\ell-1) \ar[r]_{\circ_{i}} & M(n+m+\ell-2)
}}\hspace{40pt}\underset{\text{Linear associativity for the left infinitesimal operations }}{\xymatrix{
O(n)\times O(m)\times M(\ell)\ar[r]^{\circ_{k}} \ar[d]_{\circ_{i}} & O(n)\times M(m+\ell-1) \ar[d]^{\circ_{i}}\\
O(n+m-1)\times M(\ell) \ar[r]_{\circ_{k+i-1}} & M(n+m+\ell-2)
}}
\vspace{10pt}
$$
$$
\underset{\text{Compatibility with the symmetric group action 1}}{\xymatrix{
O(n)\times M(m) \ar[r]^{\circ_{i}} \ar[d]_{\sigma^{\ast}\times \tau^{\ast}} & M(n+m-1) \ar[d]^{(\sigma\circ_{\sigma(i)} \tau)^{\ast})}\\
O(n)\times M(m)\ar[r]_{\circ_{\sigma(i)}} & M(n+m-1) 
}}\hspace{80pt}
\underset{\text{Compatibility with the symmetric group action 2}}{\xymatrix{
M(n)\times O(m) \ar[r]^{\circ^{i}} \ar[d]_{\sigma_{\ast}\times \tau^{\ast}} & M(n+m-1) \ar[d]^{(\sigma\circ_{\sigma(i)} \tau)^{\ast}}\\
M(n)\times O(m)\ar[r]_{\circ^{\sigma(i)}} & M(n+m-1) 
}}\vspace{10pt}
$$
$$
\underset{\text{Compatibility with the unit of the operad}}{\xymatrix@R=23pt{
M(n)\times O(1)\ar[dr]_{\circ^{i}} & M(n) \ar[r] \ar[l] \ar@{=}[d] & O(1)\times M(n) \ar[dl]^{\circ_{1}}\\
& M(n) & 
}}\vspace{10pt}
$$
Note that  the \textit{ramified compatibility between the left and right operations}~2 follows  from the \textit{ramified compatibility between the left and right operations}~1 and the \textit{compatibility with the symmetric group action} and was given for clarity. 

A map between $O$-Ibimodules should preserve these operations. We denote by  $\Sigma\Ibimod_{O}$ the category of $O$-Ibimodules. Given an integer $r\geq 0$, we also consider the category of $r$-truncated Ibimodules $T_{r}\Sigma\Ibimod_{O}$. An object of this category is an $r$-truncated $\Sigma$-sequence endowed with left and right operations (\ref{AA2}) which are defined under the conditions $n\leq r$ and $n+m-1\leq r$. One has an obvious truncation functor 
$$
T_{r}(-):\Sigma\Ibimod_{O}\longrightarrow T_{r}\Sigma\Ibimod_{O}.
$$
In the rest of the paper, we use the notation
$$
\begin{array}{rl}\vspace{5pt}
x\circ^{i}\theta=\circ^{i}(x;\theta), & \text{for } x\in M(n) \text{ and } \theta\in O(m), \\ 
\theta\circ_{i}x=\circ_{i}(\theta;x), & \text{for } \theta\in O(n) \text{ and } x\in M(m), 
\end{array} 
$$

\begin{expl}
If $\eta:O\rightarrow M$ is a map of $O$-bimodules, then $\eta$ is also a map of $O$-Ibimodules. Indeed, any operad is an infinitesimal bimodule over itself. Since the right operations and the right infinitesimal operations are the same, the $O$-Ibimodule structure on $M$ is given by the following left infinitesimal operations:
$$
\begin{array}{rcl} \vspace{5pt}
\circ_{i}:  O(n)\times M(m) & \longrightarrow & M(n+m-1);\\ (\theta,x) &\longmapsto & \gamma_{\ell}(\theta;\eta(\ast_{1}),\cdots,\eta(\ast_{1}),x,\eta(\ast_{1}),\cdots,\eta(\ast_{1})).
\end{array} 
$$
\end{expl}

\subsection{Properties of the category of infinitesimal bimodules}\label{Z6}

In this subsection we introduce some basic properties related to the category of $O$-Ibimodules where $O$ is a fixed operad. First, we show that the category of $O$-Ibimodules is equivalent to the category of algebras over an explicit colored operad denoted by $O_{+}$. Thereafter, we build the free bimodule functor using the language of trees. Finally, we  give a combinatorial description of the pushout for infinitesimal bimodules.

\subsubsection{Infinitesimal bimodules as algebras over a colored operad}\label{DD8}

From an operad~$O$, we build a colored operad $O_{+}$ such that the category of $O$-Ibimodules is equivalent to the category of $O_{+}$-algebras. More precisely, the colored operad $O_{+}$, with set of colors $S=\mathbb{N}$, is concentrated in arity $1$ and it is given by the formula
\begin{equation}\label{eq:O+}
O_{+}(n\,;\,m):=\underset{\alpha_{+}:[m]_{+}\rightarrow [n]_{+}}{\coprod}\,\,\, \underset{i\in [n]_{+}}{\prod} \, O(|\alpha_{+}^{-1}(i)|),
\end{equation}
where $[n]_{+}$ is the set obtained from $[n]$ by adding a basepoint denoted by $0$. The map $\alpha_{+}$ is a map of sets preserving the basepoint. In order to define   operadic compositions 
$$
\circ_{1}:O_{+}(n\,;\,m)\times O_{+}(k\,;\,n)\longrightarrow O_{+}(k\,;\,m),
$$
we fix two pointed maps  $\alpha_{+}:[m]_{+}\rightarrow [n]_{+}$ and $\beta_{+}:[n]_{+}\rightarrow [k]_{+}$ and we build a map of the form
$$
\underset{i\in [n]_{+}}{\prod}\hspace{5pt}O(\,|\alpha^{-1}_{+}(i)|\,)\times \underset{j\in [k]_{+}}{\prod}\hspace{5pt}O(\,|\beta^{-1}_{+}(j)|\,) \longrightarrow \underset{j\in [k]_{+}}{\prod}\hspace{5pt}O(\,|(\beta_{+}\circ\alpha_{+})^{-1}(j)|\,).
$$

\noindent For this purpose, we rewrite the left hand side term as follows: \vspace{5pt}
$$
\underset{\text{Part 1}}{\underbrace{O(\,|\alpha^{-1}_{+}(0)|\,)\times O(\,|\beta^{-1}_{+}(0)|\,) \times \underset{i\in \beta_{+}^{-1}(0)\setminus \{0\}}{\prod}\hspace{-5pt}O(\,|\alpha_{+}^{-1}(i)|\,)}} \times \underset{j\in [k]}{\prod}\hspace{5pt} \underset{\text{Part 2}}{\underbrace{ O(\,|\beta^{-1}_{+}(j)|\,) \times \underset{i\in \beta^{-1}_{+}(j)}{\prod} \hspace{5pt} O(\,|\alpha^{-1}_{+}(i)|\,) }}. \vspace{3pt}
$$
In Part 2, as in formula \eqref{C9}, we use the operadic structure of~$O$ in order to get an element in $O(\, |(\beta_{+}\circ\alpha_{+})^{-1}(j)|\,)$ with $j\neq 0$.
 For Part 1, let $\beta_{+}^{-1}(0)\setminus \{0\} =\{i_1,\ldots,i_p\}$, $a_0\in O(\,|\alpha^{-1}_{+}(0)|\,)$,
 $b_0\in O(\,|\beta^{-1}_{+}(0)|\,)$, $a_{r}\in O(\,|\alpha^{-1}_{+}(i_r)|\,)$, $1\leq r\leq p$. Then the element
 $(a_0,b_0,a_{1},\ldots,a_{p})$ in the Part~1 product is sent to $a_0\circ_1 b_0(id,a_{1},\ldots,a_{p})$.
% we first use the composition $O(\,|\alpha^{-1}_{+}(0)|\,)\circ_{1} O(\,|\beta^{-1}_{+}(0)\,)$, which corresponds to the left operation indexed by the basepoint, followed by the operadic compositions $\circ_{i}$.

\begin{pro}[Proposition 4.9 in \cite{AT}]
The category of $O$-Ibimodules is equivalent to the category of $O_{+}$-algebras. %\todo{citation}
\end{pro}

Note that since the operad $O_+$ has only unary operations, it can be viewed as an enriched in $Top$ category.  One has that  an $O_+$-algebra is the same thing as an $O_+$-shaped diagram in $Top$:
\[
Alg_{O_+}=Top^{O_+}.
\]

\subsubsection{The free Ibimodule functor}\label{CC2}

In what follows, we introduce the left adjoints of the forgetful functors
$$
\mathcal{U}^{\Sigma}: \Sigma\Ibimod_{O}\longrightarrow \Sigma Seq \hspace{15pt}\text{and} \hspace{15pt} \mathcal{U}^{\TT_{r}\Sigma}: \TT_{r}\Sigma\Ibimod_{O}\longrightarrow \TT_{r}\Sigma Seq,
$$
denoted by $\mathcal{IF}_{O}^\Sigma$ and  $\mathcal{IF}_{O}^{T_{r}\Sigma}$, respectively. As usual in the operadic theory, the free functor can be described as a coproduct indexed by a particular set of trees. In that case, we use the set of \textit{pearled trees} which are pairs $T=(T\,;\,p)$ where $T$ is a planar rooted tree, with leaves labelled by a permutation, and $p$ is a particular vertex, called the \textit{pearl}. A pearled tree is said to be \textit{reduced} if each vertex is connected to the pearl by an inner edge.  We denote by $p\mathbb{P}_{n}$ and $rp\mathbb{P}_{n}$ the sets of pearled trees and reduced pearled trees, respectively, having exactly $n$ leaves.%\vspace{-5pt}

\begin{figure}[!h]
\begin{center}
\includegraphics[scale=0.3]{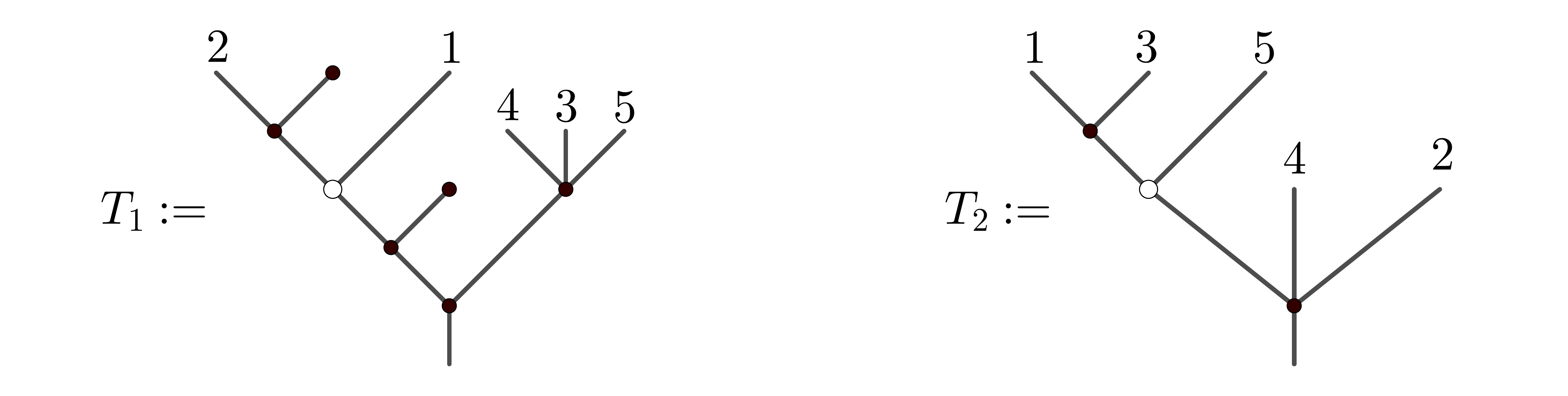}%\vspace{-15pt}
\caption{Examples of a pearled tree $T_{1}\in p\mathbb{P}_{5}$ and a reduced pearled tree $T_{2}\in rp\mathbb{P}_{5}$.}
\end{center}%\vspace{-30pt}
\end{figure}

%\newpage

\begin{const}\label{FF9}
Let $M=\{M(n)\}$ be a $\Sigma$-sequence. The space $\mathcal{IF}_{O}^\Sigma(M)(n)$ is obtained from the set of reduced pearled trees by indexing the pearl by a point in $M$ whereas the other vertices are indexed by points in the operad~$O$. More precisely, one has 
\begin{equation}\label{GG4}
\mathcal{IF}_{O}^\Sigma(M)(n)=\left. \left(
\underset{T\in rp\mathbb{P}_{n}}{\coprod}\hspace{5pt}  M(|p|)\times \underset{v\in V(T)\setminus \{p\}}{\prod} O(|v|) \right)\,\,
\right/\!\!\sim.
\end{equation}
A point in $\mathcal{IF}_{O}^\Sigma(M)$ is denoted by $[T\,;\,x_{p}\,;\,\{\theta_{v}\}]$ where $T$ is a reduced pearled tree, $x_{p}$ is a point in $M$ and $\{\theta_{v}\}_{v\in V(T)\setminus \{p\}}$ is a family of points in $O$. The equivalence relation is generated by the following relations:
\vspace{5pt}

\begin{itemize}
\item[$i)$] \textit{The unit relation}: if a vertex is indexed by the unit of the operad~$O$, then we can remove it.

\vspace{-10pt}
\begin{figure}[!h]
\begin{center}
\includegraphics[scale=0.35]{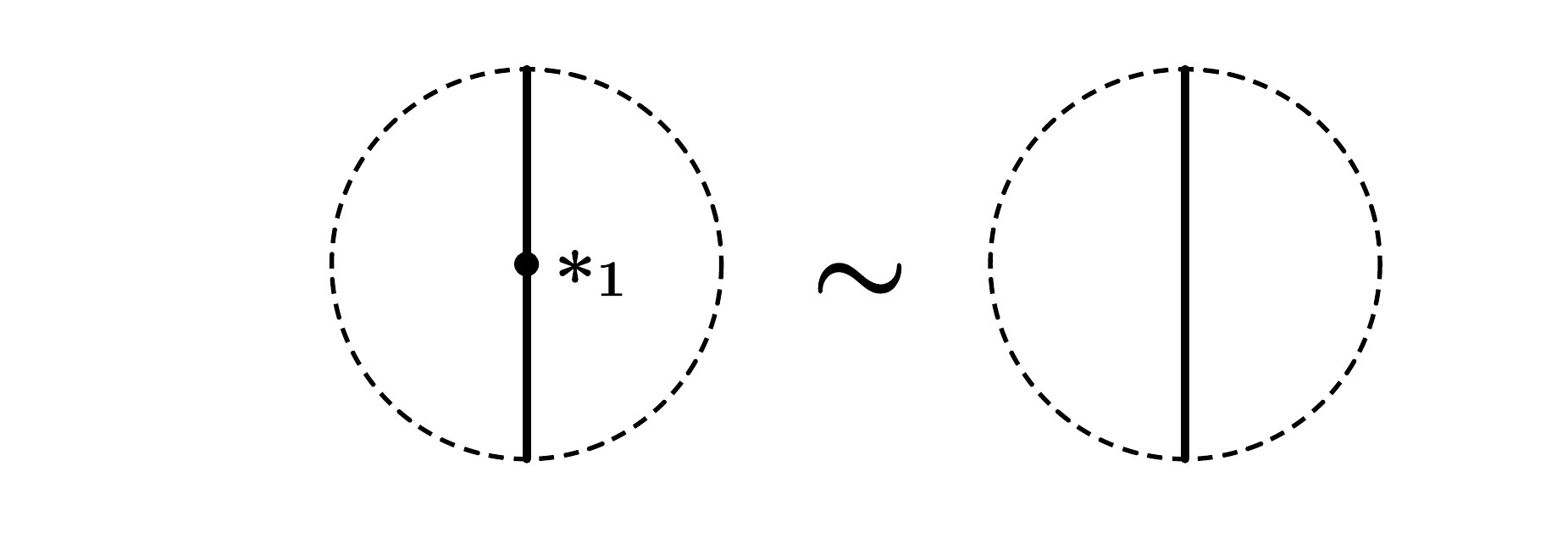}\vspace{-10pt}
\caption{Illustration of the unit relation.}
\end{center}
\end{figure}

\item[$ii)$] \textit{The compatibility with the symmetric group action}: if a vertex is labelled by $x\cdot \sigma$, with $x$ a point in  $O(n)$ or $M(n)$ and $\sigma\in \Sigma_{n}$, then we can remove $\sigma$ by permuting the incoming edges.

\begin{figure}[!h]
\begin{center}
\includegraphics[scale=0.5]{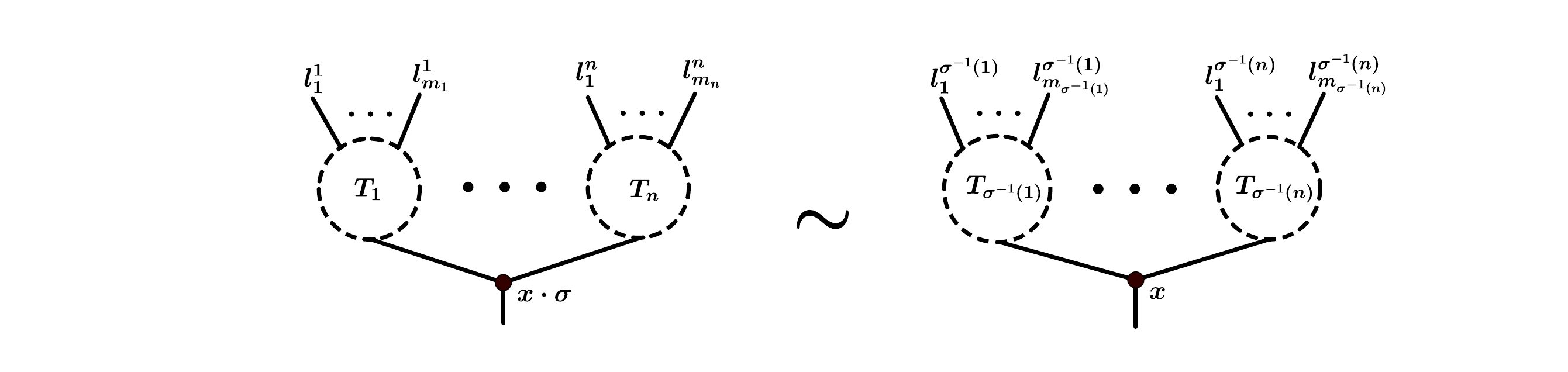}\vspace{-10pt}
\caption{Illustration of the compatibility with the symmetric group.}%\vspace{-15pt}
\end{center}
\end{figure}

\end{itemize}

%\newpage

The right infinitesimal operation $\circ^{i}$ (respectively the left infinitesimal operation $\circ_{i}$) of a point $[T\,;\,x_{p}\,;\,\{\theta_{v}\}]$ with an element $\theta\in O(m)$ consists in grafting the $m$-corolla indexed by $\theta$ (respectively the reduced pearled tree $T$)  into the $i$-th leaf of the reduced tree with section $T$ (respectively the $m$-corolla indexed by $\theta$). If the obtained element  contains an inner edge joining two consecutive vertices other than a pearl, then we contract it using the operadic structure of~$O$.

\begin{figure}[!h]
\begin{center}
\includegraphics[scale=0.45]{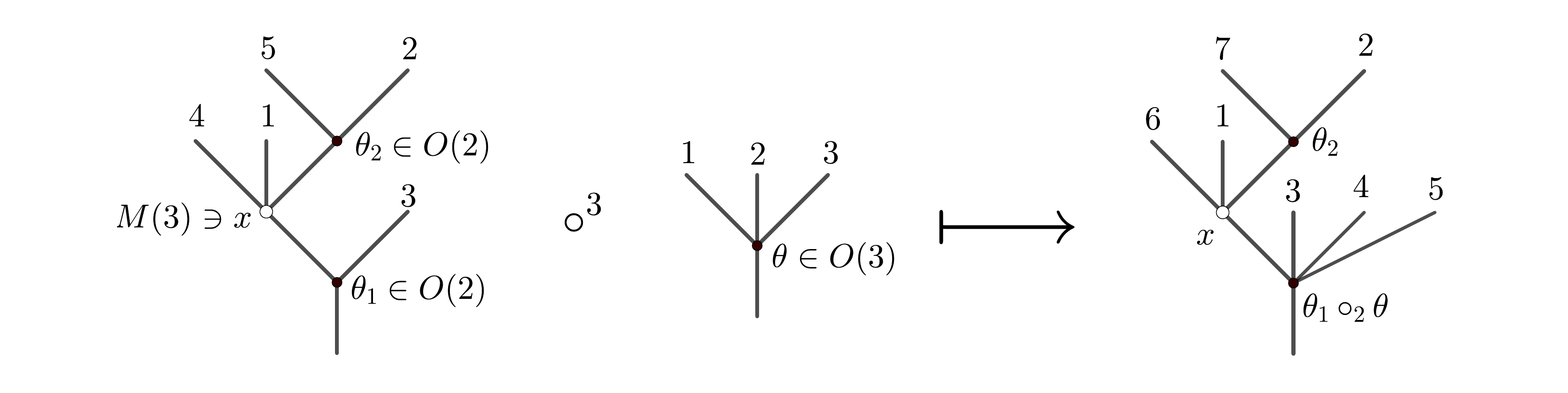}\vspace{-10pt}
\caption{Illustration of the right infinitesimal operation $\circ^{3}:\mathcal{IF}_{O}^\Sigma(M)(5)\times O(3)\rightarrow \mathcal{IF}_{O}^\Sigma(M)(7)$.}\label{G7}
\end{center}\vspace{-10pt}
\end{figure}

Similarly, the free $r$-truncated bimodule functor $\mathcal{IF}_{O}^{T_{r}\Sigma}$ is obtained from the formula (\ref{G4}) by taking the restriction of the coproduct to the reduced pearled trees having at most $r$ leaves and such that the pearl has at most $r$ incoming edges. The equivalence relation, the left and right infinitesimal operations are defined in the same way.  Finally, one has two functors:
$$
\mathcal{IF}_{O}^\Sigma:\Sigma Seq\longrightarrow\Sigma \Ibimod_{O}\hspace{15pt}\text{and} \hspace{15pt} \mathcal{IF}_{O}^{T_{r}\Sigma}:T_{r}\Sigma Seq\longrightarrow T_{r}\Sigma\Ibimod_{O}.
$$
\end{const}

\begin{thm}
One has the following adjunctions:
\begin{equation}\label{GG5}
\mathcal{IF}_{O}^\Sigma:\Sigma Seq\rightleftarrows \Sigma \Ibimod_{O}:\mathcal{U}^\Sigma\hspace{15pt}\text{and} \hspace{15pt} \mathcal{IF}_{O}^{T_{r}\Sigma}:T_{r}\Sigma Seq\rightleftarrows T_{r}\Sigma\Ibimod_{O}:\mathcal{U}^{T_{r}\Sigma}.
\end{equation}
\end{thm}

\begin{proof}
The proof is similar to the proof of Theorem \ref{THMadj}. For any $O$-Ibimodule $M'$ and for any morphism of $\Sigma$-sequences $f:M\rightarrow M'$, we can build a unique map of $O$-Ibimodules $\tilde{f}:\mathcal{IF}_{O}^\Sigma(M)\rightarrow M'$ by induction on the number of vertices, such that $f=\tilde{f}\circ i$. We refer the refer to \cite[Proposition 2.3]{Duc1} for more details.
\end{proof}

\subsubsection{Combinatorial description of the pushout}\label{BB7}

Let $O$ be a topological operad. Contrary to the bimodule case, the left infinitesimal operations are unary. As a consequence, the pushout in the category of $O$-Ibimodules coincides with the pushout in the underlying category of sequences. More precisely, for any pushout diagrams 
\begin{equation}\label{GG2}
\xymatrix{
A\ar[r]^{f_{1}} \ar[d]_{f_{2}} & C\\
B &
}
\hspace{30pt} 
\xymatrix@R=3pt{
& A_{r}\ar[r]^{f_{1}} \ar[dd]_{f_{2}} & C_{r} & \\
(respectively& & & )\\
& B_{r} & & 
}
\end{equation}
in the category of $O$-Ibimodules (respectively, $r$-truncated $O$-Ibimodules) we introduce the $\Sigma$-sequences
\begin{equation}\label{GG1}
D(n)= \left( B(n)\underset{A(n)}{\bigsqcup} C(n)\right)
\hspace{15pt}\text{ and } \hspace{15pt} D_{r}(n)= \left( B_{r}(n)\underset{A_{r}(n)}{\bigsqcup} C_{r}(n)\right).
\end{equation}
The above sequences inherit a (possibly truncated) infinitesimal bimodule structure over $O$ and one has the following statement:

\begin{pro}
One has the following identities in the category of (possibly truncated) $O$-Ibimodules:
$$
D=\underset{\hspace{40pt}\Sigma\Ibimod_{O}}{\mathrm{colim}}\big( B \longleftarrow A \longrightarrow C\big)\hspace{15pt}\text{and}\hspace{15pt} D_{r}=\underset{\hspace{40pt}T_{r}\Sigma\Ibimod_{O}}{\mathrm{colim}}\big( B_{r} \longleftarrow A_{r} \longrightarrow C_{r}\big).
$$
\end{pro}

\begin{proof}
Let $h_{1}:C\rightarrow D'$ and $h_{2}:B\rightarrow D'$ be two maps of $O$-Ibimodules such that $h_{1}\circ f_{1}=h_{2}\circ f_{2}$. Then there exists a unique map of $O$-Ibimodule $\delta:D\rightarrow D'$ given by 
$$
\delta(x)=\left\{
\begin{array}{cl}\vspace{5pt}
h_{1}(x) & \text{if } x\in C, \\ 
h_{2}(x) & \text{if } x\in B.
\end{array} 
\right.
$$
This map is well defined and proves that $D$ satisfies the universal property of the pushout.
\end{proof}

\subsection{Model category structure}\label{HH2}

By applying the transfer principle \ref{E3} to the adjunctions introduced in Section \ref{CC2}

\begin{equation}\label{AA8}
\mathcal{IF}_{O}^\Sigma:\Sigma Seq\rightleftarrows\Sigma \Ibimod_{O}:\mathcal{U}^\Sigma\hspace{15pt}\text{and} \hspace{15pt} \mathcal{IF}_{O}^{T_{r}\Sigma}:T_{r}\Sigma Seq\rightleftarrows T_{r}\Sigma\Ibimod_{O}:\mathcal{U}^{T_{r}\Sigma},
\end{equation}
we get the following statement:

\begin{thm}\label{G9}
Let $O$ be any topological operad.
The category of (truncated) infinitesimal bimodules $\Sigma \Ibimod_{O}$ (respectively, $T_r\Sigma \Ibimod_{O}$, $r\geq 0$) inherits a cofibrantly generated model category structure, %\todo{Condition sur O. Preuve: le foncteur oublie et l'argeument du petit objet.} 
  called the projective model category structure, in which all objects are fibrant.  The model structure in question makes the adjunctions (\ref{AA8}) into  Quillen adjunctions. More precisely, a (possibly truncated) Ibimodule map $f$  is a weak equivalence (respectively, a fibration) if and only if the induced map $\mathcal{U}^\Sigma(f)$ is a weak equivalence (respectively, a fibration) in the category of (possibly truncated) $\Sigma$-sequences. 
\end{thm}

\begin{proof}
Similar to the proof of Theorem \ref{ProjectBimod}. 
\end{proof}

\subsubsection{Relative left properness of the projective model category}\label{CP0}

\begin{thm}\label{th:properness_ibimod}
For any topological operad~$O$, the projective model category  $\Sigma\Ibimod_{O}$ is right proper. It is left proper provided $O$ is componentwise cofibrant.
\end{thm}

\begin{proof}
Since all the objects in $\Sigma\Ibimod_{O}$ are fibrant, this category is right proper (we refer the reader to the proof 
of Theorem~\ref{C0} for more details). Furthermore, as explained in Section ~\ref{BB7}, the pushout in the category of $O$-Ibimodules coincides with the pushout in the category of (non-$\Sigma$) sequences, which is obviously left proper as the category $Top$ is such.  By part (a) of the next theorem, 
 cofibrations in  $\Sigma\Ibimod_{O}$ are componentwise cofibrations provided $O$ is componentwise cofibrant.  We conclude that
 in this case, $\Sigma\Ibimod_{O}$ is also left proper. %relative to  objects with cofibrant components.
\end{proof}

\begin{thm}\label{th:proj_ibim_cof}
(a) If $O$ is a  componentwise cofibrant operad, then cofibrations  in the category $\Sigma\Ibimod_O$ are cofibrations componentwise. 
In particular, the class of  componentwise cofibrant objects  is closed under cofibrations and cofibrant objects have cofibrant components.

(b) If $O$ is a $\Sigma$-cofibrant operad, then cofibrations  in the category $\Sigma\Ibimod_O$ are $\Sigma$-cofibrant. In particular,  the class of $\Sigma$-cofibrant objects  is closed under cofibrations and cofibrant objects are $\Sigma$-cofibrant.
\end{thm}

\begin{proof}
It is enough to check that the generating cofibrations in $\Sigma\Ibimod_O$ are cofibrations in $Top$ componentwise, in the case (a), and
 are $\Sigma$-cofibrations, in the case (b). The first is a consequence of the pushout-product axiom, while the second is proved
in exactly the same way as the similar statement for operads \cite[Corollary 5.2]{BM}.  
\end{proof}

\subsubsection{Extension/restriction adjunction for the projective model category of infinitesimal bimodules} \label{EE8}

Let $\phi:O\rightarrow O'$ be a map of operads. Similarly to the category of algebras (see Theorem~\ref{E9}) and bimodules (see Section \ref{E8}), we show that the projective model categories of $O$-Ibimodules and $O'$-Ibimodules are Quillen equivalent under some conditions on the operads. For this purpose, we recall the construction of the \textit{restriction} functor $\phi^{\ast}$ and of the \textit{extension} functor $\phi_{!}$ in the context of infinitesimal bimodules:
$$
\phi_{!}:\Sigma\Ibimod_{O}\rightleftarrows \Sigma\Ibimod_{O'}:\phi^{\ast}.
$$

\noindent \textit{$\bullet$ The restriction functor.} The restriction functor $\phi^{\ast}$ sends an $O'$-Ibimodule $M$ to the $O$-Ibimodule $\phi^{\ast}(M)=\{\phi^{\ast}(M)(n)=M(n),\,n\geq 0\}$ in which the $O$-bimodule structure is defined using the $O'$-bimodule structure of $M$
as follows:
$$
\begin{array}{rcl}\vspace{3pt}
\circ^{i}: \phi^{\ast}(M)(n)\times O(m) & \longrightarrow & \phi^{\ast}(M)(n+m-1);\\ \vspace{10pt}
 x\,;\,\theta & \longmapsto & x\circ^{i}\phi(\theta), \\ \vspace{3pt}
\circ_{i}:O(n)\times \phi^{\ast}(M)(m) & \longrightarrow & \phi^{\ast}(M)(n+m-1); \\
\theta\,;\,x & \longmapsto & \phi(\theta)\circ_{i}x.
\end{array} 
$$

\noindent \textit{$\bullet$ The extension functor.} The extension functor $\phi_{!}$ is obtained as a quotient of the free $O'$-Ibimodule functor introduced in Section \ref{CC2}. More precisely, if $M$ is an $O$-Ibimodule, then the extension functor is given by the formula 
$$
\phi_{!}(M)(n)=\mathcal{IF}^{\Sigma}_{O'}(\mathcal{U}^{\Sigma}(M))(n)/\sim
$$ 
where the equivalence relation is generated by the axiom which consists in contracting inner edges having a vertex $v$  indexed by a point of the form $\phi_{1}(\theta)$ using the $O$ infinitesimal bimodule structure of $M$ as illustrated in the following picture:\vspace{5pt}

\hspace{-30pt}\includegraphics[scale=0.35]{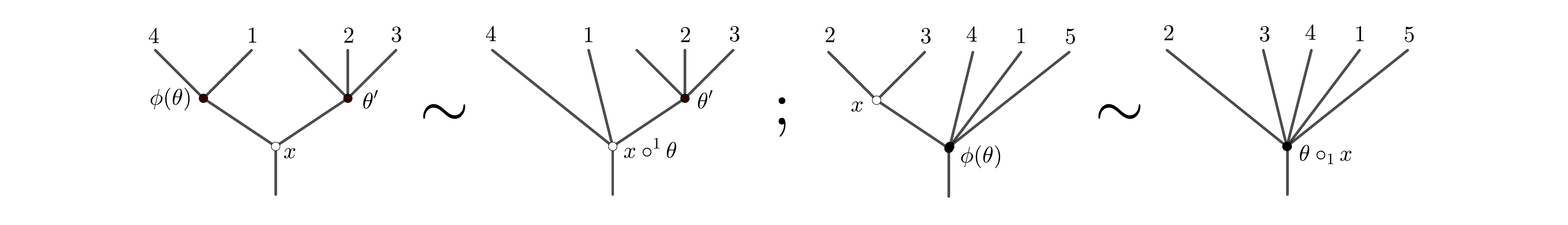}\vspace{5pt}

The $O'$-infinitesimal bimodule structure on the free object is compatible with the equivalence relation and provides an $O'$-Ibimodule structure on $\phi_{!}(M)$. Let us remark that, similarly to the bimodule case, the $O$-Ibimodule map $M\rightarrow \phi^{\ast}(\phi_{!}(M))$, sending a point $x\in M(n)$ to the $n$-corolla indexed by $x$, is not necessarily injective. 

\begin{thm}\label{G6_2}
Let $\phi:O\rightarrow O'$ be a weak equivalence between  operads with cofibrant 
components. The extension and restriction functors, as well as their truncated versions, give rise to Quillen equivalences: 
\begin{equation}\label{eq:bimod_ind_restr4}
\phi_!\colon\Sigma\Ibimod_{O}\rightleftarrows\Sigma\Ibimod_{O'}\colon\phi^*,
\end{equation}
\begin{equation}\label{eq:tr_bimod_ind_restr}
\phi_!\colon\TT_r \Sigma\Ibimod_{O}\rightleftarrows\TT_r \Sigma\Ibimod_{O'}\colon\phi^*.
\end{equation}
\end{thm}\vspace{5pt}

\begin{proof}
As explained in Section \ref{DD8}, the projective model category of infinitesimal bimodules over an operad is equivalent to the projective model category of algebras over a specific colored operad. We denote by $O_{+}$ and $O_{+}'$ the corresponding colored operads associated to $O$ and $O'$, respectively. One has that  the adjunction~\eqref{eq:bimod_ind_restr4} is induced  by the extension/restriction adjunction between the categories of algebras:
$$
\phi_! :\Sigma\Ibimod_{O}= Alg_{O_{+}} \rightleftarrows Alg_{O'_{+}} = \Sigma\Ibimod_{O'}:\phi^*.
$$
From the fact that $O$ and $O'$ have cofibrant components and from the explicit formula~\eqref{eq:O+} for the components of $O_+$ and $O'_+$, we obtain that the colored operads $O_{+}$ and $O'_{+}$ have cofibrant components. Since they are concentrated in arity one, they
are $\Sigma$-cofibrant. Applying \cite[Theorem~4.4]{BM}, we conclude
   that the extension/restriction adjunction between the categories of algebras is a Quillen equivalence. 
   
   The truncated case is done similarly by restricting the operads $O_+$ and $O'_+$  to their subsets of colors.
\end{proof}%\vspace{5pt}

\section{The Reedy model category of $O$ infinitesimal bimodules}

Let $O$ be a reduced operad. From now on, we denote by $\Lambda \Ibimod_{O}$ and $T_{r}\, \Lambda \Ibimod_{O}$ the categories of $O$-Ibimodules and truncated  $O$-Ibimodules, respectively, equipped with the Reedy model category structures. 
Note that as categories, $\Lambda\Ibimod_O=\Sigma\Ibimod_O$. Only the model structure is different.
Similarly to the case of reduced operads and bimodules,
this structure is transferred from the categories $\Lambda Seq$ and $T_{r}\,\Lambda Seq $ along the adjunctions
\begin{equation}\label{FF7}
\begin{array}{rcl}\vspace{8pt}
\mathcal{IF}_{O}^{\Lambda}:\Lambda Seq & \rightleftarrows & \Lambda \Ibimod_{O}:\mathcal{U}^{\Lambda}, \\ 
\mathcal{IF}_{O}^{T_{r}\Lambda}:T_{r}\Lambda Seq& \rightleftarrows & T_{r}\Lambda \Ibimod_{O}:\mathcal{U}^{\Lambda},
\end{array} 
\end{equation}
where both free functors are obtained from the functors $\mathcal{IF}_{O}^\Sigma$ and $\mathcal{IF}_{O}^{T_{r}\Sigma}$ by taking the restriction of the coproduct (\ref{GG4}) to the reduced pearled trees without univalent vertices other than the pearl. In other words, one has
$$
\mathcal{IF}_{O}^{\Lambda}(M)\coloneqq \mathcal{IF}_{O_{>0}}^{\Sigma}(M)
, \hspace{15pt} \text{and}\hspace{15pt} \mathcal{IF}_{O}^{T_{r}\Lambda}(M)\coloneqq \mathcal{IF}_{O_{>0}}^{T_{r}\Sigma}(M).
$$
By construction, the above $\Sigma$-sequences are equipped with a (truncated) infinitesimal bimodule structure over $O_{>0}$. We can extend this structure in order to get a (truncated) $O$ infinitesimal bimodule structure using the operadic structure of~$O$ and the $\Lambda$ structure of $M$.\vspace{10pt}

\hspace{-55pt}\includegraphics[scale=0.32]{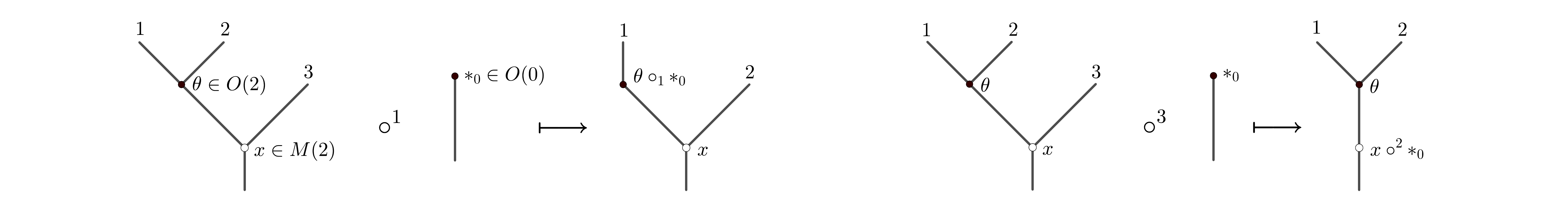}\vspace{-10pt}
 
\begin{figure}[!h]
\caption{Illustration of the right action by $\ast_{0}$.}\vspace{-5pt}
\end{figure}

\begin{thm}\label{ZZ5}
Let $O$ be a reduced well-pointed operad.
The categories $\Lambda \Ibimod_{O}$ and $T_{r}\Lambda \Ibimod_{O}$, with $r\geq 0$, admit cofibrantly generated model category structures, %\todo{Condition sur O. Preuve: le foncteur oublie et l'argeument du petit objet.} 
  called Reedy model category structures, transferred from $\Lambda Seq$ and $T_{r}\Lambda Seq$, respectively, along the adjunctions $(\ref{FF7})$. In particular, these model category structures make the pairs of functors $(\ref{FF7})$ into Quillen adjunctions.
\end{thm}

\begin{proof}
The proof is similar to the proof of Theorem \ref{Z5}. The path object is given by the same formula \eqref{PathObject} and the functorial fibrant coresolution (for which we need the assumption $Q$ is well-pointed) is defined in  Subsection~\ref{ZZ2}. 
\end{proof}

\subsection{Properties of the Reedy model category of infinitesimal bimodules}

This subsection is divided into two parts. The first one is devoted to the construction of an explicit fibrant coresolution functor for infinitesimal bimodules. In the second part, we  characterize (acyclic) cofibrations in the Reedy model category of Ibimodules, we prove properness 
and we study extension/restriction adjunctions.
% The last part consists in extending the properties introduced in Section \ref{HH2} to the Reedy model category. 

\subsubsection{Fibrant replacement functor for infinitesimal bimodules}\label{ZZ2}

Let $O$ be a reduced operad. In this subsection we produce a construction of a Reedy fibrant coresolution for infinitesimal bimodules if the operad $O$ is well-pointed. More precisely, we show that the coresolution introduced in Section \ref{Z2} gives rise to a coresolution for infinitesimal bimodules too. In other words, in both cases, bimodules and infinitesimal bimodules, we only need the right module part of the structures in order to get Reedy fibrant replacements. In the following, we overuse the notation introduced in  Section~\ref{Z2}. Given an infinitesimal bimodule $M$ over $O$, we consider the space $M^{f}(n)$  from~\eqref{eq:Mf} taking $Q=O$. The family  $M^{f}=\{M^{f}(n),\,n\geq 0\}$ admits a $\Sigma$-structure and a right module structure from  operations %\eqref{Finish1} and
 \eqref{Finish2}.\vspace{5pt}

In order to introduce the left infinitesimal operations, we need the following notation. For a tree $T\in \mathbb{P}[n]$, we denote by $T_{1},\ldots, T_{n}$ the sub-trees grafting to the root of $T$ according to the planar order. In particular, one has the identity $T=c_{n}( T_{1},\ldots, T_{n})$ where $c_{n}$ is the $n$-corolla. The number of leaves of each tree $T_{j}$, with $1\leq j \leq n$, is denoted by $n_{j}[T]$. For any tree $T\in \mathbb{P}[n]$ and $1\leq j \leq n$, we also consider the operations
$$
\begin{array}{llcl}\vspace{5pt}
\beta_{j}: & D(T) & \longrightarrow & O(n_{j}[T]); \\ 
& \{\theta_{v}\}_{v\in V(T)\setminus \{r\}} & \longmapsto & \eta(\{\theta_{v}\}_{v\in V(T_{j})}),
\end{array} 
$$
given by composing the points of $O$ indexing the vertices of the sub-tree $T_{j}$. \vspace{5pt}

Finally, by using the operation $\Gamma_{k}^{m}$ introduced in Section \ref{Z2}, one can define the left infinitesimal action
\begin{equation*}
\begin{array}{rcl}\vspace{5pt}
\circ_{i}:O(n)\times M^{f}(m) & \longrightarrow & M^{f}(n+m-1); \\ 
\theta \,\,,\,\,\{f_{T}\}_{T\in \mathbb{P}[n]} & \longmapsto & \{(\theta\circ_{i}f)_{T}\}_{T\in \mathbb{P}[n+m-1]},
\end{array} 
\end{equation*}
where $(\theta\circ_{i}f)_{T}$ is the composite map:
$$
\xymatrix{
H(T)\times D(T) \ar[r]^{(\theta\circ_{i}f)_{T}} \ar[d] & M(|T|) \\
\left(\underset{\substack{1\leq j \leq n+m-1 \\ j\notin \{i,\ldots ,i+m-1\}}}{\displaystyle\prod} O(n_{j}[T])\right) \times  \left( H(\Gamma_{i-1}^{m}(T)) \times D(\Gamma_{i-1}^{m}(T)) \right)    \ar[r] & O\left( \underset{\substack{1\leq j \leq n+m-1 \\ j\notin \{i,\ldots ,i+m-1\}}}{\displaystyle\sum} n_{j}[T]  \right) \times M(|\Gamma_{i-1}^{m}(T)|) \ar[u]
}
$$
If we denote by $I$ the set $\{1,\ldots, i-1\}\cup \{i+m,\ldots n+m-1\}$, then the left vertical map is given by the product $\prod_{j\in I} \beta_{j}$ to get the first factor and the operation induced by $\Gamma_{i-1}^{m}$, removing the incoming edges of the root of $T$ corresponding to the set $I$, in order to get the second factor.  The lower horizontal map is given by the map $f_{\Gamma_{i-1}^{m}(T)}$ on the second factor and the following map 
$$
\begin{array}{rcl}\vspace{5pt}
\theta(-,\cdots,id,\cdots,-):\left(\underset{j\in \{1,\ldots, i-1\} }{\displaystyle\prod} O(n_{j}[T])\right) \times \left(\underset{j\in \{i+m,\ldots n+m-1\} }{\displaystyle\prod} O(n_{j}[T])\right) & \longrightarrow &  O\left( \underset{\substack{1\leq j \leq n+m-1 \\ j\notin \{i,\ldots ,i+m-1\}}}{\displaystyle\sum} n_{j}[T]  \right); \\ 
 (\theta_{1},\ldots,\theta_{i-1});(\theta_{1}',\ldots,\theta_{n-i-1}')& \longmapsto & \theta(\theta_{1},\ldots,\theta_{i-1},id,\theta_{1}',\ldots,\theta_{n-i-1}'),\end{array} 
$$
using the operadic structure of $O$, on the first factor. Finally, the right vertical map is obtained using the left infinitesimal operation $\circ_{\ell}$, with $\ell=\sum_{1\leq j\leq i-1}n_{j}[T]+1$. The fact that the compatibility relations between the left and the right infinitesimal operations (defined at the beginning of Section \ref{SectProjIbimod}) are satisfied is a consequence of the following observations:
\begin{itemize}[leftmargin=10pt]
\item[$\blacktriangleright$] \textit{Ramified compatibility between the left and right operations}: for $T\in \mathbb{P}[n+m+\ell-2]$, $i\in \{1,\ldots,n-1\}$ and $j\in \{i+1,\ldots, n\}$, one has
$$
\Gamma^{m}_{i-1}(T)=\Gamma_{i-1}^{m}\big( \delta_{j+m-1;\ell}(T)\big).
$$

\item[$\blacktriangleright$] \textit{Linear compatibility between the left and right operations}: for $T\in \mathbb{P}[n+m+\ell-2]$, $i\in \{1,\ldots,n\}$ and $k\in \{1,\ldots, m\}$, one has
$$
\Gamma^{m}_{i-1}\big( \delta_{k+i-1;\ell}(T)\big)=  \delta_{k;\ell}\big(\Gamma^{m+\ell-1}_{i-1}(T)\big).
$$
\end{itemize}

\begin{pro}
The map $\eta:M\rightarrow M^{f}$ is a weak equivalence of $O$-Ibimodules. Furthermore, if the operad~$O$ is well-pointed, then the $O$-Ibimodule $M^{f}$ is Reedy fibrant.
\end{pro}

\begin{proof}
The reader can easily check that $\eta$ is a map of $O$ infinitesimal bimodules. The other statements are consequences Propositions~\ref{ProFibrantEq} and~\ref{MF}.
\end{proof}

\begin{rmk}\label{r:trunc_ibim_Reedy} 
The same strategy can be used in order to get a fibrant replacement functor for  $r$-truncated reduced $O$-Ibimodules.
The fibrant replacement should be defined as a subspace of the product with an additional restriction
$|T|\leq r$. The  constraints are the same.
\end{rmk}

\subsubsection{Characterization of cofibrations /  left properness / extension-restriction adjunction}\label{Fin2}

In this section, we show that the properties related to the Reedy model category of reduced bimodules introduced in Section~\ref{SectPropReedyBimod} admit counterparts in the context of infinitesimal bimodules. It means that we are able to give a characterization of Reedy cofibrations and we prove that $\Lambda\Ibimod_{O}$ is left proper relative to componentwise cofibrant objects. We also prove that the extension/restriction adjunction gives rise to a Quillen equivalence between Reedy model categories of Ibimodules under some conditions on the operads.
%
% Finally we compare the 
%Projective and Reedy model category structures and we give a construction of a Reedy cofibrant replacement in $\Lambda\Ibimod_O$ provided
%$O$ is componentwise cofibrant.

\begin{thm}\label{CC1}
Let $O$ be a reduced  well-pointed operad. %\todo{Condition sur l'opérade O.} 
 A morphism $\phi\colon M\rightarrow N$ in the category of (possibly truncated)  $O$-Ibimodules is a Reedy cofibration if and only if $\phi$ is a cofibration in the projective model category of (possibly truncated) $O_{>0}$-Ibimodules.
\end{thm}

\begin{proof}[Idea of the proof]
The strategy used for the proof of Theorem \ref{C1} works with no change in the context of infinitesimal bimodules. Again, we introduce an adjunction $ar_{s}:\Lambda Ibimod_{O}\rightleftarrows \Lambda Ibimod_{O}:cosk_{s}$ where $ar_{s}$ and  $cosk_{s}$ are the arity filtration and the coskeleton functors, respectively. More precisely, if $L_{s}$ and $R_{s}$ denote the left adjoint and the right adjoint, respectively, to the truncation functor $T_{s}:\Lambda Ibimod_{O}\rightarrow T_{s}\Lambda Ibimod_{O}$, then one has the following identities:
$$
ar_{s}=L_{s}\circ T_{s} \hspace{20pt}\text{and} \hspace{20pt} cosk_{s}=R_{s}\circ T_{s}.
$$

In particular, the coskeleton functor is given by the formula \eqref{DefCoskeleton} and inherits an infinitesimal bimodule structure over $O$. The $\Lambda$-structure and the right infinitesimal operations are given by \eqref{CoskeletonLambda} and \eqref{CoskeletonStruct}, respectively. In order to define the left infinitesimal operations, we recall the following notation. Let $n,m>0$, $\ell\in \{1,\ldots,n\}$ and $h\in \Lambda_{+}([i]\,;\,[n+m-1])$. If we denote by $\ell_{1}\in \Lambda_{+}([m]\,;\,[n+m-1])$ and $\ell_{2}\in \Lambda_{+}([n]\,;\,[n+m-1])$ the morphisms
$$
\begin{array}{rclcrcl}\vspace{3pt}
\ell_{1}:[m] & \longrightarrow & [n+m-1]; & \hspace{10pt}\text{and}\hspace{10pt} & \ell_{2}:[n] & \longrightarrow & [n+m-1]; \\ 
\alpha & \longmapsto & \alpha +\ell, &  & \alpha & \longmapsto & \left\{
 \begin{array}{ll}
 \alpha & \text{if } \alpha\leq \ell, \\ 
 \alpha+m & \text{if } \alpha > \ell,
 \end{array} 
 \right.
\end{array} 
$$

\noindent then there exist unique morphisms $h_{1}$ and $h_{2}$ making the following diagrams commute:
$$
\xymatrix{
[i]\ar[r]^{h} & [n+m-1] \\
[|\mathrm{Im}(\ell_{1})\cap Im(h)|]\ar[u] \ar[r]_{\hspace{30pt}h_{1}} & [m]\ar[u]_{l_{1}} 
}\hspace{30pt}
\xymatrix{
[i]\ar[r]^{h} & [n+m-1] \\
[|\mathrm{Im}(\ell_{2}\setminus \{\ell\})\cap Im(h)|]\ar[u] \ar[r]_{\hspace{40pt}h_{2}} & [n]\ar[u]_{l_{2}} 
}
$$
Finally, if we denote by $\overline{\ell}=\ell-|\{\alpha\in [i]\,|\, h(\alpha)<\ell\}|$, then the left infinitesimal operations are given by
$$
\begin{array}{rcl}\vspace{5pt}
\circ_{i}:O(n)\times cosk_{s}(M)(m) & \longrightarrow & cosk_{s}(M)(n+m-1);  \\ 
\theta\,;\,\{x_{u}\}_{u\in \Lambda_{+}([i]\,;\,[m])}^{0\leq i\leq s} & \longmapsto & \{h_{2}^{\ast}(\theta)\circ_{\overline{\ell}}x_{h_{1}}\}_{h\in \Lambda_{+}([i]\,;\,[n+m-1])}^{0\leq i\leq s}.
\end{array} 
$$

The rest of the proof is the same as the proof of Theorem \ref{C1}. It consists in using the adjunction $(ar_{s},cosk_{s})$ in order to define by induction a solution to the lifting problem.   
\end{proof}

\begin{thm}\label{th:properness_ibimod2}
For any reduced well-pointed operad~$O$, the Reedy model category $\Lambda\Ibimod_O$ is right proper.
 It  is left proper provided~$O$ is componentwise cofibrant. In the latter case, cofibrations are componentwise cofibrations, and, as a consequence, the class of componentwise cofibrant objects is closed under cofibrations and 
 cofibrant Ibimodules are componentwise cofibrant. If $O$ is $\Sigma$-cofibrant, then the cofibrations are $\Sigma$-cofibrations, the class of $\Sigma$-cofibrant objects is closed under cofibrations,
 and cofibrant objects are $\Sigma$-cofibrant.  
\end{thm}

\begin{proof}
Right properness immediately follows from Theorem~\ref{th:properness_ibimod}. Indeed, $\Lambda\Ibimod_O=\Sigma\Ibimod_O$ as categories.
Therefore, they have the same  pullbacks. Moreover, a Reedy fibration is always a projective fibration. For left properness  and the properties of cofibrations, we use the characterization
of cofibrations Theorem~\ref{CC1} together with the analogous properties of $\Sigma\Ibimod_{O_{>0}}$ established in Theorems~\ref{th:properness_ibimod} and~\ref{th:proj_ibim_cof}.
% Moreover, they have the same class of weak equivalences.
%The proof is similar to the proof of Theorem \ref{ThmProperness}. Any pushout diagram $j\circ f=i\circ g$ (see Diagram \eqref{H3}), with $f$ a weak equivalence and $g$ a Reedy cofibration in the Reedy model category $\Lambda \Ibimod_{O}$, induces a pushout diagram $j\circ f=i\circ g$ in the projective model category $\Sigma \Ibimod_{O_{>0}}$. Due to Theorem \ref{CC1}, $g$ is still a cofibration in the projective model category and $f$ is a weak equivalence. The statement follows from the fact that  $\Sigma \Ibimod_{O_{>0}}$ is left proper. 
\end{proof}

Let $\phi:O\rightarrow O'$ be a weak equivalence of reduced operads. Similarly to Section \ref{EE8}, we show that the Reedy model categories of $O$-Ibimodules and $O'$-Ibimodules are Quillen equivalent. By abuse of notation, we denote by $\phi^{\ast}$ and $\phi_{!}$ the restriction functor and the extension functor, respectively, between the Reedy model categories:
$$
\phi_{!}:\Lambda \Ibimod_{O}\rightleftarrows \Lambda \Ibimod_{O'}:\phi^{\ast}.
$$
In the same way as in Section \ref{EE8}, for any $M\in \Lambda \Ibimod_{O}$ and $M'\in \Lambda \Ibimod_{O'}$, one has
$$
\begin{array}{rcl}\vspace{7pt}
\phi_{!}(M) & = & \{\phi_{!}(M)(n)=\mathcal{F}_{O}^{\Lambda}(\mathcal{U}^\Lambda(M))(n)/\sim,\,\,\,n\geq 0 \}, \\ 
\phi^{\ast}(M') & = & \{\phi^{\ast}(M')(n)=M'(n),\,\,n\geq 0\}. 
\end{array} 
$$ 

\begin{thm}\label{eq:ind_rest_reduced_ibimod}
Let $\phi\colon O\rightarrow O'$ be a weak equivalence between reduced componentwise cofibtant operads. One has Quillen equivalences
\begin{equation}\label{eq:bimod_ind_restr5}
\phi_!\colon\Lambda\Ibimod_{O}\rightleftarrows\Lambda\Ibimod_{O'}\colon\phi^*,
\end{equation}
\begin{equation}\label{eq:tr_bimod_ind_restr}
\phi_!\colon\TT_r\Lambda\Ibimod_{O}\rightleftarrows\TT_r\Lambda\Ibimod_{O'}\colon\phi^*.
\end{equation}
\end{thm}

\begin{proof}
%The proof is similar to the proof of Theorem \ref{ThmExt/rest}. 
Since the restriction functor creates weak equivalences, one has to check that, for any Reedy cofibrant object $M$ in $\Lambda \Ibimod_{O}$, the adjunction unit 
\begin{equation*}
M\longrightarrow \phi^{\ast}(\phi_{!}(M))
\end{equation*}
is a weak equivalence. Due to the characterization of Reedy cofibrations, $M$ is also cofibrant in the projective model category of $O_{>0}$-Ibimodules. Since $\phi_{>0}:O_{>0}\rightarrow O'_{>0}$ is still a weak equivalence between componentwise  cofibrant operads, by Theorem~\ref{G6}, the pair of functors $((\phi_{>0})_{!}\,;\, (\phi_{>0})^{\ast})$ gives rise to a Quillen equivalence and the map $M(n)\rightarrow (\phi_{>0})^{\ast}\big((\phi_{>0})_{!}(M)\big)(n)$ is a weak equivalence. The statement is a consequence of the identity 
$$
\phi^{\ast}(\phi_{!}(M))=(\phi_{>0})^{\ast}\big( (\phi_{>0})_{!}(M)\big).
$$
\end{proof}

\subsection{The connection between the model category structures on infinitesimal bimodules}\label{Fin4}

Similarly to the operadic case in \cite{FTW}, we build a Quillen adjunction between the projective and the Reedy model categories of infinitesimal bimodules over a reduced operad~$O$. Furthermore, if $M$ and $N$ are two infinitesimal bimodules, then we show that there is a weak equivalence between the derived mapping spaces:
$$
\Sigma\Ibimod_{O}^{h}(M\,;\,N)\simeq \Lambda \Ibimod_{O}^{h}(M\,;\,N).
$$

For completeness of exposition, at the end of the subsection, we explain how to adapt the Boardman-Vogt resolution (well known for operads, see \cite{BM2}) to the context of infinitesimal bimodules. We refer the reader to~\cite{DT} where this construction was defined by the first and third authors. 
Using this construction we define a functorial cofibrant replacement in the categories $\Sigma\Ibimod_O$ and $\Lambda\Ibimod_O$
provided $O$ is componentwise cofibrant.

\subsubsection{Quillen adjunction between the model category structures}\label{E7_2}

Let $O$ be a reduced operad. The projective and the Reedy model categories of infinitesimal bimodules over $O$ have the same set of weak equivalence and induce the same homotopy category. Consequently, one has the following statement about the adjunctions
\begin{equation}\label{eq:unitar_bim2}
\begin{array}{rcl}\vspace{5pt}
id\colon\Sigma\Ibimod_{O} & \rightleftarrows & \Lambda\Ibimod_{O}\colon id, \\ 
id\colon\TT_r \Sigma \Ibimod_{O} & \rightleftarrows & \TT_r\Lambda\Ibimod_{O}\colon id.
\end{array} 
\end{equation}

\begin{thm}\label{th:S-L_equiv_Ibim}
For any well-pointed reduced operad~$O$, the pairs of functors \eqref{eq:unitar_bim2} form Quillen equivalences. Furthermore,  for any pair $M,\, N\in\Lambda\Ibimod_{O}$, one has
\begin{equation}\label{eq:der_bim_map2}
\Sigma\Ibimod_{O}^h( M,N)\simeq \Lambda\Ibimod_{O}^h(M,N).
\end{equation}
Moreover, if $M,\, N\in\TT_r\Lambda\Ibimod_{O}$, with $r\geq 0$, then one has
\begin{equation}\label{eq:der_tr_bim_map3}
\TT_r\Sigma\Ibimod_{O}^h( M, N)\simeq \TT_r\Lambda\Ibimod_{O}^h(M,N).
\end{equation}
\end{thm}

\begin{proof}
The proof is similar to that of Theorem~\ref{th:S-L_equiv_bim}.
% that the pairs of functors \eqref{eq:unitar_bim2} form Quillen adjunctions is similar to the proof of Proposition~\ref{ProCompMC} in which we check that the right adjoint functors preserve fibrations and acyclic fibrations. The identifications \eqref{eq:der_bim_map2} and \eqref{eq:der_tr_bim_map2} are induced by taking a Projective cofibrant replacement of $M$ and a Reedy fibrant replacement of~$N$
%and using the fact that any projective cofibration is a Reedy cofibration and any Reedy fibration is a projective fibration. 
\end{proof}

\subsubsection{Cofibrant resolution in the projective/Reedy model category}\label{5.2.1}

Let $O$ be an operad not necessarily reduced. From an $O$-Ibimodule $M$, we build an $O$-Ibimodule $\mathcal{I}b_{O}(M)$. The points of $\mathcal{I}b_{O}(M)(n)$, $n\geq 0$, are equivalence classes $[T\,;\, \{t_{v}\}\,;\, x_{p}\,;\,\{\theta_{v}\}]$, where $T\in p\mathbb{P}_n$ (see Section \ref{CC2}) is a pearled tree, $x_{p}$ is a point in $M$ labelling the pearl and $\{\theta_{v}\}_{v\in V(T)\setminus \{p\}}$ is a family of points in $O$ labelling the vertices  other than the pearl. Furthermore, $\{t_{v}\}_{v\in V(T)\setminus V^{p}(T)}$ is a family of real numbers in the interval $[0\,,\,1]$ indexing the vertices which are not pearls. According to the orientation toward the pearl, if $e$ is an inner edge, then $t_{s(e)}\geq t_{t(e)}$. In other words, the closer to the pearl is a vertex, the  smaller is the corresponding number. The space $\mathcal{I}b_{O}(M)(n)$ is a quotient of the subspace of 
\begin{equation}\label{eq:union_stree2}
%\left.
\underset{T\in p\mathbb{P}_n}{\coprod}\,\,\,M(|p|)\,\,\times\underset{v\in V(T)\setminus \{p\}}{\prod}\,\big[\,O(|v|)\times [0\,,\,1]\big]
\end{equation} 
determined by the restrictions on the families $\{t_{v}\}$. The equivalence relation is generated by the unit condition $(i)$ and the compatibility with the symmetric group relation $(ii)$ of Construction \ref{FF9} as well as the following~conditions:

\begin{itemize}%[topsep=0pt, leftmargin=18pt]

\item[$iii)$] If two consecutive vertices, connected by an edge $e$, are indexed by the same real number $t\in [0\,,\,1]$, then~$e$ is contracted using the operadic structure of~$O$. The vertex produced by this edge contraction is indexed by the real number~$t$.\vspace{5pt}

\hspace{-57pt}\includegraphics[scale=0.32]{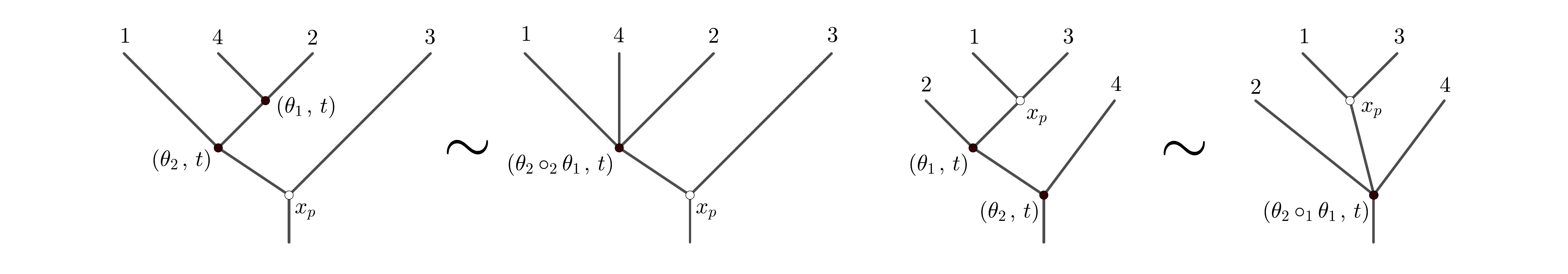}\vspace{5pt}

\item[$iv)$]  If a vertex connected to the pearl is indexed by $0$, then we contract the inner edge connecting them using the infinitesimal bimodule structure of $M$. In that case the new vertex, produced by the contraction, becomes the pearl.\vspace{5pt}

\hspace{-66pt}\includegraphics[scale=0.32]{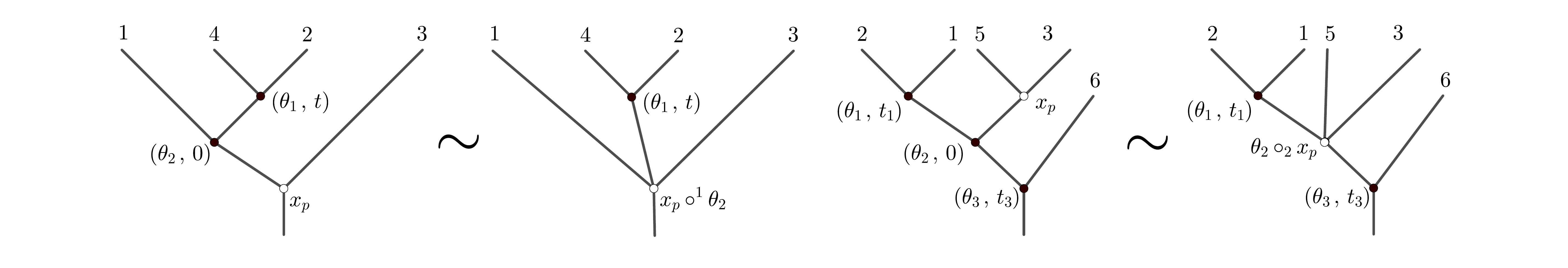}\vspace{5pt}

\end{itemize}

 Let us describe the $O$-Ibimodule structure. Let $\theta\in O(n)$ and $[T\,;\, \{t_{v}\}\,;\, x_{p}\,;\,\{\theta_{v}\}]$ be a point in $\mathcal{I}b_{O}(M)(m)$. The right infinitesimal operation $[T\,;\, \{t_{v}\}\,;\, x_{p}\,;\,\{\theta_{v}\}]\circ^{i}\theta$ consists in grafting the $n$-corolla labelled by $\theta$ to the $i$-th incoming edge of $T$ and indexing the new vertex by $1$. Similarly, the left infinitesimal operation $\theta\circ_{i}[T\,;\, \{t_{v}\}\,;\, x_{p}\,;\,\{\theta_{v}\}]$ consists in grafting the pearled tree $T$ to the $i$-th incoming edge of $n$-corolla labelled by $\theta$ and indexing the new vertex by $1$.%\vspace{-10pt}

\begin{figure}[!h]
\begin{center}
\includegraphics[scale=0.3]{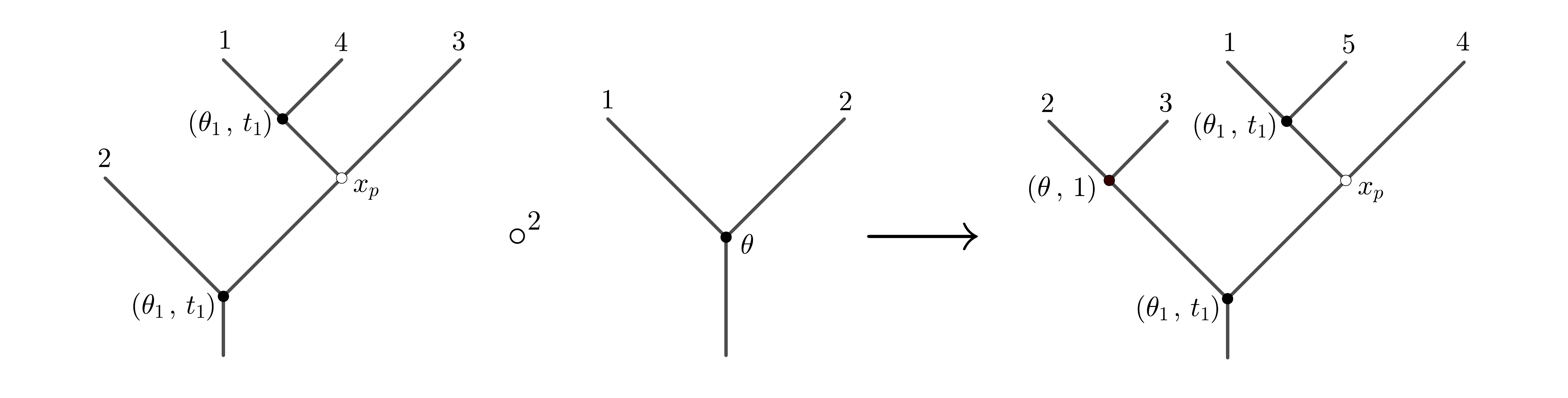}\vspace{-10pt}
\caption{Illustration of the right infinitesimal operation.}
\end{center}%\vspace{-10pt}
\end{figure}

One has an obvious inclusion of $\Sigma$-sequences $\iota \colon M\to \mathcal{I}b_{O}(M)$, where each element
$x\in M(n)$ is sent to an $n$-corolla labelled by $x$, whose only vertex is a pearl.  Furthermore, the following map:
\begin{equation}\label{DD3}
\mu:\mathcal{I}b_{O}(M)\rightarrow M\,\,;\,\, [T\,;\, \{t_{v}\}\,;\, x_{p}\,;\,\{\theta_{v}\}]\mapsto [T\,;\, \{0\}\,;\, x_{p}\,;\,\{\theta_{v}\}],
\end{equation}
is defined by sending the real numbers indexing the vertices other than the pearl to $0$. The so obtained element is identified to the pearled corolla labelled by a point in $M$.  It is easy to see that $\mu$ is an $O$-Ibimodule map.

In order to get resolutions for truncated infinitesimal bimodules, one considers a filtration in $\mathcal{I}b_{O}(M)$ according to  the number of {\it geometrical inputs} which is the number of leaves plus the number of univalent vertices other than the pearl. A point in $\mathcal{I}b_{O}(M)$ is said to be \textit{prime} if the real numbers indexing the vertices are strictly smaller than $1$. Otherwise, a point is said to be \textit{composite} and can be associated to a \textit{prime component} as shown in Figure~\ref{BB1}. More precisely, the prime component is obtained by removing the vertices indexed by $1$. \vspace{-10pt}

\begin{figure}[!h]
\begin{center}
\includegraphics[scale=0.4]{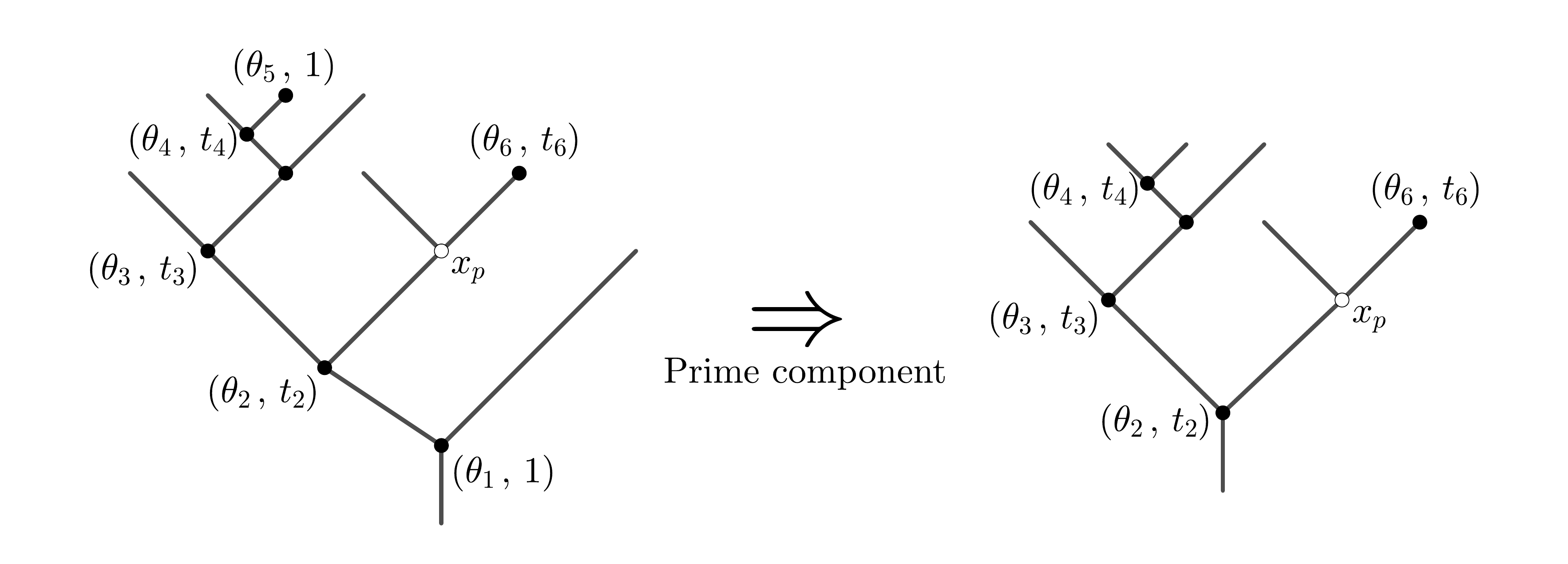}\vspace{-10pt}
\caption{A composite point and its prime components.}\label{BB1}
\end{center}
\end{figure}

A prime point is in the $r$-th filtration layer $\mathcal{I}b_{O}(M)_{r}$ if the number of its geometrical inputs is at most~$r$. Similarly, a composite point is in the $r$-th filtration layer if its prime component is in $\mathcal{I}b_{O}(M)_{r}$. For instance, the composite point in Figure~\ref{BB1} is in the filtration layer $\mathcal{I}b_{O}(M)_{6}$. For each $r$, $\mathcal{I}b_{O}(M)_{r}$ is an $O$-Ibimodule and one has the following filtration of $\mathcal{I}b_{O}(M)$:
\begin{equation}\label{BB3}
\xymatrix{
 \mathcal{I}b_{O}(M)_{0} \ar[r] & \mathcal{I}b_{O}(M)_{1}\ar[r] & \cdots \ar[r] & \mathcal{I}b_{O}(M)_{r-1} \ar[r] & \mathcal{I}b_{O}(M)_{r} \ar[r] & \cdots \ar[r] &  \mathcal{I}b_{O}(M).
}
\end{equation}

\begin{thm}[Theorem 3.10 in \cite{DT}]\label{th:BV_proj_Ibimod}
Assume that  $O$ is a $\Sigma$-cofibrant operad, and $M$ is a $\Sigma$-cofibrant 
$O$-Ibimodule. Then the objects $\mathcal{I}b_{O}(M)$ and $\TT_{r}\mathcal{I}b_{O}(M)_{r}$ are cofibrant replacements of $M$ and $\TT_{r}M$ in the categories  $\Sigma\Ibimod_{O}$ and $\TT_{r}\Sigma\Ibimod_{O}$, respectively. In particular, the maps $\mu$ and $\TT_{r}\mu|_{\TT_{r}\mathcal{I}b_{O}(M)_{r}}$ are weak equivalences.
\end{thm}

Now we slightly change the above construction in order to produce Reedy cofibrant replacements for $O$-Ibimodules when $O$ is a reduced operad. Let $M$ be an infinitesimal bimodule over $O$. As a $\Sigma$-sequence, we set
\[
\mathcal{I}b_{O}^\Lambda(M):=\mathcal{I}b_{O_{>0}}(M)
\]
The superscript $\Lambda$
is to emphasize that we get a cofibrant replacement in the Reedy model category structure. The right and left action by the positive arity components is defined as it is on $\mathcal{I}b_{O_{>0}}(M)$. The right action by $*_0\in O(0)$ is defined in the obvious way as the right  action by $*_0$ on $a$ in the vertex $(a,t)$ connected to the leaf labelled by $i$ as illustrated in the Figure \ref{A1}.\vspace{10pt}

Note that since the arity zero component of $O_{>0}$ is empty, in the union~\eqref{eq:union_stree2} 
we can consider only trees whose all non-pearl vertices have arities $\geq 1$. We denote this set by 
$p\mathbb{P}^{\geq 1}_n$. In other words, the space $\mathcal{I}b_{O}^\Lambda(M)$ can be obtained as the restriction of the coproduct \eqref{eq:union_stree2} to this set.

\begin{pro}{\cite[Proposition 3.12]{DT}}\label{p:BV_reedy_Ibimod}
Assume that $O$ is a reduced  $\Sigma$-cofibrant operad, and $M$ is a $\Sigma$-cofibrant $O$-Ibimodule. Then the objects $\mathcal{I}b_{O}^\Lambda(M)$ and $\TT_{r}\mathcal{I}b_{O}^\Lambda(M)$ are cofibrant
 replacements of $M$ and $\TT_{r}M$ in the categories $\Lambda\Ibimod_{O}$ and $\TT_{r}\Lambda\Ibimod_{O}$, respectively. In particular, the maps $\mu$ and $\TT_{r}\mu$ are weak equivalences.
\end{pro}

\begin{proof}
The map $\mu:\mathcal{I}b^{\Lambda}_{O}(M)\rightarrow M$, which changes the assignment of the real numbers indexing the vertices to $0$, is a homotopy equivalence. More precisely, it is a deformation retract in the category of $\Sigma$-sequences in which the homotopy consists in bringing the real numbers to $0$. Furthermore, as a consequence of Theorem \ref{th:BV_proj_Ibimod}, $\mathcal{I}b^{\Lambda}_{O}(M)=\mathcal{I}b_{O_{>0}}(M)$ is cofibrant in the projective model category of $O_{>0}$-Ibimodules. Due to Theorem~\ref{CC1}, $\mathcal{I}b_{O}^{\Lambda}(M)$ is also Reedy cofibrant and it gives rise to a cofibrant resolution of $M$ in the Reedy model category $\Lambda \Ibimod_{O}$. The same arguments work for the truncated case. Note that $\TT_{r}\mathcal{I}b_{O}^\Lambda(M)_r=
\TT_{r}\mathcal{I}b_{O}^\Lambda(M)$, since arity zero non-pearl vertices are not permitted.
\end{proof}

By means of Theorem~\ref{th:BV_proj_Ibimod} and Proposition~\ref{p:BV_reedy_Ibimod}, we construct a functorial cofibrant replacement in $\Sigma\Ibimod_O$ and $\Lambda\Ibimod_O$ assuming 
that $O$ is componentwise cofibrant. We adapt notation from Subsection~\ref{ss:cof_repl_bimod}. Given an $O$-Ibimodule $M$, we define
$M'_\infty:=|S_\bullet M|\times E_\infty$, $O'_\infty:=|S_\bullet|\times E_\infty$. One has that both $M'_\infty$ and $O'_\infty$ are $\Sigma$-cofibrant
and $M'_\infty$ is an $O'_\infty$-bimodule. Let
\[
\phi\colon O'_\infty\xrightarrow{\simeq} O,\quad\text{ and }\quad \phi_0\colon M'_\infty\xrightarrow{\simeq} M
\]
be the natural projections. Note that $\phi_0$ can be viewed as a map of $O'_\infty$-Ibimodules.

\begin{pro}\label{p:cof_repl_Ibimod}
(a) Assume that  $O$ is a componentwise cofibrant operad. Let
$M$ be any $O$-Ibimodule.
Then the objects $\phi_!\left(\mathcal{I}b_{O'_\infty}(M'_\infty)\right)$ and $\phi_!\left(\TT_{r}\mathcal{I}b_{O'_\infty}(M'_\infty)_r\right)$ are cofibrant
 replacements of $M$ and $\TT_{r}M$ in the categories $\Sigma\Ibimod_{O}$ and $\TT_{r}\Sigma\Ibimod_{O}$, respectively.

(b) Assume in addition that the operad~$O$ is reduced. Then the objects $\phi_!\left(\mathcal{I}b_{O'_\infty}^\Lambda(M'_\infty)\right)$ and $\phi_!\left(\TT_{r}\mathcal{I}b_{O'_\infty}^\Lambda(M'_\infty)\right)$ are cofibrant
 replacements of $M$ and $\TT_{r}M$ in the categories $\Lambda\Ibimod_{O}$ and $\TT_{r}\Lambda\Ibimod_{O}$, respectively.
\end{pro}

\begin{proof}
The result is an immediate consequence of Theorems~\ref{G6_2}, \ref{eq:ind_rest_reduced_ibimod}, \ref{th:BV_proj_Ibimod} and Proposition~\ref{p:BV_reedy_Ibimod}.
\end{proof}

\appendix
\section{Equivariant homotopy theory}\label{s:appendix}

\subsection{Projective cofibrations for monoidal action}\label{ss:ap_monoid}

In what follows, for any topological monoid $G$, we consider the projective model structure on the category of $G$-spaces denoted $G$-$Top$. We refer the reader to Subsection \ref{MCSSpace} for more details. We start by recalling the statement of Berger-Moerdijk concerning the pushout product axiom. 

\begin{lmm}{\cite[Lemma 2.5.3]{BM2}}\label{D2}
Let $1\rightarrow G_{1}\rightarrow G \rightarrow G_{2}\rightarrow 1$ be a short exact sequence of discrete groups. Let $A\rightarrow B$ be a $G_{2}$-cofibration and $X\rightarrow Y$ be a $G$-equivariant $G_{1}$-cofibration. Then the pushout product map $(A\times Y)\cup_{A\times X}(B\times X)\rightarrow B\times Y$ is a $G$-cofibration. Moreover, the latter is acyclic if $A\rightarrow B$ or $X\rightarrow Y$ is.
\end{lmm}

We use an analogue of this result (namely, Lemmas~\ref{l:push_prod1} and~\ref{l:push_prod2} below) that can be applied to $\Sigma_{n}\wr O(1)$ and  to $( \Sigma_{k}\times \Sigma_{n-k})\wr O(1)$, which are not groups, but monoids, and also are not discrete. In fact non-discreteness is not a big problem. It is not hard to see that Berger-Moerdijk's proof 
of \cite[Lemma~2.5.3]{BM2} works for topological groups as well. On the contrary, for their proof it is critical that the action is by groups. 
In fact the statement of Lemma~\ref{D2} in general does not hold when $G$, $G_1$, and $G_2$ are monoids. Indeed,
consider $G_1=1$ and $G_2=G=\NN$ the monoid of non-negative natural numbers with the natural addition operation. Then for $A\rightarrow B$
being $S^{k-1}\times\NN\rightarrow D^k\times\NN$, and $X\rightarrow Y$ being $\NN\rightarrow\NN$, $i\mapsto i+1$,  the pushout product map  $(A\times Y)\cup_{A\times X}(B\times X)\rightarrow B\times Y$ is not an $\NN$-cofibration.

\begin{lmm}\label{l:push_prod0}
Let $\Gamma$ be a topological monoid, $A\rightarrow B$ be a $\Gamma$-cofibration, $X\rightarrow Y$ be a cofibration (in $Top$), then
the pushout-product map  $(A\times Y)\cup_{A\times X}(B\times X)\rightarrow B\times Y$ is a $\Gamma$-cofibration, where $X$ and $Y$ are regarded as spaces endowed with a trivial
 action of~$\Gamma$. Moreover, the latter is acyclic if $A\rightarrow B$ or $X\rightarrow Y$ is.
\end{lmm}

The proof of this lemma is identical to that of \cite[Lemma~2.5.2]{BM2} in which $\Gamma$ is not a topological monoid but a discrete group. For the convenience of the reader the argument is given below.

\begin{proof}
Let $Z\rightarrow W$ be a trivial $\Gamma$-fibration. One has to show that any square 
\begin{equation}\label{eq:lift1}
\xymatrix{
(A\times Y)\cup_{A\times X}(B\times X)
 \ar[r] \ar[d] & Z \ar[d]\\
B\times Y \ar[r] & W
}
\end{equation}
 has a $\Gamma$-equivariant lift. Since $Top$ is cartesian closed, the existence of such lift is equivalent to the existence of a $\Gamma$-equivariant lift
 of the square
 \begin{equation}\label{eq:lift2}
\xymatrix{
A
 \ar[r] \ar[d] & Map(Y,Z) \ar[d]\\
B \ar[r] & Map(Y,W)\times_{Map(X,W)} Map(X,Z).
}
\end{equation}
Here the $\Gamma$-action on the mapping spaces $Map(Y,Z)$, $Map(Y,W)$, etc. is defined through the action on the target: $(f\cdot\gamma)(y):=f(y)\cdot\gamma$. Since the left vertical arrow is a $\Gamma$-cofibration, it is enough to show that the right arrow is a trivial $\Gamma$-fibration.
On the other hand, since the forgetful functor $\Gamma\text{-}Top\rightarrow Top$ creates fibrations and weak equivalences, it suffices 
to verify that the right arrow is a trivial fibration in $Top$, in other words that any square
$$
\xymatrix{
S^{k-1}
 \ar[r] \ar[d] & Map(Y,Z) \ar[d]\\
D^k \ar[r] & Map(Y,W)\times_{Map(X,W)} Map(X,Z)
}
$$
 has a lift in $Top$. The latter is equivalent to the pushout-product property that the map
$(S^{k-1}\times Y)\cup_{S^{k-1}\times X}(D^k\times X)\rightarrow D^k\times Y$ is a cofibration. 

The acyclicity statement is proved similarly by starting with any not-necessarily trivial $\Gamma$-fibration $Z\to W$.
 \end{proof}

Given a homomorphism of monoids $j\colon \Gamma\to\Gamma'$, one gets an extension-restriction %\todo{extension/restriction} 
adjunction
\begin{equation}\label{eq:adj_mon}
j_!\colon\Gamma\text{-}Top\rightleftarrows\Gamma'\text{-}Top\colon j^*.
\end{equation}

\begin{lmm}\label{l:adj_mon}
For any morphism of topological monoids $j\colon \Gamma\to\Gamma'$, the extension-restriction %\todo{extension/restriction} 
adjunction~\eqref{eq:adj_mon}
is a Quillen adjunction. In particular, the extension functor $j_!$ preserves cofibrations and acyclic cofibrations. The restriction functor 
$j^*$ preserves fibrations and weak equivalences. Moreover, if $\Gamma'$ is cofibrant as a right $\Gamma$-module, then
$j^*$ also preserves cofibrations.
\end{lmm}

\begin{proof}
All the statements follow from definition except the last one, which is implied by the fact that the restriction functor preserves colimits and
sends the generating cofibrations $S^{k-1}\times\Gamma'\to D^k\times\Gamma'$ to cofibrations. The latter assertion 
is a consequence of Lemma~\ref{l:push_prod0} applied to $A\rightarrow B$ being $\emptyset\rightarrow\Gamma'$, and $X\rightarrow Y$ 
being $S^{k-1}\rightarrow D^k$.
\end{proof}

\begin{defi}\label{d:mon_ex_seq}
\begin{enumerate}[label=(\alph*)]
\item A sequence of morphisms of topological monoids 
\begin{equation}\label{eq:mon_ex_seq}
1\rightarrow\Gamma_1\xrightarrow{i} \Gamma\xrightarrow{p}\Gamma_2\to 1
\end{equation}
is called {\it short exact sequence} if  %$p\circ i$ is constant, 
$\Gamma_2$ is a quotient of $\Gamma$ as a topological space ($p$ is a quotient map), and $i$ is a homeomorphism onto $p^{-1}(1)$.
\item A short exact sequence~\eqref{eq:mon_ex_seq} of topological monoids is called {\it split-surjective}, if $p$ admits a continuous section 
$s\colon\Gamma_2\rightarrow\Gamma$ which is a morphism of monoids, and the map 
$\Gamma_2\times\Gamma_1\rightarrow\Gamma$, $(\gamma_2,\gamma_1)\mapsto s(\gamma_2)\cdot i(\gamma_1)$, is surjective.
\end{enumerate}
\end{defi}

For a split-surjective short exact sequence~\eqref{eq:mon_ex_seq}, thanks to the inclusions $i$ and $s$, the monoids $\Gamma_1$ and
$\Gamma_2$ can be viewed as subobjects of~$\Gamma$. For this reason in the sequel, we will be omitting $i$ and $s$ considering $\Gamma_1,\Gamma_2\subset\Gamma$. Note also
that if $\Gamma_2$ is a group, the map $\Gamma_2\times\Gamma_1\rightarrow\Gamma$, $(\gamma_2,\gamma_1)\mapsto \gamma_2\cdot\gamma_1$, is always surjective (and, in fact, bijective).

\begin{lmm}\label{l:push_prod1}
Let~\eqref{eq:mon_ex_seq} be a split-surjective short exact sequence of monoids. Let also $\Gamma_2$ fit into
a split-surjective short exact sequence
\begin{equation}\label{eq:mon_ex_seq2}
1\rightarrow\Gamma_2^0\rightarrow \Gamma_2\xrightarrow{p_2}G_2\to 1
\end{equation}
with $G_2$ being a group. Also assume that all elements of $\Gamma_2^0$ commute with those from $\Gamma_1$ inside $\Gamma$.    Let $A\rightarrow B$ be a $\Gamma_{2}$-cofibration and $X\rightarrow Y$ be a $\Gamma$-equivariant $\Gamma_{1}$-cofibration  with both $X$ and $Y$ having
trivial action of $\Gamma_2^0$. Then the pushout product map $(A\times Y)\cup_{A\times X}(B\times X)\rightarrow B\times Y$ is a $\Gamma$-cofibration. Moreover, the latter is acyclic if $A\rightarrow B$ or $X\rightarrow Y$ is.
\end{lmm}

\begin{expl}\label{ex:push_prod1}
%As example one can take $\Gamma=\Gamma_1\times\Gamma_2$ and $G_2=1$.
 If $A\rightarrow B$ is a $\Gamma_2$-cofibration 
and $X\rightarrow Y$ is  a $\Gamma_1$-cofibration, then the pushout product map   $(A\times Y)\cup_{A\times X}(B\times X)\rightarrow B\times Y$
is a $\Gamma_1\times\Gamma_2$-cofibration. (Take $\Gamma=\Gamma_1\times\Gamma_2$ and $G_2=1$.)
\end{expl}

\begin{proof}[Proof of Lemma~\ref{l:push_prod1}]
Let $Z\rightarrow W$ be a trivial $\Gamma$-fibration. One has to show that the square~\eqref{eq:lift1} has a $\Gamma$-equivariant lift.
By adjunction such lift defines a lift in the square~\eqref{eq:lift2}. If $\Gamma$ were a group then the mapping spaces $Map(Y,Z)$, $Map(Y,W)$, etc., 
would be endowed with a natural right $\Gamma$-action: $(f\cdot \gamma)(y):=f(y\cdot\gamma^{-1})\cdot\gamma$. The lift in question
would arise from a lift in~\eqref{eq:lift1} if and only if it were $\Gamma$-equivariant. However, in our more general situation the mapping spaces
do not get a natural $\Gamma$-action and the condition on the induced lift of~\eqref{eq:lift2} is less obvious. Denote by $Map_{\Gamma_1}(Y,Z)$ ($Map_{\Gamma_1}(Y,W)$, etc.) the subspace of $Map(Y,Z)$ of $\Gamma_1$-equivariant maps. 
%consisting of maps $f\colon Y\to Z$ satisfying $f(y\cdot\gamma_1)=f(y)\cdot\gamma_1$,
%for any $y\in Y$, $\gamma_1\in\Gamma_1$. 
Since the action of $\Gamma_1$ is trivial on $A$ and on $B$, the induced lift of~\eqref{eq:lift2}
must factor through a lift in the square
 \begin{equation}\label{eq:lift3}
\xymatrix{
A
 \ar[r] \ar[d] & Map_{\Gamma_1}(Y,Z) \ar[d]\\
B \ar[r] & Map_{\Gamma_1}(Y,W)\times_{Map_{\Gamma_1}(X,W)} Map_{\Gamma_1}(X,Y).
}
\end{equation}
On the other hand, the spaces $Map_{\Gamma_1}(Y,Z)$, $Map_{\Gamma_1}(Y,W)$, etc., have a natural $\Gamma_2$-action defined as follows:
\[
(f\cdot\gamma_2)(y):=f(y\cdot p_2(\gamma_2)^{-1})\cdot\gamma_2,
\]
where $\gamma_2\in\Gamma_2$.
One checks that if $f\in  Map_{\Gamma_1}(Y,Z)$, then so is $f\cdot\gamma_2$:
$$
\begin{array}{lll}\vspace{5pt}
(f\cdot\gamma_2)(y\cdot\gamma_1)& =f(y\cdot\gamma_1\cdot p_2(\gamma_2)^{-1})\cdot\gamma_2 & \\ \vspace{5pt}
&  =
f(y\cdot p_2(\gamma_2)^{-1}\cdot p_2(\gamma_2)\cdot\gamma_1\cdot p_2(\gamma_2)^{-1})\cdot\gamma_2 & \\ \vspace{5pt}
 & = f(y\cdot p_2(\gamma_2)^{-1})\cdot p_2(\gamma_2)\cdot\gamma_1\cdot p_2(\gamma_2)^{-1}\cdot\gamma_2 &   \\ \vspace{5pt}
 & =
f(y\cdot p_2(\gamma_2)^{-1})\cdot p_2(\gamma_2)\cdot p_2(\gamma_2)^{-1}\cdot\gamma_2\cdot\gamma_1 &  \\ 
 & =f(y\cdot p_2(\gamma_2)^{-1})\cdot\gamma_2\cdot\gamma_1 & \hspace{-80pt} =
(f\cdot\gamma_2)(y)\cdot\gamma_1.
\end{array} 
$$
The third equation is obtained using the fact that $p_2(\gamma_2)\cdot\gamma_1\cdot p_2(\gamma_2)^{-1}\in\Gamma_1$ and $f\in Map_{\Gamma_1}(Y,Z)$. The fourth equation uses that $ p_2(\gamma_2)^{-1}\cdot\gamma_2\in\Gamma_2^0$ and that $\Gamma_1$ commutes with $\Gamma_2^0$. 

We claim that a map $B\rightarrow Map(Y,Z)$  is adjoint to a $\Gamma$-equivariant map $F\colon B\times Y\rightarrow Z$ if and only if
it factors through $Map_{\Gamma_1}(Y,Z)$ and the map $B\rightarrow Map_{\Gamma_1}(Y,Z)$ is $\Gamma_2$-equivariant. (The same  holds for
the maps $A\rightarrow Map(Y,Z)$, $B\rightarrow Map(Y,W)$, etc.) 

Indeed, let $F\colon  B\times Y\rightarrow Z$ be a $\Gamma$-equivariant map. Since $\Gamma_1$ acts trivially on $B$,
\[
F(b,y)\cdot\gamma_1=F(b\cdot\gamma_1,y\cdot\gamma_1)=F(b,y\cdot\gamma_1).
\]
Thus, $F(b,-)\in Map_{\Gamma_1}(Y,Z)$. To check that the induced map $B\rightarrow Map_{\Gamma_1}(Y,Z)$ is $\Gamma_2$-equivariant,
we need to make sure that $F(b\cdot\gamma_2,y)=F(b,y\cdot p_2(\gamma_2)^{-1})\cdot\gamma_2$. One has
\[
F(b,y\cdot p_2(\gamma_2)^{-1})\cdot\gamma_2=F(b\cdot \gamma_2,y\cdot p_2(\gamma_2)^{-1}\cdot\gamma_2)=
F(b\cdot\gamma_2,y).
\]
The last equation uses the fact that $p_2(\gamma_2)^{-1}\cdot\gamma_2\in\Gamma^0_2$ and that $\Gamma_2^0$ acts trivially on $Y$.

In the other direction, let $F\colon  B\times Y\rightarrow Z$ be the adjoint of a $\Gamma_2$-equivariant map $B\rightarrow  Map_{\Gamma_1}(Y,Z)
\rightarrow Map(Y,Z)$. One has to check that $F$ is $\Gamma$-equivariant. Since the product map $\Gamma_2\times\Gamma_1\to \Gamma$ is
surjective, each $\gamma\in \Gamma$ can be written as $\gamma=\gamma_2\cdot\gamma_1$, $\gamma_1\in\Gamma_1$, $\gamma_2\in\Gamma_2$. We need to check that $F(b\cdot\gamma_2\cdot\gamma_1,y\cdot\gamma_2\cdot\gamma_1)=F(b,y)\cdot
\gamma_2\cdot\gamma_1$. One has
$$
\begin{array}{lll}\vspace{5pt}
F(b\cdot\gamma_2\cdot\gamma_1,y\cdot\gamma_2\cdot\gamma_1)  &  =F(b\cdot\gamma_2,y\cdot\gamma_2\cdot\gamma_1) & \text{since }\Gamma_1 \text{ acts trivially on } B, \\ \vspace{5pt}
 & =
F(b\cdot\gamma_2,y\cdot\gamma_2)\cdot\gamma_1 & \text{since } F(b\cdot\gamma_2,-)\in Map_{\Gamma_1}(Y,Z),\\ \vspace{5pt}
 & =F(b,y\cdot\gamma_2\cdot p_2(\gamma_2)^{-1})\cdot\gamma_2\cdot\gamma_1 & \text{since } F \text{ is } \Gamma_2\text{-equivariant},  \\ 
 & = 
F(b,y)\cdot\gamma_2\cdot\gamma_1 & \text{since } \Gamma_2^0 \text{ acts trivially on } Y.
\end{array} 
$$
%The first equation uses the fact that $\Gamma_1$ acts trivially on $B$. The second equation is because $F(b\cdot\gamma_2,-)\in Map_{\Gamma_1}(Y,Z)$. The
%third equation is because of the $\Gamma_2$-equivariance. The fourth equation is due to the trivial action of $\Gamma_2^0$ on~$Y$.

As a consequence of the above, the square~\eqref{eq:lift1} has a $\Gamma$-equivariant lift if and only if the square~\eqref{eq:lift3}
has a $\Gamma_2$-equivariant lift. On the other hand, the left vertical map in~\eqref{eq:lift3} is a $\Gamma_2$-cofibration. Therefore
a lift exists provided the right vertical map in~\eqref{eq:lift3} is a trivial $\Gamma_2$-fibration. The latter holds provided that the map
is a trivial fibration in $Top$, which follows from the fact that $(S^{k-1}\times Y)\cup_{S^{k-1}\times X}(D^k\times X)\rightarrow D^k\times Y$ is a $\Gamma_1$-cofibration, see Lemma~\ref{l:push_prod0}. 
\end{proof}

\begin{lmm}\label{l:left_mod}
Let $\Gamma$ be a topological monoid and $Z$ be a $\Gamma$-cofibrant space.
Then the functor
\[
Z\times_\Gamma -\colon \Gamma^{op}\text{-}Top\rightarrow Top
\]
from left $\Gamma$-modules to spaces, sends left $\Gamma$-spaces that are cofibrant in $Top$ to cofibrant spaces, and preserves the weak equivalences between such objects.
\end{lmm}

%\todo{Find a reference or give a proof.}
\begin{proof}
Since a retract of a cofibrant space is a cofibrant space and retract of a weak equivalence is a weak equivalence, we can assume that $Z$ is $\Gamma$-cellular:
$Z=\mathrm{colim}_{\alpha<\lambda} Z_\alpha$, where each map $Z_{<\alpha}:=\mathrm{colim}_{\beta<\alpha} Z_\beta\to Z_\alpha$ is a pushout of a generating $\Gamma$-cofibration. One has to check two statements.
\begin{itemize}
\item If the statement of the lemma holds for $Z_{<\alpha}$, then it does for $Z_\alpha$.

\item If the statement of the lemma holds for each $Z_\beta$, $\beta<\alpha$, then it does for $Z_{<\alpha}$.
\end{itemize}

For any generating $\Gamma$-cofibration $S^{k-1}\times \Gamma\to D^k\times\Gamma$ and any cofibrant in $Top$
left $\Gamma$-space $A$, the induced map 
\begin{equation}\label{eq:Gamma_A}
(S^{k-1}\times\Gamma)\times_\Gamma A\to (D^k \times\Gamma)\times_\Gamma A
\end{equation}
is $S^{k-1}\times A\to D^k\times A$. Since $A$ is cofibrant, this map~\eqref{eq:Gamma_A} is a cofibration. On the other hand,
for any weak equivalence of $\Gamma$-spaces $A\to B$, the induced maps $S^{k-1}\times A\to S^{k-1}\times B$ and 
$D^k\times A\to D^k\times B$ are weak equivalences. The first statement above follows by applying
Proposition~\ref{D0} and recalling that $Top$ is left proper.

To prove the second statement, we notice that the sequence $Z_\alpha\times_\Gamma A$, $\alpha<\lambda$, is a 
sequence of cofibrations, being pushouts of cofibrations of the form~\eqref{eq:Gamma_A}. Therefore,
\[
\mathrm{colim}_{\alpha<\lambda} (Z_\alpha\times_\Gamma A)\simeq \mathrm{hocolim}_{\alpha<\lambda} (Z_\alpha\times_\Gamma A).
\]
One has the same weak equivalence of spaces for $B$. Since $Z_\alpha\times_\Gamma A \to Z_\alpha\times_\Gamma B$,
$\alpha<\lambda$, are all weak equivalences,
\[
\mathrm{hocolim}_{\alpha<\lambda} (Z_\alpha\times_\Gamma A)\simeq \mathrm{hocolim}_{\alpha<\lambda} (Z_\alpha\times_\Gamma B).
\]
On the other hand, the functors $(-)\times_\Gamma A$  and $(-)\times_\Gamma B$ preserve colimits.
We conclude
\[
Z\times_\Gamma A=\mathrm{colim}_{\alpha<\lambda} (Z_\alpha\times_\Gamma A)\simeq
\mathrm{colim}_{\alpha<\lambda} (Z_\alpha\times_\Gamma B) = Z\times_\Gamma B.
\]
\end{proof}

\vspace{5pt}
\noindent {\bf Compatible action of  a monoid and a group:}
We say that a topological monoid $\Gamma$ is endowed with a right action of a group $K$ if one is given a map $\Gamma\times K\to \Gamma$, $(\gamma,k)\mapsto \gamma^k$,
which is a right $K$-action on the set $\Gamma$ and for every $k\in K$, the map $(-)^k\colon\Gamma\to\Gamma$ is a monoid homomorphism. Given such an action, one
defines the semi-direct product monoid $\Gamma\rtimes K$. Its underlying set is $\Gamma\times K$, while multiplication is as follows
\[
(\gamma_1,k_1)\cdot (\gamma_2,k_2)=(\gamma_1\cdot\gamma_2^{k_1^{-1}},k_1\cdot k_2).
\]
A right $\Gamma\rtimes K$-space $A$ can equivalently be seen as a right $\Gamma$-module with a right $K$-action compatible in the sense
\[
(a\cdot\gamma)\cdot k=(a\cdot k)\cdot \gamma^k,\quad\text{for any $a\in A$,  $\gamma\in \Gamma$, $k\in K$.}
\]

Note that the same map $\Gamma\times K\to \Gamma$, defines a   right $K$-action on $\Gamma^{op}$ -- the monoid with the reversed multiplication. A right $\Gamma^{op}\rtimes K$-space $X$ can
equivalently be seen as a left $\Gamma$-module with a right $K$-action compatible in the sense
\[
(\gamma\cdot x)\cdot k =\gamma^k\cdot(x\cdot k), \quad\text{for any $x\in X$, $\gamma\in \Gamma$, $k\in K$.}
\]

\begin{lmm}\label{l:push_prod2}
Let $1\to K_1\to K\to K_2\to 1$ be a short exact sequence of topological groups. Let $\Gamma$ be a topological monoid endowed with a right $K_2$-action. 
If $A\to B$ is a $\Gamma\rtimes K_2$-cofibration and $X\to Y$ is a $\Gamma^{op}\rtimes K$-equivariant $K_1$-cofibration, then
the pushout-product map  $(A\times_\Gamma Y)\cup_{A\times_\Gamma X}(B\times_\Gamma X)\rightarrow B\times_\Gamma Y$ is a $K$-cofibration. Moreover, the latter is acyclic if $A\rightarrow B$ or $X\rightarrow Y$ is.
\end{lmm}

\begin{proof}
Let $Z\rightarrow W$ be a trivial $K$-fibration. One has to show that any square 
\begin{equation}\label{eq:lift4}
\xymatrix{
(A\times_\Gamma Y)\cup_{A\times_\Gamma X}(B\times_\Gamma X)
 \ar[r] \ar[d] & Z \ar[d]\\
B\times_\Gamma Y \ar[r] & W
}
\end{equation}
 has a $K$-equivariant lift. One has a homeomorphism of mapping spaces.%\footnote{We implicitly use~\cite[Lemma~A.1]\ref{DT} that the natural inclusion $Map(X/\sim,Y)\subset Map(X,Y)$
 %is always a homemorphism on its image in $Top$, which is not true in the category of all topological spaces.}
 \[
 Map_K(B\times_\Gamma Y,Z)=Map_{\Gamma\rtimes K_2}\left(B,Map_{K_1}(Y,Z)\right).
 \]
 (One has similar homeomorphisms for $Map_K(A\times_\Gamma Y,Z)$, etc.)  The $K$-action on $B\times_\Gamma Y$ is the diagonal one: $(b,y)\cdot k=(b\cdot k,y\cdot k)$.
  The action of $\Gamma\rtimes K_2$ on $Map_{K_1}(Y,W)$ is defined as follows:
 \[
 (f\cdot(\gamma,k_2))(y)=f(\gamma\cdot(y\cdot k^{-1}))\cdot k,
 \]
 where $k\in K$ is any point in the preimage of $k_2\in K_2$.

Using this, the existence of a lift in~\eqref{eq:lift4} is equivalent to the existence of a $\Gamma\rtimes K_2$-equivariant lift
 of the square
 \begin{equation}\label{eq:lift5}
\xymatrix{
A
 \ar[r] \ar[d] & Map_{K_1}(Y,Z) \ar[d]\\
B \ar[r] & Map_{K_1}(Y,W)\times_{Map_{K_1}(X,W)} Map_{K_1}(X,Z).
}
\end{equation}
Since $A\to B$ is a $\Gamma\rtimes K_2$-cofibration, one has to check that the right vertical arrow is a trivial $\Gamma\rtimes K_2$-fibration, or, equivalently a fibration in $Top$. 
%Here the $\Gamma$-action on the mapping spaces $Map(Y,Z)$, $Map(Y,W)$, etc. is defined through the action on the target: $(f\cdot\gamma)(y):=f(y)\cdot\gamma$. Since the left vertical arrow is a $\Gamma$-cofibration, it is enough to show that the right arrow is a trivial $\Gamma$-fibration.
%On the other hand, since the forgetful functor $\Gamma\text{-}Top\rightarrow Top$ creates fibrations and weak equivalences, it suffices 
%to verify that the right arrow is a trivial fibration in $Top$, in other words that any square
In other words, one has to check that any square below
$$
\xymatrix{
S^{k-1}
 \ar[r] \ar[d] & Map_{K_1}(Y,Z) \ar[d]\\
D^k \ar[r] & Map_{K_1}(Y,W)\times_{Map_{K_1}(X,W)} Map_{K_1}(X,Z)
}
$$
 has a lift in $Top$. The latter is equivalent to the pushout-product property that the map
$(S^{k-1}\times Y)\cup_{S^{k-1}\times X}(D^k\times X)\rightarrow D^k\times Y$ is a $K_1$-cofibration, true by Lemmas~\ref{D2} or~\ref{l:push_prod0}.

%The acyclicity statement is proved similarly by starting with any not-necessarily trivial $\Gamma$-fibration $Z\to W$.

\end{proof}

\subsection{Cellularly equivariant cofibrations}\label{ss:ap_cell}
\begin{defi}\label{d:ap_cell}
Let $G$ be a discrete group.
\begin{enumerate}[label=(\alph*)]
\item A $G$-equivariant map $X_0\rightarrow X_1$ is called a {\it $G$-equivariant cell attachment} if it
fits into a pushout diagram
\begin{equation}\label{eq:equiv_cell}
\xymatrix{
S^{k-1}\times(H\backslash G)
 \ar[r] \ar[d] & D^k\times(H\backslash G) \ar[d]\\
X_0\ar[r] & X_1,
}
\end{equation}
where $H\subset G$ is a subgroup of $G$.
\item A $G$-equivariant map $X_0\rightarrow X$ is called a {\it cellularly $G$-equivariant cofibration} if it is a $G$-equivariant retract of a possibly transfinite sequence
of $G$-equivariant cell attachments.
\end{enumerate}
\end{defi}

In fact there exists a model structure on $G\text{-}Top$ for which cofibrations are exactly the cellularly $G$-equivariant cofibrations~\cite{Farjoun}. This model structure produces a different (from projective)
homotopy category as it has a smaller class of equivalences for which one has to take into account all orbit subspaces. Any $G$-space is still fibrant in this model structure.
We do not use this more subtle model structure on $G\text{-}Top$. We just need a few technical lemmas below.

\begin{lmm}\label{l:app_restr}
For any cellularly $G$-equivariant cofibration $X_0\rightarrow X$ and any $G$-space $Y$, the induced map
\[
Map_G(X,Y)\rightarrow Map_G(X_0,Y)
\]
is a Serre fibration.
\end{lmm}

\begin{proof}
A retract of a Serre fibration is a Serre fibration as well as is the limit of a tower of Serre fibrations. Therefore it is enough to consider the case of a $G$-equivariant cell attachment $X_0\rightarrow X_1$
as in~\eqref{eq:equiv_cell}. One has a pullback square
\[
\xymatrix{
 Map_G(X_1,Y)   
 \ar[r] \ar[d] & Map_G(D^k\times(H\backslash G),Y) \ar[d]\\
Map_G(X_0,Y)\ar[r] & Map_G(S^{k-1}\times(H\backslash G),Y).
}
\]
The left vertical map is a Serre fibration provided the right vertical map is one. One has
\begin{align*}
Map_G(D^k\times(H\backslash G),Y)&= Map_{H}(D^{k},Y) = Map(D^{k},Y^{H});\\
 Map_G(S^{k-1}\times(H\backslash G),Y)&= Map_{H}(S^{k-1},Y) =  Map(S^{k-1},Y^{H}).
 \end{align*}
Therefore, the right vertical map in the square above is the map $Map(D^{k},Y^{H}) \to Map(S^{k-1},Y^{H})$,
which is a Serre fibration.  This follows from the fact that the invariant space $Y^{H}$ is fibrant (like any topological space) and the map from $S^{k-1}$ to $D^{k}$ is a cofibration.

%One has to show that any square
%\[
%\xymatrix{
% D^j\times\{0\}   
% \ar[r] \ar[d] & Map_G(D^k\times(H\backslash G),Y) \ar[d] & \hspace{-35pt}\cong Map_{H}(D^{k},Y) \cong  Map(D^{k},Y^{H})\\
%D^j\times [0,1]\ar[r] & Map_G(S^{k-1}\times(H\backslash G),Y)& \hspace{-25pt} \cong Map_{H}(S^{k-1},Y) \cong  Map(S^{k-1},Y^{H})
%}
%\]
%has a lift. This follows from the fact that the invariant space $Y^{H}$ is fibrant (like any topological space) and the map from $S^{k-1}$ to $D^{k}$ is a cofibration.
\end{proof}

Lemmas~\ref{l:app_ss}-\ref{l:app_sigma} below help to recognize cellularly equivariant cofibrations.

\begin{lmm}\label{l:app_ss}
The realization of any $G$-equivariant inclusion of simplicial $G$-sets is a cellularly $G$-equivariant cofibration.
\end{lmm}

\begin{proof}
Obvious.
\end{proof}

\begin{lmm}\label{l:app_ind_restr}
For any homomorphism $\phi:G_1\to G_2$ of discrete groups, both the restriction and extension functors
\[
\phi_!\colon G_1\text{-}Top\rightleftarrows G_2\text{-}Top\colon \phi^*
\]
preserve cellularly equivariant cofibrations.
\end{lmm}
\begin{proof}
Obvious.
\end{proof}

\begin{lmm}\label{l:app_push_prod}
If
\begin{equation}\label{eq:app_XandY}
X_0\rightarrow X \quad \text{and}\quad Y_0\rightarrow Y
\end{equation}
are cellularly $G$-equivariant cofibrations, then so is the pushout-product map
\begin{equation}\label{eq:app_push_prod}
(X\times Y_0)\cup_{X_0\times Y_0}(X_0\times Y)\rightarrow X\times Y.
\end{equation}
\end{lmm}

\begin{proof}
For generating $G$-equivariant cell attachments $S^{k_1-1}\times (H_1\backslash G)\rightarrow D^{k_1}\times (H_1\backslash G)$ and 
$S^{k_2-1}\times (H_2\backslash G)\rightarrow D^{k_2}\times (H_2\backslash G)$, the map~\eqref{eq:app_push_prod} becomes
\[
S^{k_1+k_2-1}\times (H_1\backslash G)\times  (H_2\backslash G)\rightarrow D^{k_1+k_2}\times (H_1\backslash G)\times (H_2\backslash G).
\]

The $G$-set $(H_1\backslash G)\times  (H_2\backslash G)$ is isomorphic to a disjoint union of identical $G$-sets $(H_1\cap H_2)\backslash G$. Thus, the statement
of the lemma holds in this case. 
Similarly it is true for an arbitrary pair of $G$-equivariant cell attachments. 

In case $X=\mathrm{colim}_{\alpha<\lambda_1}X_\alpha$ and $Y=\mathrm{colim}_{\alpha<\lambda_2} Y_\alpha$ are
(possibly transfinite) sequences of $G$-equivariant cell attachments, then the inclusion~\eqref{eq:app_push_prod} is also a (transfinite) sequence 
of $G$-equivariant cell attachments $$\underset{(\alpha_1,\alpha_2)<(\lambda_1,\lambda_2)}{\mathrm{colim}} Z_{\alpha_1,\alpha_2},$$ where the set 
$\lambda_1\times\lambda_2$ is given the lexicographical order and therefore is also an ordinal. The spaces $Z_{\alpha_1,\alpha_2}$ 
are defined recursively:
\[
Z_{\alpha_1,\alpha_2}=\left(\underset{(\beta_1,\beta_2)<(\alpha_1,\alpha_2)}{\mathrm{colim}} Z_{\beta_1,\beta_2}\right)
\cup (X_{\alpha_1}\times Y_{\alpha_2}),
\]
with $Z_{0,0}$ being the left-hand side of~\eqref{eq:app_push_prod}.
%$Z_{\alpha_1,\alpha_2+1}=Z_{\alpha_1,\alpha_2}\cup (X_{\alpha_1}\times Y_{\alpha_2+1})$ and $Z_{\alpha_1+1,0}=(X_{\alpha_1}\times Y)\cup (X\times Y_0)=
%\mathrm{colim}_{\alpha_2<\lambda_2} Z_{\alpha_1,\alpha_2}$, in case $\alpha_2$, respectively, $\alpha_1$, has a predecessor,
%otherwise one takes the colimit over smaller elements in $\lambda_1\times\lambda_2$.

Finally, if \eqref{eq:app_XandY} are retracts of (transfinite) sequences of cell attachments then so is the pushout-product.
\end{proof}

\begin{lmm}\label{l:app_sigma}
Let $\partial X\rightarrow X$ be a cofibration in $Top$. Then $\partial (X^{\times n})\rightarrow X^{\times n}$ is a cellularly $\Sigma_n$-equivariant cofibration.
\end{lmm}

\begin{proof}
One needs to check it first for the inclusion $S^{k-1}\to D^k$, which is done by stratifying $(D^k)^{\times n}$ into $\Sigma_n$-orbits 
and then decomposing the orbits into cells. Applying previous Lemmas~\ref{l:app_ind_restr} and~\ref{l:app_push_prod} we can conclude
that the inclusion
\[
\partial \prod_{i=1}^\ell(D^{k_i})^{\times n_i}\rightarrow \prod_{i=1}^\ell(D^{k_i})^{\times n_i}
\]
is a cellularly $(\prod_{i=1}^\ell\Sigma_{n_i})$-equivariant cofibration. The rest of the argument is similar to the proof of the previous lemma. Assuming that $\partial X\to X$ is a possibly transfinite sequence of cell attachments $\mathrm{colim}_{\alpha<\lambda}X_\alpha$, we  decompose $X_{\alpha}^{\times n}$ extending the cellular structure of $\mathrm{colim}_{\beta<\alpha}X_{\beta}^{\times n}$.
\end{proof}

%%%%%%%%%%%%%%%%%%%%%%%%%%%%%%%%%%%%%%%%%%%%%%%%%%%%%%%%%%%%%%%%%%%
%%%%%%%%%%%%%%%%%%%%%%%%%%%%%%%%%%%%%%%%%%%%%%%%%%%%%%%%%%%%%%%%%%%
%%%%%%%%%%%%%%%%%%%%%%%%%%%%%%%%%%%%%%%%%%%%%%%%%%%%%%%%%%%%%%%%%%%
%%%%%%%%%%%%%%%%%%%%%%%%%%%%%%%%%%%%%%%%%%%%%%%%%%%%%%%%%%%%%%%%%%%
%%%%%%%%%%%%%%%%%%%%%%%%%%%%%%%%%%%%%%%%%%%%%%%%%%%%%%%%%%%%%%%%%%%

\textbf{Acknowledgment:} The authors would like to thank C.~Berger, P.~Bressie, R.~Campos, M.~Ching, D.~E.~Farjoun, N.~Idrissi, M.~Weber, and T.~Willwacher for communication.

\end{document}